\documentclass[11pt,a4paper,reqno]{amsart}
\usepackage[margin=25mm]{geometry}
\usepackage{amsmath,amssymb,amsthm,mathtools,mathrsfs,longtable}
\usepackage[T1]{fontenc}
\usepackage{lmodern,microtype,needspace,enumitem}
\AtBeginDocument{%
  \begingroup
  \normalfont\footnotesize\fontencoding{OT1}\selectfont
  \endgroup
}
\usepackage[colorlinks=true,linkcolor=blue,citecolor=blue,urlcolor=blue]{hyperref}
\setlist[enumerate]{leftmargin=2.2em,itemsep=3pt,topsep=4pt}
\allowdisplaybreaks[2]
\numberwithin{equation}{section}
\newtheorem{theorem}{Theorem}[section]
\newtheorem{lemma}[theorem]{Lemma}
\newtheorem{proposition}[theorem]{Proposition}
\newtheorem{corollary}[theorem]{Corollary}
\newtheorem{definition}[theorem]{Definition}
\newtheorem{problem}{Problem}[section]
\theoremstyle{definition}
\newtheorem{example}[theorem]{Example}
\theoremstyle{remark}
\newtheorem{remark}[theorem]{Remark}
\newcommand{\C}{\mathbb C}
\newcommand{\R}{\mathbb R}
\newcommand{\Z}{\mathbb Z}
\newcommand{\N}{\mathbb N}

\newcommand{\cF}{\mathcal F_t^2}
\newcommand{\dbar}{\overline{\partial}}
\newcommand{\norm}[1]{\|#1\|}
\DeclareMathOperator{\Hess}{Hess}
\DeclareMathOperator{\supp}{supp}

\DeclareMathOperator{\Tr}{Tr}
\DeclareMathOperator{\rank}{rank}
\title[Schatten limits for semi-commutators in Toeplitz quantization]
{Schatten limits for semi-commutators in Toeplitz quantization on Fock space}
\author{Jie Qin}
\address{School of Mathematics and Statistics,
Chongqing Technology and Business University,
Chongqing 400067, China}
\email{qinjie24520@163.com}
\date{}
\begin{document}
\begin{abstract}
We characterize, for every $0<p<\infty$, the bounded symbols $f$
such that $T_f^{(t)}T_g^{(t)}-T_{fg}^{(t)}$ tends to zero in the
unnormalized Schatten $p$-norm or quasi-norm as $t\to0^{+}$ for every
fixed bounded symbol $g$, solving a problem posed by Hu and Virtanen in [Anal. PDE, IDA and Hankel operators on Fock spaces, 16 (2023), 2041-2077].
We work on weighted Fock spaces over $\mathbb C^n$ whose defining
weights have uniformly positive and bounded real Hessians.
Only constants satisfy the limit when $p<2n$, while the critical
class at $p=2n$ consists exactly of bounded symbols with weak real gradient
in $L^{2n}(\mathbb R^{2n})$.
When $p>2n$, the limit is characterized by relative weak compactness
in $L^1(\mathbb R^{2n})$ of normalized local mean oscillation
densities. Uniform convergence over $\|g\|_\infty\le1$ is characterized,
for $p\ge2$, by vanishing total mass of the normalized densities.
For $p>2n$, a continuous compactly supported counterexample shows
that uniformly bounded density mass and convergence for
$g=\overline f$ do not imply convergence for every fixed bounded
second symbol.
The symbol classes form weight-independent unital algebras, with
further consequences for Toeplitz products, commutators, and
finite-rank approximation of semi-commutators.
The Bargmann transform transfers the Gaussian case of this
classification to anti-Wick quantization by unitary equivalence.
For the same convergence problem on compact K\"ahler manifolds
with positive line bundles, we establish a H\"older sufficient
condition, including the endpoint, for each fixed bounded
second symbol.
An example on the projective line shows that trace norm
convergence can fail even for a smooth first symbol and a fixed
bounded second symbol.
The main proofs combine local holomorphic approximation and Schatten estimates
with Sobolev approximation, singular integral methods, and the
Dunford--Pettis theorem to connect the operator limits with weak
compactness of the normalized densities.
\end{abstract}
\thanks{The author was supported by the National Natural Science Foundation of China (12501154).}
\maketitle

\section{Introduction and main results}\label{sec:intro}
\subsection{Setting and problem}
Let $n\geq1$ and put $d=2n$. We identify $\C^n$ with
$\R^d$ by $z_j=x_{2j-1}+ix_{2j}$, $1\leq j\leq n$, and
write $dv(z)=dx_1\cdots dx_d$ for the Lebesgue measure.
Fix a real-valued function
$\varphi\in C^2(\C^n)$ such that
\begin{equation}\label{eq:weight}
 mI_d\leq\Hess_{\R}\varphi(z)
 =\left(
   \frac{\partial^2\varphi}{\partial x_\ell\partial x_k}(z)
  \right)_{\ell,k=1}^d
 \leq MI_d,\quad z\in\C^n,
\end{equation}
where $0<m\leq M<\infty$ and $I_d$ is the $d\times d$
identity matrix. The matrix inequalities mean that
$$
 m|v|^2
 \leq\sum_{\ell,k=1}^d
 \frac{\partial^2\varphi}{\partial x_\ell\partial x_k}(z)
 v_\ell v_k
 \leq M|v|^2
$$
for every $z\in\C^n$ and $v=(v_1,\ldots,v_d)\in\R^d$.
Following the scaling in
\cite[Section 6]{HuVirtanen2023}, we define
\begin{equation}\label{E82}
 d\mu_t(z)=t^{-n}e^{-2\varphi(z/\sqrt t)} dv(z),\quad
 L_t^2=L^2(\C^n,d\mu_t),\quad
 \cF=L_t^2\cap\mathcal O(\C^n)
\end{equation}
for $t>0$. Here $\mathcal O(\Omega)$ denotes
the holomorphic functions on a domain $\Omega\subseteq\C^n$. On $L_t^2$, we define
the inner product and its induced norm by
$$
 \langle u,h\rangle_t=\int_{\C^n}u\overline{h}d\mu_t,
 \quad \norm{u}_{L_t^2}=\langle u,u\rangle_t^{1/2},
 \quad u,h\in L_t^2.
$$
The Fock space $\cF$ is a Hilbert space with the inner product
inherited from $L_t^2$.

Let $P_t:L_t^2\to\cF$ be the orthogonal projection.
For $f\in L^\infty(\C^n)$, let $M_f:L_t^2\to L_t^2$
be given by $M_fu=fu$ for $u\in L_t^2$.
 The Toeplitz and Hankel operators are
\begin{equation}\label{E1}
 T_f^{(t)}=P_tM_f:\cF\to\cF,\quad
 H_f^{(t)}=(I-P_t)M_f:\cF\to L_t^2\ominus\cF.
\end{equation}
For $f,g\in L^\infty(\C^n)$, we consider the semi-commutator
\begin{equation}\label{E2}
 D_t(f,g)=T_f^{(t)}T_g^{(t)}-T_{fg}^{(t)}:\cF\to\cF.
\end{equation}
The adjoint
$(H_f^{(t)})^*:L_t^2\ominus\cF\to\cF$ satisfies
$(H_f^{(t)})^*u=P_t(\overline{f} u)$ for $u\in L_t^2\ominus\cF$.
Since $H_g^{(t)}$ takes values in this orthogonal complement,
\eqref{E2} yields
$$
 D_t(f,g)=-(H_{\overline{f}}^{(t)})^*H_g^{(t)}.
$$

For a bounded operator $A$, $\norm{A}$ denotes its operator
norm. We include space subscripts on operator norms only when
needed. The parameter $t>0$ is the semiclassical parameter.
For the standard Gaussian weights, it plays the role of
Planck's constant in Toeplitz quantization
\cite[Section 1]{BauerCoburnHagger2018}.
The limit $\norm{D_t(f,g)}\to0$ as $t\to0^{+}$ expresses
asymptotic preservation of multiplication, one of the basic
requirements in quantization. Berezin developed the general
framework in \cite{Berezin1974,Berezin1975}, and Bordemann,
Meinrenken, and Schlichenmaier established operator norm
estimates for the classical limit on compact quantizable
K\"ahler manifolds \cite{BordemannMeinrenkenSchlichenmaier1994}.
For reproducing kernels and composition expansions in this
setting and on complex domains, see
\cite{Englis1996,Schlichenmaier2010,MaMarinescu2012}.

On Fock space, the operator calculus of Berger and Coburn
\cite{BergerCoburn1986} and the deformation estimates of
Coburn \cite{Coburn1992} connect Toeplitz products with
quantization. Composition formulas for suitable classes of
possibly unbounded symbols are studied in \cite{Bauer2009},
while Hagger obtains product expansions for symbols with
finitely many bounded uniformly continuous derivatives
\cite{Hagger2019}. The zeroth order limit also admits less
regular symbols. Bauer and Coburn proved
$\norm{D_t(f,g)}\to0$ for bounded uniformly continuous symbols
\cite{BauerCoburn2016}. Bauer, Coburn, and Hagger extended
this limit to a bounded first symbol whose mean oscillation
vanishes at small radii, uniformly in the center, and an
arbitrary bounded second symbol. They also characterized
this symbol class by requiring both orders of the
semi-commutator to vanish for every bounded test
\cite{BauerCoburnHagger2018}. For the weights and scaling
used here, Hu and Virtanen characterized the first symbols
for which the limit holds for every bounded $g$ in terms
of the local analytic distance of $\overline{f}$
\cite[Theorem 6.4]{HuVirtanen2023}.

We study the corresponding limit in Schatten classes.
For $0<p<\infty$, $\mathcal S_p(\cF)$ consists of compact operators on $\cF$
whose singular values have summable $p$-th powers.
We write $\norm{\cdot}_{\mathcal S_p(\cF)}$ for the associated norm
when $p\geq1$ and quasi-norm when $0<p<1$, and assign it the
value $+\infty$ outside $\mathcal S_p(\cF)$
\cite[Section 1.4]{Zhu1990}. For brevity, we write
$\norm{\cdot}_{\mathcal S_p}$ for this norm or quasi-norm on
$\cF$. The Schatten norm or quasi-norm depends on all singular
values, so convergence in operator norm alone does not imply
convergence in a Schatten norm or quasi-norm. On weighted Bergman spaces of the unit ball,
Tang, Wang, and Zheng estimate remainders in product
expansions under differentiability assumptions on the symbols
\cite{TangWangZheng2023}.
Here we consider the unnormalized Schatten limit for each
fixed bounded second symbol, allowing the first symbol to
have no prescribed differentiability. Hu and Virtanen posed
the following problem for Fock spaces.

\begin{problem}[{Hu and Virtanen, \cite[p. 2075]{HuVirtanen2023}}]
\label{problem:hu-virtanen}
Characterize the functions $f\in L^\infty(\C^n)$ for which
$$
 \lim_{t\rightarrow0^{+}}\norm{D_t(f,g)}_{\mathcal S_2}=0
$$
for every $g\in L^\infty(\C^n)$.
More generally, consider the same problem with
$\mathcal S_2$ replaced by $\mathcal S_p$, where $0<p<\infty$.
\end{problem}

For each $0<p<\infty$, set
\begin{equation}\label{eq:property}
 \mathcal A_p=\left\{f\in L^\infty(\C^n):
 \lim_{t\to0^{+}}\norm{D_t(f,g)}_{\mathcal S_p}=0
 \text{ for every fixed }g\in L^\infty(\C^n)\right\}.
\end{equation} 
The factorization \eqref{eq:factorization} relates this problem
to Schatten estimates for Hankel operators. On a fixed space,
we omit the scale superscript and write $H_f$.
For weighted Fock spaces over $\C^n$, including those whose
weights satisfy \eqref{eq:weight}, Hu and Virtanen
characterized membership of a single Hankel operator in
$\mathcal S_q$, $0<q<\infty$, by integral distance to
holomorphic functions \cite[Theorem 1.1]{HuVirtanen2022}.
On the standard Fock space, Xia and Zheng obtained a
simultaneous criterion for $H_f,H_{\overline{f}}\in\mathcal S_q$,
$1\leq q<\infty$, using Gaussian standard deviation
\cite[Theorem 1.1]{XiaZheng2004}. For bounded symbols on
this space, Hilbert--Schmidt membership is preserved by conjugation
\cite[Theorem 2.4]{Bauer2004}. Compactness and extensions
to other operator ideals are treated in
\cite{HaggerVirtanen2021,Xia2023}.
Asghari, Hu, and Virtanen obtained Schatten criteria on
Fock spaces over $\C$ with doubling Laplacian measure
\cite{AsghariHuVirtanen2024}.

Related Hankel criteria on Bergman spaces describe
boundedness, compactness, and Schatten membership for
general symbols on the unit disk \cite{Luecking1992}, and
simultaneous membership of $H_f,H_{\overline{f}}$ in
$\mathcal S_q$, $0<q<\infty$, on standard weighted spaces
of the unit ball. The latter is characterized by local
mean oscillation with respect to invariant volume for
symbols in the corresponding weighted $L^2$ space
\cite[Theorem 1]{Pau2016}.
Other fixed-space results concern antiholomorphic symbols
on radial weighted Fock spaces
\cite{BommierHatoYoussfi2007,SeipYoussfi2013} and singular-value
estimates for anti-analytic Bergman symbols \cite{BourassEtAl2022}
or general symbols on weighted planar Bergman spaces,
including Fock spaces on $\C$ \cite{FanWangZeng2026}.
For background on mean oscillation, Fock spaces, and Hankel
forms, see
\cite{Bauer2005,PeralaSchusterVirtanen2014,Zhu2012,JansonPeetreRochberg1987}.
These results concern individual Hankel operators or a
conjugate pair on a fixed space. Membership in $\mathcal A_p$
requires a scale-dependent semi-commutator to converge for
every fixed bounded second symbol, and the distinction
between this limit and convergence uniform in the second
symbol is part of the problem.

\subsection{Main results}
To state the role of differentiability, write $\nabla f$ for
the weak gradient in the real coordinates, with the norm convention
$$\norm{\nabla f}_{L^{q}(\R^d,dv)}=\norm{|\nabla f|}_{L^{q}(\R^d,dv)}.$$
Our first result describes how the limit depends on $p$
relative to the real dimension $d$.

\begin{theorem}\label{thm:main}
Suppose \eqref{eq:weight} holds, and put $d=2n$.
\begin{enumerate}[label=\textup{(\alph*)}]
\item If $0<p<d$, then
$$
 \mathcal A_p=\{f\in L^\infty(\C^n):f\text{ is constant almost everywhere}\}.
$$
\item At $p=d$, the limit class is given by
$$
 \mathcal A_d=\{f\in L^\infty(\C^n):|\nabla f|\in L^d(\R^d,dv)\}.
$$
\item If $p>d$, $f\in L^\infty(\C^n)$, and $|\nabla f|\in L^p(\R^d,dv)$, then
\begin{equation}\label{eq:super-gradient}
 \sup_{\norm{g}_{L^\infty(\C^n)}\leq1}\norm{D_t(f,g)}_{\mathcal S_p}
 \leq C t^{1/2-n/p}\norm{\nabla f}_{L^{p}(\R^d,dv)}\rightarrow0\quad(t\rightarrow0^{+}).
\end{equation}
Here $C>0$ depends only on $n,p$, and the weight bounds $m,M$ in \eqref{eq:weight}.
In addition, $\mathcal A_d\subseteq\mathcal A_p$ for every $p>d$.
\end{enumerate}
\end{theorem}
The factor $t^{1/2-n/p}$ in \eqref{eq:super-gradient}
records the dependence of the gradient estimate on the
semiclassical scale and dimension. It tends to zero when
$p>d$. At $p=d$, the gradient bound in
Lemma \ref{lem:gradient} (b) is independent of the scale.
Theorem \ref{thm:main} (b) identifies
the critical limit class with the bounded functions whose
weak gradient belongs to $L^d(\R^d,dv)$, without requiring
$f$ itself to belong to $L^d(\R^d,dv)$.
For the original Hilbert--Schmidt problem, this condition
is $|\nabla f|\in L^2(\R^2,dv)$ when $n=1$, whereas only
constant symbols satisfy the limit when $n\geq2$.

For $p>d$, part \textup{(c)} gives a sufficient gradient condition.
To complete the characterization of $\mathcal A_p$ without
assuming weak differentiability, we use local analytic distance
and ordinary mean oscillation. For $z\in\C^n$ and $r>0$, let $B(z,r)$
denote the open Euclidean ball centered at $z$ with radius $r$,
and let $|B(z,r)|$ be its Lebesgue measure.
For $1\leq q<\infty$
and $u\in L^q_{\rm loc}(\C^n,dv)$, define the local mean by
\begin{equation}\label{E4}
 M_{q,r}(u)(z)=
 \bigg(\frac1{|B(z,r)|}\int_{B(z,r)}|u(w)|^q dv(w)\bigg)^{1/q}.
\end{equation}
The corresponding local analytic distance is given by
\begin{equation}\label{E106}
 G_{q,r}(u)(z)=\inf_{h\in\mathcal O(B(z,r))}M_{q,r}(u-h)(z).
\end{equation}
These are the local means and integral distances used in
\cite[Section 3A]{HuVirtanen2023}.
The first index $q$ in $M_{q,r}$ and $G_{q,r}$ is the exponent
used in the local average, which we call the local exponent.
For $a>0$, we write
$G_{q,r}^{a}(u)(z)=(G_{q,r}(u)(z))^{a}$.
We use the same convention for powers of other functions
and functionals.
Writing
$$
 f_{B(z,r)}=\frac1{|B(z,r)|}\int_{B(z,r)}f(w) dv(w),
$$
we define the ordinary mean oscillation by
\begin{equation}\label{E6}
 \operatorname{MO}_{1,r}(f)(z)=
 \frac1{|B(z,r)|}\int_{B(z,r)}|f(w)-f_{B(z,r)}| dv(w).
\end{equation}
When $t$ and a spatial scale $r$ occur together, we take
$t=r^2$ unless otherwise stated. For a fixed $p>0$,
define the nonnegative densities
\begin{equation}\label{eq:energy-definitions}
 e_r(f)(z)=r^{-d}G_{1,r}^{p}(\overline{f})(z)\quad\text{and}\quad
 \omega_r(f)(z)=r^{-d}\operatorname{MO}_{1,r}^p(f)(z).
\end{equation}
Their dependence on $p$ is suppressed in the notation.
Let $\mathcal W$ be the class of relatively weakly compact
subsets of $L^1(\R^d,dv)$.
By the Dunford--Pettis characterization
\cite[Theorem 4.30]{Brezis2011}, recalled in
Section \ref{sec:weak-localization}, this condition controls
both concentration on small sets and mass outside large
balls, in addition to bounding the total mass.
The next theorem uses it to characterize the limit for
each fixed bounded second symbol.

\begin{theorem}\label{thm:weakcriterion}
Suppose \eqref{eq:weight} holds, $2\leq p<\infty$, and
$f\in L^\infty(\C^n)$. The following are equivalent.
\begin{enumerate}[label=\textup{(\alph*)}]
\item $f\in\mathcal A_p$.
\item $\{e_r(f):0<r<r_0\}\in\mathcal W$ for some $r_0>0$.
\item $\{\omega_r(f):0<r<r_0\}\in\mathcal W$ for some $r_0>0$.
\end{enumerate}
\end{theorem}
This criterion completes the characterization for $p>d$
without a differentiability assumption. This together with
Theorem \ref{thm:main} describes $\mathcal A_p$ for every
finite $p>0$ using only the first symbol and Lebesgue measure.
In particular, $\mathcal A_p$ is the same for all weights
satisfying \eqref{eq:weight}, as recorded in
Corollary \ref{cor:algebra-weight}.

The mass of the densities need not tend to zero when the
limit holds for each fixed test. To compare this limit with
convergence uniform over the unit ball of $L^\infty(\C^n)$,
we use the testing quantity defined for $f\in L^\infty(\C^n)$
and $0<p<\infty$ by
\begin{equation}\label{E5}
 m_{p,t}(f)=\sup\{\norm{D_t(f,g)}_{\mathcal S_p}:
                  g\in L^\infty(\C^n),\ \norm{g}_{L^\infty(\C^n)}\leq1\}.
\end{equation}
The value $+\infty$ is allowed. The next theorem estimates
this quantity and characterizes its convergence to zero
as $t\rightarrow0^{+}$.

\begin{theorem}\label{thm:uniform}
Suppose \eqref{eq:weight} holds and $2\leq p<\infty$. There are
$\rho,R,c,C>0$, depending only on $n,m,M,p$, such that
\begin{equation}\label{eq:two-sided}
 c r^{-d/p}\norm{G_{1,\rho r}(\overline{f})}_{L^{p}(\R^d,dv)}
 \leq m_{p,r^2}(f)
 \leq C r^{-d/p}\norm{G_{1,Rr}(\overline{f})}_{L^{p}(\R^d,dv)}
\end{equation}
for every bounded $f$ and every $r>0$. These bounds give
the equivalence
\begin{equation}\label{eq:uniform-characterization}
 \lim_{t\rightarrow0^{+}}m_{p,t}(f)=0
 \quad\Longleftrightarrow\quad
 \lim_{r\rightarrow0^{+}}r^{-d/p}\norm{G_{1,r}(\overline{f})}_{L^{p}(\R^d,dv)}=0.
\end{equation}
\end{theorem}
Theorems \ref{thm:weakcriterion} and \ref{thm:uniform}
distinguish convergence for each fixed bounded second
symbol from convergence uniform over all such symbols
of norm at most one. For $2\leq p<\infty$, relative weak
compactness of the densities $\omega_r(f)$ characterizes
$f\in\mathcal A_p$ by
Theorem \ref{thm:weakcriterion}, while
$\norm{\omega_r(f)}_{L^1(\R^d,dv)}\to0$ as $r\rightarrow0^{+}$
characterizes uniform convergence by
Corollary \ref{cor:uniform-real}. This difference is explicit
at $p=d$, where Theorem \ref{thm:main} (b) admits nonconstant
$f\in\mathcal A_d$ and Corollary \ref{cor:critical-bound} gives
$$
 \liminf_{t\rightarrow0^{+}}m_{d,t}(f)
 \geq c\norm{\nabla f}_{L^d(\R^d,dv)}>0,
$$
where $c>0$ is independent of $f$. More generally, uniform
convergence for $0<p\leq d$ holds only for constant first
symbols by Remark \ref{R1}.

In the supercritical range, even a uniform bound on $m_{p,t}(f)$
and convergence for the conjugate test $g=\overline{f}$
do not imply convergence
for every fixed bounded test. We write $C_c(\R^d)$ for continuous
compactly supported functions on $\R^d$.

\begin{proposition}\label{prop:continuous-concentration}
Suppose $p>d=2n$. There are real-valued $f\in C_c(\R^d)$
and $g\in L^\infty(\R^d)$ such that
$$
 \begin{gathered}
 \sup_{0<t<1}m_{p,t}(f)<\infty,\quad
 \lim_{t\rightarrow0^{+}}\norm{D_t(f,f)}_{\mathcal S_p}=0,\quad
 \limsup_{t\rightarrow0^{+}}\norm{D_t(f,g)}_{\mathcal S_p}>0.
 \end{gathered}
$$
The same functions work for every weight satisfying
\eqref{eq:weight}.
\end{proposition}
Remark \ref{R2} shows that the densities $\omega_r(f)$ have
uniformly bounded mass but concentrate on sets of vanishing measure. The weak compactness criterion also gives sufficient
conditions without assuming weak differentiability.
For compactly supported H\"older symbols, we control the
densities through the H\"older seminorm
\begin{equation}\label{E44}
 [f]_\alpha=\sup_{x\ne y}\frac{|f(x)-f(y)|}{|x-y|^\alpha}
\end{equation}
for $0<\alpha\leq1$.
Finiteness of $[f]_\alpha$ means H\"older continuity
of exponent $\alpha$, with $\alpha=1$ corresponding to
Lipschitz continuity. Let $C_c^{0,\alpha}(\R^d)$ denote
the compactly supported functions with $[f]_\alpha<\infty$.
The next theorem applies the criteria to this class.

\begin{theorem}\label{thm:holder}
Suppose \eqref{eq:weight} holds, $p>d$, and
$d/p\leq\alpha\leq1$. Every $f\in C_c^{0,\alpha}(\R^d)$ belongs to
$\mathcal A_p$. If $\alpha>d/p$, the convergence is uniform over
$\norm{g}_{L^\infty(\C^n)}\leq1$, with the estimate
$$
 \sup_{\norm{g}_{L^\infty(\C^n)}\leq1}\norm{D_t(f,g)}_{\mathcal S_p}
 \leq C_f t^{\alpha/2-n/p}.
$$
Here $C_f>0$ may depend on $n,m,M,p,\alpha$, the support
of $f$, and $\norm{f}_{L^\infty(\C^n)}+[f]_\alpha$, but not on $t$ or $g$.
At $\alpha=d/p$, the limit also holds for every fixed bounded $g$.
\end{theorem}
Example \ref{ex:lacunary} provides, for each
$p>d$, a function $f\in C_c^{0,d/p}(\R^d)$ that does not
satisfy the weak gradient condition in Theorem \ref{thm:main} (c).
Its membership in $\mathcal A_p$ by Theorem \ref{thm:holder}
shows that this sufficient condition is not necessary.

The characterization of $\mathcal A_p$ also yields its algebra
properties in Corollary \ref{cor:algebra-weight}. For every
$0<p<\infty$, this class is a unital algebra closed under complex
conjugation, which allows us to extend the semi-commutator limit
to finite Toeplitz products and commutators.

\begin{corollary}\label{P1}
Suppose \eqref{eq:weight} holds and let $0<p<\infty$.
All symbols below are independent of $t$.
\begin{enumerate}[label=\textup{(\alph*)}]
\item If $k\geq2$ is an integer,
$f_1,\ldots,f_{k-1}\in\mathcal A_p$, and
$f_k\in L^\infty(\C^n)$, then
$$
 \norm{T_{f_1}^{(t)}\cdots T_{f_k}^{(t)}
             -T_{f_1\cdots f_k}^{(t)}}_{\mathcal S_p(\cF)}
 \longrightarrow0\quad(t\rightarrow0^{+}).
$$
\item If $f\in\mathcal A_p$ and $g\in L^\infty(\C^n)$, then
$$
 \norm{[T_f^{(t)},T_g^{(t)}]}_{\mathcal S_p(\cF)}
 \longrightarrow0\quad(t\rightarrow0^{+}),
$$
where $[A,B]=AB-BA$ for operators on the same space.
\end{enumerate}
\end{corollary}
Passing to an algebraic quotient of operator families expresses
preservation of multiplication and conjugation as a unital
$*$-homomorphism in Proposition \ref{P2}. For compactly supported
H\"older symbols with $\alpha>d/p$, the quantitative estimate
in Theorem \ref{thm:holder} also gives singular value bounds
and finite rank approximation through Proposition \ref{P3}.

The Bargmann transform transfers the classification to Gaussian
anti-Wick quantization by unitary equivalence in
Corollary \ref{cor:antiwick}.
For Toeplitz quantization on compact K\"ahler manifolds,
the H\"older sufficient condition in Theorem \ref{thm:holder}
has the following analogue. To state it, let $(X,\omega)$ be a
compact connected K\"ahler manifold of complex dimension $n\geq1$ with complex
structure $J$, and let $L\to X$ be a holomorphic line bundle
with a smooth Hermitian metric $h^L$ of strictly positive curvature.
For each integer $k\geq1$, write
$$\mathcal H_k=H^0(X,L^{\otimes k})$$ for the holomorphic section
space, with the $L^2$ inner product induced by $h^L$
and $dv_X=\omega^n/n!$.
For $a\in L^\infty(X)$, let
$T_{a,k}:\mathcal H_k\to\mathcal H_k$ be the orthogonal compression
of multiplication by $a$, and write
$$D_k(f,g)=T_{f,k}T_{g,k}-T_{fg,k}$$ for bounded $f,g$.
Here $C^{0,\alpha}(X)$ consists of the continuous functions with
finite H\"older seminorm $[f]_\alpha$, defined as in \eqref{E44}
using the distance induced by $\omega$ and $J$.
The geometric setting and operator construction are detailed in
Section \ref{sec:compact-model}.

\begin{proposition}\label{P4}
Let $0<\alpha\leq1$ and $2n/\alpha\leq p<\infty$.
For $f\in C^{0,\alpha}(X)$ and every fixed
$g\in L^\infty(X)$, we have
$$
 \norm{D_k(f,g)}_{\mathcal S_p(\mathcal H_k)}\longrightarrow0
 \quad(k\longrightarrow\infty).
$$
Moreover, there are $C>0$ and $k_0\geq1$, depending only on
$(X,\omega,J)$, $(L,h^L)$, $\alpha$, and $p$, such that
\begin{equation}\label{E215}
 \sup_{\norm{g}_{L^\infty(X)}\leq1}
 \norm{D_k(f,g)}_{\mathcal S_p(\mathcal H_k)}
 \leq C[f]_\alpha k^{n/p-\alpha/2},\qquad k\geq k_0.
\end{equation}
In particular, the convergence is uniform over this unit ball
when $p>2n/\alpha$.
\end{proposition}

Example \ref{Eg1} gives a smooth first symbol and a fixed bounded
second symbol on the projective line for which the trace norm
has a positive limit, although every Schatten $p$-norm with
$1<p<\infty$ tends to zero.

\subsection{Proof strategy and organization}
The necessary conditions and the two testing criteria require
estimates beyond those obtained directly from the Hankel
factorization. The local exponent is $1$ in both $G_{1,\rho r}$ and
$G_{1,Rr}$ in \eqref{eq:two-sided}, whereas the corresponding
Hankel estimate uses $G_{2,r}$. More precisely, for bounded
$u$ and $r>0$, Lemma \ref{lem:scaled-kernel} (b) and
\cite[Theorem 1.1]{HuVirtanen2022} yield
$$
 \norm{H_u^{(r^2)}}_{\mathcal S_q(\mathcal F_{r^2}^2,L_{r^2}^2\ominus\mathcal F_{r^2}^2)}
 \asymp r^{-d/q}\norm{G_{2,r}(u)}_{L^q(\R^d,dv)},
 \quad 0<q<\infty.
$$
Here $q$ is the Schatten exponent and the exponent of the
outer $L^q$ norm, while the local exponent in $G_{2,r}$ is $2$.
The notation $\asymp$ denotes two-sided estimates with positive
constants depending only on $n,m,M,q$.
For the conjugate test, the factorization gives
$$
 \norm{D_t(f,\overline{f})}_{\mathcal S_p}
 =\norm{H_{\overline{f}}^{(t)}}_{\mathcal S_{2p}(\cF,L_t^2\ominus\cF)}^2.
$$
For nonzero bounded $f$ and $2\leq p<\infty$, testing
\eqref{E5} with $g=\overline{f}/\norm{f}_{L^\infty(\C^n)}$
gives a lower bound by this identity.
The factorization \eqref{eq:factorization}, \eqref{E7}, and
$\norm{H_g^{(t)}}\leq\norm{g}_{L^\infty(\C^n)}$ give the upper bound.
Combining these bounds, we obtain
$$
 \frac{\norm{H_{\overline{f}}^{(t)}}_{\mathcal S_{2p}(\cF,L_t^2\ominus\cF)}^2}
      {\norm{f}_{L^\infty(\C^n)}}
 \leq m_{p,t}(f)
 \leq\norm{H_{\overline{f}}^{(t)}}_{\mathcal S_p(\cF,L_t^2\ominus\cF)},
$$
whose lower and upper bounds involve the distinct Schatten
exponents $2p$ and $p$. These bounds alone do not give
the estimate with local exponent one in \eqref{eq:two-sided}.

The conjugate test also has a restricted role in the
classification. By \eqref{eq:conjugate-identity} and
\cite[Theorem 1.1]{HuVirtanen2022}, Proposition \ref{prop:conjugate}
derives constancy from convergence of $D_t(f,\overline{f})$
to zero in $\mathcal S_p$ as $t\to0^{+}$ for $0<p\leq n$.
For $n<p<2n$, the Gaussian
example in Proposition \ref{prop:gaussian} satisfies this limit
without being constant. Our proof of constancy in this
remaining subcritical range and of the necessary gradient
condition at $p=2n$ uses additional bounded tests.

For the necessary conditions in Theorem \ref{thm:main} with
$1\leq p\leq d$, the scale at which a fixed test gives a
Schatten bound may depend on that test. A Baire argument
gives a common bound for $m_{p,t}(f)$ at small scales.
Local phase tests then control derivatives of mollifications
through analytic distance, with the factor $r^{d/p-1}$ giving
the dimension threshold in Proposition \ref{prop:necessity}.
Schatten inclusion extends the subcritical conclusion to
$0<p<1$. For sufficiency, the gradient estimate proves
\eqref{eq:super-gradient} when $p>d$. At $p=d$, its bound is
independent of the scale, allowing us to pass from convergence
for compactly supported smooth symbols to the full gradient
class by approximation in the gradient norm, as established
in Section \ref{sec:critical}.

We prove necessity in Theorem \ref{thm:weakcriterion} by representing
normalized finite rank trace pairings along sequences of scales
tending to zero by $L^1$ functions. The assumed semi-commutator
convergence forces their integrals against each fixed $L^\infty$
function to tend to zero, producing weakly null sequences in $L^1$.
The Dunford--Pettis criterion then controls their absolute integrals
uniformly on small sets and outside large balls.
Local phase tests transfer this control to the normalized
analytic distance densities.
To pass to ordinary mean oscillation, the analytic decomposition
first controls the complex derivatives of its smooth part and
local averages of its remainder. For the operators used to
recover the real gradient, Lemma \ref{lem:ui-stability} preserves
relative weak compactness in $L^1(\R^d,dv)$ of the corresponding
families of $p$-th power densities. This reconstruction and a
Poincar\'e inequality give the equivalent mean oscillation
condition in Lemma \ref{lem:ui-analytic-real}.

Conversely, we use the mean oscillation criterion to control
the gradient and remainder of a mollification at $s=\lambda r$, with
$0<\lambda<1/4$ fixed. Lemma \ref{lem:smooth-family-ui} uses
relative weak compactness of the normalized gradient densities
to make the smooth contribution tend to zero for each fixed test.
Positive Toeplitz estimates and comparison of local averages
bound the remainder by $C\lambda^{d/p}$, uniformly in $r$ and
$\lambda$. We take the upper limit as $r\to0^{+}$ before letting
$\lambda\to0^{+}$.

Local phase tests also give the lower bound in Theorem
\ref{thm:uniform}. For the upper bound, the analytic decomposition
at scale $r$ controls both the scaled complex derivatives of its
smooth part and the local $L^1$ averages of its remainder by
$G_{1,4r}$.
Gradient estimates handle the smooth part, and positive Toeplitz
estimates retain this local exponent for the remainder.
This gives the upper bound with local exponent one in
\eqref{eq:two-sided}.
For Theorem \ref{thm:holder}, a common integrable majorant for
$e_r(f)$, $0<r<1$, gives membership in $\mathcal A_p$ through
Theorem \ref{thm:weakcriterion}, including at $\alpha=d/p$.
When $\alpha>d/p$, the H\"older estimate and the upper bound in
\eqref{eq:two-sided} give the quantitative rate.

On compact K\"ahler manifolds, off-diagonal decay of the
Bergman kernel and the dimension bound for holomorphic sections
give the uniform H\"older estimate in Proposition \ref{P4}.
Smooth approximation of each fixed bounded second symbol improves
its Hankel Schatten bound to a little-$o$ estimate, so the
semi-commutator factorization also gives convergence at the endpoint.
Radial diagonalization on the projective line gives the explicit
trace norm obstruction in Example \ref{Eg1}.

The paper is organized around the classification and the two
convergence criteria. After the preliminaries in Section
\ref{sec:prelim}, Sections \ref{sec:necessary} and
\ref{sec:critical} prove Theorem \ref{thm:main}.
Section \ref{sec:remaining} establishes Theorem
\ref{thm:weakcriterion}, and Section \ref{sec:uniform} proves
Theorems \ref{thm:uniform} and \ref{thm:holder}.
Section \ref{sec:examples} examines the scope of these results
through a H\"older example, a Gaussian conjugate test, and
concentration examples culminating in Proposition
\ref{prop:continuous-concentration}.
Section \ref{sec:quantization} proves Corollary \ref{P1} and
develops the algebra of operator families and the quantitative
approximation consequences of the main estimates.
Section \ref{sec:other-models} transfers the classification to
Gaussian anti-Wick operators and proves Proposition \ref{P4}
in the compact K\"ahler setting, ending with an explicit example
on the projective line.

\section{Notation and preliminaries}\label{sec:prelim}
In this section, we fix the notation for function spaces and weak derivatives
and collect the Schatten, kernel, and analytic distance estimates
used below. We also record the properties of finite sign averages
needed in the local testing argument.
\subsection{Function spaces and weak derivatives}\label{sec:notation}
We use the same notation for functions and sets under the
identification $\C^n\cong\R^d$.
For a measurable set $E\subseteq\R^d$ and a positive measure
$\nu$ on $E$, let $L^q(E,d\nu)$, $0<q<\infty$, consist of
measurable complex-valued functions $u$ such that
$$
 \norm{u}_{L^q(E,d\nu)}=
 \left(\int_E|u|^qd\nu\right)^{1/q}<\infty.
$$
For $q=\infty$, we use the essential supremum with respect
to $\nu$. 
For a measurable set $E\subseteq\R^d$, we write $|E|$ for
its Lebesgue measure and $\mathbf1_E$ for its indicator.
If $0<|E|<\infty$ and $u\in L^1(E,dv)$,
its average over $E$ is $u_E=|E|^{-1}\int_Eu dv$.

For a vector-valued function $u=(u_1,\ldots,u_N):\R^d\to\C^N$, where
$N\geq1$ is an integer, we use the Euclidean length and the
corresponding norm or quasi-norm
$$
 |u|=\left(\sum_{j=1}^N|u_j|^2\right)^{1/2},\quad
 \norm{u}_{L^{q}(\R^d,dv)}=\norm{|u|}_{L^{q}(\R^d,dv)}.
$$
The space $[L^q(\R^d,dv)]^N$ consists of vectors whose
components belong to $L^q(\R^d,dv)$. With this convention,
$M_{q,r}(u)$ in \eqref{E4} means $M_{q,r}(|u|)$.

For an open set $\Omega\subseteq\R^d$ and $0<q<\infty$,
the space $L^q_{\rm loc}(\Omega,dv)$ consists of measurable
functions $u$ such that $\int_K|u|^q dv<\infty$ for every
compact set $K\subset\Omega$.
  For an integer $k\geq0$, $C^k(\Omega)$ consists of functions
with continuous real derivatives through order $k$.
We write $C^\infty(\Omega)$ for smooth functions and
$C_c^\infty(\Omega)$ for those with compact support in
$\Omega$. We write $A\Subset\Omega$ if $\overline A$ is a compact
subset of $\Omega$. For $h\in C_c^\infty(\Omega)$, define its zero extension by
\begin{equation}\label{E93}
 h_0(x)=
 \begin{cases}
 h(x),&x\in\Omega,\\
 0,&x\in\R^d\setminus\Omega.
 \end{cases}
\end{equation}
Since $\supp h\Subset\Omega$, the function $h_0$ vanishes
near each boundary point and belongs to $C_c^\infty(\R^d)$.
We use this convention in other Euclidean dimensions
and retain the original function notation when no
distinction is needed.
For measurable symbols, $\supp$ denotes the essential
closed support.

The weak derivatives used below are understood in the
Sobolev sense \cite[Section 9.1, p. 263--264]{Brezis2011}.
\begin{definition}\label{def:weak-gradient}
Let $u\in L^1_{\rm loc}(\R^d,dv)$ and $1\leq\ell\leq d$.
A function $g_\ell\in L^1_{\rm loc}(\R^d,dv)$ is the weak
derivative of $u$ in the coordinate $x_\ell$ if
\begin{equation}\label{E3}
 \int_{\R^d}u\partial_{x_\ell}h dv
 =-\int_{\R^d}g_\ell h dv,
 \quad h\in C_c^\infty(\R^d).
\end{equation}
We write $g_\ell=\partial_{x_\ell}u$ for the weak partial derivative
of $u$ with respect to $x_\ell$. When all first weak
derivatives exist, we set
$$
 \nabla u=(\partial_{x_\ell}u)_{\ell=1}^d,\quad
 |\nabla u|=\bigg(\sum_{\ell=1}^d|\partial_{x_\ell}u|^2\bigg)^{1/2},\quad
 \norm{\nabla u}_{L^{q}(\R^d,dv)}=\norm{|\nabla u|}_{L^{q}(\R^d,dv)}.
$$
\end{definition}
The coordinates of differentiation are real, but $u$ and
its weak derivatives may be complex-valued. We write
$u=\Re u+i\Im u$, where $\Re u$ and $\Im u$ denote
its real and imaginary parts, respectively.
On an open subset of $\R^d$, we use the same integral identity
with test functions supported in that subset. The definition applies separately to the real and imaginary
parts, as do the Sobolev results used below.  The differential operators in
complex coordinates are defined from real derivatives by
\begin{equation}\label{E13}
 \partial_{z_j}=\tfrac12(\partial_{x_{2j-1}}-i\partial_{x_{2j}}),\quad
 \partial_{\overline{z}_j}=\tfrac12(\partial_{x_{2j-1}}+i\partial_{x_{2j}})
\end{equation}
for $1\leq j\leq n$.
For a function $u$ with first weak derivatives, we write
\begin{equation}\label{E107}
 \dbar u=(\partial_{\overline{z}_j}u)_{j=1}^n,\quad
 |\dbar u|=\bigg(\sum_{j=1}^n|\partial_{\overline{z}_j}u|^2\bigg)^{1/2}.
\end{equation}
For a holomorphic function $h$, the derivatives defined in
\eqref{E13} satisfy $\partial_{\overline{z}_j}h=0$,
and $\partial_{z_j}h$ agrees with its usual complex partial derivative.
Equations \eqref{E13} and \eqref{E107} together with
$|a+ib|^2\leq2(|a|^2+|b|^2)$ for $a,b\in\C$ give
$|\dbar u|\leq2^{-1/2}|\nabla u|$ almost everywhere.
In particular, if $0<q\leq\infty$ and
$|\nabla u|\in L^q(\R^d,dv)$, then
\begin{equation}\label{E57}
 \norm{\dbar u}_{L^q(\R^d,dv)}
 \leq\frac1{\sqrt2}\norm{\nabla u}_{L^q(\R^d,dv)}<\infty.
\end{equation}

To pass to limits without assuming weak differentiability,
we use distributional derivatives
\cite[Sections 6.11--6.13]{Rudin1991}.
\begin{definition}\label{def:distributional-derivative}
Let $u\in L^1_{\rm loc}(\R^d,dv)$ and $1\leq\ell\leq d$.
Its associated distribution and distributional derivative
are defined by
\begin{equation}\label{E48}
 \langle u,h\rangle=\int_{\R^d}uh dv,\quad
 \langle\partial_{x_\ell}u,h\rangle
 =-\int_{\R^d}u\partial_{x_\ell}h dv,
 \quad h\in C_c^\infty(\R^d).
\end{equation}
The pairings are bilinear. More generally, for a distribution
$T$ on $\R^d$, we define
\begin{equation}\label{E159}
 \langle\partial_{x_\ell}T,h\rangle
 =-\langle T,\partial_{x_\ell}h\rangle,
 \quad h\in C_c^\infty(\R^d),\quad 1\leq\ell\leq d.
\end{equation}
Higher derivatives are defined by iterating \eqref{E159}.
By \eqref{E3} and \eqref{E48}, a function
$g_\ell\in L^1_{\rm loc}(\R^d,dv)$ is the weak derivative
of $u$ exactly when
\begin{equation}\label{E160}
 \langle\partial_{x_\ell}u,h\rangle=\int_{\R^d}g_\ell h dv
\end{equation}
for every $h\in C_c^\infty(\R^d)$.
The complex distributional derivatives are defined by
\eqref{E13}, so
\begin{equation}\label{E49}
 \langle\partial_{\overline{z}_j}u,h\rangle
 =-\int_{\R^d}u\partial_{\overline{z}_j}h dv,
 \quad h\in C_c^\infty(\R^d),\quad 1\leq j\leq n.
\end{equation}
The condition $\dbar u=0$ in distributions means that the
integral in \eqref{E49} vanishes for every such $h$ and $j$.
\end{definition}
Throughout the paper, a condition such as
$|\nabla u|\in L^p(\Omega,dv)$ on an open set
$\Omega\subseteq\R^d$ means that all weak first derivatives
in Definition \ref{def:weak-gradient} exist and their
Euclidean length belongs to $L^p(\Omega,dv)$.
For $u\in C^1(\Omega)$, integration by parts identifies its
classical first derivatives with its weak derivatives in
\eqref{E3} and the functions representing its distributional
derivatives \cite[Section 6.13, p. 158]{Rudin1991}.
We retain the same notation for derivatives and gradients
under these identifications. In the distributional setting,
$\nabla u$ denotes the vector of derivatives in \eqref{E48},
with test functions taken in $C_c^\infty(\Omega)$.
We explicitly mark identities asserted only in distributions.

Constants $c,C>0$ may vary, with dependence stated locally.
We write $a\lesssim b$ for $a\leq Cb$ and $a\asymp b$ when
both $a\lesssim b$ and $b\lesssim a$ hold. Unless stated otherwise, the bounds are
uniform in the scale and symbols. All scale parameters are
positive, and their limits at zero are taken through positive values.

\subsection{Schatten classes and diagonal estimates}\label{sec:schatten}
Let $X,Y$ be complex Hilbert spaces, and let $A:X\to Y$ be bounded.
Its operator norm is given by
$$
 \norm{A}=\sup_{\norm{u}_X\leq1}\norm{Au}_Y.
$$ Its approximation
numbers are
\begin{equation}\label{E102}
 s_j(A)=\inf\{\norm{A-R}:R:X\to Y\text{ is bounded},\ \rank R<j\},
 \qquad j\geq1.
\end{equation}
When $A$ is compact, these are its singular values in
nonincreasing order, counted with multiplicity
\cite[Theorem 1.34]{Zhu1990}. For $0<q<\infty$, the Schatten
class $\mathcal S_q(X,Y)$ consists of compact operators
$A:X\to Y$ with
$$
 \norm{A}_{\mathcal S_q(X,Y)}=\bigg(\sum_{j=1}^\infty s_j^q(A)\bigg)^{1/q}<\infty.
$$
We write $\mathcal S_q(X)$ if $X=Y$. When displayed, the spaces
in the norm subscript specify the domain and codomain of the operator.
The same series with approximation numbers defines the extended
value $+\infty$ for a bounded operator outside $\mathcal S_q(X,Y)$.
For $0<q<\infty$ and bounded $A:X\to Y$, \eqref{E102} with
$j=1$ and the definition of the Schatten norm give
\begin{equation}\label{E128}
 \norm{A}=s_1(A)
 \leq\bigg(\sum_{j=1}^\infty s_j^q(A)\bigg)^{1/q}
 =\norm{A}_{\mathcal S_q(X,Y)}.
\end{equation}
Each $\cF$ is a separable Hilbert space
\cite[Section 3, p. 369--370]{SchusterVarolin2012}, so these
conventions apply to the operators considered here.
If $Y$ is a closed subspace of $L_t^2$ and $A:\cF\to Y$
is bounded, the operator $\widetilde A:L_t^2\to L_t^2$
defined by $\widetilde A(v+w)=Av$ for
$v\in\cF$ and $w\in L_t^2\ominus\cF$ acts as zero on
the orthogonal complement. With respect to
$L_t^2=\cF\oplus(L_t^2\ominus\cF)$, we have
\begin{equation}\label{E138}
 \widetilde A^*\widetilde A=(A^*A)\oplus0.
\end{equation}
This identity shows that the natural inclusion of $Y$
and the zero extension preserve the operator norm and
nonzero singular values.

We use the ideal and H\"older inequalities for Schatten classes
\cite[Theorems 2.7 and 2.8]{Simon2005}. Let $q\geq1$.
In the first inequality below, take Hilbert spaces $X_0,X,Y,Y_0$,
$A\in\mathcal S_q(X,Y)$, and bounded operators
$B:Y\to Y_0$ and $R:X_0\to X$. In the second, take
$B\in\mathcal S_{2q}(X,Y)$ and $A\in\mathcal S_{2q}(Y,Z)$,
where $Z$ is a Hilbert space. These operators satisfy
\begin{equation}\label{E7}
 \begin{aligned}
 \norm{BAR}_{\mathcal S_q(X_0,Y_0)}
 &\leq\norm{B}\norm{A}_{\mathcal S_q(X,Y)}\norm{R}\quad\text{and}\quad
 \norm{AB}_{\mathcal S_q(X,Z)}
 \leq\norm{A}_{\mathcal S_{2q}(Y,Z)}\norm{B}_{\mathcal S_{2q}(X,Y)}.
 \end{aligned}
\end{equation}
To derive \eqref{E7} from the cited results on one Hilbert
space, take $H=X_0\oplus X\oplus Y\oplus Y_0$ for the first
inequality and $H=X\oplus Y\oplus Z$ for the second.
For each summand $U$, let $\iota_U:U\to H$ be the canonical
inclusion and $P_U=\iota_U^*:H\to U$ the orthogonal projection.
For a bounded operator $T:U\to V$ between summands, define
$\widetilde T=\iota_VTP_U:H\to H$. Since $P_V\iota_V=I_V$,
its adjoint satisfies
$$
 \widetilde T^*\widetilde T
 =\iota_UT^*(P_V\iota_V)TP_U=\iota_UT^*TP_U.
$$
Relative to $H=U\oplus U^\perp$, this has the form of
\eqref{E138}, so $\norm{\widetilde T}=\norm{T}$.
When $T$ is compact, so is $\widetilde T$, with the same
nonzero singular values and multiplicities, hence the same
Schatten norms. For a bounded $S:V\to W$ between summands,
$P_V\iota_V=I_V$ gives
$$
 \widetilde S\widetilde T
 =\iota_WS(P_V\iota_V)TP_U
 =\iota_WSTP_U=\widetilde{ST}.
$$
Norm preservation and this product identity give
\eqref{E7} from the cited inequalities on $H$.

For a compact positive operator $A$, let $\lambda_j(A)$
be its eigenvalues in nonincreasing order, counted with
multiplicity. If $q>0$, $A,B:\cF\to\cF$ are bounded,
$0\leq A\leq CB$ for a fixed $C>0$, and
$B\in\mathcal S_q(\cF)$ is positive, then $A$ is compact.
Indeed, let $A^{1/2}$ be its positive square root
\cite[Section 1.2, p. 7]{Zhu1990}. For every sequence
$(h_k)$ in $\cF$ converging weakly to zero, compactness of $B$ gives
$$
 \norm{A^{1/2}h_k}_{L_t^2}^2
 =\langle Ah_k,h_k\rangle_t
 \leq C\langle Bh_k,h_k\rangle_t\longrightarrow0\quad(k\rightarrow\infty).
$$
By \cite[Theorems 1.14 and 1.15, p. 11]{Zhu1990},
$A^{1/2}$ and $A=A^{1/2}A^{1/2}$ are compact. The variational
formula \cite[Section 1.6, Exercise 8, p. 31]{Zhu1990},
applied to $A$ and $CB$, followed by summation of the
$q$-th powers, gives
\begin{equation}\label{E8}
 A\in\mathcal S_q(\cF),\quad
 \lambda_j(A)\leq C\lambda_j(B),\quad
 \norm{A}_{\mathcal S_q}\leq C\norm{B}_{\mathcal S_q}
\end{equation}
for $q>0$.
Note that both inequalities use the same constant $C$ as in $A\leq CB$.
For a bounded operator $J:\cF\to Y$, where $Y$ is a
Hilbert space, compactness of $J$ is equivalent to compactness
of $J^*J$. In this case, the singular values give
\begin{equation}\label{E9}
 \norm{J}_{\mathcal S_s(\cF,Y)}^2=\norm{J^*J}_{\mathcal S_{s/2}},
 \qquad 0<s<\infty,
\end{equation}
whenever either side is finite, by
\cite[Theorems 1.17 and 1.26]{Zhu1990}.

The following diagonal estimate follows from
\cite[Proposition 2.6, p. 19]{Simon2005}
by taking the symmetric sequence norm to be the $\ell^p$ norm
and the two orthonormal sets to be the same basis.
We apply it below to normalized kernels through their
synthesis operator.

\begin{lemma}\label{lem:diagonal}
Let $X$ be a separable Hilbert space with an orthonormal basis
$(e_j)_{j=0}^{\infty}$. For $1\leq p<\infty$, there is a positive
constant $C_p$, depending only on $p$, such that every
$A\in\mathcal S_p(X)$ satisfies
$$
 \sum_{j=0}^{\infty}|\langle Ae_j,e_j\rangle_X|^p
 \leq C_p\norm{A}_{\mathcal S_p(X)}^p.
$$
\end{lemma}

\subsection{Kernel estimates and localization}\label{sec:kernels}
Let $K_t$ be the reproducing kernel of $\cF$ and put $K=K_1$. 
For $a\in\C^n$, the normalized reproducing kernel $k_{t,a}$ at $a$ is $k_{t,a}=K_t(\cdot,a)/K_t^{1/2}(a,a)$. The reproducing identity is given by
\begin{equation}\label{E10}
 \begin{gathered}
 h(z)=P_t h(z)=\int_{\C^n}K_t(z,w)h(w)d\mu_t(w),\quad h\in\cF.
 \end{gathered}
\end{equation}
For $r=\sqrt t$, define $U_t :L_t^2\longrightarrow L_1^2$ and $ V_t:L_t^2\longrightarrow L^2(\C^n,dv)$ by 
\begin{equation}\label{E11}
 \begin{aligned}
  (U_tu)(\zeta)=u(r\zeta)\quad\text{and}\quad
  (V_tu)(z)&=r^{-d/2}e^{-\varphi(z/r)}u(z).
 \end{aligned}
\end{equation}
On unweighted $L^2(\C^n,dv)$, set
\begin{equation}\label{E12}
 \Pi_t=V_tP_tV_t^{-1}\quad \text{and}\quad
 p_t(z,w)=r^{-d}e^{-\varphi(z/r)}K_t(z,w)e^{-\varphi(w/r)}.
\end{equation}
 We regard $\Pi_t$ as an operator
on this whole space, with
range $V_t\cF$.

Dilation reduces the kernel estimates to scale one and fixes
the powers of $r$ in the local bounds. Throughout the paper, $P=P_1:L_1^2\to\mathcal F_1^2$
denotes the orthogonal projection at scale one, and we write
$T_u=T_u^{(1)}$.

\begin{lemma}\label{lem:scaled-kernel}
Suppose \eqref{eq:weight} holds and $r=\sqrt t$.
\begin{enumerate}[label=\textup{(\alph*)}]
\item
The maps $U_t$ and $V_t$ are unitary, with
$U_t\cF=\mathcal F_1^2$. For every $f\in L^\infty(\C^n)$,
the following identities hold.
\begin{equation}
\begin{aligned}\label{eq:unitary-dilation}
 &U_tP_tU_t^{-1}=P,\quad K_t(z,w)=K(z/r,w/r),\quad U_tM_fU_t^{-1}=M_{f(r\cdot)},\quad V_tM_fV_t^{-1}=M_f.
\end{aligned}
\end{equation}
The orthogonal projection $$\Pi_t:L^2(\C^n,dv)\to V_t\cF$$ has
integral kernel $p_t$.

\item
The unitary $U_t$ restricts to unitary maps
$\cF\to\mathcal F_1^2$ and
$L_t^2\ominus\cF\to L_1^2\ominus\mathcal F_1^2$.
We use the same notation for these restrictions.
For $u,g\in L^\infty(\C^n)$, the following identities hold.
\begin{equation}\label{E55}
 U_tT_u^{(t)}U_t^{-1}=T_{u(r\cdot)},\quad
 U_tH_u^{(t)}U_t^{-1}=H_{u(r\cdot)}^{(1)},\quad
 U_tD_t(u,g)U_t^{-1}=D_1(u(r\cdot),g(r\cdot)).
\end{equation}
These identities preserve the operator norms and the corresponding Schatten norms or quasi-norms for
$0<q<\infty$.
\end{enumerate}
\end{lemma}
\begin{proof}
We first prove \textup{(a)}. By
\cite[Section 6, p. 2071]{HuVirtanen2023}, $U_t$ is
unitary, maps $\cF$ onto $\mathcal F_1^2$, and satisfies
$U_tP_tU_t^{-1}=P$. To obtain the kernel identity, take
$h\in\cF$ and $w\in\C^n$. The reproducing property gives
$$
 h(w)=(U_th)(w/r)
 =\langle U_th,K(\cdot,w/r)\rangle_1
 =\langle h,U_t^{-1}K(\cdot,w/r)\rangle_t.
$$
Uniqueness of the reproducing kernel therefore yields
$K_t(\cdot,w)=U_t^{-1}K(\cdot,w/r)$, and evaluation at $z$
gives $K_t(z,w)=K(z/r,w/r)$.

For the passage to unweighted coordinates, \eqref{E11} and
the definition of $\mu_t$ in \eqref{E82} give
$$
 \norm{V_tu}_{L^2(\C^n,dv)}^2
 =\int_{\C^n}|u(z)|^2r^{-d}e^{-2\varphi(z/r)} dv(z)
 =\norm{u}_{L_t^2}^2.
$$
Multiplication by $r^{d/2}e^{\varphi(z/r)}$ is the inverse
of $V_t$, so $V_t$ is unitary. Hence
$\Pi_t=V_tP_tV_t^{-1}$ is the orthogonal projection onto
$V_t\cF$. To identify its kernel, take $h\in L^2(\C^n,dv)$.
By the reproducing property and self-adjointness of $P_t$,
$P_tu(z)=\langle u,K_t(\cdot,z)\rangle_t$ for $u\in L_t^2$.
Substituting $V_t^{-1}h$ and writing $\mu_t$ in terms of
its density, we obtain
\begin{align*}
(\Pi_t h)(z)
&=r^{-d/2}e^{-\varphi(z/r)}
  \int_{\C^n}K_t(z,w)(V_t^{-1}h)(w)d\mu_t(w)\\
&=\int_{\C^n}r^{-d}e^{-\varphi(z/r)}
  K_t(z,w)e^{-\varphi(w/r)}h(w) dv(w)\\
&=\int_{\C^n}p_t(z,w)h(w) dv(w),
\end{align*}
which proves the kernel formula in \eqref{E12}. The definitions in \eqref{E11} also give the two
multiplication identities in \eqref{eq:unitary-dilation}.

For \textup{(b)}, combining these multiplication identities
with the projection identity and the operator definitions
\eqref{E1} and \eqref{E2} proves \eqref{E55}.
The asserted equalities of operator and Schatten norms
follow from these unitary equivalences. The proof is complete.
\end{proof}

 \begin{remark}\label{R3}
In $V_tM_fV_t^{-1}=M_f$, the multiplier between $V_t$ and
$V_t^{-1}$ acts on $L_t^2$, whereas the multiplier on the
right acts on $L^2(\C^n,dv)$, both by pointwise multiplication. Throughout the paper,  we use the same notation $M_f$ for multiplication by $f$
on the relevant spaces and for its restrictions,
with the domain and codomain determined by context. 
\end{remark}  

For $f,g\in L^\infty(\C^n)$, the operator definitions
\eqref{E1} and \eqref{E2} give the factorization
\cite[(6-9)]{HuVirtanen2023}
\begin{equation}\label{eq:factorization}
 D_t(f,g)=-P_tM_f(I-P_t)M_g
         =-\bigl(H_{\overline{f}}^{(t)}\bigr)^*H_g^{(t)}
         \quad\text{on }\cF.
\end{equation}
The corresponding commutator satisfies
$[T_f^{(t)},T_g^{(t)}]=D_t(f,g)-D_t(g,f)$.
Applying \eqref{E9} to $H_{\overline{f}}^{(t)}$ and using
\eqref{eq:factorization} gives
\begin{equation}\label{eq:conjugate-identity}
 \norm{D_t(f,\overline{f})}_{\mathcal S_p}
 =\norm{H_{\overline{f}}^{(t)}}_{\mathcal S_{2p}(\cF,L_t^2\ominus\cF)}^2
\end{equation}
for $0<p<\infty$.
The identity also holds when both sides are infinite.
The operator $D_t(f,g)P_t:L_t^2\to L_t^2$ is the zero
extension of $D_t(f,g):\cF\to\cF$ defined before
\eqref{E138}.
To estimate $D_t(f,g)$ on the fixed space $L^2(\C^n,dv)$,
we transport this extension by $V_t$ and define
\begin{equation}\label{E83}
 \widetilde D_t(f,g)=V_tD_t(f,g)P_tV_t^{-1}:L^2(\C^n,dv)\to L^2(\C^n,dv).
\end{equation}
By \eqref{E138} and the unitarity of $V_t$ in
Lemma \ref{lem:scaled-kernel} (a), we have
\begin{equation}\label{E43}
 \norm{\widetilde D_t(f,g)}_{\mathcal S_q(L^2(\C^n,dv))}
 =\norm{D_t(f,g)}_{\mathcal S_q}
\end{equation}
for $0<q<\infty$.
For $h\in L_t^2$, the vector $P_th$ belongs to $\cF$.
Applying \eqref{E1} and \eqref{E2} to $P_th$ gives
$$
 D_t(f,g)P_t=P_tM_f(P_t-I)M_gP_t
 \quad\text{on }L_t^2.
$$
Conjugating by $V_t$ and using \eqref{eq:unitary-dilation},
\eqref{E83}, and $\Pi_t^2=\Pi_t$, we obtain
\begin{equation}\label{eq:unweighted-semicommutator}
 \widetilde D_t(f,g)=\Pi_tM_f(\Pi_t-I)M_g\Pi_t
 =\Pi_t[M_f,\Pi_t]M_g\Pi_t,
\end{equation}
where $[M_f,\Pi_t]=M_f\Pi_t-\Pi_tM_f$.

By \cite[Section 2A, p. 2045--2046]{HuVirtanen2023},
the real Hessian bounds in \eqref{eq:weight} imply complex
Hessian bounds with constants depending only on $m,M$.
This verifies the weight assumptions in the results of Hu
and Virtanen used below, so we will not repeat the verification in the rest of the paper.
In particular, the kernel
estimates at scale one are given in \cite[Theorem 2]{Delin1998}
and \cite[Proposition 3.3]{SchusterVarolin2012}, and are recorded
in \cite[Lemma 2.2]{HuVirtanen2023}. Setting the two kernel variables
equal gives $K(\zeta,\zeta)e^{-2\varphi(\zeta)}\asymp1$.
The diagonal lower and upper bounds control the normalization
in the Gram estimate and the local lower bound, respectively.
The dilation in \eqref{eq:unitary-dilation} then yields the following estimates.

\begin{lemma}\label{L1}
Suppose \eqref{eq:weight} holds and $r=\sqrt t$. There are constants
$c,C,\rho>0$, depending only on $n,m,M$, such that
\begin{equation}\label{eq:gram}
 \begin{gathered}
 |\langle k_{t,a},k_{t,b}\rangle_t|\leq Ce^{-c|a-b|/r},\quad
 K_t(z,z)e^{-2\varphi(z/r)}\leq C,
 \quad|p_t(z,w)|\leq Cr^{-d}e^{-c|z-w|/r}
 \end{gathered}
\end{equation}
for $a,b,z,w\in\C^n$.
The normalized kernels also satisfy the local lower bound
\begin{equation}\label{E105}
 |k_{t,a}(z)|^2e^{-2\varphi(z/r)}\geq c
 \quad\text{if }|z-a|<\rho r.
\end{equation}
In particular, $k_{t,a}$ satisfies $\norm{k_{t,a}}_{L_t^2}=1$
and has no zero on $B(a,\rho r)$.
\end{lemma}

The next lemma transfers the diagonal estimate to the
normalized reproducing kernels $k_{t,a}$ defined before \eqref{E10},
with a bound independent of the scale and the lattice translation. Here $\ell^2(\Z^d)$ denotes
the Hilbert space of square-summable complex sequences
$c=(c_j)_{j\in\Z^d}$, with norm
$\norm{c}_{\ell^2(\Z^d)}=(\sum_{j\in\Z^d}|c_j|^2)^{1/2}$.

\begin{lemma}\label{lem:lattice}
Let $\rho>0$ be as in Lemma \ref{L1}, and fix $L>2\rho$.
For $r>0$ and $b\in[0,Lr)^d$, set
$$
 a_j=b+Lrj,\qquad k_j=k_{r^2,a_j},\qquad j\in\Z^d.
$$
\begin{enumerate}[label=\textup{(\alph*)}]
\item
The balls $B(a_j,\rho r)$ are pairwise disjoint, and the
synthesis operator
$$
 B_{r,b}:\ell^2(\Z^d)\longrightarrow\mathcal F_{r^2}^2,
 \qquad B_{r,b}c=\sum_{j\in\Z^d}c_jk_j,
$$
satisfies $\norm{B_{r,b}}\leq C_0$, where $C_0>0$ depends only on
$n,m,M,L$.

\item
For $1\leq p<\infty$, the normalized kernels satisfy
\begin{equation}\label{eq:lattice-diagonal}
 \sum_{j\in\Z^d}|\langle Ak_j,k_j\rangle_{r^2}|^p
 \leq C_p\norm{A}_{\mathcal S_p(\mathcal F_{r^2}^2)}^p,
 \qquad A\in\mathcal S_p(\mathcal F_{r^2}^2).
\end{equation}
Here $C_p>0$ depends only on $n,m,M,p,L$.
\end{enumerate}
\end{lemma}
\begin{proof}
For \textup{(a)}, the lattice centers satisfy
$$
 |a_i-a_j|=Lr|i-j|\geq Lr>2\rho r,\qquad i\neq j.
$$
So the balls are pairwise disjoint, and the scaled centers
$a_j/r=b/r+Lj$ satisfy $$|a_i/r-a_j/r|\geq L$$ for $i\neq j$.
Lemma \ref{lem:scaled-kernel} (a), together with the
definition of $k_{t,a}$ before \eqref{E10}, gives
$U_{r^2}k_j=k_{1,a_j/r}$.
For every finitely supported $c\in\ell^2(\Z^d)$, unitarity
and \cite[Lemma 2.4 and its proof]{HuVirtanen2022} yield
\begin{equation}\label{E199}
\|\sum_jc_jk_j\|_{L_{r^2}^2}
 =\|\sum_jc_jk_{1,a_j/r}\|_{L_1^2}
 \leq C_0\norm{c}_{\ell^2(\Z^d)},
\end{equation}
where $C_0>0$ depends only on $n,m,M,L$.
For $c=(c_j)_{j\in\Z^d}\in\ell^2(\Z^d)$ and an integer
$N\geq1$, define its finite truncation
$c^{(N)}=(c_j^{(N)})_{j\in\Z^d}$ by
$$
 c_j^{(N)}=
 \begin{cases}
 c_j,&|j|\leq N,\\
 0,&|j|>N.
 \end{cases}
$$
Each $c^{(N)}$ has finite support, and square summability gives
$$
 \norm{c-c^{(N)}}_{\ell^2(\Z^d)}^2
 =\sum_{|j|>N}|c_j|^2
 \longrightarrow0\quad(N\rightarrow\infty).
$$
For integers $1\leq N'<N$, applying \eqref{E199} to
$c^{(N)}-c^{(N')}$ gives
$$
 \bigg\|\sum_{N'<|j|\leq N}c_jk_j\bigg\|_{L_{r^2}^2}
 \leq C_0\bigg(\sum_{N'<|j|\leq N}|c_j|^2\bigg)^{1/2}
 \rightarrow0\quad(N',N\rightarrow\infty).
$$
Then $\sum_{|j|\leq N}c_jk_j$ is Cauchy in
$\mathcal F_{r^2}^2$. By completeness, we may define
$$
 B_{r,b}c=\lim_{N\rightarrow\infty}\sum_{|j|\leq N}c_jk_j
 \quad\text{in }\mathcal F_{r^2}^2.
$$
Passing to the limit in \eqref{E199} yields
$$
 \norm{B_{r,b}c}_{L_{r^2}^2}
 \leq C_0\norm{c}_{\ell^2(\Z^d)}.
$$
Thus, the synthesis map extends uniquely to a bounded
operator from $\ell^2(\Z^d)$ to $\mathcal F_{r^2}^2$,
with $\norm{B_{r,b}}\leq C_0$.

To prove \textup{(b)}, let $(e_j)_{j\in\Z^d}$ be the canonical
orthonormal basis of $\ell^2(\Z^d)$, so that $B_{r,b}e_j=k_j$.
Denote the constant from Lemma \ref{lem:diagonal} by $c_p>0$.
For $A\in\mathcal S_p(\mathcal F_{r^2}^2)$, that lemma and
\cite[Theorem 1.27]{Zhu1990}, together with the first inequality
in \eqref{E7} and the bound $\norm{B_{r,b}}\leq C_0$ from
\textup{(a)}, give
\begin{align*}
 \sum_{j\in\Z^d}|\langle Ak_j,k_j\rangle_{r^2}|^p
 &=\sum_{j\in\Z^d}
   |\langle B_{r,b}^*AB_{r,b}e_j,e_j\rangle_{\ell^2(\Z^d)}|^p\leq c_p
   \norm{B_{r,b}^*AB_{r,b}}_{\mathcal S_p(\ell^2(\Z^d))}^p\\
 &\leq c_p
   \norm{B_{r,b}}_{\ell^2(\Z^d)\to\mathcal F_{r^2}^2}^{2p}
   \norm{A}_{\mathcal S_p(\mathcal F_{r^2}^2)}^p\leq c_p C_0^{2p}
   \norm{A}_{\mathcal S_p(\mathcal F_{r^2}^2)}^p.
\end{align*}
With $C_p=c_p C_0^{2p}$, this is \eqref{eq:lattice-diagonal},
which completes the proof.
\end{proof}

\subsection{Analytic distance and derivative estimates}\label{sec:analytic-prelim}
We first recall the coordinate derivative estimate of Hu
and Virtanen and its consequence for the real gradient.

\begin{lemma}[{\cite[Lemma 5.1]{HuVirtanen2023}}]
\label{lem:complex-derivatives}
For $1<q<\infty$, there are $C_q,C_{n,q}>0$, depending only on
$n,q$, such that every $u\in C^2(\C^n)\cap L^\infty(\C^n)$ satisfies
\begin{equation}\label{E14}
\|{\nabla u}\|_{L^q(\R^d,dv)}\leq C_{n,q}\|{\dbar u}\|_{L^q(\R^d,dv)} \quad\text{and}\quad \|{\partial_{z_j}u}\|_{L^q(\R^d,dv)}\leq C_q\|{\partial_{\overline{z}_j}u}\|_{L^q(\R^d,dv)}
 \end{equation}
for $1\leq j\leq n$.
In particular, $|\dbar u|\in L^q(\R^d,dv)$ implies
$|\nabla u|\in L^q(\R^d,dv)$, without requiring $u\in L^q(\C^n,dv)$.
\end{lemma}
To obtain the real gradient estimate in \eqref{E14}, the coordinate
identities in \eqref{E13} give
$$
 |\nabla u|^2=2\sum_{j=1}^n
 \bigl(|\partial_{z_j}u|^2+|\partial_{\overline{z}_j}u|^2\bigr).
$$
The cited coordinate estimate and the triangle inequality therefore yield
\begin{align*}
 \norm{\nabla u}_{L^q(\R^d,dv)}
 &\leq\sqrt2\sum_{j=1}^n
 \bigl(\norm{\partial_{z_j}u}_{L^q(\R^d,dv)}
      +\norm{\partial_{\overline{z}_j}u}_{L^q(\R^d,dv)}\bigr)\\
 &\leq\sqrt2 n(C_q+1)\norm{\dbar u}_{L^q(\R^d,dv)}.
\end{align*}
Thus, we may take $C_{n,q}=\sqrt2 n(C_q+1)$ in \eqref{E14}.

To control derivatives and the remainder at scale $r$,
we use the construction in
\cite[Lemmas 3.3 and 3.6, p. 2051--2054]{HuVirtanen2023}.
We first fix its auxiliary radius, then use dilation to obtain
constants independent of $r$.

\begin{lemma}\label{lem:analytic-decomposition}
For $u\in L^\infty(\C^n)$ and $r>0$, there are
$u_{1,r}\in C^\infty(\C^n)\cap L^\infty(\C^n)$ and
$u_{2,r}\in L^\infty(\C^n)$ such that $u=u_{1,r}+u_{2,r}$ and
\begin{equation}\label{E15}
 \begin{gathered}
 \norm{u_{1,r}}_{L^\infty(\C^n)}+\norm{u_{2,r}}_{L^\infty(\C^n)}
 \leq C\norm{u}_{L^\infty(\C^n)},\\
 r|\dbar u_{1,r}(z)|+M_{1,r}(u_{2,r})(z)
 \leq CG_{1,4r}(u)(z),\quad z\in\C^n.
 \end{gathered}
\end{equation}
The constant $C>0$ depends only on $n$ and the fixed smooth
partition used in the construction, and is independent of $u,r,z$.
\end{lemma}
\begin{proof}
Fix $r>0$ and set $h(\zeta)=u(r\zeta)$, so that
$h\in L^\infty(\C^n)\subset L^1_{\rm loc}(\C^n,dv)$.
We use the construction in
\cite[(2-8) and (3-11)--(3-12)]{HuVirtanen2023}
with local exponent $q=1$ and auxiliary radius $2$.
For each dimension, fix the lattice centers
$\{a_j\}\subset\C^n$ and the associated smooth partition
of unity $\{\psi_j\}$.
Distinct centers have a fixed positive separation,
and the partition functions satisfy
$$
 \begin{gathered}
 \psi_j\in C_c^\infty(B(a_j,1)),\quad
 \psi_j\geq0,\quad \sum_j\psi_j=1,\quad
 \sup_j\norm{\dbar\psi_j}_{L^\infty(\C^n)}\leq C_\psi,
 \end{gathered}
$$
where $C_\psi>0$ depends only on $n$ and the fixed partition.
By \cite[Lemma 3.3, (3-2)]{HuVirtanen2023}, choose
$g_j\in\mathcal O(B(a_j,2))$ such that
$$
 M_{1,2}(h-g_j)(a_j)=G_{1,2}(h)(a_j).
$$
Since $g_j\psi_j\in C_c^\infty(B(a_j,2))$, define its
zero extension as in \eqref{E93} by
$$
 g_{j,0}(\zeta)=
 \begin{cases}
 g_j(\zeta)\psi_j(\zeta),&\zeta\in B(a_j,2),\\
 0,&\zeta\in\C^n\setminus B(a_j,2).
 \end{cases}
$$
These functions satisfy
$$
 g_{j,0}\in C_c^\infty(\C^n)\quad\text{and}\quad
 \supp g_{j,0}\subseteq\supp\psi_j\Subset B(a_j,1).
$$
To define the decomposition on all of $\C^n$, set
$$
 h_1=\sum_jg_{j,0},\quad h_2=h-h_1.
$$
To verify local finiteness, fix $\zeta\in\C^n$ and set
$$
 J_\zeta=\{j:B(\zeta,1)\cap\supp\psi_j\neq\varnothing\}
 \subseteq\{j:a_j\in B(\zeta,2)\}.
$$
We write $\#J$ for the number of elements of an index set $J$.
The lattice bound in \cite[(2-9)]{HuVirtanen2023}
gives an integer $N_n\geq1$, depending only on $n$, such that
$$
 \#J_\zeta
 \leq\#\{j:a_j\in B(\zeta,2)\}
 \leq N_n.
$$
Since $g_{j,0}=0$ on $B(\zeta,1)$ for $j\notin J_\zeta$,
we have
$$
 h_1(w)=\sum_{j\in J_\zeta}g_{j,0}(w),\quad w\in B(\zeta,1).
$$
This finite sum of smooth functions proves
$h_1\in C^\infty(\C^n)$. The regularity follows from
the partition in \cite[(3-11)]{HuVirtanen2023},
although its Lemma 3.6 states only $C^2$ regularity.

To bound the two parts, take any $j$ and $\zeta\in B(a_j,1)$.
The inclusion $B(\zeta,1)\subset B(a_j,2)$ and the submean
estimate used in the proof of
\cite[Lemma 3.3]{HuVirtanen2023} yield
\begin{align*}
 |g_j(\zeta)|
 &\leq M_{1,1}(g_j)(\zeta)
 \leq2^dM_{1,2}(g_j)(a_j)\leq2^d\bigl(M_{1,2}(h)(a_j)+G_{1,2}(h)(a_j)\bigr)\leq2^{d+1}\norm{h}_{L^\infty(\C^n)}.
\end{align*}
Here the choice of $g_j$ gives the third inequality,
and the zero holomorphic approximant in \eqref{E106}
gives the last one. Set $C_1=2^{d+1}$.
The definition of $g_{j,0}$ and the nonnegative partition
of unity therefore imply
$$
 |h_1(\zeta)|\leq\sum_j|g_{j,0}(\zeta)|
 \leq C_1\norm{h}_{L^\infty(\C^n)}\sum_j\psi_j(\zeta)
 =C_1\norm{h}_{L^\infty(\C^n)}
$$
for every $\zeta\in\C^n$.
This bound together with $h_2=h-h_1$ proves
$$
 \norm{h_1}_{L^\infty(\C^n)}
 +\norm{h_2}_{L^\infty(\C^n)}
 \leq(1+2C_1)\norm{h}_{L^\infty(\C^n)}.
$$
The derivative and remainder bounds follow from
\cite[Lemma 3.6, (3-13)]{HuVirtanen2023} applied to this
decomposition with $q=1$ and auxiliary radius $2$.
Dropping its nonnegative term $M_{1,1}(\dbar h_1)$ gives
$$
 |\dbar h_1|+M_{1,1}(h_2)\leq C_0G_{1,4}(h),
$$
where $C_0>0$ depends only on $n$ and the fixed partition.

To return to scale $r$, set $u_{j,r}(z)=h_j(z/r)$ for
$j=1,2$. The chain rule and definitions \eqref{E4} and
\eqref{E106} give
$$
 \begin{gathered}
 r\dbar u_{1,r}(z)=(\dbar h_1)(z/r),\quad
 M_{1,r}(u_{2,r})(z)=M_{1,1}(h_2)(z/r),\quad
 G_{1,4r}(u)(z)=G_{1,4}(h)(z/r).
 \end{gathered}
$$
For the last identity, $g\mapsto g(r\cdot)$ maps
$\mathcal O(B(z,4r))$ bijectively onto
$\mathcal O(B(z/r,4))$, and the substitution $w=r\zeta$
preserves the normalized averages.
Since $\norm{u_{j,r}}_{L^\infty(\C^n)}
=\norm{h_j}_{L^\infty(\C^n)}$ for $j=1,2$, the norm bound
and derivative estimate for $h_1,h_2$ prove \eqref{E15}
with $C=\max\{C_0,1+2C_1\}$, independently of $r$.
The proof is complete.
\end{proof}

\subsection{Averages over signs}\label{sec:sign-averages}
The proof of Lemma \ref{lem:phase} uses finite averages to
separate a diagonal coefficient from the remaining terms.

\begin{definition}\label{def:sign-average}
Let $J\subset\Z^d$ be finite. We write
$$
 \{-1,1\}^J
 =\bigl\{\varepsilon=(\varepsilon_j)_{j\in J}:
          \varepsilon_j\in\{-1,1\}\text{ for every }j\in J\bigr\}.
$$
Each vector $\varepsilon$ specifies one sign for each index
$j\in J$, giving $2^{\#J}$ possible sign vectors.
Given $u:\{-1,1\}^J\to\C$, its average over the $2^{\#J}$
equally weighted sign choices is defined by
\begin{equation}\label{E45}
 \mathbb E u(\varepsilon)
 =2^{-\#J}\sum_{\varepsilon\in\{-1,1\}^J}u(\varepsilon).
\end{equation}
Given $i\in J$, we use $\mathbb E_{\neq i}u$ for the uniform
average over the $2^{\#J-1}$ choices of the remaining signs,
with $\varepsilon_i$ held fixed. This average remains a function
of $\varepsilon_i$.
\end{definition}

These averages are expectations on finite uniform probability
spaces \cite[Example 1.1.2 and Section 1.6]{Durrett2019}.
They are linear, preserve constants, and preserve inequalities
between real-valued functions \cite[Theorem 1.6.1, p. 25]{Durrett2019}.
For complex-valued functions, linearity follows by applying
the real-valued result to real and imaginary parts, as in
\cite[Section 3.3.1]{Durrett2019}. In particular, both averages
commute with finite sums. By \eqref{E45}, averaging
$\mathbb E_{\neq i}u$ over $\varepsilon_i=1,-1$ recovers
$\mathbb E u$.

Let $u:\{-1,1\}^J\to\C$ and $1\leq p<\infty$.
The triangle inequality gives $|\mathbb E u|\leq\mathbb E|u|$.
Applying Jensen's inequality \cite[Theorem 1.6.2]{Durrett2019}
to the nonnegative variable $|u|$ and the convex function
$x\mapsto x^p$ on $[0,\infty)$, for each of the two averages, yields
\begin{equation}\label{E47}
 |\mathbb E u|^p\leq\mathbb E|u|^p\quad\text{and}\quad
 |\mathbb E_{\neq i}u|^p\leq\mathbb E_{\neq i}|u|^p.
\end{equation}
The second inequality holds for each $i\in J$ and each fixed
value of $\varepsilon_i$. The following lemma identifies the term retained when one sign
is fixed and all remaining signs are averaged.
\begin{lemma}\label{L2}
Let $J\subset\Z^d$ be a nonempty finite set, let $i\in J$,
and let $c_j\in\C$, $j\in J$, be fixed coefficients.
With $\varepsilon_i=\sigma\in\{-1,1\}$ held fixed,
Definition \ref{def:sign-average} gives
\begin{equation}\label{E46}
 \mathbb E_{\neq i}\bigg(\sum_{j\in J}\varepsilon_jc_j\bigg)
 =\frac1{2^{\#J-1}}
   \sum_{\substack{\varepsilon\in\{-1,1\}^J\\
                   \varepsilon_i=\sigma}}
   \bigg(\sigma c_i+
         \sum_{\substack{j\in J\\j\neq i}}\varepsilon_jc_j\bigg)=\sigma c_i.
\end{equation}
\end{lemma}
\begin{proof}
Fix $\varepsilon_i=\sigma\in\{-1,1\}$.
If $J=\{i\}$, the average equals $\sigma c_i$ directly.
Suppose $\#J\geq2$.
Among the $2^{\#J-1}$ sign choices with $\varepsilon_i=\sigma$,
each remaining sign takes each value $1$ and $-1$
exactly $2^{\#J-2}$ times.
By Definition \ref{def:sign-average}, we obtain
\begin{align*}
\mathbb E_{\neq i}\bigg(\sum_{j\in J}\varepsilon_jc_j\bigg)
&=\frac1{2^{\#J-1}}
  \sum_{\substack{\varepsilon\in\{-1,1\}^J\\
                  \varepsilon_i=\sigma}}
  \bigg(\sigma c_i+\sum_{\substack{j\in J\\j\neq i}}
                         \varepsilon_jc_j\bigg)\\
&=\sigma c_i+\frac1{2^{\#J-1}}
  \sum_{\substack{j\in J\\j\neq i}}
  \bigl(2^{\#J-2}c_j-2^{\#J-2}c_j\bigr)\\
&=\sigma c_i=\varepsilon_i c_i.
\end{align*}
This proves \eqref{E46}.
\end{proof}

\section{The bounded tests and necessary conditions}\label{sec:necessary}
We derive the necessary conditions in Theorem \ref{thm:main}
by passing from bounded tests to local analytic distance
and then to derivatives of mollified symbols.
We also obtain a constancy criterion from the conjugate test.

\subsection{Bounded tests and local analytic distance}
We first use the Schatten norm for $p\geq1$ to pass from
convergence for fixed tests to a common bound for the testing norm
$m_{p,t}(f)$ in \eqref{E5}.

\begin{lemma}\label{lem:baire}
If $1\leq p<\infty$ and $f\in\mathcal A_p$, there are fixed
positive numbers $t_0$ and $a$ such that $m_{p,t}(f)\leq a$
for $0<t<t_0$. They may depend on $f,p$ and the fixed weight
$\varphi$, but not on the second symbol $g$ or on $t$.
\end{lemma}
\begin{proof}
Fix $1\leq p<\infty$ and $f\in\mathcal A_p$. For each
integer $j\geq1$, set
$$
 E_j=\bigg\{g\in L^\infty(\C^n):
 \sup_{0<t<1/j}\norm{D_t(f,g)}_{\mathcal S_p}\leq j\bigg\}.
$$
For each fixed $g\in L^\infty(\C^n)$, \eqref{eq:property}
provides $\tau_g>0$ such that
$\norm{D_t(f,g)}_{\mathcal S_p}\leq1$ for $0<t<\tau_g$.
Choose an integer $j\geq1$ with $1/j<\tau_g$. Then
$$
 \sup_{0<t<1/j}\norm{D_t(f,g)}_{\mathcal S_p}\leq1\leq j,
$$
so $g\in E_j$ and $L^\infty(\C^n)=\bigcup_{j=1}^\infty E_j$.

To apply Baire's theorem to this cover, we show that each
$E_j$ is closed. Fix $j\geq1$ and let
$g_\ell\in E_j$ converge to $g$ in $L^\infty(\C^n)$ as
$\ell\rightarrow\infty$. At each fixed $t\in(0,1/j)$,
\eqref{E1}, \eqref{E2}, and $\norm{P_t}=1$ give
$$
 \norm{D_t(f,g_\ell)-D_t(f,g)}
 \leq2\norm{f}_{L^\infty(\C^n)}\norm{g_\ell-g}_{L^\infty(\C^n)}\rightarrow0.
$$
Each $D_t(f,g_\ell)$ is compact, so its operator norm limit
$D_t(f,g)$ is compact by \cite[Theorem 1.18]{Zhu1990}.
By the approximation number formula in
\cite[Theorem 1.34]{Zhu1990}, each $k\geq1$ satisfies
$$
 |s_k(D_t(f,g_\ell))-s_k(D_t(f,g))|
 \leq\norm{D_t(f,g_\ell)-D_t(f,g)}
 \longrightarrow0\quad(\ell\rightarrow\infty).
$$
Fatou's lemma applied to the $p$-th powers of these singular
values yields
$$
 \norm{D_t(f,g)}_{\mathcal S_p}^p
 \leq\liminf_{\ell\rightarrow\infty}
              \norm{D_t(f,g_\ell)}_{\mathcal S_p}^p
 \leq j^p.
$$
Since $t\in(0,1/j)$ was arbitrary, we have $g\in E_j$,
so $E_j$ is closed.

The sets $E_j$ form a countable closed cover of the Banach
space $L^\infty(\C^n)$. By Baire's theorem
\cite[Theorem 2.2]{Rudin1991}, at least one of them has nonempty
interior. We can therefore choose an integer $j_0\geq1$,
a function $g_0\in L^\infty(\C^n)$, and $\delta>0$ such that
$$
 \left\{g\in L^\infty(\C^n):
              \norm{g-g_0}_{L^\infty(\C^n)}<\delta\right\}
 \subseteq E_{j_0}.
$$
For $\norm{h}_{L^\infty(\C^n)}\leq1$, we have
$$
 \norm{g_0\pm(\delta/2)h-g_0}_{L^\infty(\C^n)}
 \leq\delta/2<\delta,
$$
so $g_0\pm(\delta/2)h\in E_{j_0}$.
Linearity in the second symbol and the triangle inequality
for $p\geq1$ now give
\begin{align*}
 \delta\norm{D_t(f,h)}_{\mathcal S_p}
 &=\norm{D_t(f,g_0+(\delta/2)h)
          -D_t(f,g_0-(\delta/2)h)}_{\mathcal S_p}\\
 &\leq\norm{D_t(f,g_0+(\delta/2)h)}_{\mathcal S_p}
      +\norm{D_t(f,g_0-(\delta/2)h)}_{\mathcal S_p}\\
 &\leq2j_0
 \end{align*}
for $0<t<1/j_0$.
Taking the supremum over $\norm{h}_{L^\infty(\C^n)}\leq1$
and using \eqref{E5}, we obtain
$$
 m_{p,t}(f)\leq\frac{2j_0}{\delta}
 \quad\text{for }0<t<\frac1{j_0}.
$$
Thus, $t_0=1/j_0$ and $a=2j_0/\delta$ have the required
properties and depend only on the fixed $f,p,\varphi$.
The proof is complete.
\end{proof}

Lemma \ref{lem:baire} gives a common bound for the testing norm
$m_{p,t}(f)$ in \eqref{E5}. We next relate this norm to local
analytic distance by choosing tests at each scale on the
disjoint balls of Lemma \ref{lem:lattice} (a).

\begin{lemma}\label{lem:phase}
Let $1\leq p<\infty$ and $f\in L^\infty(\C^n)$.
For each $r>0$, the function $G_{1,\rho r}(\overline{f})$
is continuous and satisfies
\begin{equation}\label{eq:phase}
 \norm{G_{1,\rho r}(\overline{f})}_{L^{p}(\R^d,dv)}
 \leq C_p r^{d/p}m_{p,r^2}(f),\qquad r>0.
\end{equation}
Here $\rho$ is the radius in Lemma \ref{L1}, and $C_p>0$
depends only on $n,m,M,p$ and a fixed lattice spacing $L>2\rho$.
\end{lemma}
\begin{proof}
We first prove continuity to ensure measurability when integrating
over the lattice translation below.
Fix $r>0$ and put $B=B(0,\rho r)$.
Translation in
\eqref{E106} gives
$$
 G_{1,\rho r}(\overline{f})(z)
 =\inf_{h\in\mathcal O(B)}\frac1{|B|}
   \int_B|\overline{f}(z+w)-h(w)| dv(w).
$$
For $h\in\mathcal O(B)$, the triangle inequality yields
\begin{align*}
 G_{1,\rho r}(\overline{f})(z)
 &\leq\frac1{|B|}\int_B
       |\overline{f}(z+w)-\overline{f}(z_0+w)| dv(w)\\
 &\quad+\frac1{|B|}\int_B
       |\overline{f}(z_0+w)-h(w)| dv(w).
\end{align*}
Taking the infimum over $h$ and applying the same argument
with $z,z_0$ interchanged, we obtain
\begin{equation}\label{E110}
\begin{aligned}
 &|G_{1,\rho r}(\overline{f})(z)-G_{1,\rho r}(\overline{f})(z_0)|\leq\frac1{|B|}\int_B
       |\overline{f}(z+w)-\overline{f}(z_0+w)| dv(w)
       \rightarrow0\quad(z\rightarrow z_0).
\end{aligned}
\end{equation}
To justify the limit in \eqref{E110}, fix $z_0\in\C^n$
and set $g=\overline{f}\mathbf1_{B(z_0,\rho r+1)}$. Then we have
$$\norm{g}_{L^1(\R^d,dv)}
\leq |B(z_0,\rho r+1)|\norm{f}_{L^\infty(\C^n)}<\infty.$$
For $|z-z_0|<1$, the substitution $x=z_0+w$ and
\cite[Lemma 4.3, p. 114]{Brezis2011}, applied with $N=d=2n$
and $q=1$ to the real and imaginary parts of $g$, give
\begin{align*}
 \frac1{|B|}\int_B
 |\overline{f}(z+w)-\overline{f}(z_0+w)| dv(w)
 &\leq\frac1{|B|}
       \norm{g(\cdot+z-z_0)-g}_{L^1(\R^d,dv)}\rightarrow0
\end{align*}
as $z\rightarrow z_0$. This proves \eqref{E110} and the
continuity of $G_{1,\rho r}(\overline{f})$.

To prove \eqref{eq:phase}, we use the normalized kernel
localization in \cite[(4.7)]{HuVirtanen2022}, keeping the
volume factors in $\C^n$. The additional phase and sign choices
below estimate the local analytic distance with exponent one.
If the right-hand side of \eqref{eq:phase} is infinite,
the estimate holds. We may therefore assume that
$m_{p,t}(f)<\infty$, where $t=r^2$.
Fix $b\in[0,Lr)^d$ and take the lattice points
$a_j=b+Lrj$, $j\in\Z^d$, as in Lemma \ref{lem:lattice}.
Let $k_j=k_{t,a_j}$ be the normalized reproducing kernel
at $a_j$ defined before \eqref{E10}, and put $B_j=B(a_j,\rho r)$.
Note that the balls $B_j$ are pairwise disjoint by
Lemma \ref{lem:lattice} (a). The lower bound \eqref{E105}
shows that $k_j$ has no zeros on $B_j$, so we may define
$$
 u_j=H_{\overline{f}}^{(t)}k_j\quad\text{and}\quad
 h_j=\frac{P_t(\overline{f} k_j)}{k_j}
$$ on $B_j$.
The Hankel image $u_j\in L_t^2\ominus\cF$ is defined on
$\C^n$, while the quotient $h_j\in\mathcal O(B_j)$ satisfies
$u_j=k_j(\overline{f}-h_j)$ on $B_j$.

To remove the phase of $k_j\overline{u_j}$ in the local
integrals, set $c_j(w)=k_j(w)\overline{u_j(w)}$ and define
the measurable symbol $g_j$ by
\begin{equation}\label{E29}
 g_j(w)=
 \begin{cases}
 \overline{c_j(w)}/|c_j(w)|,&w\in B_j,\ c_j(w)\neq0,\\
 0,&\text{otherwise}.
 \end{cases}
\end{equation}
These functions satisfy $|g_j|\leq\mathbf1_{B_j}$ on $\C^n$
and $g_jc_j=|c_j|$ on $B_j$. Fix a finite set $J\subset\Z^d$.
Using the notation of Definition \ref{def:sign-average},
for each sign vector $\varepsilon=(\varepsilon_j)_{j\in J}
\in\{-1,1\}^J$, we define
$$
 g_\varepsilon(w)=\sum_{j\in J}\varepsilon_jg_j(w)
 =\begin{cases}
 \varepsilon_j\overline{c_j(w)}/|c_j(w)|,
 &w\in B_j,\ j\in J,\ c_j(w)\neq0,\\
 0,&\text{otherwise}.
 \end{cases}
$$
The piecewise expression follows from \eqref{E29} and
the disjointness of the balls. In particular, $g_\varepsilon$ satisfies $|g_\varepsilon(w)|\leq1$
for every $w\in\C^n$, so
$\norm{g_\varepsilon}_{L^\infty(\C^n)}\leq1$.

We relate these local integrals to the diagonal entries of
$D_t(f,g_\varepsilon)$ by setting
$$c_{ij}=\int_{B_j}g_jk_i\overline{u_i}d\mu_t$$ for $i,j\in J$.
Using $u_i\perp\cF$ in \eqref{eq:factorization}, we compute
 \begin{align*}
 \langle D_t(f,g_\varepsilon)k_i,k_i\rangle_t
 &=-\langle H_{g_\varepsilon}^{(t)}k_i,u_i\rangle_t=-\int_{\C^n} g_\varepsilon k_i\overline{u_i}d\mu_t
 =-\sum_{j\in J}\varepsilon_j c_{ij}.
 \end{align*}
With $i\in J$ and $\varepsilon_i\in\{-1,1\}$ fixed,
Lemma \ref{L2} removes the terms with $j\neq i$ from the
average and gives
$$
 \mathbb E_{\neq i}\bigg(\sum_{j\in J}\varepsilon_jc_{ij}\bigg)
 =\varepsilon_ic_{ii}.
$$
Applying the second inequality in \eqref{E47} for each fixed
$\varepsilon_i$ and then averaging over its two values as in
\eqref{E45}, we obtain
\begin{equation}\label{E30}
 \begin{aligned}
 \mathbb E\bigg|\sum_{j\in J}\varepsilon_jc_{ij}\bigg|^p
 &=\frac12\sum_{\sigma\in\{-1,1\}}
   \left.\bigg(\mathbb E_{\neq i}
   \bigg|\sum_{j\in J}\varepsilon_jc_{ij}\bigg|^p\bigg)
   \right|_{\varepsilon_i=\sigma}\\
 &\geq |c_{ii}|^p
 =\left(\int_{B_i}|k_iu_i|d\mu_t\right)^p.
 \end{aligned}
\end{equation}
The last equality uses \eqref{E29} in the form
$g_ik_i\overline{u_i}=|k_iu_i|$ on $B_i$.
Let $C_{p,0}>0$ denote the constant in \eqref{eq:lattice-diagonal}.
Fix $\varepsilon\in\{-1,1\}^J$ and set
$A_\varepsilon=D_t(f,g_\varepsilon)$. Since
$\norm{g_\varepsilon}_{L^\infty(\C^n)}\leq1$, definition \eqref{E5}
gives $\norm{A_\varepsilon}_{\mathcal S_p}\leq m_{p,t}(f)<\infty$,
so $A_\varepsilon\in\mathcal S_p(\cF)$.
Since $J\subset\Z^d$ and all summands are nonnegative,
we enlarge the finite sum and apply \eqref{eq:lattice-diagonal}
in Lemma \ref{lem:lattice} (b) to obtain
\begin{align*}
 \sum_{i\in J}|\langle A_\varepsilon k_i,k_i\rangle_t|^p
 &\leq\sum_{i\in\Z^d}|\langle A_\varepsilon k_i,k_i\rangle_t|^p\leq C_{p,0}\norm{A_\varepsilon}_{\mathcal S_p}^p.
 \end{align*}
Here $C_{p,0}$ is independent of $J$ and $\varepsilon$.
Using linearity and monotonicity of $\mathbb E$, as recorded
in Section \ref{sec:sign-averages}, we sum \eqref{E30}
over $i\in J$ and average the preceding bound to obtain
 \begin{align}\label{E147}
 \sum_{i\in J}\left(\int_{B_i}|k_iu_i|d\mu_t\right)^p
 &\leq\mathbb E\sum_{i\in J}
   |\langle D_t(f,g_\varepsilon)k_i,k_i\rangle_t|^p\leq C_{p,0}\mathbb E\norm{D_t(f,g_\varepsilon)}_{\mathcal S_p}^p
 \leq C_{p,0} m_{p,t}^p(f).
 \end{align}
 
To bound local analytic distance by these integrals, we use
$h_i\in\mathcal O(B_i)$ as a competitor in \eqref{E106}.
Let $C_0>0$ be the lower bound constant in \eqref{E105}, and set $C_1=C_0|B(0,1)|\rho^d>0$.
Writing $d\mu_t$ as in \eqref{E82}, we use \eqref{E105}
and $u_i=k_i(\overline{f}-h_i)$ to estimate
 \begin{align}\label{E148}
 \int_{B_i}|k_iu_i|d\mu_t
 &=\int_{B_i}|k_i|^2|\overline{f}-h_i|d\mu_t\geq C_0r^{-d}\int_{B_i}|\overline{f}-h_i| dv
 \geq C_1G_{1,\rho r}(\overline{f})(a_i).
 \end{align}
The last step uses $|B_i|=|B(0,1)|\rho^dr^d$ and the definition of
$G_{1,\rho r}$. Take $J_N=\Z^d\cap[-N,N]^d$ for
$N=1,2,\ldots$. Estimates \eqref{E147} and \eqref{E148}
hold for every $J_N$,
with constants independent of $N$. As these finite sets
increase to $\Z^d$, we pass to the limit in the nonnegative
sums to obtain
\begin{equation}\label{E65}
 \sum_{j\in\Z^d}G_{1,\rho r}^{p}(\overline{f})(b+Lrj)
 =\lim_{N\to\infty}\sum_{j\in J_N}G_{1,\rho r}^{p}(\overline{f})(b+Lrj)
 \leq C_{p,0}C_1^{-p}m_{p,t}^p(f).
\end{equation}
Set $Q=[0,Lr)^d$ and $F(z)=G_{1,\rho r}^{p}(\overline{f})(z)$.
By \eqref{E110}, $F$ is measurable. Since $F\geq0$ and
the sets $Q+Lrj$, $j\in\Z^d$, partition $\R^d$,
Tonelli's theorem and \eqref{E65} yield
\begin{align*}
 \int_{\R^d}G_{1,\rho r}^{p}(\overline{f})(z) dv(z)
 &=\sum_{j\in\Z^d}\int_{Q+Lrj}F(z) dv(z)=\int_{[0,Lr)^d}\sum_{j\in\Z^d}
        G_{1,\rho r}^{p}(\overline{f})(b+Lrj) dv(b)\\ 
 &\leq C_{p,0}C_1^{-p}(Lr)^d m_{p,t}^p(f).
 \end{align*}
Taking the $p$-th root proves \eqref{eq:phase} with
$C_p=C_{p,0}^{1/p}C_1^{-1}L^{d/p}$, which has the stated
dependence. The proof is complete.
\end{proof}

\subsection{Mollification and derivative estimates}
To obtain the derivative estimates, we mollify at a radius
controlled by the fixed constant $\rho>0$ from Lemma \ref{L1},
which depends only on $n,m,M$. Under $\C^n\cong\R^d$,
smoothness and convolution are understood in real coordinates.
For each dimension, fix a nonnegative function
$\eta_0\in C_c^\infty(B(0,1))$ with $\int_{\R^d}\eta_0 dv=1$,
and set
$$
 \eta(z)=\rho^{-d}\eta_0(z/\rho),\quad z\in\C^n.
$$
Then $\eta\in C_c^\infty(B(0,\rho))$ is nonnegative and has
unit mass. For $r>0$ and $u\in L^1_{\rm loc}(\R^d,dv)$, define
\begin{equation}\label{E16}
 \eta_r(z)=r^{-d}\eta(z/r)\quad\text{and}\quad
 (u*\eta_r)(z)=\int_{\R^d}u(z-w)\eta_r(w) dv(w).
\end{equation}
Convolution of vector-valued functions is taken componentwise.
This scaling preserves positivity and unit mass and gives
\begin{equation}\label{E53}
 \supp\eta_r=r\supp\eta\subset B(0,\rho r),\quad r>0.
\end{equation}

For $u_r,u\in L^1_{\rm loc}(\R^d,dv)$, the convergence
$u_r\rightarrow u$ in this space as $r\rightarrow0^+$ means that
$$
 \int_K|u_r-u| dv\rightarrow0\quad(r\rightarrow0^+)
$$
for every fixed compact set $K\subseteq\R^d$.
This local convergence implies convergence in distributions,
meaning that $$\int_{\R^d}(u_r-u)h dv\rightarrow0$$
as $r\rightarrow0^+$ for every $h\in C_c^\infty(\R^d)$.
Differentiation preserves convergence in distributions
\cite[Section 6.16 and Theorem 6.17]{Rudin1991}.
The following lemma provides smooth bounded approximants
converging in $L^1_{\rm loc}(\R^d,dv)$ for use in these
distributional limits.

\begin{lemma}\label{L6}
Let $u\in L^\infty(\C^n)$ and set $u_r=u*\eta_r$ for $r>0$,
with $\eta_r$ as in \eqref{E16}.
The functions $u_r$ belong to $C^\infty(\R^d)$ and satisfy
\begin{equation}\label{E50}
 \begin{gathered}
 \norm{u_r}_{L^\infty(\C^n)}\leq\norm{u}_{L^\infty(\C^n)}
 \quad(r>0),\quad
 u_r\rightarrow u\quad\text{in }L^1_{\rm loc}(\R^d,dv)
 \quad(r\rightarrow0^+).
 \end{gathered}
\end{equation}
\end{lemma}
\begin{proof}
Local integrability of $u$ and $\eta_r\in C_c^\infty(\R^d)$
allow us to apply \cite[Proposition 4.20, p. 107]{Brezis2011},
with $N=d=2n$, to the real and imaginary parts of $u$.
The convolution $u_r=u*\eta_r$ therefore belongs to $C^\infty(\R^d)$.
Since $\eta_r\geq0$ and $\int_{\R^d}\eta_r dv=1$, we have
$$
 |u_r(z)|
 \leq\int_{\R^d}|u(z-y)|\eta_r(y) dv(y)
 \leq\norm{u}_{L^\infty(\C^n)},\qquad z\in\R^d.
$$
This proves the norm bound in \eqref{E50}.

For the convergence, fix a compact set $K\subset\R^d$ and
choose $R>\rho$ with $K\subset B(0,R-\rho)$.
Set $v=u\mathbf1_{B(0,R)}$. The estimate
$\norm{v}_{L^1(\R^d,dv)}
\leq |B(0,R)|\norm{u}_{L^\infty(\C^n)}<\infty$
allows us to use translation continuity for $v$.
For $z\in K$, $w\in\supp\eta_r$, and $0<r<1$,
\eqref{E53} gives
$$
 |z-w|\leq |z|+|w|<R-\rho+\rho r<R.
$$
Since $K\subset B(0,R)$, we have $u=v$ on $K$.
The inclusion $z-w\in B(0,R)$ also gives
\begin{align*}
 u_r(z)
 &=\int_{B(0,\rho r)}u(z-w)\eta_r(w) dv(w)=\int_{B(0,\rho r)}v(z-w)\eta_r(w) dv(w)
 =(v*\eta_r)(z)
\end{align*}
for $z\in K$ and $0<r<1$. 
Using unit mass, the triangle inequality, and Tonelli's theorem,
we obtain
\begin{align*}
 \norm{u_r-u}_{L^1(K,dv)}
 &=\int_K\left|\int_{B(0,\rho r)}
       \bigl(v(z-y)-v(z)\bigr)\eta_r(y) dv(y)\right| dv(z)\\
 &\leq\int_{B(0,\rho r)}\eta_r(y)
       \norm{v(\cdot-y)-v}_{L^1(\R^d,dv)} dv(y)\\
 &\leq\sup_{|y|<\rho r}
       \norm{v(\cdot-y)-v}_{L^1(\R^d,dv)}
       \longrightarrow0\quad(r\rightarrow0^+).
\end{align*}
The limit follows from \cite[Lemma 4.3, p. 114]{Brezis2011},
applied with $N=d=2n$ and $q=1$ to the real and imaginary
parts of $v$. This proves \eqref{E50}.
\end{proof}

To pass to limits in derivatives of these approximants,
we next estimate $\dbar u_r$ by the local analytic distance of $u$.

\begin{lemma}\label{lem:mollifier}
Let $\eta_r$ be defined by \eqref{E16}.
\begin{enumerate}[label=\textup{(\alph*)}]
\item
For $u\in L^\infty(\C^n)$, set $u_r=u*\eta_r$.
There is a positive constant $C$, depending only on
$n,\rho$ and $\eta$, such that
\begin{equation}\label{eq:mollifier}
 |\dbar u_r(z)|
 \leq Cr^{-1}G_{q,\rho r}(u)(z),
 \quad z\in\C^n
\end{equation}
for every $q\in\{1,2\}$ and $r>0$.

\item
For each $1\leq p<\infty$, there is a
positive constant $C_p$, depending only on
$n,m,M,p,\rho$, $\eta$, and a fixed lattice spacing
$L>2\rho$ as in Lemma \ref{lem:phase}, such that every
$f\in L^\infty(\C^n)$ satisfies
\begin{equation}\label{eq:derivative-necessary}
 \norm{\dbar(\overline{f}*\eta_r)}_{L^p(\R^d,dv)}
 \leq C_p r^{d/p-1}m_{p,r^2}(f),
 \quad r>0.
\end{equation}
\end{enumerate}
\end{lemma}
\begin{proof}
For \textup{(a)}, we fix $r>0$ and $z\in\C^n$ and choose
$h\in\mathcal O(B(z,\rho r))$ to express $\dbar u_r(z)$
in terms of $u-h$.
By \eqref{E53}, the function
$w\mapsto h(w)\eta_r(z-w)$ belongs to
$C_c^\infty(B(z,\rho r))$.
Its zero extension, as in \eqref{E93}, is the function
$h_0\in C_c^\infty(\R^d)$ given by
$$
 h_0(w)=
 \begin{cases}
 h(w)\eta_r(z-w),&w\in B(z,\rho r),\\
 0,&w\in\R^d\setminus B(z,\rho r).
 \end{cases}
$$
For $1\leq j\leq n$, write $w_j=y_{2j-1}+iy_{2j}$ and set
$s=y_{2j-1}$.
We denote the remaining real coordinates by $y'\in\R^{d-1}$
and regard $h_0(s,y')$ as a function of $s$ with $y'$ fixed.
Then $dv(w)=dsdy'$, where $dy'$ is Lebesgue measure on
$\R^{d-1}$. 

Choose $a>0$ with $\supp h_0\subset(-a,a)^d$.
Fubini's theorem and the fundamental theorem of calculus give
\begin{align}\label{E144}
\int_{B(z,\rho r)}\partial_{y_{2j-1}}
  \bigl(h(w)\eta_r(z-w)\bigr) dv(w)
&=\int_{\R^{d-1}}\int_{-a}^{a}
  \partial_s h_0(s,y')dsdy'\nonumber\\
&=\int_{\R^{d-1}}
  \bigl(h_0(a,y')-h_0(-a,y')\bigr)dy'=0.
\end{align}
Expanding the derivative in the integrand, we obtain
$$
 \int_{B(z,\rho r)}h(w)\partial_{y_{2j-1}}\eta_r(z-w) dv(w)
 =-\int_{B(z,\rho r)}\partial_{y_{2j-1}}h(w)\eta_r(z-w) dv(w).
$$
Repeating the calculation for $y_{2j}$ and using the two
identities together with \eqref{E13} and the chain rule,
we obtain
 \begin{align*}
 \int_{B(z,\rho r)} h(w)\partial_{\overline{z}_j}\eta_r(z-w) dv(w)
 &=-\int_{B(z,\rho r)} h(w)\partial_{\overline{w}_j}\eta_r(z-w) dv(w)\\
 &=\int_{B(z,\rho r)}\partial_{\overline{w}_j}h(w)\eta_r(z-w) dv(w)\\
 &=0.
 \end{align*}
Here $\partial_{\overline{w}_j}h=0$ because $h$ is holomorphic.
Since $u$ is bounded and $\eta_r$ is smooth with compact
support, we may differentiate under the integral sign.
Subtracting the zero integral above yields
$$
 \partial_{\overline{z}_j}u_r(z)
 =\int_{B(z,\rho r)}(u(w)-h(w))
                         \partial_{\overline{z}_j}\eta_r(z-w) dv(w).
$$
Taking the components together according to \eqref{E107},
we have
$$
 \dbar u_r(z)=\int_{B(z,\rho r)}(u(w)-h(w))
                    \dbar\eta_r(z-w) dv(w).
$$
Set $c_0=\sup_{x\in\R^d}|\dbar\eta(x)|\in(0,\infty)$.
By \eqref{E16} and the chain rule,
$$
 \dbar\eta_r(z)=r^{-d-1}(\dbar\eta)(z/r),\quad
 |\dbar\eta_r(z)|\leq c_0r^{-d-1}.
$$
The vector integral formula and this kernel bound, followed
by H\"older's inequality and $|B(z,\rho r)|=r^d|B(0,\rho)|$,
yield
 \begin{align*}
 |\dbar u_r(z)|
 &\leq c_0r^{-d-1}|B(z,\rho r)|^{1-1/q}
          \left(\int_{B(z,\rho r)}|u-h|^q dv\right)^{1/q}\\
 &=c_0|B(0,\rho)|r^{-1}
       \left(\frac1{|B(z,\rho r)|}
       \int_{B(z,\rho r)}|u-h|^q dv\right)^{1/q}.
 \end{align*}
Since $h\in\mathcal O(B(z,\rho r))$ was arbitrary, taking
the infimum over $h$ and using \eqref{E106} proves
\eqref{eq:mollifier} with $C=c_0|B(0,\rho)|$.

For \textup{(b)}, let $C_{p,0}>0$ be the constant in \eqref{eq:phase}.
Applying \eqref{eq:mollifier} with $u=\overline{f}$ and $q=1$,
taking the $L^p(\R^d,dv)$ norm, and then using
\eqref{eq:phase}, we obtain
 \begin{align*}
 \norm{\dbar(\overline{f}*\eta_r)}_{L^p(\R^d,dv)}
 &\leq Cr^{-1}
       \norm{G_{1,\rho r}(\overline{f})}_{L^p(\R^d,dv)}\leq CC_{p,0}r^{d/p-1}m_{p,r^2}(f).
 \end{align*}
This proves \eqref{eq:derivative-necessary} with
$C_p=CC_{p,0}$. The proof is complete.
\end{proof}

\subsection{Necessary conditions and the conjugate test}
In the following proposition, we combine \eqref{E50} with
\eqref{eq:derivative-necessary}. The factor $r^{d/p-1}$
determines the different conclusions for $p<d$ and $p=d$.

\begin{proposition}\label{prop:necessity}
If $1\leq p<d$ and $f\in\mathcal A_p$, then $f$ is constant almost
everywhere. If $f\in\mathcal A_d$, then $|\nabla f|\in L^d(\R^d,dv)$.
\end{proposition}
\begin{proof}
Suppose $1\leq p\leq d$ and $f\in\mathcal A_p$.
Fix $a,t_0>0$ as in Lemma \ref{lem:baire}, and put
$r_0=\sqrt{t_0}$ and $u_r=\overline{f}*\eta_r$, with $\eta_r$
as in \eqref{E16}. 
For $0<r<r_0$, we have $0<t=r^2<t_0$, with $t\rightarrow0^{+}$
as $r\rightarrow0^{+}$. The local approximation in \eqref{E50}, applied to
$u=\overline{f}$, gives
\begin{equation}\label{E66}
 \norm{u_r-\overline{f}}_{L^1(K,dv)}\longrightarrow0
 \quad(r\rightarrow0^{+})
\end{equation}
for every compact set $K\subset\R^d$.
We use this convergence to pass to the limit in the
integration by parts identities for the smooth functions $u_r$.

For $1\leq p<d$, combining \eqref{eq:derivative-necessary}
with the bound $m_{p,r^2}(f)\leq a$ from Lemma \ref{lem:baire}
gives
\begin{align*}
 \norm{\dbar u_r}_{L^p(\R^d,dv)}
 &\leq C_p r^{d/p-1}m_{p,r^2}(f)\leq C_p a r^{d/p-1}\rightarrow0
 \quad(r\rightarrow0^{+})
 \end{align*}
for $0<r<r_0$, where $C_p>0$ is the constant in \eqref{eq:derivative-necessary}.
The limit follows from $d/p-1>0$ and the independence of
$C_p,a$ from $r$.
Let $\psi\in C_c^\infty(\R^d)$ and $1\leq j\leq n$.
Write $p'=p/(p-1)$ when $p>1$ and $p'=\infty$ when $p=1$.
Since $u_r\psi\in C_c^\infty(\R^d)$, we apply the
integration by parts calculation in \eqref{E144} to this
product in each real coordinate and use \eqref{E13} to obtain
$$
 \int_{\R^d}u_r\partial_{\overline{z}_j}\psi dv
 =-\int_{\R^d}\partial_{\overline{z}_j}u_r\psi dv.
$$
Combining this identity with \eqref{E66} on $K=\supp\psi$
and H\"older's inequality, we obtain
 \begin{align}\label{E67}
 \left|\int_{\R^d}\overline{f}\partial_{\overline{z}_j}\psi dv\right|
 &=\lim_{r\rightarrow0^{+}}
   \left|\int_{\R^d}u_r\partial_{\overline{z}_j}\psi dv\right|=\lim_{r\rightarrow0^{+}}
   \left|\int_{\R^d}\partial_{\overline{z}_j}u_r\psi dv\right|\nonumber\\
 &\leq\norm{\psi}_{L^{p'}(\R^d,dv)}
       \lim_{r\rightarrow0^{+}}\norm{\dbar u_r}_{L^p(\R^d,dv)}=0.
 \end{align}
The integral is therefore zero for every $\psi\in C_c^\infty(\R^d)$
and $1\leq j\leq n$. By \eqref{E49}, this means that
$\dbar\overline{f}=0$ in distributions.

To deduce constancy, fix $\varepsilon>0$.
For $z\in\C^n$ and $1\leq j\leq n$, \eqref{E53} gives
$\eta_\varepsilon(z-\cdot)\in C_c^\infty(\R^d)$.
Since $f$ is bounded, we differentiate the convolution
under the integral sign and use the chain rule and
\eqref{E49} to obtain
\begin{equation}\label{E213}
\begin{aligned}
 \partial_{\overline{z}_j}u_\varepsilon(z)
 &=\int_{\R^d}\overline{f}(w)
       \partial_{\overline{z}_j}\eta_\varepsilon(z-w) dv(w)=-\int_{\R^d}\overline{f}(w)
       \partial_{\overline{w}_j}\eta_\varepsilon(z-w) dv(w)\\
 &=\left\langle\partial_{\overline{w}_j}\overline{f},
       \eta_\varepsilon(z-\cdot)\right\rangle=0.
\end{aligned}
\end{equation}
The last equality uses $\dbar\overline{f}=0$ in distributions.
Equations \eqref{E213} and \eqref{E50} give
$$
 \dbar u_\varepsilon=0\quad\text{and}\quad
 \norm{u_\varepsilon}_{L^\infty(\C^n)}
 \leq\norm{f}_{L^\infty(\C^n)}.
$$
Since $u_\varepsilon$ is smooth by Lemma \ref{L6},
\eqref{E13}, \eqref{E107}, and $\dbar u_\varepsilon=0$
show that it is entire.
The bound in \eqref{E50} allows us to apply Liouville's theorem
along the complex coordinate lines, obtaining
$u_\varepsilon\equiv c_\varepsilon$ for some $c_\varepsilon\in\C$,
and hence $\nabla u_\varepsilon=0$.
For $1\leq\ell\leq d$ and $\psi\in C_c^\infty(\R^d)$,
the argument leading to \eqref{E67} applies with
$\partial_{x_\ell}$ in place of $\partial_{\overline{z}_j}$.
Using \eqref{E66} on $K=\supp\psi$, we obtain
\begin{align*}
 \int_{\R^d}\overline{f}\partial_{x_\ell}\psi dv
 &=\lim_{\varepsilon\rightarrow0^{+}}
   \int_{\R^d}u_\varepsilon\partial_{x_\ell}\psi dv=-\lim_{\varepsilon\rightarrow0^{+}}
   \int_{\R^d}\partial_{x_\ell}u_\varepsilon\psi dv=0.
\end{align*}
The last equality uses $\partial_{x_\ell}u_\varepsilon=0$
for every $\varepsilon>0$, as proved above by Liouville's theorem.
Since $\ell$ and $\psi$ were arbitrary, \eqref{E48} gives
$\nabla\overline{f}=0$ in distributions.

To conclude that $f$ is constant, let $v\in L^\infty(\R^d)$
satisfy $\nabla v=0$ in distributions. We show that
\begin{equation}\label{E151}
 v=c\quad\text{almost everywhere on }\R^d
\end{equation}
for some $c\in\C$.
For every integer $N\geq1$ and $h\in\{\Re v,\Im v\}$,
\eqref{E160} with $g_\ell=0$ for $1\leq\ell\leq d$
gives zero weak derivatives, and we have
$$
 \norm{h}_{L^1(B(0,N),dv)}
 \leq |B(0,N)|\norm{v}_{L^\infty(\R^d)}<\infty\quad\text{and}\quad \nabla h=0.
$$
The zero-gradient criterion
\cite[Section 9.1, Remark 7]{Brezis2011}, applied to these
real functions on $B(0,N)$, yields $c_N\in\C$ such that
$v=c_N$ almost everywhere on $B(0,N)$.
Since $B(0,N)\subset B(0,N+1)$, we obtain
$$
 |c_{N+1}-c_N||B(0,N)|
 =\int_{B(0,N)}|c_{N+1}-c_N| dv=0.
$$
It follows that $c_N=c_1$ for all $N\geq1$. As
$\bigcup_{N\geq1}B(0,N)=\R^d$, this proves \eqref{E151}.
Taking $v=\overline{f}$ gives $f=\overline{c_1}$
almost everywhere on $\R^d$.

Suppose now $p=d$. Since $f\in\mathcal A_d\subseteq L^\infty(\C^n)$
by \eqref{eq:property}, we have $\overline{f}\in L^\infty(\C^n)$.
Lemma \ref{L6}, applied with $u=\overline{f}$, yields
$u_r=\overline{f}*\eta_r\in C^\infty(\R^d)\cap L^\infty(\C^n)$
for every $r>0$.
Let $C_d>0$ be the constant in
\eqref{eq:derivative-necessary} with $p=d$.
Since $r_0^2=t_0$, we insert the bound from Lemma \ref{lem:baire}
into that estimate to obtain
$$
 \sup_{0<r<r_0}\norm{\dbar u_r}_{L^d(\R^d,dv)}
 \leq C_d\sup_{0<r<r_0}m_{d,r^2}(f)
 \leq C_da<\infty.
$$
Applying the first estimate in \eqref{E14} from
Lemma \ref{lem:complex-derivatives} to $u_r$ with $q=d>1$ gives
$$
 \sup_{0<r<r_0}\norm{\nabla u_r}_{L^{d}(\R^d,dv)}
 \leq C\sup_{0<r<r_0}\norm{\dbar u_r}_{L^{d}(\R^d,dv)}<\infty,
$$
where $C=C_{n,d}>0$ is the constant in \eqref{E14} with
$q=d=2n$ and depends only on $n$.
Starting with the sequence $r_0/(j+1)$, we use the uniform
gradient bound and $d>1$ to apply reflexivity
\cite[Theorem 4.10, p. 95]{Brezis2011} and weak sequential
compactness \cite[Theorem 3.18, p. 69]{Brezis2011}
successively to the real and imaginary parts of the $d$
derivatives. This gives a common subsequence
$0<r_k<r_0$ with $r_k\rightarrow0^{+}$ as $k\rightarrow\infty$.
Combining the paired limits gives
$g=(g_1,\ldots,g_d)\in[L^d(\R^d,dv)]^d$ such that
\begin{equation}\label{E150}
 \partial_{x_\ell}u_{r_k}\rightharpoonup g_\ell
 \quad\text{in }L^d(\R^d,dv),\quad 1\leq\ell\leq d.
\end{equation}
Here $\rightharpoonup$ denotes weak convergence, meaning
$$
 \int_{\R^d}(\partial_{x_\ell}u_{r_k}-g_\ell)h dv
 \longrightarrow0\quad(k\rightarrow\infty)
$$
for every $h\in L^{d/(d-1)}(\R^d,dv)$ and $1\leq\ell\leq d$.

To identify $g$ as the weak gradient of $\overline{f}$, fix
$1\leq\ell\leq d$ and $\psi\in C_c^\infty(\R^d)$.
The weak convergence in \eqref{E150}, integration by parts as in
\eqref{E144} for $u_{r_k}\psi\in C_c^\infty(\R^d)$,
and \eqref{E66} along the subsequence $r=r_k$ give
 \begin{align}\label{E154}
 \int_{\R^d}g_\ell\psi dv
 &=\lim_{k\rightarrow\infty}
   \int_{\R^d}\partial_{x_\ell}u_{r_k}\psi dv=-\lim_{k\rightarrow\infty}
   \int_{\R^d}u_{r_k}\partial_{x_\ell}\psi dv=-\int_{\R^d}\overline{f}\partial_{x_\ell}\psi dv.
 \end{align}
By \eqref{E3}, identity \eqref{E154} yields
$$
 \nabla\overline{f}=g,\quad \nabla f=\overline{g},\quad
 |\nabla f|=|g|\in L^d(\R^d,dv),
$$
where the second equality follows by taking
$\overline{\psi}$ as the test function in \eqref{E154}
and conjugating. This completes the proof.
\end{proof}

The inclusion of Schatten classes gives the remaining range
$0<p<1$.

\begin{corollary}\label{cor:smallp}
For $0<p<1$, $\mathcal A_p$ consists only of constants. More
generally, $\mathcal A_p$ is contained in $\mathcal A_q$
whenever $0<p<q<\infty$.
\end{corollary}
\begin{proof}
For $0<p<q<\infty$, every $a\in\ell^p$ belongs to $\ell^q$
and satisfies $\norm{a}_{\ell^q}\leq\norm{a}_{\ell^p}$.
Applying this to the singular values in the definition of
Schatten classes \cite[Section 1.4, p. 18]{Zhu1990} gives
\begin{equation}\label{E146}
 \mathcal S_p(X)\subseteq\mathcal S_q(X)\quad\text{and}\quad
 \norm{A}_{\mathcal S_q(X)}\leq\norm{A}_{\mathcal S_p(X)}
\end{equation}
for a complex Hilbert space $X$ and $A\in\mathcal S_p(X)$.
For $f\in\mathcal A_p$ and each fixed $g\in L^\infty(\C^n)$,
we apply \eqref{E146} with $X=\cF$ to obtain
$$
 0\leq\norm{D_t(f,g)}_{\mathcal S_q}
 \leq\norm{D_t(f,g)}_{\mathcal S_p}
 \longrightarrow0\quad(t\rightarrow0^{+}).
$$
By \eqref{eq:property}, $\mathcal A_p\subseteq\mathcal A_q$.
For $0<p<1$, taking $q=1$ gives
$\mathcal A_p\subseteq\mathcal A_1$. Since $1<d$, Proposition
\ref{prop:necessity} shows that every $f\in\mathcal A_p$ is
constant almost everywhere, which completes the proof.
\end{proof}
To estimate the conjugate test, fix $\rho>0$ as in Lemma
\ref{L1}. For $1\leq p<\infty$, \cite[Remark 2.5]{HuVirtanen2022}
shows that different fixed radii give equivalent
$L^{2p}$ norms of the local analytic distance.
Taking the radius to be $\rho$ and the Schatten exponent to be $2p$,
we obtain from \cite[Theorem 1.1, (1.4)]{HuVirtanen2022}
a constant $C_{p,0}>0$, depending only on $n,m,M,p$, such that
\begin{equation}\label{E54}
\norm{G_{2,\rho}(u)}_{L^{2p}(\R^d,dv)}
\leq C_{p,0}
\norm{H_u^{(1)}}_{\mathcal S_{2p}(\mathcal F_1^2,L_1^2\ominus\mathcal F_1^2)}
\end{equation}
for every $u\in L^\infty(\C^n)$.
The factorization \eqref{eq:factorization} gives
$-D_t(f,\overline{f})=(H_{\overline{f}}^{(t)})^*H_{\overline{f}}^{(t)}$.
We combine its Schatten norm identity
\eqref{eq:conjugate-identity} with \eqref{E54} to derive
the following bounds for the conjugate test.

\begin{proposition}\label{prop:conjugate}
For $f\in L^\infty(\C^n)$, $p>0$, and $t>0$, set
$\delta_{p,t}(f)=\norm{D_t(f,\overline{f})}_{\mathcal S_p}$.
We use $r=\sqrt t$ and the mollifier $\eta_r$ in \eqref{E16}.
If $1\leq p<\infty$, then
\begin{equation}\label{eq:conjugate-derivative}
 \norm{\dbar(\overline{f}*\eta_r)}_{L^{2p}(\R^d,dv)}
 \leq C_p r^{n/p-1}\delta_{p,r^2}^{1/2}(f).
\end{equation}
Here $C_p>0$ depends only on $n,m,M,p$ and the chosen mollifier
$\eta$. For every $0<p\leq n$, the condition
$\delta_{p,t}(f)\rightarrow0$ as $t\rightarrow0^{+}$ is equivalent to
$f$ being constant almost everywhere.
\end{proposition}
\begin{proof}
To prove \eqref{eq:conjugate-derivative}, fix $1\leq p<\infty$
and $r>0$. We may assume $\delta_{p,r^2}(f)<\infty$ and set
$h_r(\zeta)=\overline{f}(r\zeta)$.
For $s>0$, the change of variables $w=r\xi$ makes
$h\mapsto h(\cdot/r)$ a bijection from
$\mathcal O(B(\zeta,s))$ onto
$\mathcal O(B(r\zeta,sr))$.
Since $dv(w)=r^d dv(\xi)$ and
$|B(r\zeta,sr)|=r^d|B(\zeta,s)|$,
definition \eqref{E106} gives
\begin{equation}\label{E149}
G_{q,s}(h_r)(\zeta)
=G_{q,sr}(\overline{f})(r\zeta),\quad \zeta\in\C^n,
\end{equation}
for $1\leq q<\infty$ and $s>0$.
Taking $q=2$ and $s=\rho$ in \eqref{E149} and changing
variables once more, we obtain
$$
\int_{\R^d}G_{2,\rho}^{2p}(h_r)(\zeta) dv(\zeta)
=r^{-d}\int_{\R^d}G_{2,\rho r}^{2p}(\overline{f})(z) dv(z).
$$
Taking the $2p$-th root and applying \eqref{E54} to $h_r$, 
we obtain
\begin{align}\label{E31}
r^{-d/(2p)}\norm{G_{2,\rho r}(\overline{f})}_{L^{2p}(\R^d,dv)}
&=\norm{G_{2,\rho}(h_r)}_{L^{2p}(\R^d,dv)}\leq C_{p,0}
\norm{H_{h_r}^{(1)}}_{\mathcal S_{2p}(\mathcal F_1^2,L_1^2\ominus\mathcal F_1^2)}
\nonumber\\
&=C_{p,0}
\norm{H_{\overline{f}}^{(r^2)}}_{\mathcal S_{2p}(\mathcal F_{r^2}^2,L_{r^2}^2\ominus\mathcal F_{r^2}^2)}
=C_{p,0}\delta_{p,r^2}^{1/2}(f).
\end{align}
The last two equalities use the Hankel identity and norm
invariance in Lemma \ref{lem:scaled-kernel} (b), followed by
\eqref{eq:conjugate-identity}.

Let
$C_\eta=|B(0,\rho)|\sup_{x\in\R^d}|\dbar\eta(x)|>0$
be the constant in \eqref{eq:mollifier}.
Applying \eqref{eq:mollifier} with $u=\overline{f}$ and $q=2$,
taking the $L^{2p}(\R^d,dv)$ norm, and using \eqref{E31},
we obtain
\begin{align*}
\norm{\dbar(\overline{f}*\eta_r)}_{L^{2p}(\R^d,dv)}
&\leq C_\eta r^{-1}
\norm{G_{2,\rho r}(\overline{f})}_{L^{2p}(\R^d,dv)}\leq C_\eta C_{p,0}r^{d/(2p)-1}
\delta_{p,r^2}^{1/2}(f).
\end{align*}
Since $d=2n$, this proves \eqref{eq:conjugate-derivative}
with $C_p=C_\eta C_{p,0}>0$, which has the stated dependence.

Suppose now that $0<p\leq n$ and
$\delta_{p,t}(f)\rightarrow0$ as $t\rightarrow0^{+}$.
If $1\leq p\leq n$, then $n/p-1\geq0$, so
\eqref{eq:conjugate-derivative} gives
$$
\norm{\dbar(\overline{f}*\eta_r)}_{L^{2p}(\R^d,dv)}
\rightarrow0\quad(r\rightarrow0^{+}).
$$
Combining this limit with the local convergence in \eqref{E50},
we repeat the calculation leading to \eqref{E67} with exponent
$2p$ to obtain $\dbar\overline{f}=0$ in distributions.
Set $u_\varepsilon=\overline{f}*\eta_\varepsilon$.
Lemma \ref{L6} and \eqref{E213} show that $u_\varepsilon$ is
entire, and \eqref{E50} gives
$\norm{u_\varepsilon}_{L^\infty(\C^n)}\leq\norm{f}_{L^\infty(\C^n)}$.
Liouville's theorem along the complex coordinate lines yields
$\nabla u_\varepsilon=0$. By the local convergence in
\eqref{E50} and the distributional derivative definition \eqref{E48},
$$
 \nabla\overline{f}
 =\lim_{\varepsilon\rightarrow0^+}\nabla u_\varepsilon=0
 \quad\text{in distributions}.
$$
Equation \eqref{E151} now shows that $f$ is constant almost everywhere.

For $0<p<1$, \eqref{E146} with $q=1$ gives
$$
\delta_{1,t}(f)\leq\delta_{p,t}(f)\rightarrow0
\quad(t\rightarrow0^{+}).
$$
The case $p=1$, already proved, again shows that $f$ is
constant almost everywhere. Conversely, constant symbols
have zero semi-commutators by \eqref{E2}, which completes
the proof.
\end{proof}

We record two necessary local conditions in the supercritical range.
For $p>d$ and $f\in\mathcal A_p$, there exists $r_0>0$ such that
\begin{equation}\label{eq:super-necessary}
 \sup_{0<r<r_0}r^{-d/p}\norm{G_{1,r}(\overline{f})}_{L^{p}(\R^d,dv)}<\infty
 \quad\text{and}\quad
 \lim_{r\rightarrow0^{+}}r^{-d/(2p)}
                       \norm{G_{2,r}(\overline{f})}_{L^{2p}(\R^d,dv)}=0.
\end{equation}
Indeed, Lemma \ref{lem:baire} gives $a,t_0>0$ such that
$m_{p,t}(f)\leq a$ for $0<t<t_0$.
With $r_0=\rho\sqrt{t_0}$, \eqref{eq:phase} at $r=s/\rho$ yields
$$
 s^{-d/p}\norm{G_{1,s}(\overline{f})}_{L^p(\R^d,dv)}
 \leq C_p\rho^{-d/p}a,\quad 0<s<r_0.
$$
For the fixed bounded symbol $g=\overline{f}$,
\eqref{E31} with the same substitution and
\eqref{eq:property} give
\begin{align*}
 s^{-d/(2p)}\norm{G_{2,s}(\overline{f})}_{L^{2p}(\R^d,dv)}
 &\leq C\rho^{-d/(2p)}
 \norm{D_{(s/\rho)^2}(f,\overline{f})}_{\mathcal S_p(\mathcal F_{(s/\rho)^2}^2)}^{1/2}\rightarrow0\quad(s\rightarrow0^{+}),
\end{align*}
where $C>0$ depends only on $n,m,M,p$.
Remark \ref{R2} will show that the two conditions in
\eqref{eq:super-necessary} are not sufficient for membership in
$\mathcal A_p$, even for continuous compactly supported symbols.

\section{Gradient estimates and sufficiency at \texorpdfstring{$p=d$}{p=d}}\label{sec:critical}
To prove sufficiency at $p=d$, we establish Schatten estimates
for $p\geq2$ and operator norm bounds. The gradient estimates
also give the supercritical rate in Theorem \ref{thm:main} (c).
At $p=d$, combining the Schatten and operator norm bounds
yields convergence for smooth compactly supported symbols,
and gradient approximation extends the result to bounded
symbols with $|\nabla f|\in L^d(\R^d,dv)$.
Throughout this section, we write $r=\sqrt t$.

\subsection{Schatten and operator norm estimates}
Recall that $T_u$ denotes the Toeplitz operator at scale one.
For $f\in L^\infty(\C^n)$, define
$J_f:\cF\to L_t^2$ by $J_fh=fh$ for $h\in\cF$.
Its Hilbert space adjoint $J_f^*:L_t^2\to\cF$ is characterized by
$$
\langle J_fh,u\rangle_t=\langle h,J_f^*u\rangle_t,
\quad h\in\cF,\quad u\in L_t^2.
$$
The orthogonality of $P_t$ gives
\begin{equation}\label{E68}
 J_f^*u=P_t(\overline{f} u),\quad u\in L_t^2,
\end{equation}
and hence $J_f^*J_f=T_{|f|^2}^{(t)}$.
This identity together with \eqref{E9} reduces the Schatten
estimate for $J_f$ to one for a positive Toeplitz operator.
We apply \cite[Lemma 3.1]{IsralowitzVirtanenWolf2015}
and record the constants and dependence on the scale
in the following lemma.

\begin{lemma}\label{lem:positive}
\begin{enumerate}[label=\textup{(\alph*)}]
\item
Let $1\leq q<\infty$ and $u\geq0$ with $u\in L^q(\C^n,dv)$. The form
$$
 \langle T_u^{(t)}h,h\rangle_t=\int_{\C^n}u|h|^2d\mu_t,
 \qquad h\in\cF,
$$
defines a bounded positive operator satisfying
\begin{equation}\label{eq:positive}
 \norm{T_u^{(t)}}_{\mathcal S_q}^q\leq Cr^{-d}\norm{u}_{L^{q}(\R^d,dv)}^q.
\end{equation}
Here $C>0$ depends only on $n,m,M,q$.

\item
If $p\geq2$ and $f\in L^p(\C^n,dv)\cap L^\infty(\C^n)$, then
\begin{equation}\label{eq:multiplier}
 \norm{J_f}_{\mathcal S_p(\cF,L_t^2)}
 \leq Cr^{-d/p}\norm{f}_{L^{p}(\R^d,dv)}\quad\text{and}\quad
 \norm{H_f^{(t)}}_{\mathcal S_p(\cF,L_t^2\ominus\cF)}
 \leq Cr^{-d/p}\norm{f}_{L^{p}(\R^d,dv)}.
\end{equation}
Here $C>0$ depends only on $n,m,M,p$.
\end{enumerate}
\end{lemma}
\begin{proof}
We first prove \textup{(a)} for $t=1$.
Applying \cite[Lemma 3.1]{IsralowitzVirtanenWolf2015}
to $u\geq0$ with $u\in L^q(\C^n,dv)$, we obtain
a positive operator $T_u\in\mathcal S_q(\mathcal F_1^2)$.
To record the dependence of the norm bound, let $\lambda_j$
be the positive eigenvalues of $T_u$, counted with
multiplicity, and let $h_j$ be corresponding orthonormal
eigenvectors. Since $\int_{\C^n}|h_j|^2d\mu_1=1$,
the measure $d\nu_j=|h_j|^2d\mu_1$ is a probability measure.
The reproducing property and Bessel's inequality give
$$
\sum_j|h_j(z)|^2
=\sum_j|\langle h_j,K_1(\cdot,z)\rangle_1|^2
\leq K_1(z,z),\quad z\in\C^n.
$$
Applying Jensen's inequality to $s\mapsto s^q$ with respect
to $\nu_j$, we sum over $j$ and use Tonelli's theorem to
bring the sum inside the integral. The preceding Bessel bound
and the second estimate in \eqref{eq:gram} with $t=r=1$ then give
\begin{align*}
\norm{T_u}_{\mathcal S_q(\mathcal F_1^2)}^q
&=\sum_j\lambda_j^q
=\sum_j\left(\int_{\C^n}u|h_j|^2d\mu_1\right)^q\leq\int_{\C^n}u^q\sum_j|h_j|^2d\mu_1\\
&\leq\int_{\C^n}u^q(z)K_1(z,z)e^{-2\varphi(z)} dv(z)\leq C\norm{u}_{L^q(\R^d,dv)}^q.
\end{align*}
Here $C>0$ is the constant in that kernel estimate and
depends only on $n,m,M$.
For general $t$ and $h\in\cF$, the substitution $z=r\zeta$
in the defining quadratic form gives
\begin{align*}
 \int_{\C^n}u|h|^2d\mu_t
 &=\int_{\C^n}u(r\zeta)|(U_th)(\zeta)|^2d\mu_1(\zeta)=\langle T_{u(r\cdot)}U_th,U_th\rangle_1.
\end{align*}
Since $u(r\cdot)\in L^q(\C^n,dv)$, the result at scale one
makes this form bounded. Polarization identifies its operator
as $T_u^{(t)}=U_t^{-1}T_{u(r\cdot)}U_t$, extending the first
identity in \eqref{E55} to these nonnegative symbols.
We may therefore use the unitary
invariance of the Schatten norm in Lemma \ref{lem:scaled-kernel} (b)
to obtain
\begin{align*}
 \norm{T_u^{(t)}}_{\mathcal S_q}^q
 &=\norm{T_{u(r\cdot)}}_{\mathcal S_q(\mathcal F_1^2)}^q\leq C\int_{\R^d}u^q(r\zeta) dv(\zeta)
 =Cr^{-d}\int_{\R^d}u^q(z) dv(z).
 \end{align*}
The last equality uses $dv(z)=r^d dv(\zeta)$ and proves
\eqref{eq:positive}.

For \textup{(b)}, \eqref{E68} gives
$J_f^*J_f=T_{|f|^2}^{(t)}$. Since $p/2\geq1$ and
$|f|^2\in L^{p/2}(\C^n,dv)$, \eqref{E9} and
\eqref{eq:positive} with $q=p/2$ yield
\begin{align*}
\norm{J_f}_{\mathcal S_p(\cF,L_t^2)}^p
&=\norm{T_{|f|^2}^{(t)}}_{\mathcal S_{p/2}}^{p/2}\leq Cr^{-d}\int_{\R^d}|f|^p dv.
\end{align*}
Here $C>0$ is the constant in \eqref{eq:positive}
with $q=p/2$.
Since $H_f^{(t)}=(I-P_t)J_f$ and $\norm{I-P_t}\leq1$,
the first inequality in \eqref{E7} and the preceding bound give
$$
\norm{H_f^{(t)}}_{\mathcal S_p(\cF,L_t^2\ominus\cF)}
\leq\norm{J_f}_{\mathcal S_p(\cF,L_t^2)}
\leq C^{1/p}r^{-d/p}\norm{f}_{L^p(\R^d,dv)}.
$$
This proves both estimates in \eqref{eq:multiplier}
with the positive constant $C^{1/p}$. The proof is complete.
\end{proof}

The bound in \eqref{eq:multiplier} uses integrability of the
symbol. We now estimate Hankel operators through the weak
derivatives of bounded symbols, retaining the scale factor.

\begin{lemma}\label{lem:gradient}
Let $2\leq p<\infty$.
\begin{enumerate}[label=\textup{(\alph*)}]
\item
For bounded $u$ with $|\nabla u|\in L^p(\R^d,dv)$, we have
\begin{equation}\label{eq:hankel-gradient}
 \norm{H_u^{(t)}}_{\mathcal S_p(\cF,L_t^2\ominus\cF)}
 \leq Cr^{1-d/p}\norm{\dbar u}_{L^{p}(\R^d,dv)}.
\end{equation}
The estimate \eqref{eq:hankel-gradient} also holds for bounded
smooth $u$ whenever $|\dbar u|\in L^p(\R^d,dv)$, without first assuming that
the other derivatives belong to $L^p(\C^n,dv)$.

\item
For bounded $f$ with $|\nabla f|\in L^p(\R^d,dv)$ and any
bounded $g$, we have
\begin{equation}\label{eq:gradient-semicommutator}
 \norm{D_t(f,g)}_{\mathcal S_p}
 \leq Cr^{1-d/p}\norm{\nabla f}_{L^{p}(\R^d,dv)}\norm{g}_{L^\infty(\C^n)}.
\end{equation}
\end{enumerate}
In both estimates, $C>0$ depends only on $n,m,M,p$.
\end{lemma}
\begin{proof}
We first prove \textup{(a)}. To control the local mean
$M_{2,1}(|\dbar u|)$ in the Hankel estimate below, we
establish an $L^a$ bound for the means in \eqref{E4}.
Let $s>0$, $1\leq q\leq a<\infty$, and $h\in L^a(\R^d,dv)$.
Set $B_s=B(0,s)$. For fixed $z$, the change of variables $w=z-\xi$
maps $B(z,s)$ onto $B_s$ and preserves Lebesgue measure.
Since $a/q\geq1$, Jensen's inequality, Tonelli's theorem,
and translation invariance give
\begin{align*}
\int_{\R^d}M_{q,s}^a(h)(z) dv(z)
&=\int_{\R^d}
  \left(\frac1{|B_s|}\int_{B_s}|h(z-w)|^q dv(w)\right)^{a/q} dv(z)\\
&\leq\frac1{|B_s|}\int_{B_s}
  \int_{\R^d}|h(z-w)|^a dv(z) dv(w)=\norm{h}_{L^a(\R^d,dv)}^a.
\end{align*}
Taking the $a$-th root yields
\begin{equation}\label{E56}
 \norm{M_{q,s}(h)}_{L^a(\R^d,dv)}
 \leq\norm{h}_{L^a(\R^d,dv)}.
\end{equation}

First, assume that $u\in L^\infty(\C^n)\cap C^\infty(\C^n)$
with $|\nabla u|\in L^p(\R^d,dv)$.
The smoothness and boundedness of $u$ give
$u\in C^2(\C^n)\cap L^2_{\rm loc}(\C^n,dv)$.
By \eqref{E57} and \eqref{E14}, the condition
$|\nabla u|\in L^p(\R^d,dv)$ is equivalent to
$|\dbar u|\in L^p(\R^d,dv)$.
Since $p\geq2$, \eqref{E56} with $q=2$, $a=p$,
$s=1$, and $h=|\dbar u|$ gives
$$
 |\dbar u|+M_{2,1}(|\dbar u|)\in L^p(\R^d,dv).
$$
This together with $u\in C^2(\C^n)\cap L^2_{\rm loc}(\C^n,dv)$
verifies the decomposition conditions in
\cite[Theorem 2.6, (2.9)]{HuVirtanenCorr2023} for $u=u+0$,
with exponent $s=p$ and auxiliary radius $1$.
Applying \cite[Theorem 1.1, (1.4)]{HuVirtanen2022}
to $u$ and combining its norm estimate with the corrected
decomposition estimate \cite[(2.10)]{HuVirtanenCorr2023}
and \eqref{E56}, we obtain
\begin{align*}
 \norm{H_u^{(1)}}_{\mathcal S_p(\mathcal F_1^2,L_1^2\ominus\mathcal F_1^2)}
 &\leq C\norm{G_{2,1}(u)}_{L^p(\R^d,dv)}\leq C\norm{M_{2,1}(|\dbar u|)}_{L^p(\R^d,dv)}\leq C\norm{\dbar u}_{L^p(\R^d,dv)},
\end{align*}
where $G_{2,1}$ in \eqref{E106} is the local analytic distance
denoted by $G_1$ in the cited paper.
Set $r=\sqrt t$ and $k(\zeta)=u(r\zeta)$. By the chain rule,
$\dbar k(\zeta)=r(\dbar u)(r\zeta)$. The Hankel norm invariance
in Lemma \ref{lem:scaled-kernel} (b), the preceding estimate, and
the change of variables $z=r\zeta$ give
$$
 \norm{H_u^{(t)}}_{\mathcal S_p(\cF,L_t^2\ominus\cF)}
 =\norm{H_k^{(1)}}_{\mathcal S_p(\mathcal F_1^2,L_1^2\ominus\mathcal F_1^2)}
 \leq C\norm{\dbar k}_{L^{p}(\R^d,dv)}
 =Cr^{1-d/p}\norm{\dbar u}_{L^{p}(\R^d,dv)}.
$$
Note that the constant $C>0$ depends only on $n,m,M,p$
and is independent of $t$. This proves \eqref{eq:hankel-gradient}
for smooth $u$.

To extend the smooth estimate to weak derivatives, suppose that $u$
is bounded with $|\nabla u|\in L^p(\R^d,dv)$, and let
$u_\varepsilon=u*\eta_\varepsilon$, using the mollifier in
\eqref{E16}.
Since $u$ is bounded and $\eta_\varepsilon\in C_c^\infty(\R^d)$,
we may differentiate under the integral sign to obtain
\begin{equation}\label{E145}
 \partial_{x_\ell}u_\varepsilon(x)
 =\int_{\R^d}u(y)
   \partial_{x_\ell}\eta_\varepsilon(x-y) dv(y),
 \quad 1\leq\ell\leq d.
\end{equation}
For each fixed $x\in\R^d$, the function
$y\mapsto\eta_\varepsilon(x-y)$ belongs to
$C_c^\infty(\R^d)$.
Using
$\partial_{x_\ell}\eta_\varepsilon(x-y)
=-\partial_{y_\ell}\eta_\varepsilon(x-y)$,
we apply the weak derivative identity \eqref{E3} to obtain
\begin{align*}
 \partial_{x_\ell}u_\varepsilon(x)
 &=-\int_{\R^d}u(y)
       \partial_{y_\ell}\eta_\varepsilon(x-y) dv(y)=\int_{\R^d}(\partial_{y_\ell}u)(y)
       \eta_\varepsilon(x-y) dv(y)\\
 &=\bigl((\partial_{x_\ell}u)*\eta_\varepsilon\bigr)(x).
\end{align*}
These component identities, together with \eqref{E13},
\eqref{E107}, and the linearity of convolution, yield
\begin{equation}\label{E17}
 \nabla u_\varepsilon=(\nabla u)*\eta_\varepsilon
 \quad\text{and}\quad
 \dbar u_\varepsilon=(\dbar u)*\eta_\varepsilon.
\end{equation}
By \eqref{E57} with $q=p$, the hypothesis on $\nabla u$ gives
$|\dbar u|\in L^p(\R^d,dv)$.
Since $\eta_\varepsilon\geq0$ and
$\int_{\R^d}\eta_\varepsilon dv=1$, Minkowski's integral
inequality and translation invariance applied to \eqref{E17} yield
\begin{align}\label{E108}
 \norm{\dbar u_\varepsilon}_{L^p(\R^d,dv)}
 &=\left\|\int_{\R^d}\eta_\varepsilon(y)
       \dbar u(\cdot-y) dv(y)\right\|_{L^p(\R^d,dv)}\nonumber\\
 &\leq\int_{\R^d}\eta_\varepsilon(y)
       \norm{\dbar u(\cdot-y)}_{L^p(\R^d,dv)} dv(y)\nonumber\\
 &=\norm{\dbar u}_{L^p(\R^d,dv)}
 \leq\frac1{\sqrt2}\norm{\nabla u}_{L^p(\R^d,dv)}<\infty.
\end{align}

By \eqref{E50}, the real-valued functions
$|u_\varepsilon-u|$ tend to zero in $L^1$ on every ball.
Applying \cite[Theorem 4.9, p. 94]{Brezis2011} on the nested
balls $B(0,N)$, $N=1,2,\ldots$, and taking a diagonal subsequence gives
a sequence $\varepsilon_j>0$ with
$\varepsilon_j\rightarrow0^{+}$ as $j\rightarrow\infty$ such that
$u_{\varepsilon_j}\rightarrow u$ almost everywhere.
Fix $t>0$. Extending the Hankel operators from $\cF$ to
$L_t^2$ by zero as before \eqref{E138}, we obtain
bounded operators $A_j,A:L_t^2\to L_t^2$ given by
$$
 A_j=H_{u_{\varepsilon_j}}^{(t)}P_t,\qquad A=H_u^{(t)}P_t.
$$
For $h\in L_t^2$, the definition of the Hankel operator in
\eqref{E1} gives
$$
 (A_j-A)h=(I-P_t)\bigl((u_{\varepsilon_j}-u)P_th\bigr).
$$
Using $\norm{I-P_t}\leq1$ and
$|u_{\varepsilon_j}-u|\leq2\norm{u}_{L^\infty(\C^n)}$,
we apply dominated convergence with the integrable majorant
$4\norm{u}_{L^\infty(\C^n)}^2|P_th|^2$ to obtain
\begin{align*}
\norm{(A_j-A)h}_{L_t^2}^2
&\leq\int_{\C^n}|u_{\varepsilon_j}-u|^2|P_th|^2d\mu_t\longrightarrow0\quad(j\rightarrow\infty).
\end{align*}
Then $A_j\to A$ strongly on $L_t^2$.

Applying the smooth case of \eqref{eq:hankel-gradient}
to each $u_{\varepsilon_j}$ and then using \eqref{E108}, we obtain
\begin{align*}
 \sup_{j\geq1}\norm{A_j}_{\mathcal S_p(L_t^2)}
 &=\sup_{j\geq1}
   \norm{H_{u_{\varepsilon_j}}^{(t)}}_
   {\mathcal S_p(\cF,L_t^2\ominus\cF)}\leq Cr^{1-d/p}\sup_{j\geq1}
   \norm{\dbar u_{\varepsilon_j}}_{L^p(\R^d,dv)}\\
 &\leq Cr^{1-d/p}\norm{\dbar u}_{L^p(\R^d,dv)}<\infty.
\end{align*}
The equality follows from \eqref{E138}.
The strong convergence of $A_j$ to $A$ also implies weak
operator convergence. Using this convergence and the uniform
Schatten bound above, we apply the last assertion of the
noncommutative Fatou lemma
\cite[Theorem 2.7 (d), p. 19]{Simon2005} with the symmetric
sequence norm $\ell^p$, $p<\infty$, to obtain
\begin{align*}
\norm{A}_{\mathcal S_p(L_t^2)}
&\leq\sup_{j\geq1}\norm{A_j}_{\mathcal S_p(L_t^2)}\leq Cr^{1-d/p}\norm{\dbar u}_{L^p(\R^d,dv)}.
\end{align*}
By \eqref{E138}, $A$ and $H_u^{(t)}$ have the same
Schatten norm, so the preceding bound proves
\eqref{eq:hankel-gradient}.

For \textup{(b)}, conjugation in \eqref{E3} gives
$\nabla\overline{f}=\overline{\nabla f}$ almost everywhere,
so \eqref{E57} and \eqref{eq:hankel-gradient} apply to
$\overline{f}$ with $|\nabla\overline{f}|=|\nabla f|$.
These estimates together with \eqref{eq:factorization}, the first
inequality in \eqref{E7}, and
$\norm{H_g^{(t)}}\leq\norm{g}_{L^\infty(\C^n)}$ give
\begin{align*}
 \norm{D_t(f,g)}_{\mathcal S_p}
 &\leq\norm{H_{\overline{f}}^{(t)}}_{\mathcal S_p(\cF,L_t^2\ominus\cF)}
       \norm{H_g^{(t)}}\leq Cr^{1-d/p}\norm{\dbar\overline{f}}_{L^p(\R^d,dv)}
       \norm{g}_{L^\infty(\C^n)}\\
 &\leq Cr^{1-d/p}\norm{\nabla f}_{L^p(\R^d,dv)}
       \norm{g}_{L^\infty(\C^n)}.
\end{align*}
This proves \eqref{eq:gradient-semicommutator} and completes the proof.
\end{proof}

\begin{remark}\label{rem:corrigendum}
We use Theorem 2.6 of \cite{HuVirtanen2022} with the
decomposition norm corrected in \cite{HuVirtanenCorr2023}.
The $L^p(\R^d,dv)$ norm of the local quadratic mean is bounded
by $\norm{\dbar u}_{L^{p}(\R^d,dv)}$ when $p/2\geq1$, as noted in
\cite[Remark 1]{HuVirtanenCorr2023}. This is the restriction
$p\geq2$ used in Lemma \ref{lem:gradient} (a).
\end{remark}

The estimate in \eqref{eq:hankel-gradient} controls one Hankel
factor in \eqref{eq:factorization}. To approximate the bounded
second symbol locally, we also need an operator norm bound
for the other factor in terms of the H\"older seminorm
in \eqref{E44}.

\begin{lemma}\label{lem:holder-op}
If $f$ is bounded and H\"older continuous of exponent
$0<\alpha\leq1$, with H\"older seminorm $[f]_\alpha$, then
\begin{equation}\label{E18}
 \norm{H_f^{(t)}}\leq C_\alpha r^\alpha[f]_\alpha.
\end{equation}
Here $C_\alpha>0$ depends only on $n,m,M,\alpha$.
\end{lemma}
\begin{proof}
Put $f_r(\zeta)=f(r\zeta)$. Applying
\cite[Theorem 1.1 (a)]{HuVirtanen2023} to $f_r$
with both Lebesgue exponents equal to two and using
the Hankel norm invariance in Lemma \ref{lem:scaled-kernel} (b),
we obtain
$$
 \norm{H_f^{(t)}}_{\cF\to L_t^2\ominus\cF}=\norm{H_{f_r}^{(1)}}_{\mathcal F_1^2\to L_1^2\ominus\mathcal F_1^2}
 \leq C\sup_{\zeta\in\C^n}G_{2,1}(f_r)(\zeta).
$$
For each fixed $\zeta$, the constant function
$w\mapsto f_r(\zeta)$ is holomorphic on $B(\zeta,1)$.
Choosing this function in the definition of $G_{2,1}$
and using \eqref{E44}, we obtain
$$
 G_{2,1}(f_r)(\zeta)
 \leq\left(\frac1{|B(0,1)|}\int_{B(\zeta,1)}
       |f(r\xi)-f(r\zeta)|^2 dv(\xi)\right)^{1/2}
 \leq r^\alpha[f]_\alpha,
$$
which proves \eqref{E18}.
\end{proof}

\subsection{Sufficiency for smooth compactly supported symbols}
We next use the exponential decay of the kernel to estimate
semi-commutators whose symbols have separated supports.

\begin{lemma}\label{lem:separated}
Suppose $f,g\in L^\infty(\C^n)$, $f$ has compact support, and
$\operatorname{dist}(\supp f,\supp g)\geq\delta>0$.
These symbols satisfy
\begin{equation}\label{E19}
 \norm{D_t(f,g)}_{\mathcal S_p}^2
 \leq Cr^{-d}e^{-c\delta/r}\norm{f}_{L^{2}(\R^d,dv)}^2\norm{g}_{L^\infty(\C^n)}^2
 \rightarrow0\quad(t\rightarrow0^{+})
\end{equation}
for $p\geq2$.
Here $c,C>0$ depend only on $n,m,M$.
\end{lemma}
\begin{proof}
Fix $t>0$, put $r=\sqrt t$, and set
$A=M_f\Pi_tM_g$ on $L^2(\C^n,dv)$.
By Lemma \ref{lem:scaled-kernel} (a), this operator has the
integral kernel
$$
\kappa_t(z,w)=f(z)p_t(z,w)g(w),\quad z,w\in\C^n.
$$
To obtain a Hilbert--Schmidt bound for $A$, we first
estimate the square integral of $\kappa_t$.
Let $C_0,c>0$ be the constants in the third estimate of
\eqref{eq:gram}, and set
$$
a_c=\int_{\R^d}e^{-c|y|} dv(y)\in(0,\infty).
$$
Using $e^{-2c|x|/r}\leq e^{-c\delta/r}e^{-c|x|/r}$
for $|x|\geq\delta$, we change variables by $x=w-z$
and $x=ry$ to obtain the following estimate
\begin{align}\label{E60}
\int_{\{w\in\C^n:|w-z|\geq\delta\}}
e^{-2c|z-w|/r} dv(w)
&=\int_{\R^d\setminus B(0,\delta)}
e^{-2c|x|/r} dv(x)\nonumber\\
&\leq e^{-c\delta/r}
\int_{\R^d}e^{-c|x|/r} dv(x)=a_c r^de^{-c\delta/r}
\end{align}
for each fixed $z\in\C^n\cong\R^d$.

The support separation implies that $\kappa_t(z,w)=0$ almost
everywhere on $\{|z-w|<\delta\}$.
Applying the third estimate in \eqref{eq:gram} and
then \eqref{E60} to the inner integral, we obtain
\begin{align}\label{E58}
\int_{\C^n}\int_{\C^n}|\kappa_t(z,w)|^2 dv(w) dv(z)
&\leq C_0^2r^{-2d}
\norm{g}_{L^\infty(\C^n)}^2
\int_{\C^n}|f(z)|^2
\bigl(a_c r^de^{-c\delta/r}\bigr) dv(z)\nonumber\\
&=C_0^2a_c r^{-d}e^{-c\delta/r}
\norm{f}_{L^2(\R^d,dv)}^2
\norm{g}_{L^\infty(\C^n)}^2<\infty.
\end{align}
The compact support and boundedness of $f$ make the right-hand
side of \eqref{E58} finite, so
\cite[Theorem 3.5 and (3.8), p. 51--52]{Zhu1990} gives
$A\in\mathcal S_2(L^2(\C^n,dv))$ and
\begin{equation}\label{E59}
\norm{A}_{\mathcal S_2(L^2(\C^n,dv))}^2
=\int_{\C^n}\int_{\C^n}|\kappa_t(z,w)|^2 dv(w) dv(z).
\end{equation}

The estimate for $A$ in \eqref{E58} and \eqref{E59} is on
$L^2(\C^n,dv)$. To apply it to $D_t(f,g)$, we use
$\widetilde D_t(f,g)$ from \eqref{E83}, which acts on
the same space. The support separation gives $fg=0$ almost
everywhere, so \eqref{eq:unweighted-semicommutator} and
the definition of $A$ yield
$$
 \widetilde D_t(f,g)=\Pi_tM_f\Pi_tM_g\Pi_t=\Pi_tA\Pi_t.
$$
Using \eqref{E43} for the first equality and the ideal
inequality in \eqref{E7}, with $\norm{\Pi_t}\leq1$, we obtain
\begin{align*}
\norm{D_t(f,g)}_{\mathcal S_p}
&=\norm{\widetilde D_t(f,g)}_{\mathcal S_p(L^2(\C^n,dv))}=\norm{\Pi_tA\Pi_t}_{\mathcal S_p(L^2(\C^n,dv))}\\
&\leq\norm{A}_{\mathcal S_p(L^2(\C^n,dv))}
\leq\norm{A}_{\mathcal S_2(L^2(\C^n,dv))}.
\end{align*}
The last inequality follows from \eqref{E146} with exponents
$2$ and $p$ when $p>2$, and is an equality when $p=2$.
Squaring this inequality and applying \eqref{E59}
followed by \eqref{E58} proves the bound in \eqref{E19}
with $C=C_0^2a_c>0$.
Note that both $C$ and $c$ depend only on $n,m,M$.
Finally, $r^{-d}e^{-c\delta/r}$ tends to zero as
$r\rightarrow0^{+}$ for every fixed $\delta>0$.
With $r=\sqrt t$, \eqref{E19} therefore gives the limit
as $t\rightarrow0^{+}$. This completes the proof.
\end{proof}

These estimates allow approximation of a bounded second symbol
near the support of a smooth first symbol when $p=d$.

\begin{proposition}\label{prop:critical-smooth}
Every $f\in C_c^\infty(\R^d)$ belongs to $\mathcal A_d$.
\end{proposition}
\begin{proof}
Put $r=\sqrt t$ and fix $g\in L^\infty(\C^n)$.
Choose a bounded open set $\Omega\subseteq\R^d$ such that
$\operatorname{dist}(\supp f,\Omega^c)>0$.
Since $g\mathbf1_\Omega\in L^d(\R^d,dv)$,
we apply the density result
\cite[Corollary 4.23, p. 109]{Brezis2011}, with exponent $d$,
to the real and imaginary parts of $g\mathbf1_\Omega$.
For each $\varepsilon>0$, approximating each part within
$\varepsilon/2$ and combining the approximants as real and
imaginary parts gives $h\in C_c^\infty(\R^d)$ satisfying
$$
 \norm{g\mathbf1_\Omega-h}_{L^d(\R^d,dv)}<\varepsilon.
$$
Set $u=g\mathbf1_\Omega-h$ and $k=g\mathbf1_{\Omega^c}$.
Then $g=h+u+k$, all three functions are bounded, and
$\norm{u}_{L^d(\R^d,dv)}<\varepsilon$.
By linearity in the second symbol, we have
$$
 D_t(f,g)=D_t(f,h)+D_t(f,u)+D_t(f,k).
$$

We estimate the first two terms using the factorization
\eqref{eq:factorization} and the ideal inequality in
\eqref{E7}.
For the smooth term, the condition $f,h\in C_c^\infty(\R^d)$
gives $\norm{\nabla f}_{L^d(\R^d,dv)}<\infty$ and
$[f]_1,[h]_1<\infty$.
Applying \eqref{eq:hankel-gradient} to $\overline{f}$ with
$p=d$, using $|\nabla\overline{f}|=|\nabla f|$, we obtain
\begin{equation}\label{E200}
 \begin{aligned}
 \norm{D_t(f,h)}_{\mathcal S_d}
 &\leq
 \norm{H_{\overline{f}}^{(t)}}_{\mathcal S_d(\cF,L_t^2\ominus\cF)}
 \norm{H_h^{(t)}}\leq C\norm{\dbar\overline{f}}_{L^d(\R^d,dv)}
       \norm{H_h^{(t)}}\\
 &\leq C_1r\norm{\nabla f}_{L^d(\R^d,dv)}[h]_1
 \rightarrow0\quad(t\rightarrow0^{+}).
 \end{aligned}
\end{equation}
The last inequality uses \eqref{E57} and the case
$\alpha=1$ of \eqref{E18}, with $C,C_1>0$ depending only on
$n,m,M$.
The limit holds because $f,h$ are fixed and $r=\sqrt t$.

For the function $u$, we use \eqref{E18} for the first
Hankel operator and \eqref{eq:multiplier} for the second.
Since $d=2n\geq2$ and $[\overline{f}]_1=[f]_1$, these estimates give
\begin{equation}\label{E201}
 \begin{aligned}
 \norm{D_t(f,u)}_{\mathcal S_d}
 &\leq \norm{H_{\overline{f}}^{(t)}}
       \norm{H_u^{(t)}}_{\mathcal S_d(\cF,L_t^2\ominus\cF)}\leq C_2r[f]_1r^{-1}\norm{u}_{L^d(\R^d,dv)}
 \leq C_2[f]_1\varepsilon,
 \end{aligned}
\end{equation}
where $C_2>0$ depends only on $n,m,M$.

Since $k=g\mathbf1_{\Omega^c}$ and $\Omega^c$ is closed,
we have $\supp k\subseteq\Omega^c$. By the choice of
$\Omega$, we obtain
$$
 \operatorname{dist}(\supp f,\supp k)
 \geq\operatorname{dist}(\supp f,\Omega^c)>0.
$$
We may therefore apply Lemma \ref{lem:separated} with
$p=d$ to obtain
$\norm{D_t(f,k)}_{\mathcal S_d}\rightarrow0$
as $t\rightarrow0^{+}$.
Keeping $\varepsilon$ and $h$ fixed while $t\rightarrow0^{+}$,
we combine \eqref{E200}, \eqref{E201}, and the limit for
$D_t(f,k)$ by the triangle inequality to obtain
\begin{align*}
 \limsup_{t\rightarrow0^{+}}\norm{D_t(f,g)}_{\mathcal S_d}
 &\leq\limsup_{t\rightarrow0^{+}}
 \bigl(\norm{D_t(f,h)}_{\mathcal S_d}
       +\norm{D_t(f,u)}_{\mathcal S_d}
       +\norm{D_t(f,k)}_{\mathcal S_d}\bigr)\\
 &\leq C_2[f]_1\varepsilon.
\end{align*}
Letting $\varepsilon\rightarrow0^{+}$ proves the limit for every
fixed $g\in L^\infty(\C^n)$, so \eqref{eq:property}
implies $f\in\mathcal A_d$. This completes the proof.
\end{proof}
\subsection{Approximation and proof of the main theorem}
To extend Proposition \ref{prop:critical-smooth} to bounded
symbols with $|\nabla f|\in L^d(\R^d,dv)$, we approximate
their gradients using a logarithmic cutoff.

\begin{lemma}\label{lem:critical-density}
If $f\in L^\infty(\R^d)$ and $|\nabla f|\in L^d(\R^d,dv)$, with $d>1$,
there are $u_j\in C_c^\infty(\R^d)$ such that
$\norm{\nabla(f-u_j)}_{L^{d}(\R^d,dv)}\rightarrow0$ as $j\rightarrow\infty$.
\end{lemma}
\begin{proof}
For $R>2$, define the radial Lipschitz cutoff
$$
 \chi_R(x)=
 \begin{cases}
 1,&|x|\leq R,\\
 \log(R^2/|x|)/\log R,&R<|x|<R^2,\\
 0,&|x|\geq R^2.
 \end{cases}
$$
The bound $0\leq\chi_R\leq1$ gives
$\chi_R\in L^\infty(\R^d)$.
Its continuous radial profile is constant on each of $[0,R]$
and $[R^2,\infty)$, with derivative $-1/(s\log R)$ for $R<s<R^2$.
Using $\bigl||x|-|y|\bigr|\leq|x-y|$, we obtain
$[\chi_R]_1\leq(R\log R)^{-1}$.
We write $(\tau_h\chi_R)(x)=\chi_R(x+h)$ for $x,h\in\R^d$.
For every $\omega\Subset\R^d$ and $h\in\R^d$, the
Lipschitz estimate yields
$$
 \norm{\tau_h\chi_R-\chi_R}_{L^\infty(\omega)}
 \leq[\chi_R]_1|h|
 \leq\frac{|h|}{R\log R}.
$$
These facts verify condition \textup{(iii)} of
\cite[Proposition 9.3, p. 267]{Brezis2011} with
$p=\infty$. Hence, $\chi_R$ has first weak derivatives
with $|\nabla\chi_R|\in L^\infty(\R^d,dv)$. 
Away from the spheres $|x|=R$ and $|x|=R^2$, it is smooth,
and these weak derivatives coincide with its classical
derivatives by \cite[Section 9.1, Remark 2, p. 264]{Brezis2011}.
On the annulus $R<|x|<R^2$, differentiation gives
$$
 \nabla\chi_R(x)=-\frac{x}{|x|^2\log R}\quad\text{and}\quad
 |\nabla\chi_R(x)|=\frac1{|x|\log R},
$$
while the gradient vanishes almost everywhere outside
this annulus.

Let $\sigma_{d-1}$ denote the surface area of the unit
sphere in $\R^d$. Integration in polar coordinates gives
\begin{equation}\label{E62}
 \begin{aligned}
 \norm{\nabla\chi_R}_{L^d(\R^d,dv)}^d
 &=\frac1{(\log R)^d}
   \int_{\{R<|x|<R^2\}}|x|^{-d} dv(x)=\frac{\sigma_{d-1}}{(\log R)^d}
   \int_R^{R^2}s^{-d}s^{d-1}ds\\
 &=\sigma_{d-1}(\log R)^{1-d}
   \longrightarrow0\quad(R\rightarrow\infty),
 \end{aligned}
\end{equation}
where the last limit uses $d>1$.

On every bounded open set $\Omega\subseteq\R^d$, boundedness
of $f,\chi_R$, the hypothesis on $\nabla f$, and \eqref{E62} give
$$
 f,\chi_R\in L^d(\Omega,dv)\cap L^\infty(\Omega)\quad\text{and}\quad
 |\nabla f|,|\nabla\chi_R|\in L^d(\Omega,dv).
$$
The product rule \cite[Proposition 9.4]{Brezis2011} with $p=d$,
applied to the real and imaginary parts of $f$, therefore yields
\begin{align*}
 \nabla(\chi_Rf)
 &=\chi_R\nabla f+f\nabla\chi_R
 \end{align*}
 and
 \begin{align*}
  \nabla(f-\chi_Rf)
 &=\nabla f-\nabla(\chi_Rf)
  =(1-\chi_R)\nabla f-f\nabla\chi_R.
   \end{align*}
Taking $\Omega=B(0,j)$ for $j=1,2,\ldots$ shows that
these identities hold almost everywhere on $\R^d$.
Using $0\leq1-\chi_R\leq\mathbf1_{\{|x|>R\}}$,
we obtain
\begin{equation}\label{E63}
 \begin{aligned}
 \norm{\nabla(f-\chi_Rf)}_{L^d(\R^d,dv)}
 &\leq
 \norm{\mathbf1_{\{|x|>R\}}\nabla f}_{L^d(\R^d,dv)}+\norm{f}_{L^\infty(\R^d)}
        \norm{\nabla\chi_R}_{L^d(\R^d,dv)}\\
 &\rightarrow0\quad(R\rightarrow\infty).
 \end{aligned}
\end{equation}
The first term tends to zero by dominated convergence,
since $|\nabla f|^d$ is integrable.
The second tends to zero by \eqref{E62}.

With the cutoff approximation controlled by \eqref{E63},
fix $R>2$ and set $f_R=\chi_Rf$.
We now approximate $f_R$ by smooth functions while controlling
the change in its gradient.
The bounds $|f_R|\leq\norm{f}_{L^\infty(\R^d)}$ and
$\supp f_R\subseteq\overline{B(0,R^2)}$ give
$f_R\in L^d(\R^d,dv)\cap L^\infty(\R^d)$.
The product rule $\nabla f_R=\chi_R\nabla f+f\nabla\chi_R$
and \eqref{E62} also give
$$
 \norm{\nabla f_R}_{L^d(\R^d,dv)}
 \leq\norm{\nabla f}_{L^d(\R^d,dv)}
     +\norm{f}_{L^\infty(\R^d)}
      \norm{\nabla\chi_R}_{L^d(\R^d,dv)}<\infty.
$$
Write $v=\nabla f_R=(v_\ell)_{\ell=1}^d$.
Since $|\Re v_\ell|,|\Im v_\ell|\leq|v|\in L^d(\R^d,dv)$,
we apply \cite[Lemma 4.3, p. 114]{Brezis2011} with
$N=d=2n$ and $q=d$ to these real-valued components and obtain
\begin{align*}
 \norm{v(\cdot-y)-v}_{L^d(\R^d,dv)}
 &\leq\sum_{\ell=1}^d
       \norm{v_\ell(\cdot-y)-v_\ell}_{L^d(\R^d,dv)}
 \rightarrow0\quad(y\rightarrow0).
\end{align*}
Using \eqref{E17} and the support and unit mass of
$\eta_\varepsilon$, we apply Minkowski's integral inequality
as in \eqref{E108} to obtain
\begin{equation}\label{E64}
 \begin{aligned}
 \norm{\nabla(f_R*\eta_\varepsilon)-\nabla f_R}_{L^d(\R^d,dv)}
 &=\norm{(\nabla f_R)*\eta_\varepsilon-\nabla f_R}_{L^d(\R^d,dv)}\\
 &=\left\|\int_{B(0,\rho\varepsilon)}\eta_\varepsilon(y)
       \bigl(v(\cdot-y)-v\bigr) dv(y)\right\|_{L^d(\R^d,dv)}\\
 &\leq\int_{B(0,\rho\varepsilon)}\eta_\varepsilon(y)
       \norm{\nabla f_R(\cdot-y)-\nabla f_R}_{L^d(\R^d,dv)} dv(y)\\
 &\leq\sup_{|y|<\rho\varepsilon}
       \norm{\nabla f_R(\cdot-y)-\nabla f_R}_{L^d(\R^d,dv)}
       \longrightarrow0\quad(\varepsilon\rightarrow0^{+}).
 \end{aligned}
\end{equation}
The limit follows from the translation continuity of
$\nabla f_R$ proved above.

For each integer $j\geq1$, we use \eqref{E63} to choose
$R_j>2$ and then apply \eqref{E64} with $R=R_j$ to choose
$\varepsilon_j>0$ so that
$$
 \begin{gathered}
 \norm{\nabla(f-f_{R_j})}_{L^d(\R^d,dv)}
 <\frac1{2j} \quad\text{and}\quad
 \norm{\nabla(f_{R_j}*\eta_{\varepsilon_j})
             -\nabla f_{R_j}}_{L^d(\R^d,dv)}
 <\frac1{2j}.
 \end{gathered}
$$
Set $u_j=f_{R_j}*\eta_{\varepsilon_j}$.
The bound $\supp f_{R_j}\subseteq\overline{B(0,R_j^2)}$
and \eqref{E53} give
$\supp u_j\subseteq\overline{B(0,R_j^2+\rho\varepsilon_j)}$.
Lemma \ref{L6} gives $u_j\in C^\infty(\R^d)$, so
$u_j\in C_c^\infty(\R^d)$.
The estimates for $\nabla(f-f_{R_j})$ and
$\nabla(f_{R_j}-u_j)$ in $L^d(\R^d,dv)$,
together with the triangle inequality, yield
\begin{align*}
 \norm{\nabla(f-u_j)}_{L^d(\R^d,dv)}
 &\leq\norm{\nabla(f-f_{R_j})}_{L^d(\R^d,dv)}
      +\norm{\nabla(f_{R_j}-u_j)}_{L^d(\R^d,dv)}\\
 &<\frac1j\longrightarrow0\quad(j\rightarrow\infty).
 \end{align*}
This proves the required approximation, which completes the proof.
\end{proof}

\begin{proof}[Proof of Theorem \ref{thm:main}]
Part (a) follows from Proposition \ref{prop:necessity} and
Corollary \ref{cor:smallp}. The necessity in part (b) is also
contained in that proposition. To prove sufficiency in part (b), we use the functions $u_j$
from Lemma \ref{lem:critical-density}. For fixed bounded $g$,
\eqref{eq:gradient-semicommutator} at $p=d$ gives
$$
 \norm{D_t(f-u_j,g)}_{\mathcal S_d}
 \leq C\norm{\nabla(f-u_j)}_{L^{d}(\R^d,dv)}\norm{g}_{L^\infty(\C^n)},
$$
uniformly in $t>0$. For each fixed $j$, Proposition
\ref{prop:critical-smooth} gives
$\norm{D_t(u_j,g)}_{\mathcal S_d}\rightarrow0$ as $t\rightarrow0^{+}$.
Using $D_t(f,g)=D_t(f-u_j,g)+D_t(u_j,g)$ and the
triangle inequality, we obtain
$$
 \limsup_{t\rightarrow0^{+}}\norm{D_t(f,g)}_{\mathcal S_d}
 \leq C\norm{\nabla(f-u_j)}_{L^{d}(\R^d,dv)}\norm{g}_{L^\infty(\C^n)}.
$$
The gradient convergence in Lemma \ref{lem:critical-density}
makes the right-hand side tend to zero as $j\rightarrow\infty$.
Since $g$ was arbitrary, \eqref{eq:property} gives $f\in\mathcal A_d$.

For part (c), we take the supremum over
$\norm{g}_{L^\infty(\C^n)}\leq1$ in
\eqref{eq:gradient-semicommutator}, as in the definition
\eqref{E5}. Since $p>d=2n$ and $r=\sqrt t$, we have
$r^{1-d/p}=t^{1/2-n/p}\rightarrow0$ as $t\rightarrow0^{+}$,
which proves \eqref{eq:super-gradient}. The inclusion
$\mathcal A_d\subseteq\mathcal A_p$ follows from
Corollary \ref{cor:smallp}. The proof is complete.
\end{proof}

\section{Uniform integrability and convergence for fixed symbols}\label{sec:remaining}
The proof of Theorem \ref{thm:weakcriterion} begins with the
weak compactness forced by convergence for fixed bounded tests.
We transfer this property from analytic distance to ordinary
mean oscillation, then prove sufficiency by smoothing at a fixed
fraction of the operator scale. The resulting characterization
also yields the algebra properties of $\mathcal A_p$ and its
independence of the weight.
Unless otherwise stated, we assume $2\leq p<\infty$
in the rest of this section.

\subsection{Weak compactness and localization}\label{sec:weak-localization}
Let $\mathcal W$ be the class of relatively weakly compact
subsets of $L^1(\R^d,dv)$. In particular, 
$\mathscr H\in\mathcal W$ implies
$\int_{\R^d}|h| dv<\infty$ for every $h\in\mathscr H$. The duality theorem
\cite[Theorem 4.14, p. 99]{Brezis2011} shows that, for
$h_j,h\in L^1(\R^d,dv)$,
$h_j\rightharpoonup h$ as $j\rightarrow\infty$ if and only if
$$\int_{\R^d}(h_j-h)g dv\rightarrow0$$ for every
$g\in L^\infty(\R^d)$.
By the Dunford--Pettis criterion
\cite{DunfordPettis1940}, in the form given in
\cite[Theorem 4.30, p. 115]{Brezis2011}, a family
$\mathscr H\subseteq L^1(\R^d,dv)$ belongs to $\mathcal W$
if and only if the following conditions hold.
Here $E$ ranges over measurable subsets of $\R^d$.
\begin{equation}\label{eq:intrinsic-three}
 \begin{gathered}
 \sup_{h\in\mathscr H}\norm{h}_{L^1(\R^d,dv)}<\infty,\quad
 \lim_{\delta\rightarrow0^{+}}\sup_{h\in\mathscr H}
 \sup_{|E|\leq\delta}\int_E|h| dv=0,\\
 \lim_{R\rightarrow\infty}\sup_{h\in\mathscr H}
 \int_{|z|>R}|h(z)| dv(z)=0.
 \end{gathered}
\end{equation}

For the last condition, we apply
\cite[Exercise 4.36 (1), condition (e), p. 129]{Brezis2011}
with $\Omega=\R^d$ and $\Omega_j=B(0,j)$, $j\geq1$.
These sets have finite measure and increase to $\R^d$.
Since the integrals over their complements are
nonincreasing in the radius, the limit over integer
radii is equivalent to the limit as $R\rightarrow\infty$.
These conditions express boundedness in
$L^1(\R^d,dv)$, uniform absolute continuity of the integrals,
and uniform smallness of the integrals outside large balls.
Keeping the first and third conditions in
\eqref{eq:intrinsic-three}, we may replace the second by
the level set condition
\begin{equation}\label{eq:dp-three}
 \lim_{K\rightarrow\infty}\sup_{h\in\mathscr H}
                 \int_{\{|h|>K\}}|h| dv=0.
\end{equation}
This replacement is given in
\cite[Exercise 4.36 (1), p. 129]{Brezis2011}.
For complex functions, weak convergence is equivalent to
weak convergence of the real and imaginary parts, and
relative weak compactness is equivalent to that of both
component families. Applying the real criteria to these
components and using
$|\Re h|,|\Im h|\leq|h|\leq|\Re h|+|\Im h|$
gives the same criteria in \eqref{eq:intrinsic-three} and
\eqref{eq:dp-three}. 

Every subfamily of a member of $\mathcal W$ belongs to
$\mathcal W$. The criterion in \eqref{eq:intrinsic-three}
also yields the following domination property.

\begin{lemma}\label{L9}
Let $N\geq1$ be a fixed integer and let $\Gamma$ be an
index set. Suppose that
$\{h_{j,\gamma}:\gamma\in\Gamma\}\in\mathcal W$
for $1\leq j\leq N$.
If measurable functions $k_\gamma:\R^d\to\C$ satisfy
\begin{equation}\label{E22}
 |k_\gamma|\leq C\sum_{j=1}^N|h_{j,\gamma}|
 \quad\text{almost everywhere},\quad \gamma\in\Gamma,
\end{equation}
where $C>0$ is independent of $\gamma$, then
$\{k_\gamma:\gamma\in\Gamma\}\in\mathcal W$.
\end{lemma}

\begin{proof}
For every measurable $E\subseteq\R^d$, \eqref{E22} gives
$$
 \sup_{\gamma\in\Gamma}\int_E|k_\gamma| dv
 \leq C\sum_{j=1}^N
       \sup_{\gamma\in\Gamma}\int_E|h_{j,\gamma}| dv.
$$
With $E=\R^d$, this estimate gives a uniform $L^1$ bound.
Since $N$ is finite and $C$ is independent of $\gamma$,
taking the supremum over $|E|\leq\delta$ and letting
$\delta\rightarrow0^+$ verifies the second condition in
\eqref{eq:intrinsic-three}.
Applying the same estimate to $E=\{|z|>R\}$ and letting
$R\rightarrow\infty$ verifies the third condition.
The three conditions show that $\{k_\gamma:\gamma\in\Gamma\}$
belongs to $\mathcal W$, which completes the proof.
\end{proof}

For $f\in L^\infty(\R^d)$, $t>0$, and a measurable set
$E\subseteq\R^d$, define the localized testing norm
\begin{equation}\label{E23}
 m_{p,t}(f,E)=\sup\{\norm{D_t(f,g)}_{\mathcal S_p}:
        g\in L^\infty(\R^d),\ |g|\leq\mathbf1_E\text{ almost everywhere}\}.
\end{equation}
This is the localized form of \eqref{E5}, with
$m_{p,t}(f,E)\leq m_{p,t}(f)$ and
$m_{p,t}(f,\R^d)=m_{p,t}(f)$.
As before, the supremum is allowed to be infinite.
Convergence for fixed bounded tests yields uniform control
on small sets and outside large balls.

\begin{lemma}\label{lem:localized-testing}
If $f\in\mathcal A_p$ and $p\geq2$, then
\begin{align}
 \lim_{\delta\rightarrow0^{+}}\limsup_{t\rightarrow0^{+}}
   \sup_{|E|\leq\delta}m_{p,t}(f,E)&=0
       \label{eq:test-smallsets}
\end{align}
and \begin{equation}\label{eq:test-tails}
   \lim_{R\rightarrow\infty}\limsup_{t\rightarrow0^{+}}
   m_{p,t}(f,B(0,R)^c)=0.
    \end{equation}
\end{lemma}
\begin{proof}
Lemma \ref{lem:baire} gives fixed numbers $a,t_0>0$ with
$m_{p,t}(f)\leq a$ for $0<t<t_0$. Let $p'=p/(p-1)$ be the conjugate exponent of $p$,
so that $1/p+1/p'=1$. Let $0<t_j<t_0$ satisfy $t_j\rightarrow0^{+}$
as $j\rightarrow\infty$, and let $A_j$ be any finite-rank operator
on $\mathcal F_{t_j}^2$
with $\norm{A_j}_{\mathcal S_{p'}(\mathcal F_{t_j}^2)}\leq1$.
Write $\Tr$ for the trace on $\mathcal S_1(\mathcal F_{t_j}^2)$
and define $\ell_j:L^\infty(\R^d)\to\C$ by
$$
 \ell_j(g)=\Tr(D_{t_j}(f,g)A_j),\qquad g\in L^\infty(\R^d).
$$
We represent this functional by an $L^1(\R^d,dv)$ density
and use its weak convergence to control localized tests. Choose an integer $N_j\geq1$ and vectors
$u_{j,\nu},k_{j,\nu}\in\mathcal F_{t_j}^2$, $1\leq\nu\leq N_j$,
such that
$$
 A_jh=\sum_{\nu=1}^{N_j}\langle h,u_{j,\nu}\rangle_{t_j}
                  k_{j,\nu},\quad h\in\mathcal F_{t_j}^2.
$$

Define
$$
 q_j(z)=-t_j^{-n}e^{-2\varphi(z/\sqrt{t_j})}
   \sum_{\nu=1}^{N_j}k_{j,\nu}(z)
       \overline{H_{\overline{f}}^{(t_j)}u_{j,\nu}(z)}.
$$
Each summand is integrable by the Cauchy--Schwarz inequality in $L_{t_j}^2$.
The finite-rank representation of $A_j$ and
\eqref{eq:factorization} give
\begin{align*}
 \ell_j(g)
 &=\sum_{\nu=1}^{N_j}
       \langle D_{t_j}(f,g)k_{j,\nu},u_{j,\nu}\rangle_{t_j}=-\sum_{\nu=1}^{N_j}\int_{\C^n}
       gk_{j,\nu}
       \overline{H_{\overline{f}}^{(t_j)}u_{j,\nu}}d\mu_{t_j}
   =\int_{\R^d}g(z)q_j(z) dv(z).
 \end{align*}
The second equality uses \eqref{E1} and
$H_{\overline{f}}^{(t_j)}u_{j,\nu}\perp\mathcal F_{t_j}^2$
to replace $H_g^{(t_j)}k_{j,\nu}$ by $gk_{j,\nu}$
in the inner product.
Using the trace bound \cite[Lemma 1.36 (c), p. 27--28]{Zhu1990},
definition \eqref{E5}, and
$\norm{A_j}_{\mathcal S_{p'}(\mathcal F_{t_j}^2)}\leq1$,
we obtain
\begin{align*}
 \norm{q_j}_{L^1(\R^d,dv)}
 &=\sup_{\norm{g}_{L^\infty(\R^d)}\leq1}|\ell_j(g)|\leq
 \norm{A_j}_{\mathcal S_{p'}(\mathcal F_{t_j}^2)}
 \sup_{\norm{g}_{L^\infty(\R^d)}\leq1}
\norm{D_{t_j}(f,g)}_{\mathcal S_p(\mathcal F_{t_j}^2)}\\
 &\leq m_{p,t_j}(f)\leq a.
\end{align*}
The supremum in the first equality is attained by taking
$g=\overline{q_j}/|q_j|$ where $q_j\neq0$ and $g=0$
elsewhere, while the last inequality follows from
$0<t_j<t_0$ and the bound in Lemma \ref{lem:baire}.
For every fixed $g\in L^\infty(\R^d)$, \eqref{eq:property} gives
$$
 \left|\int_{\R^d}gq_j dv\right|
 =|\ell_j(g)|
 \leq\norm{D_{t_j}(f,g)}_{\mathcal S_p(\mathcal F_{t_j}^2)}
 \longrightarrow0\quad(j\rightarrow\infty).
$$
By the duality characterization at the beginning of this
subsection, $q_j\rightharpoonup0$ in $L^1(\R^d,dv)$.
Hence, $\{0\}\cup\{q_j:j\geq1\}$ is weakly compact and
$\{q_j:j\geq1\}\in\mathcal W$.
For measurable $E_j\subseteq\R^d$ and positive $R_j$,
\eqref{eq:intrinsic-three} therefore yields
\begin{equation}\label{eq:dual-density-ui}
 |E_j|\rightarrow0\ \Longrightarrow\ \int_{E_j}|q_j| dv\rightarrow0,
 \quad
 R_j\rightarrow\infty\ \Longrightarrow\
       \int_{|z|>R_j}|q_j(z)| dv(z)\rightarrow0
\end{equation}
as $j\rightarrow\infty$.

Since $(t_j)$ and $(A_j)$ were arbitrary subject to the stated
bounds, \eqref{eq:dual-density-ui} applies to the operators
chosen below to test the failure of \eqref{eq:test-smallsets}.
If \eqref{eq:test-smallsets} failed, the definition \eqref{E23}
would give an $\varepsilon>0$, a sequence $t_j\in(0,t_0)$ with $t_j\rightarrow0^{+}$,
measurable sets
$E_j\subseteq\R^d$ with $|E_j|\rightarrow0$, and $g_j$ satisfying
$$|g_j|\leq\mathbf1_{E_j}\quad \text{and}\quad
\norm{D_{t_j}(f,g_j)}_{\mathcal S_p(\mathcal F_{t_j}^2)}>\varepsilon.$$
The finite-rank norm formula in the proof of
\cite[Lemma 1.37, p. 28--29]{Zhu1990} allows us to choose finite-rank
operators $A_j$ on $\mathcal F_{t_j}^2$, with
$\norm{A_j}_{\mathcal S_{p'}(\mathcal F_{t_j}^2)}\leq1$,
so that their associated densities $q_j$ satisfy
$$
 \frac{\varepsilon}{2}
 <|\Tr(D_{t_j}(f,g_j)A_j)|
 =\left|\int_{\R^d}g_jq_j dv\right|
 \leq\int_{E_j}|q_j| dv
 \longrightarrow0\quad(j\rightarrow\infty).
$$
The limit follows from \eqref{eq:dual-density-ui}, contradicting
$\varepsilon>0$. If \eqref{eq:test-tails} failed, the same construction with
$|g_j|\leq\mathbf1_{B(0,R_j)^c}$ and $R_j\rightarrow\infty$
would contradict the second assertion of
\eqref{eq:dual-density-ui}. This proves both limits. The proof is complete.
\end{proof}

For $S\subseteq\R^d$ and $s>0$, write
$$
 S+B(0,s)=\{x+y:x\in S,\ |y|<s\}
         =\bigcup_{x\in S}B(x,s).
$$
This enlargement is open. Restricting the phase construction in
Lemma \ref{lem:phase} to centers in $S$ gives the following
localized estimate.

\begin{lemma}\label{lem:localized-phase}
The radius $\rho$ is fixed as in Lemma \ref{L1}.
For bounded $f$, $r>0$, and a measurable set $S\subseteq\R^d$,
put $E=S+B(0,\rho r)$. Then
\begin{equation}\label{eq:phase-localized}
 r^{-d}\int_S G_{1,\rho r}^{p}(\overline{f})(z) dv(z)
       \leq C m_{p,r^2}^p(f,E).
\end{equation}
Here $C>0$ depends only on $n,m,M,p$ and the fixed lattice
spacing $L>2\rho$, and is independent of $f,r,S$.
\end{lemma}
\begin{proof}
If $m_{p,r^2}(f,E)=\infty$, there is nothing to prove.
Assume it is finite and fix $L>2\rho$ as in
Lemma \ref{lem:lattice}.
For $b\in[0,Lr)^d$, retain the lattice centers
$a_i=b+Lri$ that belong to $S$.
Lemma \ref{lem:lattice} (a) and the choice of the centers give
$B_i\cap B_j=\varnothing$ for $i\neq j$ and $B_i\subseteq E$,
where $B_i=B(a_i,\rho r)$.
Let $k_i=k_{r^2,a_i}$ be the normalized reproducing kernel
defined before \eqref{E10}, and set
$u_i=H_{\overline{f}}^{(r^2)}k_i$.
Define $g_i$ by \eqref{E29}, with $c_i=k_i\overline{u_i}$.

For a finite set $J$ of the selected indices and each
sign vector $\epsilon=(\epsilon_i)_{i\in J}\in\{-1,1\}^J$,
we set
$$
 g_\epsilon=\sum_{i\in J}\epsilon_i g_i.
$$
Since $|g_i|\leq\mathbf1_{B_i}$ and the balls are disjoint,
we have $|g_\epsilon|\leq\mathbf1_E$, so
$g_\epsilon$ is admissible in \eqref{E23}.
For these indices, we use \eqref{E148} and the first two
inequalities in \eqref{E147}, followed by \eqref{E23}, to obtain
\begin{align*}
 \sum_{i\in J}G_{1,\rho r}^{p}(\overline{f})(a_i)
 &\leq C_0\mathbb E
       \norm{D_{r^2}(f,g_\epsilon)}_{\mathcal S_p}^{p}\leq C_0m_{p,r^2}^p(f,E),
\end{align*}
where $C_0>0$ has the stated dependence and is independent
of $J$ and $b$.

To include all selected centers, take
$$
 J_N=\{i\in\Z^d\cap[-N,N]^d:a_i\in S\},
 \quad N=1,2,\ldots.
$$
The sets satisfy $J_N\subseteq J_{N+1}$ and
$\bigcup_{N=1}^{\infty}J_N=\{i\in\Z^d:a_i\in S\}$.
Since all summands are nonnegative, the same monotone convergence
argument as in \eqref{E65} yields
$$
 \sum_{\substack{i\in\Z^d\\a_i\in S}}
 G_{1,\rho r}^{p}(\overline{f})(a_i)
 =\lim_{N\rightarrow\infty}\sum_{i\in J_N}
 G_{1,\rho r}^{p}(\overline{f})(a_i)
 \leq C_0m_{p,r^2}^p(f,E).
$$
Integrating the preceding bound in $b$ and applying
Tonelli's theorem to the lattice partition, we obtain
\begin{align*}
 \int_S G_{1,\rho r}^{p}(\overline{f})(z) dv(z)
 &=\int_{[0,Lr)^d}
   \sum_{\substack{i\in\Z^d\\b+Lri\in S}}
   G_{1,\rho r}^{p}(\overline{f})(b+Lri) dv(b)\leq C_0(Lr)^d m_{p,r^2}^p(f,E).
\end{align*}
Dividing by $r^d$ proves \eqref{eq:phase-localized}
with $C=C_0L^d$. The proof is complete.
\end{proof}

Recall from \eqref{eq:energy-definitions} that
$e_r(f)(z)=r^{-d}G_{1,r}^{p}(\overline{f})(z)$. Combining the localized estimates gives the necessary weak
compactness condition for the family $e_r(f)$.

\begin{proposition}\label{prop:necessary-wcompact}
If $f\in\mathcal A_p$ and $p\geq2$, then
$\{e_r(f):0<r<r_0\}\in\mathcal W$ for some $r_0>0$.
\end{proposition}
\begin{proof}
We verify the first and third conditions in
\eqref{eq:intrinsic-three}, together with the equivalent
level set condition \eqref{eq:dp-three}.
The radius $\rho>0$ is as in Lemma \ref{L1}.
Fix the lattice spacing $L>2\rho$.
Choose $a_0,t_0>0$ from Lemma \ref{lem:baire} and
$r_0>0$ such that $(2r_0/\rho)^2<t_0$.
Applying \eqref{eq:phase} with $r=s/\rho$ and using
\eqref{eq:energy-definitions}, we obtain
\begin{equation}\label{E72}
\begin{aligned}
 \norm{e_s(f)}_{L^1(\R^d,dv)}
 &=s^{-d}\norm{G_{1,s}(\overline{f})}_{L^p(\R^d,dv)}^p\leq C_p^p\rho^{-d}m_{p,(s/\rho)^2}^p(f)
 \leq C_p^p\rho^{-d}a_0^p=b
\end{aligned}
\end{equation}
for  $0<s\leq2r_0$.
Here $C_p>0$ is the constant in \eqref{eq:phase}, and
$b>0$ is independent of $s$.

Lemma \ref{lem:localized-phase} tests a set $S$ through
its enlargement $S+B(0,s)$ when $r=s/\rho$.
To apply \eqref{eq:test-smallsets} to level sets of $e_s(f)$,
we must control the measure of this enlargement.
We first compare the local analytic distances on nested balls.
Assume that $0<s\leq\tau$ and
$B(z,s)\subseteq B(w,\tau)$.
Each $h\in\mathcal O(B(w,\tau))$ restricts to a holomorphic
function on $B(z,s)$, so \eqref{E106} gives
\begin{align*}
 G_{1,s}(\overline{f})(z)
 &\leq\frac1{|B(z,s)|}
       \int_{B(z,s)}|\overline{f}-h| dv\leq\left(\frac{\tau}{s}\right)^d
       \frac1{|B(w,\tau)|}
       \int_{B(w,\tau)}|\overline{f}-h| dv.
\end{align*}
Taking the infimum over $h$ and using
\eqref{eq:energy-definitions}, we obtain
\begin{equation}\label{E73}
 e_s(f)(z)\leq
 \left(\frac{\tau}{s}\right)^{d+dp}e_\tau(f)(w)
 \quad\text{if }B(z,s)\subseteq B(w,\tau).
\end{equation}

For $K>0$ and $0<s<r_0$, set
$$
 S_{s,K}=\{z\in\R^d:e_s(f)(z)>K\}\quad\text{and}\quad
 E_{s,K}=S_{s,K}+B(0,s).
$$
Equation \eqref{E110} with $r=s/\rho$ shows that
$G_{1,s}(\overline{f})$ is continuous. By
\eqref{eq:energy-definitions}, $e_s(f)$ is also continuous,
so $S_{s,K}$ is open and hence measurable.
If $w\in E_{s,K}$, choose $z\in S_{s,K}$ with $|w-z|<s$.
Then $B(z,s)\subseteq B(w,2s)$, and \eqref{E73} implies
$$
 K<e_s(f)(z)\leq2^{d+dp}e_{2s}(f)(w).
$$
Since $0<s<r_0$, we have $0<2s<2r_0$, and \eqref{E72}
gives $\int_{\R^d}e_{2s}(f) dv\leq b$. Combining this
bound with the preceding pointwise inequality, we obtain
\begin{equation}\label{E202}
\begin{gathered}
 E_{s,K}\subseteq
 \{w\in\R^d:e_{2s}(f)(w)>2^{-d-dp}K\},\quad
 |E_{s,K}|\leq\frac{2^{d+dp}}{K}
       \int_{E_{s,K}}e_{2s}(f) dv
 \leq\frac{2^{d+dp}b}{K}.
\end{gathered}
\end{equation}

Let $C_0>0$ be the constant in \eqref{eq:phase-localized}
and set $C_1=\rho^{-d}C_0$.
Lemma \ref{lem:localized-phase}, applied with
$r=s/\rho$ and $S=S_{s,K}$, gives
$$
 \int_{S_{s,K}}e_s(f) dv
 \leq C_1m_{p,(s/\rho)^2}^p(f,E_{s,K}).
$$
Combining this inequality with the measure bound in \eqref{E202},
we take the upper limit as $s\rightarrow0^{+}$ and use
$t=(s/\rho)^2$ to obtain
$$
 \limsup_{s\rightarrow0^{+}}
 \int_{\{e_s(f)>K\}}e_s(f) dv
 \leq C_1\left(
 \limsup_{t\rightarrow0^{+}}
 \sup_{|E|\leq2^{d+dp}b/K}m_{p,t}(f,E)
 \right)^p.
$$
Since $2^{d+dp}b/K\rightarrow0$ as $K\rightarrow\infty$,
\eqref{eq:test-smallsets} yields
\begin{equation}\label{E74}
 \lim_{K\rightarrow\infty}\limsup_{s\rightarrow0^{+}}
 \int_{\{e_s(f)>K\}}e_s(f) dv=0.
\end{equation}

To estimate the integrals outside large balls, take
$R>0$ and $0<s<\min\{r_0,R/2\}$.
The inclusion
$$
 B(0,R)^c+B(0,s)\subseteq B(0,R/2)^c
$$
allows us to apply Lemma \ref{lem:localized-phase}
with $r=s/\rho$ and $S=B(0,R)^c$.
By the preceding set inclusion and the definition
\eqref{E23}, the localized estimate yields
\begin{align*}
\int_{|z|>R}e_s(f)(z)dv(z)
&\leq C_1m_{p,(s/\rho)^2}^{p}
 \bigl(f,B(0,R)^c+B(0,s)\bigr)\leq C_1m_{p,(s/\rho)^2}^{p}
 \bigl(f,B(0,R/2)^c\bigr).
\end{align*}
For each fixed $R>0$, the restriction $s<\min\{r_0,R/2\}$
holds for all sufficiently small $s>0$. Taking the upper
limit and substituting $t=(s/\rho)^2$, we obtain
\begin{align*}
 \limsup_{s\rightarrow0^{+}}
 \int_{|z|>R}e_s(f)(z) dv(z)
 &\leq C_1\left(
 \limsup_{s\rightarrow0^{+}}
 m_{p,(s/\rho)^2}(f,B(0,R/2)^c)
 \right)^p\\
 &=C_1\left(
 \limsup_{t\rightarrow0^{+}}
 m_{p,t}(f,B(0,R/2)^c)
 \right)^p.
\end{align*}
Since $C_1$ is independent of $R$ and $R/2\rightarrow\infty$,
\eqref{eq:test-tails} makes this upper bound tend to zero
as $R\rightarrow\infty$. The integrals are nonnegative, so
\begin{equation}\label{E75}
 \lim_{R\rightarrow\infty}\limsup_{s\rightarrow0^{+}}
 \int_{|z|>R}e_s(f)(z) dv(z)=0.
\end{equation}

To complete the weak compactness argument, we need the
level set estimates and the estimates outside large balls
to hold uniformly for $0<s<r_0$.
Fix $\varepsilon>0$. By \eqref{E74}, \eqref{E75}, and the
definition of the upper limit, we can first choose
$K_0,R_0>0$ and then $0<a<r_0$ such that
\begin{equation}\label{E111}
 \sup_{0<s<a}\left(
 \int_{\{e_s(f)>K_0\}}e_s(f) dv
 +\int_{|z|>R_0}e_s(f)(z) dv(z)
 \right)<\varepsilon.
\end{equation}
It remains to control the same integrals for $a\leq s<r_0$.
The inclusion $B(z,s)\subseteq B(z,2r_0)$ holds for these scales,
so \eqref{E73} gives the common upper bound
$$
 0\leq e_s(f)\leq
 \left(\frac{2r_0}{s}\right)^{d+dp}e_{2r_0}(f)
 \leq h_a\quad\text{and}\quad
 h_a=\left(\frac{2r_0}{a}\right)^{d+dp}e_{2r_0}(f).
$$
Equation \eqref{E72}, applied at the radius $2r_0$, gives
$\norm{h_a}_{L^1(\R^d,dv)}
\leq(2r_0/a)^{d+dp}b<\infty$.
For $a\leq s<r_0$ and $K,R>0$, the inclusion
$\{e_s(f)>K\}\subseteq\{h_a>K\}$ and the bound
$0\leq e_s(f)\leq h_a$ give
\begin{equation}\label{E112}
 \int_{\{e_s(f)>K\}}e_s(f) dv
 \leq\int_{\{h_a>K\}}h_a dv\quad\text{and}\quad
 \int_{|z|>R}e_s(f)(z) dv(z)
 \leq\int_{|z|>R}h_a(z) dv(z).
\end{equation}
Since $a$ is fixed and $h_a\in L^1(\R^d,dv)$, dominated convergence and
\eqref{E112} yield
\begin{equation}\label{E113}
\begin{aligned}
 &\qquad\sup_{a\leq s<r_0}\left(
 \int_{\{e_s(f)>K\}}e_s(f) dv
 +\int_{|z|>R}e_s(f)(z) dv(z)
 \right)\\
 &\qquad\leq\int_{\{h_a>K\}}h_a dv
 +\int_{|z|>R}h_a(z) dv(z)
 \longrightarrow0\quad(K,R\rightarrow\infty).
\end{aligned}
\end{equation}
We choose $K\geq K_0$ and $R\geq R_0$ so that the upper
bound in \eqref{E113} is less than $\varepsilon$.
For these choices, we have
$$
 \{e_s(f)>K\}\subseteq\{e_s(f)>K_0\}\quad\text{and}\quad
 \{|z|>R\}\subseteq\{|z|>R_0\}.
$$
Since $e_s(f)\geq0$, \eqref{E111} also holds with
$K,R$ in place of $K_0,R_0$. The bounds in \eqref{E111}
and \eqref{E113} now cover $0<s<a$ and $a\leq s<r_0$,
respectively, and give
\begin{equation}\label{E76}
 \sup_{0<s<r_0}\left(
 \int_{\{e_s(f)>K\}}e_s(f) dv
 +\int_{|z|>R}e_s(f)(z) dv(z)
 \right)<\varepsilon.
\end{equation}

Since $\varepsilon>0$ is arbitrary, \eqref{E76} verifies
the level set condition \eqref{eq:dp-three} and the
condition outside large balls in \eqref{eq:intrinsic-three}.
Together with the uniform $L^1$ bound in \eqref{E72},
these are precisely the weak compactness conditions
stated in \eqref{eq:intrinsic-three}, with its second
condition replaced by \eqref{eq:dp-three}.
We conclude that
$\{e_s(f):0<s<r_0\}\in\mathcal W$, which completes the proof.
\end{proof}

\subsection{From analytic distance to ordinary oscillation}
For $2\leq p<\infty$ and bounded $f$, we prove in
Lemma \ref{lem:ui-analytic-real} that relative weak compactness
of the analytic distance densities for sufficiently small radii
is equivalent to the same property for the ordinary oscillation
densities. The analytic decomposition controls the $\dbar$
derivatives of its smooth part, whereas a Poincar\'e inequality
uses its real gradient. We therefore recover the real gradient
from the complex derivatives using operators that preserve
relative weak compactness of the corresponding densities.
We first establish the required operator estimates for families
satisfying $\{|u_\gamma|^p:\gamma\in\Gamma\}\in\mathcal W$,
where $\Gamma$ is an arbitrary index set.
For a fixed integer $N\geq1$, we write
$$
 u_\gamma=(u_{\gamma,1},\ldots,u_{\gamma,N}),\quad
 |u_\gamma(x)|=
 \left(\sum_{j=1}^N|u_{\gamma,j}(x)|^2\right)^{1/2},
 \quad x\in\R^d.
$$
A linear operator $S$ on $L^p(\R^d,dv)$ acts on
$u_\gamma\in[L^p(\R^d,dv)]^N$ componentwise by
$$
 Su_\gamma=(Su_{\gamma,1},\ldots,Su_{\gamma,N}).
$$

For $u\in L^1_{\rm loc}(\R^d,dv)$, the local averages in
\eqref{E4} define the centered Hardy--Littlewood maximal
operator \cite[Chapter I, Section 1.1, p. 4]{Stein1970} by
\begin{equation}\label{E24}
 \mathcal Mu(x)=\sup_{s>0}M_{1,s}(u)(x).
\end{equation}
For $u\in[L^1_{\rm loc}(\R^d,dv)]^N$, we use its Euclidean
length and the conventions
$$
 M_{1,s}(u)=M_{1,s}(|u|),\quad
 \mathcal Mu=\mathcal M(|u|).
$$
The averages are nonnegative real valued, and the maximal function
may take values in $[0,\infty]$.

To express real derivatives in terms of the components of $\dbar$,
we introduce the following Fourier multipliers.
For $x=(x_1,\ldots,x_d)$ and $y=(y_1,\ldots,y_d)$ in $\R^d$,
we define their real Euclidean inner product by
\begin{equation}\label{E109}
 x\cdot y=\sum_{\ell=1}^d x_\ell y_\ell,\quad x,y\in\R^d.
\end{equation}
For $u\in C_c^\infty(\R^d)$, let
$$\widehat u(\xi)=\int_{\R^d}u(x)e^{-ix\cdot\xi} dv(x)$$
be its Fourier transform at
$\xi=(\xi_1,\ldots,\xi_d)\in\R^d$, with $x\cdot\xi$ as in
\eqref{E109}. For tempered distributions, we use the distributional Fourier
transform in \cite[Section 7.14, p. 192]{Rudin1991} with
our normalization. Rudin's Fourier transform is
$(2\pi)^{-d/2}$ times ours
\cite[Section 7.1 (a), (c), p. 182--183]{Rudin1991}.
His first-order differential operators are
$i^{-1}\partial_{x_\ell}$ for $1\leq\ell\leq d$
\cite[Section 7.1 (d), p. 183]{Rudin1991}.
The constant normalization factor cancels on both sides of
his differentiation formula
\cite[Theorem 7.15 (b), p. 192]{Rudin1991}, which in our
notation becomes
\begin{equation}\label{E198}
 \widehat{\partial_{x_\ell}T}=i\xi_\ell\widehat T,
 \quad 1\leq\ell\leq d,
\end{equation}
for every tempered distribution $T$ on $\R^d$, with equality
in the sense of tempered distributions.

To apply Stein's results, we write
$\widehat u_S$ for the Fourier transform in
\cite[Chapter II, Section 1.3, p. 28]{Stein1970}. Its definition gives
$$
 \widehat u_S(\eta)=\int_{\R^d}u(x)e^{2\pi i x\cdot\eta} dv(x)
 \quad\text{and}\quad
 \widehat u(\xi)=\widehat u_S\!\left(-\frac{\xi}{2\pi}\right).
$$
The Plancherel identity stated there and the substitution
$\xi=-2\pi\eta$, with $dv(\xi)=(2\pi)^d dv(\eta)$, yield
\begin{equation}\label{E139}
 \norm{\widehat h}_{L^2}
 =(2\pi)^{d/2}\norm{\widehat h_S}_{L^2}
 =(2\pi)^{d/2}\norm{h}_{L^2},
 \quad h\in C_c^\infty(\R^d).
\end{equation}
By \cite[Corollary 4.23, p. 109]{Brezis2011}, applied to
real and imaginary parts, $C_c^\infty(\R^d)$ is dense in
$L^q(\R^d,dv)$ for $1\leq q<\infty$.
For $u\in L^2(\R^d,dv)$, choose such a sequence $(u_j)$
with $u_j\to u$ in that space and define
$$
 \widehat u=\lim_{j\to\infty}\widehat{u_j}
 \quad\text{in }L^2(\R^d,dv).
$$
Equation \eqref{E139} shows that this limit exists and
is independent of the approximating sequence. The same identity
remains valid for every $h\in L^2(\R^d,dv)$. The unitarity stated in
\cite[Chapter II, Section 1.3, p. 28]{Stein1970}, with the
frequency substitution above, shows that
$u\mapsto(2\pi)^{-d/2}\widehat u$ is unitary on $L^2(\R^d,dv)$.
Since $|\xi_\ell\xi_k|\leq|\xi|^2$, the following identity
therefore defines a unique bounded operator
$\mathcal R_{\ell k}:L^2(\R^d,dv)\to L^2(\R^d,dv)$
\begin{equation}\label{E25}
 \widehat{\mathcal R_{\ell k}u}(\xi)
 =\frac{\xi_\ell\xi_k}{|\xi|^2}\widehat u(\xi),
 \quad\text{a.e. }\xi\neq0,\quad 1\leq\ell,k\leq d,
\end{equation}
with the multiplier set to zero at $\xi=0$.
This multiplier is smooth away from the origin, even,
and homogeneous of degree zero. Its evenness and homogeneity
leave it unchanged under the frequency substitution
$\xi\mapsto-\xi/(2\pi)$.

For $1<q<\infty$ and
$u\in L^2(\R^d,dv)\cap L^q(\R^d,dv)$,
\cite[Chapter IV, Section 3.2, Theorem 3 and
Example (2), p. 96]{Stein1970} gives
\begin{equation}\label{E140}
 \norm{\mathcal R_{\ell k}u}_{L^q(\R^d,dv)}
 \leq C_{d,q}\norm{u}_{L^q(\R^d,dv)},
\end{equation}
where $C_{d,q}>0$ depends only on $d,q$.
The extension statement in
\cite[Chapter IV, Section 3.1, p. 94]{Stein1970}
applies to the restriction
$\mathcal R_{\ell k}:L^2\cap L^q\to L^q$.
Since $C_c^\infty(\R^d)\subseteq L^2\cap L^q$ is dense
in $L^q$, this restriction has a unique bounded extension
$L^q(\R^d,dv)\to L^q(\R^d,dv)$ satisfying
\eqref{E140} for every $u\in L^q(\R^d,dv)$.
The following lemma identifies the values of these extensions
on intersections of their domains.

\begin{lemma}\label{L10}
Let $1<q_1,q_2<\infty$, and let $S$ be a bounded linear
operator on $L^2(\R^d,dv)$.
For $i=1,2$, suppose that the restriction of $S$ to
$L^2(\R^d,dv)\cap L^{q_i}(\R^d,dv)$ has a bounded linear
extension $S_i:L^{q_i}(\R^d,dv)\to L^{q_i}(\R^d,dv)$,
so that $S_i v=Sv$ for every $v\in L^2\cap L^{q_i}$. Then
$$
 S_1u=S_2u\quad\text{almost everywhere},\qquad
 u\in L^{q_1}(\R^d,dv)\cap L^{q_2}(\R^d,dv).
$$
\end{lemma}
\begin{proof}
Fix $u\in L^{q_1}(\R^d,dv)\cap L^{q_2}(\R^d,dv)$ and set
$$
 u_j=u\mathbf1_{B(0,j)}\mathbf1_{\{|u|\leq j\}},
 \quad j\geq1.
$$
These truncations belong to
$L^2(\R^d,dv)\cap L^{q_1}(\R^d,dv)\cap L^{q_2}(\R^d,dv)$,
and dominated convergence gives
$$
 \norm{u-u_j}_{L^{q_1}(\R^d,dv)}
 +\norm{u-u_j}_{L^{q_2}(\R^d,dv)}
 \longrightarrow0\quad(j\rightarrow\infty).
$$
Since $S_1u_j=Su_j=S_2u_j$, H\"older's inequality on
$B=B(0,R)$, with $R>0$ fixed, yields
\begin{align*}
 \norm{S_1u-S_2u}_{L^1(B,dv)}
 &\leq\sum_{i=1}^2
 |B|^{1-1/q_i}
 \norm{S_i(u-u_j)}_{L^{q_i}(\R^d,dv)}\\
 &\leq\sum_{i=1}^2
 |B|^{1-1/q_i}\norm{S_i}
 \norm{u-u_j}_{L^{q_i}(\R^d,dv)}
 \rightarrow0\quad(j\rightarrow\infty).
\end{align*}
The preceding limit implies that $S_1u=S_2u$
almost everywhere on $B(0,R)$.
Taking $R=1,2,\ldots$ proves the assertion on $\R^d$. The proof is complete.
\end{proof}

Applying Lemma \ref{L10} to the operator defined in
\eqref{E25}, we see that its $L^q$ extensions agree on
intersections. We therefore denote all these extensions
by $\mathcal R_{\ell k}$. For the derivative identities below, we set
\begin{equation}\label{E84}
 \Delta_z=\sum_{\ell=1}^d\partial_{x_\ell}^2.
\end{equation}
The next lemma gives the kernel representation needed to estimate
the integrals outside large balls.

\begin{lemma}\label{L5}
For $1\leq\ell,k\leq d$, there is a kernel
$K_{\ell k}:\R^d\setminus\{0\}\to\R$ satisfying
\begin{equation}\label{E79}
 |K_{\ell k}(x-y)|\leq C_d|x-y|^{-d},\quad x\neq y,
\end{equation}
where $C_d>0$ depends only on $d$.
If $u\in L^\infty(\R^d)$ has compact support, then
\begin{equation}\label{E80}
 \mathcal R_{\ell k}u(x)
 =\int_{\R^d}K_{\ell k}(x-y)u(y) dv(y)
 \quad\text{for almost every }x\notin\supp u.
\end{equation}
The integral in \eqref{E80} is absolutely convergent.
\end{lemma}
\begin{proof}
Let $\delta_{\ell k}$ be the Kronecker delta.
Write $\sigma_{d-1}$ for the surface area of the unit sphere
$S^{d-1}$, with surface measure $d\sigma$.
Set $b_d=-d/\sigma_{d-1}$ and
$$
 h_{\ell k}(x)=x_\ell x_k-\frac{\delta_{\ell k}}d|x|^2,
 \quad
 K_{\ell k}(x)=b_d\frac{h_{\ell k}(x)}{|x|^{d+2}},
 \quad x\neq0.
$$
For $z\neq0$, these definitions give
\begin{align*}
 |h_{\ell k}(z)|
 &=\left|z_\ell z_k-\frac{\delta_{\ell k}}d|z|^2\right|\leq |z_\ell||z_k|+\frac{\delta_{\ell k}}d|z|^2
 \leq2|z|^2.
\end{align*}
The definition of $K_{\ell k}$ then gives
 $$|K_{\ell k}(z)|
 =|b_d|\frac{|h_{\ell k}(z)|}{|z|^{d+2}}
 \leq2|b_d||z|^{-d}.$$
Taking $z=x-y$ proves \eqref{E79} with $C_d=2|b_d|>0$.

Since $u$ is bounded and compactly supported, it belongs
to $L^1(\R^d,dv)\cap L^2(\R^d,dv)$.
To obtain \eqref{E80}, we first establish the representation
\begin{equation}\label{E81}
 \left(\mathcal R_{\ell k}-\frac{\delta_{\ell k}}dI\right)u(x)
 =\lim_{\varepsilon\rightarrow0^+}
 \int_{\{y\in\R^d:|x-y|>\varepsilon\}}
       K_{\ell k}(x-y)u(y) dv(y),
\end{equation}
where the limit is taken in $L^2(\R^d,dv)$. The polynomial $h_{\ell k}$ is homogeneous of degree two, so
the angular function
$\Omega_{\ell k}(x)=b_dh_{\ell k}(x)/|x|^2$
is homogeneous of degree zero, with
$K_{\ell k}(x)=\Omega_{\ell k}(x)/|x|^d$.
For $\theta=(\theta_1,\ldots,\theta_d)\in S^{d-1}$,
the identity $|\theta|=1$ gives
$$
 \Omega_{\ell k}(\theta)
 =b_d\left(\theta_\ell\theta_k-\frac{\delta_{\ell k}}d\right).
$$
Spherical symmetry then yields
$$
 \int_{S^{d-1}}\Omega_{\ell k}(\theta)d\sigma(\theta)
 =b_d\left(
 \int_{S^{d-1}}\theta_\ell\theta_kd\sigma(\theta)
 -\frac{\delta_{\ell k}}d\sigma_{d-1}\right)=0.
$$
Moreover, for $\vartheta=(\vartheta_1,\ldots,\vartheta_d)\in S^{d-1}$,
we have
\begin{align*}
 |\Omega_{\ell k}(\theta)-\Omega_{\ell k}(\vartheta)|
 &=|b_d||\theta_\ell\theta_k-\vartheta_\ell\vartheta_k|\leq2|b_d||\theta-\vartheta|.
 \end{align*}
For $0<s\leq1$, we define the spherical modulus of continuity
over pairs of points at Euclidean distance at most $s$ by
$$
 \omega_{\ell k}(s)
 =\sup_{\substack{\theta,\vartheta\in S^{d-1}\\
                  |\theta-\vartheta|\leq s}}
   |\Omega_{\ell k}(\theta)-\Omega_{\ell k}(\vartheta)|.
$$
The preceding estimate gives $\omega_{\ell k}(s)\leq2|b_d|s$,
so the Dini condition follows from
$$
 \int_0^1\frac{\omega_{\ell k}(s)}sds
 \leq\int_0^1 2|b_d|ds=2|b_d|<\infty.
$$

For $\varepsilon>0$ and $v\in L^2(\R^d,dv)$, we define
\begin{align*}
 T_\varepsilon v(x)
 &=\int_{\{z\in\R^d:|z|>\varepsilon\}}
       \frac{\Omega_{\ell k}(z)}{|z|^d}v(x-z) dv(z)=\int_{\{y\in\R^d:|x-y|>\varepsilon\}}
       K_{\ell k}(x-y)v(y) dv(y),
\end{align*}
where the second equality follows from $y=x-z$.
Equation \eqref{E79} and the Cauchy--Schwarz inequality give
$$
 \int_{|z|>\varepsilon}|K_{\ell k}(z)v(x-z)| dv(z)
 \leq C_d\left(\frac{\sigma_{d-1}}d\right)^{1/2}
       \varepsilon^{-d/2}\norm{v}_{L^2(\R^d,dv)}<\infty,
$$
so these truncated integrals are absolutely convergent.
The homogeneity, cancellation, and Dini condition verified
above allow us to apply
\cite[Chapter II, Section 4.2, Theorem 3 (a)--(b), p. 39]{Stein1970}
with exponent $2$. The operators $T_\varepsilon:L^2\to L^2$
have operator norms bounded uniformly for $\varepsilon>0$,
and there is a bounded operator
$T:L^2(\R^d,dv)\to L^2(\R^d,dv)$ such that
\begin{equation}\label{E114}
 \norm{T_\varepsilon v-Tv}_{L^2(\R^d,dv)}
 \longrightarrow0\quad(\varepsilon\longrightarrow0^+),
 \quad v\in L^2(\R^d,dv).
\end{equation}
To identify $T$, we note that $h_{\ell k}$ is harmonic, since
$$
 \sum_{j=1}^d\partial_{x_j}^2h_{\ell k}
 =2\delta_{\ell k}-\frac{\delta_{\ell k}}d2d=0.
$$
The multiplier formula in
\cite[Chapter III, Section 3.3, Theorem 5, p. 73]{Stein1970}
therefore applies to $h_{\ell k}$, with constant
$-\sigma_{d-1}/d$ for degree two.
Since $b_d(-\sigma_{d-1}/d)=1$, this multiplier formula
and definition \eqref{E25} give
\begin{align*}
 \widehat{Tv}(\xi)
 &=\frac{h_{\ell k}(\xi)}{|\xi|^2}\widehat v(\xi)=\left(\frac{\xi_\ell\xi_k}{|\xi|^2}
          -\frac{\delta_{\ell k}}d\right)\widehat v(\xi)=\widehat{\left(\mathcal R_{\ell k}
                -\frac{\delta_{\ell k}}dI\right)v}(\xi)
\end{align*}
for every $v\in L^2(\R^d,dv)$ and almost every $\xi\neq0$.
The $L^2$ Plancherel identity \eqref{E139} implies
injectivity of the Fourier transform, so we obtain
\begin{equation}\label{E141}
 T=\mathcal R_{\ell k}-\frac{\delta_{\ell k}}dI
 \quad\text{on }L^2(\R^d,dv).
\end{equation}
Applying \eqref{E114} and \eqref{E141} to the fixed
bounded, compactly supported function $u$ proves \eqref{E81}.

To obtain \eqref{E80}, we pass from the $L^2$ convergence in
\eqref{E114} to almost everywhere convergence along a subsequence.
For $i=1,2,\ldots$, we define the nonnegative functions on $\R^d$ by
$$
 v_i(x)=\left|
   \int_{\{y\in\R^d:|x-y|>1/i\}}
       K_{\ell k}(x-y)u(y) dv(y)-Tu(x)\right|.
$$
Equation \eqref{E114} with $\varepsilon=1/i$ gives
$v_i\in L^2(\R^d,dv)$ and
$$\norm{v_i}_{L^2(\R^d,dv)}
=\norm{T_{1/i}u-Tu}_{L^2(\R^d,dv)}\rightarrow0$$
as $i\rightarrow\infty$.
We apply \cite[Theorem 4.9, p. 94]{Brezis2011} to $(v_i)$
and select a subsequence $(v_{i_j})$.
With $\varepsilon_j=1/i_j\to0^+$, we then have
$|T_{\varepsilon_j}u(x)-Tu(x)|=v_{i_j}(x)\rightarrow0$
for almost every $x\in\R^d$ as $j\rightarrow\infty$.

To determine this limit outside the support, we fix
$x\notin\supp u$ and put
$a_x=\operatorname{dist}(x,\supp u)>0$.
Equation \eqref{E79} then gives
$$
 \int_{\R^d}|K_{\ell k}(x-y)u(y)| dv(y)
 \leq C_da_x^{-d}\norm{u}_{L^1(\R^d,dv)}<\infty.
$$
For $0<\varepsilon<a_x$, we have
$
 \{y\in\R^d:|x-y|\leq\varepsilon\}\cap\supp u=\varnothing
$.
This implies $$\supp u\subseteq\{y\in\R^d:|x-y|>\varepsilon\}.$$
Since $u=0$ almost everywhere outside $\supp u$, we have
\begin{align*}
 T_\varepsilon u(x)
 &=\int_{\supp u}K_{\ell k}(x-y)u(y) dv(y)=\int_{\R^d}K_{\ell k}(x-y)u(y) dv(y)
\end{align*}
for  $0<\varepsilon<a_x$.
We can now pass to the limit along $(\varepsilon_j)$ to obtain
\begin{align*}
 \left(\mathcal R_{\ell k}-\frac{\delta_{\ell k}}dI\right)u(x)
 &=\lim_{j\to\infty}
   \int_{\{y\in\R^d:|x-y|>\varepsilon_j\}}
       K_{\ell k}(x-y)u(y) dv(y)\\
 &=\int_{\R^d}K_{\ell k}(x-y)u(y) dv(y)
\end{align*}
for almost every $x\notin\supp u$.
The first equality uses the almost everywhere convergence
$T_{\varepsilon_j}u(x)\to Tu(x)$ and the operator identity
\eqref{E141}.
For the second, $\varepsilon_j\to0^+$ implies
$\varepsilon_j<a_x$ for all sufficiently large $j$, so the
formula for $T_\varepsilon u(x)$ makes each of these truncated
integrals equal to the full integral.
Since $u(x)=0$ for almost every $x\notin\supp u$, the
left-hand side equals $\mathcal R_{\ell k}u(x)$.
This proves \eqref{E80} and completes the proof.
\end{proof}

We combine the kernel estimate \eqref{E79} and representation
\eqref{E80} with the $L^q$ bounds \eqref{E140} to prove
the following preservation of weak compactness.

\begin{lemma}\label{lem:ui-stability}
Let $1<p<\infty$, let $N\geq1$ be a fixed integer,
and let $\Gamma$ be an arbitrary index set.
Suppose $u_\gamma\in[L^p(\R^d,dv)]^N$ and
$
 \{|u_\gamma|^p:\gamma\in\Gamma\}\in\mathcal W
$.
For fixed coefficients $a_0,a_{\ell k}\in\C$, set
$$
 S=a_0I+\sum_{\ell,k=1}^d a_{\ell k}\mathcal R_{\ell k}.
$$
Then $\{|Su_\gamma|^p:\gamma\in\Gamma\}\in\mathcal W $ and $
 \{|\mathcal Mu_\gamma|^p:\gamma\in\Gamma\}\in\mathcal W
$.
Moreover, for any fixed $s_0>0$ and radii
$0<s_\gamma\leq s_0$, we have
$
 \{M_{1,s_\gamma}^{p}(u_\gamma):\gamma\in\Gamma\}
 \in\mathcal W
$.
\end{lemma}

\begin{proof}
We first prove the assertions for $S$ and $\mathcal M$.
Write $T$ for either operator.
For $1<q<\infty$, the $L^q$ bounds in \eqref{E140} imply
$\norm{Sh}_{L^q(\R^d,dv)}\leq b_q\norm{h}_{L^q(\R^d,dv)}$
for every $h\in L^q(\R^d,dv)$ with values in $\C$, where
$b_q>0$ depends only on $d,q$ and the fixed coefficients of $S$.
For $u=(u_1,\ldots,u_N)\in[L^q(\R^d,dv)]^N$, applying
this bound componentwise yields
\begin{align*}
 \norm{Su}_{L^q(\R^d,dv)}
 &\leq\sum_{j=1}^N\norm{Su_j}_{L^q(\R^d,dv)}\leq b_q\sum_{j=1}^N\norm{u_j}_{L^q(\R^d,dv)}
 \leq Nb_q\norm{u}_{L^q(\R^d,dv)}.
\end{align*}
For $\mathcal M$, we apply
\cite[Chapter I, Section 1.3, Theorem 1 (c), p. 5]{Stein1970}
to the nonnegative function $|u|$ for the corresponding $L^q$ bound.
For $v\in[L^q(\R^d,dv)]^N$, the triangle inequality and
\eqref{E24} also give
\begin{align*}
 \mathcal M(u+v)(x)
 &\leq\sup_{s>0}
       \bigl(M_{1,s}(u)(x)+M_{1,s}(v)(x)\bigr)\leq\mathcal Mu(x)+\mathcal Mv(x).
\end{align*}
Combining these estimates with the linearity of $S$, we obtain
\begin{equation}\label{E77}
 \norm{Tu}_{L^q(\R^d,dv)}
 \leq c_q\norm{u}_{L^q(\R^d,dv)}\quad\text{and}\quad
 |T(u+v)|\leq|Tu|+|Tv|
\end{equation}
for $1<q<\infty$ and $u,v\in[L^q(\R^d,dv)]^N$.
Here $c_q>0$ depends only on $d,N,q$ and the fixed
coefficients of $S$. Applying \eqref{E77} with $q=p$ and using the first
condition in \eqref{eq:intrinsic-three} for the given family
$\{|u_\gamma|^p:\gamma\in\Gamma\}\in\mathcal W$, we obtain
\begin{equation}\label{E85}
 \begin{aligned}
 \sup_{\gamma\in\Gamma}\norm{|Tu_\gamma|^p}_{L^1(\R^d,dv)}
 &=\sup_{\gamma\in\Gamma}\norm{Tu_\gamma}_{L^p(\R^d,dv)}^p\leq c_p^p
      \sup_{\gamma\in\Gamma}\norm{u_\gamma}_{L^p(\R^d,dv)}^p
 <\infty.
 \end{aligned}
\end{equation}
This verifies the first condition in \eqref{eq:intrinsic-three}
for $\{|Tu_\gamma|^p:\gamma\in\Gamma\}$.

To prove the remaining two conditions, concerning uniform
absolute continuity of the integrals and their uniform
smallness outside large balls, fix $\varepsilon>0$.
Applying \eqref{eq:intrinsic-three} and \eqref{eq:dp-three}
to the family $\{|u_\gamma|^p:\gamma\in\Gamma\}$,
we choose $K,R>0$ such that
\begin{equation}\label{E115}
 \sup_{\gamma\in\Gamma}\left(
 \int_{\{|u_\gamma|>K\}}|u_\gamma|^p dv
 +\int_{|x|>R}|u_\gamma(x)|^p dv(x)
 \right)<\varepsilon^p.
\end{equation}
Define the truncated functions and remainders by
$$
 f_\gamma=u_\gamma\mathbf1_{B(0,R)}
                    \mathbf1_{\{|u_\gamma|\leq K\}},\quad
 g_\gamma=u_\gamma-f_\gamma.
$$
The definition of $g_\gamma$ gives
\begin{align*}
 \norm{g_\gamma}_{L^p(\R^d,dv)}^p
 &\leq\int_{\{|u_\gamma|>K\}}|u_\gamma|^p dv
       +\int_{|x|>R}|u_\gamma(x)|^p dv(x).
\end{align*}
Taking the supremum in $\gamma$ and using \eqref{E115},
we obtain the two bounds
\begin{equation}\label{E116}
 |f_\gamma|\leq K\mathbf1_{B(0,R)}\quad\text{and}\quad
 \sup_{\gamma\in\Gamma}
 \norm{g_\gamma}_{L^p(\R^d,dv)}<\varepsilon.
\end{equation}
For every measurable $E\subseteq\R^d$, the $L^p$ estimate in \eqref{E77} and \eqref{E116} imply
\begin{align*}
 \sup_{\gamma\in\Gamma}\int_E|Tg_\gamma|^p dv
 &\leq\sup_{\gamma\in\Gamma}\norm{Tg_\gamma}_{L^p(\R^d,dv)}^p\leq c_p^p\sup_{\gamma\in\Gamma}
        \norm{g_\gamma}_{L^p(\R^d,dv)}^p
 <c_p^p\varepsilon^p.
\end{align*}
Combining this remainder estimate with $u_\gamma=f_\gamma+g_\gamma$,
subadditivity in \eqref{E77}, and
$(a+b)^p\leq2^{p-1}(a^p+b^p)$ for $a,b\geq0$, we obtain
\begin{equation}\label{E78}
 \sup_{\gamma\in\Gamma}\int_E|Tu_\gamma|^p dv
 \leq 2^{p-1}\sup_{\gamma\in\Gamma}
                  \int_E|Tf_\gamma|^p dv
      +2^{p-1}c_p^p\varepsilon^p
\end{equation}
for every measurable $E\subseteq\R^d$.
By \eqref{E78}, it suffices to estimate the truncated functions
$f_\gamma$, whose $L^\infty$ norms are bounded by $K$ and
whose supports lie in $\overline{B(0,R)}$, uniformly in $\gamma$.

By \eqref{E116}, each component of $f_\gamma$ belongs to
$L^p(\R^d,dv)\cap L^{2p}(\R^d,dv)$, so Lemma \ref{L10}
identifies the $L^p$ and $L^{2p}$ extensions of $S$ on $f_\gamma$.
For a measurable set $E\subseteq\R^d$ of finite measure,
H\"older's inequality with exponents $2,2$ and \eqref{E77}
with $q=2p$ give
\begin{equation}\label{E117}
\begin{aligned}
 \sup_{\gamma\in\Gamma}\int_E|Tf_\gamma|^p dv
 &\leq |E|^{1/2}
       \sup_{\gamma\in\Gamma}
       \norm{Tf_\gamma}_{L^{2p}(\R^d,dv)}^p\leq c_{2p}^p|E|^{1/2}
       \sup_{\gamma\in\Gamma}
       \norm{f_\gamma}_{L^{2p}(\R^d,dv)}^p\\
 &\leq c_{2p}^pK^p|B(0,R)|^{1/2}|E|^{1/2}.
\end{aligned}
\end{equation}
The last inequality uses the pointwise bound in
\eqref{E116}. Combining \eqref{E117} with \eqref{E78}
and taking the supremum over all measurable $E$ with
$|E|\leq\delta$, we obtain
\begin{equation}\label{E86}
 \sup_{\gamma\in\Gamma}\sup_{|E|\leq\delta}
       \int_E|Tu_\gamma|^p dv
 \leq 2^{p-1}c_{2p}^pK^p|B(0,R)|^{1/2}\delta^{1/2}
      +2^{p-1}c_p^p\varepsilon^p
\end{equation}
for $\delta>0$.

We next estimate the integrals outside large balls.
Suppose first that $T=S$.
Each component of $f_\gamma$ is bounded and supported in
$\overline{B(0,R)}$, so the $L^2$ kernel representation in
Lemma \ref{L5} applies outside its support.
By Lemma \ref{L10}, this representation also
holds for the $L^p$ extensions used in \eqref{E78}.
For $|x|>2R$, the identity terms in the kernel
representation of $S$ vanish, and
$|x-y|\geq|x|/2$ for $y\in B(0,R)$.
Equations \eqref{E79} and \eqref{E80}, together with the
pointwise bound in \eqref{E116}, give
\begin{align*}
 |Sf_\gamma(x)|
 &\leq C_d\sum_{\ell,k=1}^d|a_{\ell k}|
       \int_{B(0,R)}|x-y|^{-d}|f_\gamma(y)| dv(y)\leq c_SK|B(0,R)||x|^{-d}
\end{align*}
for almost every $x$ with $|x|>2R$, where $c_S>0$ depends only
on $d,N$ and the coefficients of $S$.

For $T=\mathcal M$, if $|x|>2R$ and
$B(x,s)\cap B(0,R)\neq\varnothing$, then
$s>|x|-R>|x|/2$.
For $s\leq|x|/2$, the integral over $B(x,s)$ is zero because
$f_\gamma$ vanishes outside $B(0,R)$.
Using \eqref{E24}, the pointwise bound in \eqref{E116},
and $|B(0,s)|=|B(0,1)|s^d$, we obtain
\begin{align*}
 \mathcal Mf_\gamma(x)
 &=\sup_{s>|x|/2}\frac{1}{|B(0,s)|}
       \int_{B(x,s)}|f_\gamma(y)| dv(y)\leq\sup_{s>|x|/2}
       \frac{K|B(0,R)|}{|B(0,s)|}
 \leq\frac{2^dK|B(0,R)|}{|B(0,1)||x|^d}.
\end{align*}
Hence, with
$c_0=\max\{c_S,2^d/|B(0,1)|\}>0$,
both operators satisfy
\begin{equation}\label{E118}
 |Tf_\gamma(x)|
 \leq c_0K|B(0,R)||x|^{-d}
 \quad\text{for almost every }x\text{ with }|x|>2R.
\end{equation}

Let $\sigma_{d-1}$ denote the area of the unit sphere
in $\R^d$.
For $a>2R$, integrating \eqref{E118} in polar coordinates
gives
\begin{equation}\label{E119}
\begin{aligned}
 \sup_{\gamma\in\Gamma}
 \int_{|x|>a}|Tf_\gamma(x)|^p dv(x)
 &\leq c_0^pK^p|B(0,R)|^p\sigma_{d-1}
       \int_a^\infty s^{d-1-dp}ds\\
 &=\frac{c_0^p\sigma_{d-1}}{d(p-1)}
       K^p|B(0,R)|^p a^{-d(p-1)}.
\end{aligned}
\end{equation}
The integral is finite because $p>1$.
Substituting \eqref{E119} into \eqref{E78} with
$E=\{x\in\R^d:|x|>a\}$ yields
\begin{equation}\label{E87}
 \sup_{\gamma\in\Gamma}
       \int_{|x|>a}|Tu_\gamma(x)|^p dv(x)
 \leq\frac{2^{p-1}c_0^p\sigma_{d-1}}{d(p-1)}
       K^p|B(0,R)|^p a^{-d(p-1)}
      +2^{p-1}c_p^p\varepsilon^p.
\end{equation}

With $\varepsilon,K,R$ fixed, taking $\delta\rightarrow0^{+}$
in \eqref{E86} and $a\rightarrow\infty$ in \eqref{E87} gives
$$
 0\leq
 \lim_{\delta\rightarrow0^{+}}
 \sup_{\gamma\in\Gamma}\sup_{|E|\leq\delta}
       \int_E|Tu_\gamma|^p dv
 \leq2^{p-1}c_p^p\varepsilon^p
$$
and 
$$0\leq
 \lim_{a\rightarrow\infty}
 \sup_{\gamma\in\Gamma}
       \int_{|x|>a}|Tu_\gamma(x)|^p dv(x)
 \leq2^{p-1}c_p^p\varepsilon^p.$$
These limits and $c_p$ are independent of $\varepsilon$.
Letting $\varepsilon\rightarrow0^{+}$ therefore proves the
second and third conditions in \eqref{eq:intrinsic-three}.
This together with \eqref{E85} verifies all three
conditions and gives
$\{|Tu_\gamma|^p:\gamma\in\Gamma\}\in\mathcal W$.

Finally, for the radii $s_\gamma$ in the statement,
definition \eqref{E24} gives
$$
 0\leq M_{1,s_\gamma}^{p}(u_\gamma)
 \leq(\mathcal Mu_\gamma)^p,\quad \gamma\in\Gamma.
$$
The result just proved for $T=\mathcal M$ and
Lemma \ref{L9}, applied to the bound above, imply
$$
 \{M_{1,s_\gamma}^{p}(u_\gamma):\gamma\in\Gamma\}
 \in\mathcal W.
$$
This completes the proof.
\end{proof}

We now combine Lemma \ref{lem:ui-stability} with the analytic
decomposition to obtain weak compactness of the real gradient
densities. A Poincar\'e inequality on balls will then control
mean oscillation. For a ball $B\subset\R^d$, we write
\begin{equation}\label{E126}
 W^{1,1}(B)=
 \left\{g\in L^1(B,dv):
 \partial_{x_\ell}g\in L^1(B,dv),\quad 1\leq\ell\leq d\right\}.
\end{equation}
Recall that the derivatives in \eqref{E126} are weak derivatives
in the sense of \eqref{E3}, with integrals over $B$ and test
functions in $C_c^\infty(B)$
\cite[Section 9.1, p. 263]{Brezis2011}.

\begin{lemma}\label{lem:ui-analytic-real}
For $p\geq2$ and bounded $f$, the following are equivalent.
\begin{enumerate}
\item[(a)]
$\{e_r(f):0<r<r_0\}\in\mathcal W$ for some $r_0>0$.
\item[(b)]
$\{\omega_r(f):0<r<r_1\}\in\mathcal W$ for some $r_1>0$.
\end{enumerate}
\end{lemma}
\begin{proof}
We first prove (b)$\Rightarrow$(a).
For each ball $B(z,r)$, the constant function
$h=\overline{f_{B(z,r)}}$ is holomorphic on $B(z,r)$.
Using this approximant in \eqref{E106}, together with
\eqref{E6} and \eqref{eq:energy-definitions}, gives
\begin{equation}\label{E152}
 G_{1,r}(\overline{f})\leq\operatorname{MO}_{1,r}(f)\quad\text{and}\quad
 0\leq e_r(f)\leq\omega_r(f).
\end{equation}
Applying Lemma \ref{L9} to \eqref{E152} proves (a)
with $r_0=r_1$.

For (a)$\Rightarrow$(b), choose $r_0>0$ as in (a)
and set $u=\overline{f}$.
For $0<r<r_0/4$, Lemma \ref{lem:analytic-decomposition}
gives a decomposition
$$
 u=u_{1,r}+u_{2,r},\quad
 u_{1,r}\in C^\infty(\R^d)\cap L^\infty(\R^d),\quad
 u_{2,r}\in L^\infty(\R^d).
$$
Choose $C_D>0$ as in \eqref{E15}, independently of $r$.
Then
\begin{equation}\label{eq:scaled-decomposition-ui}
 r|\dbar u_{1,r}|+M_{1,r}(u_{2,r})
 \leq C_DG_{1,4r}(u).
\end{equation}
By \eqref{eq:energy-definitions},
$$
 |r^{1-d/p}\dbar u_{1,r}|^p
 \leq C_D^pr^{-d}G_{1,4r}^p(u)
 =C_D^p4^de_{4r}(f).
$$
Assumption (a) and Lemma \ref{L9} imply
\begin{equation}\label{E89}
 \{|r^{1-d/p}\dbar u_{1,r}|^p:0<r<r_0/4\}
 \in\mathcal W.
\end{equation}

To transfer \eqref{E89} to the real gradient by Lemma
\ref{lem:ui-stability}, we express each real derivative as
a linear combination of $\mathcal R_{\ell k}$ applied to
the components of $\dbar u_{1,r}$, with coefficients
independent of $r$.
For each fixed $0<r<r_0/4$, \eqref{E89} and the definition
of $\mathcal W$ imply
$|\dbar u_{1,r}|\in L^p(\R^d,dv)$.
Since $u_{1,r}$ is bounded and smooth, we can apply
\eqref{E14} to obtain
$$
 \norm{\nabla u_{1,r}}_{L^p(\R^d,dv)}
 \leq C_{n,p}\norm{\dbar u_{1,r}}_{L^p(\R^d,dv)}
 <\infty.
$$
Here $C_{n,p}>0$ is the constant in \eqref{E14}.

Fix $h\in C_c^\infty(\R^d)$ and $1\leq\ell,k\leq d$.
Both $h$ and $\mathcal R_{\ell k}h$ belong to
$L^2(\R^d,dv)$ and hence define tempered distributions
\cite[Section 7.12 (d), p. 191]{Rudin1991}.
Their Fourier--Plancherel transforms agree with their
distributional Fourier transforms
\cite[Section 7.14, p. 192]{Rudin1991}.
Applying \eqref{E198} twice and using \eqref{E25}, we obtain
\begin{align*}
 \widehat{\Delta_z(\mathcal R_{\ell k}h)}(\xi)
 &=-|\xi|^2\widehat{\mathcal R_{\ell k}h}(\xi)=-\xi_\ell\xi_k\widehat h(\xi)
 =\widehat{\partial_{x_\ell}\partial_{x_k}h}(\xi).
\end{align*}
The injectivity of the Fourier transform
\cite[Theorem 7.15 (a), p. 192]{Rudin1991} yields
\begin{equation}\label{E94}
 \Delta_z(\mathcal R_{\ell k}h)
 =\partial_{x_\ell}\partial_{x_k}h
 \quad\text{in distributions}.
\end{equation}

For $h\in L^p(\R^d,dv)$, the density result
\cite[Corollary 4.23, p. 109]{Brezis2011}, applied separately
to its real and imaginary parts, gives
$h_\nu\in C_c^\infty(\R^d)$ with $h_\nu\to h$ in
$L^p(\R^d,dv)$. Lemma \ref{L10} identifies the $L^2$ and
$L^p$ values of $\mathcal R_{\ell k}h_\nu$, so \eqref{E94}
holds for these approximants with the $L^p$ extension.
Equation \eqref{E140} with $q=p$ gives
$$
 \norm{\mathcal R_{\ell k}(h_\nu-h)}_{L^p(\R^d,dv)}
 \leq\norm{\mathcal R_{\ell k}}
       \norm{h_\nu-h}_{L^p(\R^d,dv)}
 \longrightarrow0\quad(\nu\rightarrow\infty).
$$
Here $\norm{\mathcal R_{\ell k}}$ is the operator norm of
its $L^p(\R^d,dv)$ extension.
For $\psi\in C_c^\infty(\R^d)$, we have
$\Delta_z\psi,\partial_{x_\ell}\partial_{x_k}\psi
\in L^{p'}(\R^d,dv)$ with $p'=p/(p-1)$.
H\"older's inequality and the two $L^p$ convergences give
\begin{align*}
 \int_{\R^d}\mathcal R_{\ell k}h\Delta_z\psi dv
 &=\lim_{\nu\rightarrow\infty}
   \int_{\R^d}\mathcal R_{\ell k}h_\nu\Delta_z\psi dv=\lim_{\nu\rightarrow\infty}
   \int_{\R^d}h_\nu\partial_{x_\ell}\partial_{x_k}\psi dv
 =\int_{\R^d}h\partial_{x_\ell}\partial_{x_k}\psi dv.
\end{align*}
The middle equality uses \eqref{E94} and the distributional
derivative definition \eqref{E159}, applied twice.
This proves \eqref{E94} for every $h\in L^p(\R^d,dv)$.

To prove the derivative representation, we show that the following
difference vanishes. For fixed $0<r<r_0/4$ and $1\leq\ell\leq d$,
define
$$
 w_\ell=\partial_{x_\ell}u_{1,r}
 -2\sum_{j=1}^n
   (\mathcal R_{\ell,2j-1}-i\mathcal R_{\ell,2j})
   (\partial_{\overline{z}_j}u_{1,r}).
$$
For this fixed $r$, equations \eqref{E14} and \eqref{E89}
give $|\nabla u_{1,r}|,|\dbar u_{1,r}|\in L^p(\R^d,dv)$.
The $L^p$ operator bound \eqref{E140} with $q=p$
gives $w_\ell\in L^p(\R^d,dv)$.
Using \eqref{E94} with $h=\partial_{\overline{z}_j}u_{1,r}$,
together with \eqref{E13} and the definition of $\Delta_z$
in \eqref{E84}, we obtain
\begin{align*}
 \Delta_z w_\ell
 &=\partial_{x_\ell}\Delta_z u_{1,r}
   -2\partial_{x_\ell}\sum_{j=1}^n
     (\partial_{x_{2j-1}}-i\partial_{x_{2j}})
     \partial_{\overline{z}_j}u_{1,r}\\
 &=\partial_{x_\ell}
   (\Delta_z u_{1,r}
   -4\sum_{j=1}^n
     \partial_{z_j}\partial_{\overline{z}_j}u_{1,r})
 =0
\end{align*}
in distributions.

With the mollifier in \eqref{E16}, set
$w_{\ell,\varepsilon}=w_\ell*\eta_\varepsilon$.
Since $w_\ell\in L^p(\R^d,dv)\subseteq L^1_{\rm loc}(\R^d,dv)$
and $\eta_\varepsilon\in C_c^\infty(\R^d)$,
\cite[Proposition 4.20, p. 107]{Brezis2011}, applied with
$N=d=2n$ to the real and imaginary parts of $w_\ell$, gives
$w_{\ell,\varepsilon}\in C^\infty(\R^d)$.
To transfer $\Delta_z w_\ell=0$ to this smooth function,
we use the derivative formula in that proposition and
\eqref{E159} applied twice to obtain
$$
 \Delta_z w_{\ell,\varepsilon}(z)
 =\int_{\R^d}w_\ell(y)\Delta_y\eta_\varepsilon(z-y) dv(y)
 =\langle\Delta_z w_\ell,\eta_\varepsilon(z-\cdot)\rangle
 =0.
$$
The first equality uses
$\Delta_z\eta_\varepsilon(z-y)=\Delta_y\eta_\varepsilon(z-y)$,
and the last uses $\Delta_z w_\ell=0$ in distributions with
the test function $\eta_\varepsilon(z-\cdot)\in C_c^\infty(\R^d)$.
Hence, $w_{\ell,\varepsilon}$ is harmonic.
The calculation in \eqref{E108}, with $\dbar u$ replaced
by $w_\ell\in L^p(\R^d,dv)$, gives
$$
 \norm{w_{\ell,\varepsilon}}_{L^p(\R^d,dv)}
 \leq\norm{w_\ell}_{L^p(\R^d,dv)}.
$$
For $x\in\R^d$ and $R>0$, the mean-value identity
\cite[Theorem 1.6, p. 6]{AxlerBourdonRamey2001}
and H\"older's inequality yield
\begin{align*}
 |w_{\ell,\varepsilon}(x)|
 &\leq\frac1{|B(0,R)|}
       \int_{B(x,R)}|w_{\ell,\varepsilon}| dv\leq |B(0,R)|^{-1/p}
       \norm{w_{\ell,\varepsilon}}_{L^p(\R^d,dv)}\\
 &\leq |B(0,1)|^{-1/p}R^{-d/p}
       \norm{w_\ell}_{L^p(\R^d,dv)}
 \rightarrow0\quad(R\rightarrow\infty).
\end{align*}
Letting $R\rightarrow\infty$ gives $w_{\ell,\varepsilon}=0$.
Since $w_\ell\in L^p(\R^d,dv)$, the translation estimate
in \eqref{E64}, now with $\nabla f_R$ replaced by $w_\ell$
and exponent $p$, yields $w_{\ell,\varepsilon}\rightarrow w_\ell$
in $L^p(\R^d,dv)$ as $\varepsilon\rightarrow0^+$.
Passing to the limit in $w_{\ell,\varepsilon}=0$ gives
$$
 \norm{w_\ell}_{L^p(\R^d,dv)}
 =\lim_{\varepsilon\rightarrow0^+}
   \norm{w_{\ell,\varepsilon}}_{L^p(\R^d,dv)}=0.
$$
The definition of $w_\ell$ then gives the following identity
almost everywhere on $\R^d$.
\begin{equation}\label{eq:gradient-riesz-ui}
 \partial_{x_\ell}u_{1,r}
 =2\sum_{j=1}^n
   (\mathcal R_{\ell,2j-1}-i\mathcal R_{\ell,2j})
   (\partial_{\overline{z}_j}u_{1,r}),\quad 1\leq\ell\leq d.
\end{equation}

For $0<r<r_0/4$, $1\leq j\leq n$, and $1\leq\ell\leq d$,
set
$$
 h_{j,r}=r^{1-d/p}\partial_{\overline{z}_j}u_{1,r},\quad
 S_{\ell j}=2(\mathcal R_{\ell,2j-1}
                  -i\mathcal R_{\ell,2j}).
$$
Since $|h_{j,r}|^p\leq|r^{1-d/p}\dbar u_{1,r}|^p$,
\eqref{E89} and Lemma \ref{L9} imply
$$
 \{|h_{j,r}|^p:0<r<r_0/4\}\in\mathcal W,
 \quad 1\leq j\leq n.
$$
Applying Lemma \ref{lem:ui-stability} with the fixed
operators $S_{\ell j}$ gives
$$
 \{|S_{\ell j}h_{j,r}|^p:0<r<r_0/4\}\in\mathcal W,
 \quad 1\leq\ell\leq d,\quad 1\leq j\leq n.
$$
By \eqref{eq:gradient-riesz-ui} and linearity, we have
$$
 r^{1-d/p}\nabla u_{1,r}
 =\bigg(\sum_{j=1}^nS_{\ell j}h_{j,r}\bigg)_{\ell=1}^d.
$$
The finite sum inequality yields
\begin{align*}
 |r^{1-d/p}\nabla u_{1,r}|^p
 &\leq
 \bigg(\sum_{\ell=1}^d\sum_{j=1}^n
                |S_{\ell j}h_{j,r}|\bigg)^p\leq(dn)^{p-1}
       \sum_{\ell=1}^d\sum_{j=1}^n
                |S_{\ell j}h_{j,r}|^p.
\end{align*}
Applying Lemma \ref{L9} to this bound proves
\begin{equation}\label{E90}
 \{|r^{1-d/p}\nabla u_{1,r}|^p:0<r<r_0/4\}
 \in\mathcal W.
\end{equation}
For the local averages, the choices $s_r=r$ and $s_0=r_0/4$
satisfy $0<s_r<s_0$.
The averaging assertion of Lemma \ref{lem:ui-stability},
applied to the vector family in \eqref{E90}, now gives
\begin{equation}\label{E91}
 \{M_{1,r}^p(|r^{1-d/p}\nabla u_{1,r}|):
       0<r<r_0/4\}\in\mathcal W.
\end{equation}

We use \eqref{E91} to control the mean oscillation
of $u_{1,r}$.
For a ball $B=B(z,r)$ and $g\in W^{1,1}(B)$,
we write $g_B$ for its average on $B$,
as defined before \eqref{E6}. The ball $B$ is convex and has
diameter $\operatorname{diam}B=\sup_{x,y\in B}|x-y|=2r$.
The definition of $g_B$ gives
$$
 \frac1{|B|}\int_B(g-g_B) dv
 =\frac1{|B|}\left(\int_Bg dv-|B|g_B\right)=0.
$$
Taking real and imaginary parts, we obtain
$(\Re(g-g_B))_B=(\Im(g-g_B))_B=0$.
We apply
\cite[Theorem 3.2]{AcostaDuran2004} to these two functions.
Combining their bounds, we obtain
\begin{equation}\label{E26}
 \frac1{|B|}\int_B|g-g_B| dv
 \leq\frac{C_Pr}{|B|}\int_B|\nabla g| dv,
 \quad g\in W^{1,1}(B),
\end{equation}
where $C_P>0$ depends only on $d$.

Since $u_{1,r}$ is smooth on $\R^d$, \eqref{E126} gives
$u_{1,r}|_B\in W^{1,1}(B)$ for every ball $B$, so
\eqref{E26} applies to $u_{1,r}|_B$.
Moreover, \eqref{E6} gives
$$
 \operatorname{MO}_{1,r}(f)
 =\operatorname{MO}_{1,r}(u)\quad\text{and}\quad
 \operatorname{MO}_{1,r}(u_{2,r})
 \leq2M_{1,r}(u_{2,r}).
$$
The decomposition $u=u_{1,r}+u_{2,r}$ and linearity of
averaging give $u_B=(u_{1,r})_B+(u_{2,r})_B$.
The triangle inequality yields
$$
 \operatorname{MO}_{1,r}(u)(z)
 \leq\operatorname{MO}_{1,r}(u_{1,r})(z)
     +\operatorname{MO}_{1,r}(u_{2,r})(z).
$$
Using $\operatorname{MO}_{1,r}(f)=\operatorname{MO}_{1,r}(u)$,
we combine \eqref{E26} with the preceding bound for
$\operatorname{MO}_{1,r}(u_{2,r})$.
Multiplication by $r^{-d/p}$ gives the first inequality in
\begin{equation}\label{E51}
 \begin{aligned}
 r^{-d/p}\operatorname{MO}_{1,r}(f)(z)
 &\leq C_Pr^{1-d/p}M_{1,r}(|\nabla u_{1,r}|)(z)
       +2r^{-d/p}M_{1,r}(u_{2,r})(z)\\
 &\leq C_PM_{1,r}(|r^{1-d/p}\nabla u_{1,r}|)(z)
       +2C_Dr^{-d/p}G_{1,4r}(u)(z).
 \end{aligned}
\end{equation}
For the second inequality, the factor $r^{1-d/p}>0$ is
moved inside $M_{1,r}$ in the gradient term, while
\eqref{eq:scaled-decomposition-ui} bounds $M_{1,r}(u_{2,r})$.
By \eqref{eq:energy-definitions},
$$
 \bigl(2C_Dr^{-d/p}G_{1,4r}(u)(z)\bigr)^p
 =(2C_D)^p4^de_{4r}(f)(z).
$$
Raising \eqref{E51} to the $p$-th power and using
$(a+b)^p\leq2^{p-1}(a^p+b^p)$ for $a,b\geq0$ gives
\begin{align*}
 0\leq\omega_r(f)(z)
 &\leq2^{p-1}C_P^p
       M_{1,r}^p(|r^{1-d/p}\nabla u_{1,r}|)(z)+2^{p-1}(2C_D)^p4^de_{4r}(f)(z).
\end{align*}
The family of local averages belongs to $\mathcal W$
by \eqref{E91}, while
$\{e_{4r}(f):0<r<r_0/4\}\in\mathcal W$ by (a).
Lemma \ref{L9} therefore yields
$$
 \{\omega_r(f):0<r<r_0/4\}\in\mathcal W.
$$
This proves (b) with $r_1=r_0/4$. The proof is complete.
\end{proof}

\subsection{Convergence for smooth families of symbols}
For the converse direction of Theorem \ref{thm:weakcriterion},
we need convergence for smooth first symbols that vary with the
scale. We express their semi-commutators through multipliers
controlled by the gradient. Relative weak compactness of the
normalized gradient densities will then allow us to pass from
smooth second symbols to arbitrary fixed bounded symbols.
Recall the unweighted projection $\Pi_t$ in \eqref{E12} and
the operator $\widetilde D_t(f,g)$ in \eqref{E83}.

For $p\geq2$ and a possibly unbounded $f\in L^p(\R^d,dv)$,
we consider $M_f$ on its natural domain
$$
 \operatorname{Dom}(M_f)
 =\{h\in L^2(\R^d,dv):fh\in L^2(\R^d,dv)\}.
$$
Since $p\geq2$, H\"older's inequality on compact sets gives
$f\in L^2_{\rm loc}(\R^d,dv)$, so
$C_c^\infty(\R^d)\subseteq\operatorname{Dom}(M_f)$.
The density result \cite[Corollary 4.23, p. 109]{Brezis2011},
applied with exponent $2$ to real and imaginary parts, shows
that this domain is dense in $L^2(\R^d,dv)$.
The next lemma gives the bounded extensions needed below.

\begin{lemma}\label{L3}
Let $2\leq p<\infty$ and $r>0$.
\begin{enumerate}[label=\textup{(\alph*)}]
\item
For every $f\in L^p(\R^d,dv)$, the product
$M_f\Pi_{r^2}$ is defined and bounded on all of
$L^2(\R^d,dv)$. The product $\Pi_{r^2}M_f$, initially
defined on $\operatorname{Dom}(M_f)$, has a unique bounded
extension to all of $L^2(\R^d,dv)$, denoted by the same symbol.
These operators satisfy
\begin{equation}\label{eq:ui-multiplier}
 \norm{M_f\Pi_{r^2}}_{\mathcal S_p(L^2(\R^d,dv))}
 +\norm{\Pi_{r^2}M_f}_{\mathcal S_p(L^2(\R^d,dv))}
 \leq C_Mr^{-d/p}\norm{f}_{L^p(\R^d,dv)},
\end{equation}
where $C_M>0$ depends only on $n,m,M,p$.

\item
Suppose $u\in C^\infty(\R^d)\cap L^\infty(\R^d)$ and
$|\nabla u|\in L^p(\R^d,dv)$.
With the maximal operator in \eqref{E24}, set
$$
 b_r=\mathcal M(|\nabla u|)\quad\text{and}\quad
 \beta_r=r^{1-d/p}b_r.
$$
There is a bounded integral operator $B_r$ on
$L^2(\R^d,dv)$ such that $\norm{B_r}\leq C_B$ and
\begin{equation}\label{eq:factor-ui}
 \widetilde D_{r^2}(u,g)
 =r^{d/p}\bigl(
   \Pi_{r^2}M_{\beta_r}B_rM_g\Pi_{r^2}
  +\Pi_{r^2}B_rM_{\beta_rg}\Pi_{r^2}\bigr)
\end{equation}
for every $g\in L^\infty(\R^d)$.
Here $C_B>0$ depends only on $n,m,M$. The kernel $b_{r,0}$
of $B_r$ satisfies
$|b_{r,0}(z,w)|\leq C'r^{-d}e^{-c'|z-w|/r}$ almost everywhere,
where $C',c'>0$ depend only on $n,m,M$.
\end{enumerate}
\end{lemma}
\begin{proof}
We first prove \textup{(a)}. Suppose that
$f\in L^\infty(\R^d)\cap L^p(\R^d,dv)$.
With $t=r^2$, let $J_f:\cF\to L_t^2$ be given by $J_fh=fh$.
The definitions in \eqref{E11} and \eqref{E12} give
$$
 M_f\Pi_{r^2}=V_{r^2}J_fP_{r^2}V_{r^2}^{-1}.
$$
Here $J_fP_{r^2}:L_t^2\to L_t^2$ is the zero extension of
$J_f:\cF\to L_t^2$. Equation \eqref{E138}, unitary
invariance, and \eqref{eq:multiplier} give
\begin{align*}
 \norm{M_f\Pi_{r^2}}_{\mathcal S_p(L^2(\R^d,dv))}
 &=\norm{J_f}_{\mathcal S_p(\cF,L_t^2)}\leq c_Mr^{-d/p}\norm{f}_{L^p(\R^d,dv)},
\end{align*}
where $c_M>0$ depends only on $n,m,M,p$.
Applying this estimate to $\overline{f}$ and using
$(M_{\overline{f}}\Pi_{r^2})^*=\Pi_{r^2}M_f$ gives the same
bound for $\Pi_{r^2}M_f$.
Adding the two bounds proves \eqref{eq:ui-multiplier}
for bounded $f$, with $C_M=2c_M$.

Now let $f\in L^p(\R^d,dv)$ and set
$$
 f_j=f\mathbf1_{B(0,j)}\mathbf1_{\{|f|\leq j\}},
 \quad j\geq1.
$$
Dominated convergence gives $f_j\rightarrow f$ in
$L^p(\R^d,dv)$ as $j\rightarrow\infty$.
On $L^2(\R^d,dv)$, define
$$
 A_j=M_{f_j}\Pi_{r^2},\quad T_j=\Pi_{r^2}M_{f_j}.
$$
Since $f_j\in L^\infty(\R^d)\cap L^p(\R^d,dv)$, the bounded
case of \eqref{eq:ui-multiplier} gives
$A_j,T_j\in\mathcal S_p(L^2(\R^d,dv))$.
Applying the same estimate to $f_j-f_k$, we obtain
\begin{align*}
 &\norm{A_j-A_k}_{\mathcal S_p(L^2(\R^d,dv))}
 +\norm{T_j-T_k}_{\mathcal S_p(L^2(\R^d,dv))}\leq C_Mr^{-d/p}\norm{f_j-f_k}_{L^p(\R^d,dv)}
 \rightarrow0
\end{align*}
as $j,k\rightarrow\infty$.
It follows that both sequences are Cauchy in
$\mathcal S_p(L^2(\R^d,dv))$.
By completeness \cite[Section 1.4, p. 26]{Zhu1990},
there are $A,T\in\mathcal S_p(L^2(\R^d,dv))$ such that
\begin{equation}\label{E129}
 A_j\rightarrow A,\quad T_j\rightarrow T
 \quad\text{in }\mathcal S_p(L^2(\R^d,dv))
 \quad(j\rightarrow\infty).
\end{equation}

To identify $A$, fix $h\in L^2(\R^d,dv)$.
Equations \eqref{E128} and \eqref{E129} give
$$
 \norm{A_j-A}
 \leq\norm{A_j-A}_{\mathcal S_p(L^2(\R^d,dv))}
 \longrightarrow0\quad(j\rightarrow\infty).
$$
This operator norm convergence implies
$$
 \norm{A_jh-Ah}_{L^2(\R^d,dv)}
 \leq\norm{A_j-A}\norm{h}_{L^2(\R^d,dv)}
 \longrightarrow0\quad(j\rightarrow\infty).
$$
The truncation definition gives $f_j\rightarrow f$ almost
everywhere. Since $\Pi_{r^2}h,Ah\in L^2(\R^d,dv)$, their values are
finite almost everywhere, so
$$
 |f_j\Pi_{r^2}h-Ah|^2\longrightarrow|f\Pi_{r^2}h-Ah|^2
 \quad\text{almost everywhere}\quad(j\rightarrow\infty).
$$
Applying Fatou's lemma to these nonnegative functions and
using $A_jh=f_j\Pi_{r^2}h$, we obtain
\begin{align*}
 \int_{\R^d}|f\Pi_{r^2}h-Ah|^2 dv
 &=\int_{\R^d}\liminf_{j\rightarrow\infty}
       |f_j\Pi_{r^2}h-Ah|^2 dv\\
       &\leq\liminf_{j\rightarrow\infty}
       \norm{A_jh-Ah}_{L^2(\R^d,dv)}^2=0.
\end{align*}
The vanishing integral implies $f\Pi_{r^2}h=Ah\in L^2(\R^d,dv)$
for every $h\in L^2(\R^d,dv)$. In terms of the natural domain of $M_f$,
this proves
\begin{equation}\label{E130}
 \Pi_{r^2}(L^2(\R^d,dv))\subseteq\operatorname{Dom}(M_f)\quad\text{and}\quad
 A=M_f\Pi_{r^2}:L^2(\R^d,dv)\to L^2(\R^d,dv).
\end{equation}
In particular, the product $M_f\Pi_{r^2}$ itself is defined
and bounded on the whole space.

To identify $T$ on the natural domain of $\Pi_{r^2}M_f$,
let $h\in\operatorname{Dom}(M_f)$. The truncations satisfy
$$
 |(f_j-f)h|^2\leq |fh|^2\in L^1(\R^d,dv),\quad
 (f_j-f)h\to0\quad\text{almost everywhere}.
$$
Dominated convergence and $\norm{\Pi_{r^2}}\leq1$ yield
$$
 \norm{T_jh-\Pi_{r^2}(fh)}_{L^2(\R^d,dv)}
 \leq\norm{(f_j-f)h}_{L^2(\R^d,dv)}
 \longrightarrow0\quad(j\rightarrow\infty).
$$
Equations \eqref{E128} and \eqref{E129} also give
$T_jh\to Th$ in $L^2(\R^d,dv)$, so
\begin{equation}\label{E131}
 Th=\lim_{j\to\infty}T_jh=\Pi_{r^2}(fh),
 \quad h\in\operatorname{Dom}(M_f).
\end{equation}
Then $T:L^2(\R^d,dv)\to L^2(\R^d,dv)$ is a bounded
extension of $\Pi_{r^2}M_f$.
Since $\operatorname{Dom}(M_f)$ is dense in $L^2(\R^d,dv)$,
as established before the lemma, this bounded extension is
unique and is denoted by $\Pi_{r^2}M_f$.
Applying \eqref{E129} and \eqref{E130} to $\overline{f}$,
whose truncations are $\overline{f_j}$, and using
$T_j=(M_{\overline{f_j}}\Pi_{r^2})^*$, we obtain
\begin{align*}
 \norm{T-(M_{\overline{f}}\Pi_{r^2})^*}
 &\leq\norm{T-T_j}
   +\norm{M_{\overline{f_j}}\Pi_{r^2}
          -M_{\overline{f}}\Pi_{r^2}}
 \longrightarrow0\quad(j\to\infty).
\end{align*}
This proves the adjoint identity on the whole space
\begin{equation}\label{E132}
 \Pi_{r^2}M_f=T=(M_{\overline{f}}\Pi_{r^2})^*
 \quad\text{on }L^2(\R^d,dv).
\end{equation}
Finally, the triangle inequality, \eqref{E129}, and
$f_j\to f$ in $L^p(\R^d,dv)$ give convergence of the
respective norms. Using \eqref{E130} and the bounded case
of \eqref{eq:ui-multiplier}, we obtain
\begin{align*}
 &\qquad\norm{M_f\Pi_{r^2}}_{\mathcal S_p(L^2(\R^d,dv))}
  +\norm{T}_{\mathcal S_p(L^2(\R^d,dv))}\\
  &=\lim_{j\to\infty}
   \bigl(\norm{A_j}_{\mathcal S_p(L^2(\R^d,dv))}
        +\norm{T_j}_{\mathcal S_p(L^2(\R^d,dv))}\bigr)\\
 &\leq C_Mr^{-d/p}\lim_{j\to\infty}\norm{f_j}_{L^p(\R^d,dv)}
  =C_Mr^{-d/p}\norm{f}_{L^p(\R^d,dv)}.
\end{align*}
This proves \textup{(a)}.

We now prove \textup{(b)}. To apply \textup{(a)} to $\beta_r$,
we first use \eqref{E77} with $T=\mathcal M$ and $q=p$ to obtain
\begin{equation}\label{E127}
 \norm{\beta_r}_{L^p(\R^d,dv)}
 =r^{1-d/p}\norm{\mathcal M(|\nabla u|)}_{L^p(\R^d,dv)}
 \leq c_pr^{1-d/p}\norm{\nabla u}_{L^p(\R^d,dv)}<\infty.
\end{equation}
In particular, $b_r$ is finite almost everywhere.
To construct $B_r$, we estimate $u(z)-u(w)$ in terms of $b_r$.
For $h\in\{\Re u,\Im u\}$, the smoothness and boundedness of $u$
give $h|_B\in W^{1,1}(B)$ on every ball $B\subset\R^d$ by
\eqref{E126}, and $|h(x)|\leq\norm{u}_{L^\infty(\R^d)}<\infty$
for all $x\in\R^d$.
Then $h\in W_{\rm loc}^{1,1}(\R^d)$, and smoothness identifies
$h$ and its derivatives with the representatives defined in
\cite[(4), p. 81]{BojarskiHajlasz1993}.
Using the definitions of the Taylor polynomials and truncated
maximal functions in \cite[p. 78--79]{BojarskiHajlasz1993},
we apply \cite[Theorem 3 (1), p. 83]{BojarskiHajlasz1993}
with derivative order one and spatial dimension $d$.
In the notation of \eqref{E4}, this yields
\begin{align*}
 |h(z)-h(w)|
 &\leq C_d|z-w|\bigg(
 \sup_{0<s<|z-w|}M_{1,s}(|\nabla h|)(z)
 +\sup_{0<s<|z-w|}M_{1,s}(|\nabla h|)(w)\bigg)\\
 &\leq C_d|z-w|(b_r(z)+b_r(w))
\end{align*}
for almost every $z\ne w$.
Here $C_d>0$ depends only on $d$. The second inequality uses
$|\nabla h|\leq|\nabla u|$ and enlarges each supremum to all
$s>0$ in \eqref{E24}, with $b_r=\mathcal M(|\nabla u|)$.
Adding the estimates for $\Re u$ and $\Im u$ and using
$|u(z)-u(w)|\leq|\Re(u(z)-u(w))|+|\Im(u(z)-u(w))|$,
we obtain
\begin{equation}\label{eq:hajlasz-pointwise}
 |u(z)-u(w)|\leq C_H|z-w|(b_r(z)+b_r(w))
\end{equation}
for almost every $(z,w)$. Here $C_H=2C_d>0$ depends only on $d$.

To factor \eqref{eq:unweighted-semicommutator}, we normalize
the kernel $p_{r^2}$ from \eqref{E12}. For almost every $(z,w)$,
set $b_{r,0}(z,w)=0$ when $b_r(z)+b_r(w)=0$, and otherwise set
\begin{equation}\label{E203}
 b_{r,0}(z,w)=
 \frac{u(z)-u(w)}{r(b_r(z)+b_r(w))}p_{r^2}(z,w).
\end{equation}
On the set where the denominator vanishes,
\eqref{eq:hajlasz-pointwise} shows that the numerator is zero
almost everywhere. Combining \eqref{eq:hajlasz-pointwise}
with the third estimate in \eqref{eq:gram}, we obtain
\begin{equation}\label{E32}
 |b_{r,0}(z,w)|\leq Cr^{-d}\frac{|z-w|}{r}
                                  e^{-c|z-w|/r}
                   \leq C'r^{-d}e^{-c'|z-w|/r}.
\end{equation}
Here $C,c>0$ depend only on $n,m,M$, and we may take
$c'=c/2>0$ and $C'=2C/(ce)>0$.
Integrating the last bound with $w=z+ry$, or with $z=w+ry$
for the other variable, cancels the factor $r^{-d}$.
We obtain a constant $C_B>0$, depending only on $n,m,M$, such that
\begin{equation}\label{E133}
 \operatorname*{esssup}_z\int_{\R^d}|b_{r,0}(z,w)| dv(w)
 +\operatorname*{esssup}_w\int_{\R^d}|b_{r,0}(z,w)| dv(z)\leq C_B.
\end{equation}
For $h\in L^2(\R^d,dv)$, define
$$
 B_rh(z)=\int_{\R^d}b_{r,0}(z,w)h(w) dv(w).
$$
Equation \eqref{E32} and the Cauchy--Schwarz inequality give
$$
 \int_{\R^d}|b_{r,0}(z,w)h(w)| dv(w)
 \leq\norm{b_{r,0}(z,\cdot)}_{L^2(\R^d,dv)}
       \norm{h}_{L^2(\R^d,dv)}<\infty
$$
for almost every $z$, proving absolute convergence.
Let $B_r^+$ denote the integral operator with kernel $|b_{r,0}|$.
The triangle inequality gives
$
 |B_rh|\leq B_r^+(|h|)
$ almost everywhere.

The two bounds in \eqref{E133} allow both integral constants
in Schur's test \cite[Corollary 3.7, p. 53]{Zhu1990} to be chosen as $C_B$.
We apply this $L^2(\R^d,dv)$ criterion to the nonnegative kernel
$|b_{r,0}|$ with auxiliary function one and obtain
$$
 \norm{B_r^+h}_{L^2(\R^d,dv)}
 \leq C_B\norm{h}_{L^2(\R^d,dv)},\quad h\in L^2(\R^d,dv).
$$
The pointwise bound above now gives $B_rh\in L^2(\R^d,dv)$ and
\begin{equation}\label{E134}
 \norm{B_r}\leq\norm{B_r^+}\leq C_B.
\end{equation}

Equation \eqref{E203}, with the zero-denominator case
covered by \eqref{eq:hajlasz-pointwise}, gives
the kernel identity
\begin{align}\label{E52}
 (u(z)-u(w))p_{r^2}(z,w)
 &=r(b_r(z)+b_r(w))b_{r,0}(z,w)=r^{d/p}(\beta_r(z)+\beta_r(w))b_{r,0}(z,w).
\end{align}
This holds almost everywhere, with the second equality using
the definition $\beta_r=r^{1-d/p}b_r$ in \textup{(b)}.
To verify \eqref{eq:factor-ui}, we use the inner product
$$
 \langle h,k\rangle_{L^2(\R^d,dv)}
 =\int_{\R^d}h(z)\overline{k(z)} dv(z),
 \qquad h,k\in L^2(\R^d,dv).
$$
For arbitrary $h,k$ in this space, set
$h_1=\Pi_{r^2}h$ and $k_1=\Pi_{r^2}k$.
To justify the integrals below, note that
$gh_1\in L^2(\R^d,dv)$ because $g\in L^\infty(\R^d)$.
Equation \eqref{E127} and
$\norm{\beta_rg}_{L^p}\leq\norm{g}_{L^\infty}\norm{\beta_r}_{L^p}$
give $\beta_r,\beta_rg\in L^p(\R^d,dv)$.
Equations \eqref{E128} and \eqref{eq:ui-multiplier} yield
\begin{align}\label{E135}
 \norm{\beta_rk_1}_{L^2(\R^d,dv)}
 &=\norm{(M_{\beta_r}\Pi_{r^2})k}_{L^2(\R^d,dv)}\leq\norm{M_{\beta_r}\Pi_{r^2}}_{\mathcal S_p(L^2(\R^d,dv))}
       \norm{k}_{L^2(\R^d,dv)}\nonumber\\
 &\leq C_Mr^{-d/p}\norm{\beta_r}_{L^p(\R^d,dv)}
       \norm{k}_{L^2(\R^d,dv)}<\infty.
\end{align}
The argument in \eqref{E135}, with $\beta_rg$ and $h$ in place
of $\beta_r$ and $k$, also gives $\beta_rgh_1\in L^2(\R^d,dv)$.
Using these integrability statements, \eqref{E134}, and the
Cauchy--Schwarz inequality, we obtain
 \begin{align}\label{E136}
 &\qquad\iint_{\R^d\times\R^d}\beta_r(z)|b_{r,0}(z,w)||g(w)h_1(w)k_1(z)| dv(w) dv(z)\nonumber\\
 &\quad\leq\norm{\beta_rk_1}_{L^{2}(\R^d,dv)}\norm{B_r^+(|gh_1|)}_{L^{2}(\R^d,dv)}<\infty
 \end{align}
and 
\begin{align}\label{E137}
 &\qquad\iint_{\R^d\times\R^d}\beta_r(w)|b_{r,0}(z,w)||g(w)h_1(w)k_1(z)| dv(w) dv(z)\nonumber\\
 &\quad\leq\norm{k_1}_{L^{2}(\R^d,dv)}\norm{B_r^+(|\beta_rgh_1|)}_{L^{2}(\R^d,dv)}<\infty.
  \end{align}

The kernel formula in \eqref{E12} shows that
the kernel of $[M_u,\Pi_{r^2}]$ is
$(u(z)-u(w))p_{r^2}(z,w)$.
Equations \eqref{E52}, \eqref{E136}, and \eqref{E137}
give absolute integrability of both terms in the kernel
decomposition. We may therefore use Fubini's theorem in
\eqref{eq:unweighted-semicommutator} to obtain
\begin{align}\label{E123}
 \langle\widetilde D_{r^2}(u,g)h,k\rangle_{L^2(\R^d,dv)}
 &=\iint_{\R^d\times\R^d}
   (u(z)-u(w))p_{r^2}(z,w)g(w)h_1(w)
   \overline{k_1(z)} dv(w) dv(z)\nonumber\\
 &=r^{d/p}\langle B_r(gh_1),\beta_rk_1\rangle_{L^2(\R^d,dv)}+r^{d/p}\langle B_r(\beta_rgh_1),k_1\rangle_{L^2(\R^d,dv)}.
\end{align}
Since $\beta_r$ is real, \eqref{E132} with $f=\beta_r$ gives
$\Pi_{r^2}M_{\beta_r}=(M_{\beta_r}\Pi_{r^2})^*$ on
$L^2(\R^d,dv)$. Here $\Pi_{r^2}M_{\beta_r}$ denotes the
bounded extension of the map
$v\mapsto\Pi_{r^2}(\beta_rv)$ on $\operatorname{Dom}(M_{\beta_r})$.
This extension acts on $B_r(gh_1)\in L^2(\R^d,dv)$ by
\eqref{E134}. Using this adjoint identity and the
self-adjointness of $\Pi_{r^2}$, we obtain
$$ \langle B_r(gh_1),\beta_rk_1\rangle_{L^2(\R^d,dv)}
 =\langle(\Pi_{r^2}M_{\beta_r})B_rM_g\Pi_{r^2}h,k\rangle_{L^2(\R^d,dv)}$$
 and
\begin{align*}
 \langle B_r(\beta_rgh_1),k_1\rangle_{L^2(\R^d,dv)}
 &=\langle\Pi_{r^2}B_r(M_{\beta_rg}\Pi_{r^2})h,k\rangle_{L^2(\R^d,dv)}.
\end{align*}
Substitution into \eqref{E123} proves
\eqref{eq:factor-ui}, since $h,k\in L^2(\R^d,dv)$ are arbitrary.
 The proof is complete.
\end{proof}

We now apply \eqref{eq:factor-ui} to a family whose normalized
gradient densities belong to $\mathcal W$.
Equation \eqref{E43} identifies the Schatten norm of the unweighted
extension with that of the semi-commutator on $\mathcal F_{r^2}^2$,
so truncating the multipliers in this representation gives uniform
estimates for second symbols supported on small sets or outside
large balls.

\begin{lemma}\label{L4}
Let $2\leq p<\infty$ and $r_0>0$. Suppose
$u_r\in C^\infty(\R^d)\cap L^\infty(\R^d)$,
$|\nabla u_r|\in L^p(\R^d,dv)$ for $0<r<r_0$, and
$$
 \{|r^{1-d/p}\nabla u_r|^p:0<r<r_0\}\in\mathcal W.
$$
\begin{enumerate}[label=\textup{(\alph*)}]
\item
There is a constant $C_0\in[0,\infty)$, depending on
the family but independent of $r,g$ and all approximation
choices, such that
\begin{equation}\label{E34}
 \norm{D_{r^2}(u_r,g)}_{\mathcal S_p(\mathcal F_{r^2}^2)}
 \leq C_0\norm{g}_{L^\infty(\R^d)},\quad 0<r<r_0,
\end{equation}
for every $g\in L^\infty(\R^d)$.

\item
There are constants $C,c_1>0$, depending only on $n,m,M,p$,
with the following property. Given $\varepsilon>0$, one can
choose $K,R>0$, independent of $r$, so that every measurable
$E\subseteq\R^d$ with $|E|<\infty$ satisfies
\begin{equation}\label{eq:smooth-smallsets}
 \sup_{|g|\leq\mathbf1_E}
 \norm{D_{r^2}(u_r,g)}_{\mathcal S_p(\mathcal F_{r^2}^2)}
 \leq C\varepsilon+CK|E|^{1/p},\quad 0<r<r_0,
\end{equation}
and
\begin{equation}\label{eq:smooth-tails}
 \sup_{|g|\leq\mathbf1_{B(0,R+1)^c}}
 \norm{D_{r^2}(u_r,g)}_{\mathcal S_p(\mathcal F_{r^2}^2)}
 \leq C\varepsilon+CK|B(0,R)|^{1/p}e^{-c_1/(pr)}
\end{equation}
for $0<r<r_0$. The same $K,R$ apply to both
estimates and may depend on $\varepsilon$ and the family $(u_r)$.
All suprema are taken over $g\in L^\infty(\R^d)$.
\end{enumerate}
\end{lemma}
\begin{proof}
We first prove \textup{(a)}. Let $c_0>0$ be the constant in
\eqref{eq:gradient-semicommutator}. The first condition in
\eqref{eq:intrinsic-three} bounds the $L^1(\R^d,dv)$ norms
of $|r^{1-d/p}\nabla u_r|^p$ uniformly for $0<r<r_0$.
Since these norms equal
$\norm{r^{1-d/p}\nabla u_r}_{L^p(\R^d,dv)}^p$, we may set
$$
 C_0=c_0\sup_{0<r<r_0}
       \norm{r^{1-d/p}\nabla u_r}_{L^p(\R^d,dv)}<\infty.
$$
Equation \eqref{eq:gradient-semicommutator} proves
\eqref{E34} with this $C_0$, which is fixed independently
of $\varepsilon,K,R$.

For \textup{(b)}, set $b_r=\mathcal M(|\nabla u_r|)$ and
$\beta_r=r^{1-d/p}b_r$, as in Lemma \ref{L3} (b).
By \eqref{E24},
$$\beta_r=\mathcal M(r^{1-d/p}|\nabla u_r|).$$
Applying Lemma \ref{lem:ui-stability} to the family
$r^{1-d/p}\nabla u_r$ gives
$\{\beta_r^p:0<r<r_0\}\in\mathcal W$.
Given $\varepsilon>0$, we use the truncation estimates
\eqref{E115} and \eqref{E116} to choose $K,R>0$ for the
family $(\beta_r)_{0<r<r_0}$. The decomposition
$$
 \beta_{r,1}=\beta_r\mathbf1_{B(0,R)}\mathbf1_{\{\beta_r\leq K\}}\quad\text{and}\quad
 \beta_{r,2}=\beta_r-\beta_{r,1}
$$
then satisfies
\begin{equation}\label{E33}
 0\leq\beta_{r,1}\leq K\mathbf1_{B(0,R)}\quad\text{and}\quad
 \sup_{0<r<r_0}\norm{\beta_{r,2}}_{L^p(\R^d,dv)}<\varepsilon.
\end{equation}
Since $\beta_{r,1},\beta_{r,2}\in L^p(\R^d,dv)$,
\eqref{E130} and \eqref{E132} give the following identity on $L^2(\R^d,dv)$
\begin{align*}
 \Pi_{r^2}M_{\beta_r}
 &=(M_{\beta_r}\Pi_{r^2})^*
 =(M_{\beta_{r,1}}\Pi_{r^2}+M_{\beta_{r,2}}\Pi_{r^2})^*=\Pi_{r^2}M_{\beta_{r,1}}+\Pi_{r^2}M_{\beta_{r,2}}.
\end{align*}
This identity decomposes the bounded extension
$\Pi_{r^2}M_{\beta_r}$. Since $g\in L^\infty(\R^d)$,
we also have $\beta_{r,\nu}g\in L^p(\R^d,dv)$ for
$\nu=1,2$, so linearity in \eqref{E130} gives
$$
 M_{\beta_rg}\Pi_{r^2}
 =M_{\beta_{r,1}g}\Pi_{r^2}+M_{\beta_{r,2}g}\Pi_{r^2}.
$$
Substituting these two identities into \eqref{eq:factor-ui},
we obtain
$$
 \widetilde D_{r^2}(u_r,g)
 =r^{d/p}\sum_{\nu=1}^2\bigl(
   \Pi_{r^2}M_{\beta_{r,\nu}}B_rM_g\Pi_{r^2}
  +\Pi_{r^2}B_rM_{\beta_{r,\nu}g}\Pi_{r^2}\bigr).
$$
By \eqref{E43}, an estimate for this sum in
$\mathcal S_p(L^2(\R^d,dv))$ gives the same estimate for
$D_{r^2}(u_r,g)$ on $\mathcal F_{r^2}^2$.

We first estimate the summand containing $\beta_{r,2}$.
The triangle inequality and \eqref{E7}, with
$\norm{\Pi_{r^2}}\leq1$,
$\norm{M_g}=\norm{g}_{L^\infty(\R^d)}$, and
$\norm{B_r}\leq C_B$ from Lemma \ref{L3} (b), give
\begin{equation}\label{E92}
\begin{aligned}
 &\qquad r^{d/p}\bigl\|
   \Pi_{r^2}M_{\beta_{r,2}}B_rM_g\Pi_{r^2}
  +\Pi_{r^2}B_rM_{\beta_{r,2}g}\Pi_{r^2}
   \bigr\|_{\mathcal S_p(L^2(\R^d,dv))}\\
 &\leq r^{d/p}C_B\bigl(
       \norm{\Pi_{r^2}M_{\beta_{r,2}}}_{\mathcal S_p(L^2(\R^d,dv))}
       \norm{g}_{L^\infty(\R^d)}
       +\norm{M_{\beta_{r,2}g}\Pi_{r^2}}_{\mathcal S_p(L^2(\R^d,dv))}\bigr)\\
 &\leq C_MC_B\bigl(
       \norm{\beta_{r,2}}_{L^p(\R^d,dv)}
       \norm{g}_{L^\infty(\R^d)}
       +\norm{\beta_{r,2}g}_{L^p(\R^d,dv)}\bigr)\\
 &\leq2C_MC_B\norm{\beta_{r,2}}_{L^p(\R^d,dv)}
                  \norm{g}_{L^\infty(\R^d)}
 \leq2C_MC_B\varepsilon\norm{g}_{L^\infty(\R^d)}.
\end{aligned}
\end{equation}
The second inequality uses \eqref{eq:ui-multiplier} with
$f=\beta_{r,2}$ and $f=\beta_{r,2}g$, where the factors
$r^{d/p}$ and $r^{-d/p}$ cancel. The last inequality
uses \eqref{E33}.

It remains to estimate the terms containing $\beta_{r,1}$.
To prove \eqref{eq:smooth-smallsets}, suppose
$|g|\leq\mathbf1_E$, where $|E|<\infty$.
By \eqref{E32} and $|g|\leq\mathbf1_E$, the kernel of
$B_rM_g$ satisfies
$$
 \iint_{\R^d\times\R^d}
 |b_{r,0}(z,w)g(w)|^2 dv(z) dv(w)
 \leq C_2r^{-d}|E|<\infty,
$$
where $C_2>0$ depends only on $n,m,M$.
Using the kernel identity \eqref{E59} for
$b_{r,0}(z,w)g(w)$ and the operator norm bound \eqref{E134}
with $\norm{g}_{L^\infty(\R^d)}\leq1$, we obtain
\begin{equation}\label{E204}
 \norm{B_rM_g}_{\mathcal S_2(L^2(\R^d,dv))}^2
 \leq C_2r^{-d}|E|\quad\text{and}\quad \norm{B_rM_g}\leq C_B.
\end{equation}
For $p\geq2$, \eqref{E102} gives $s_j(A)\leq\norm{A}$.
The definition of the Schatten norm therefore yields
\begin{equation}\label{E124}
 \norm{A}_{\mathcal S_p}^p
 =\sum_j s_j^p(A)
 \leq\norm{A}^{p-2}\sum_j s_j^2(A)
 =\norm{A}^{p-2}\norm{A}_{\mathcal S_2}^2,
 \quad A\in\mathcal S_2.
\end{equation}
Applying \eqref{E124} to $A=B_rM_g$ and using \eqref{E204},
we obtain
$$
 \norm{B_rM_g}_{\mathcal S_p(L^2(\R^d,dv))}^p
 \leq C_B^{p-2}C_2r^{-d}|E|.
$$
Taking the $p$th root gives
$$
 \norm{B_rM_g}_{\mathcal S_p(L^2(\R^d,dv))}
 \leq C_3r^{-d/p}|E|^{1/p},\quad C_3=C_B^{1-2/p}C_2^{1/p}>0.
$$
For the first term containing $\beta_{r,1}$ in
\eqref{eq:factor-ui}, we use \eqref{E7} and
$\norm{\beta_{r,1}}_{L^\infty(\R^d)}\leq K$ to obtain
\begin{align*}
 &r^{d/p}\norm{\Pi_{r^2}M_{\beta_{r,1}}B_rM_g\Pi_{r^2}}
               _{\mathcal S_p(L^2(\R^d,dv))}\leq r^{d/p}K\norm{B_rM_g}_{\mathcal S_p(L^2(\R^d,dv))}
 \leq C_3K|E|^{1/p}.
\end{align*}
For the second term, \eqref{eq:ui-multiplier} and
$\norm{\beta_{r,1}g}_{L^p(\R^d,dv)}\leq K|E|^{1/p}$
give
\begin{align*}
 &r^{d/p}\norm{\Pi_{r^2}B_rM_{\beta_{r,1}g}\Pi_{r^2}}
               _{\mathcal S_p(L^2(\R^d,dv))}\leq r^{d/p}C_BC_Mr^{-d/p}
       \norm{\beta_{r,1}g}_{L^p(\R^d,dv)}
 \leq C_BC_MK|E|^{1/p}.
\end{align*}
Combining these two bounds with \eqref{E92} proves
\eqref{eq:smooth-smallsets} whenever
$$C\geq\max\{2C_MC_B,C_3+C_BC_M\}.$$

For \eqref{eq:smooth-tails}, suppose
$|g|\leq\mathbf1_{B(0,R+1)^c}$. Since $\beta_{r,1}g=0$,
only the first term containing
$\beta_{r,1}$ remains. Set $A_r=M_{\beta_{r,1}}B_rM_g$.
The support conditions give
$|z-w|\geq1$ for
$z\in\supp\beta_{r,1}$ and $w\in\supp g$.
For the kernel $\beta_{r,1}(z)b_{r,0}(z,w)g(w)$,
we repeat the square integral calculation in \eqref{E58},
using \eqref{E32} and \eqref{E60} with $\delta=1$.
Using \eqref{E59} for the Hilbert--Schmidt norm and
\eqref{E134} with \eqref{E33} for the operator norm, we obtain
$$
 \norm{A_r}_{\mathcal S_2(L^2(\R^d,dv))}^2
 \leq C_4K^2r^{-d}|B(0,R)|e^{-c_1/r},\quad
 \norm{A_r}\leq C_BK,
$$
for $0<r<r_0$, where $C_4,c_1>0$ depend only
on $n,m,M$. Substituting these bounds into \eqref{E124}, we obtain
\begin{align*}
 r^d\norm{A_r}_{\mathcal S_p(L^2(\R^d,dv))}^p
 &\leq r^d(C_BK)^{p-2}C_4K^2r^{-d}|B(0,R)|e^{-c_1/r}\\
 &=C_B^{p-2}C_4K^p|B(0,R)|e^{-c_1/r}.
\end{align*}
Taking the $p$th root gives
$$
 r^{d/p}\norm{A_r}_{\mathcal S_p(L^2(\R^d,dv))}
 \leq C_5K|B(0,R)|^{1/p}e^{-c_1/(pr)},\quad
 C_5=C_B^{1-2/p}C_4^{1/p}>0.
$$
The remaining term in \eqref{eq:factor-ui} is
$r^{d/p}\Pi_{r^2}A_r\Pi_{r^2}$, and \eqref{E7} gives
$$
 r^{d/p}\norm{\Pi_{r^2}A_r\Pi_{r^2}}_{\mathcal S_p(L^2(\R^d,dv))}
 \leq r^{d/p}\norm{A_r}_{\mathcal S_p(L^2(\R^d,dv))}.
$$
Combining the bound for $A_r$ with \eqref{E92} and
\eqref{E43}, we take the supremum over
$|g|\leq\mathbf1_{B(0,R+1)^c}$ to obtain
\eqref{eq:smooth-tails}.
Both estimates hold with
$C=\max\{2C_MC_B,C_3+C_BC_M,C_5\}$, which proves \textup{(b)}. The proof is complete.
\end{proof}

Lemma \ref{L4} gives uniform control on small sets and outside
large balls. We use this control to extend the convergence
argument for smooth compactly supported second symbols to
every fixed bounded second symbol.

\begin{lemma}\label{lem:smooth-family-ui}
Let $2\leq p<\infty$ and $r_0>0$. For $0<r<r_0$, let
$u_r\in C^\infty(\R^d)\cap L^\infty(\R^d)$ with
$|\nabla u_r|\in L^p(\R^d,dv)$. If
$
 \{|r^{1-d/p}\nabla u_r|^p:0<r<r_0\}\in\mathcal W
$,
then
$$
 \lim_{r\rightarrow0^{+}}\norm{D_{r^2}(u_r,g)}_{\mathcal S_p(\mathcal F_{r^2}^2)}=0
 \quad\text{for every fixed }g\in L^\infty(\R^d).
$$
\end{lemma}
\begin{proof}
The $L^1(\R^d,dv)$ bound for the gradient densities in the
first condition of \eqref{eq:intrinsic-three} gives
\begin{equation}\label{E205}
 \sup_{0<r<r_0}
 \norm{r^{1-d/p}\nabla u_r}_{L^p(\R^d,dv)}<\infty.
\end{equation}
For fixed $h\in C_c^\infty(\R^d)$, the fundamental theorem
of calculus gives $[h]_1\leq\norm{\nabla h}_{L^\infty(\R^d)}$.
Using \eqref{E18} with $\alpha=1$ and this bound on $[h]_1$,
we obtain
$$
 \norm{H_h^{(r^2)}}
 \leq C_1r[h]_1
 \leq C_1r\norm{\nabla h}_{L^\infty(\R^d)}.
$$
For the other factor, we apply \eqref{eq:hankel-gradient}
and \eqref{E57} with $q=p$ to $\overline{u}_r$ and use
$|\nabla\overline{u}_r|=|\nabla u_r|$ to obtain
\begin{align*}
 \norm{H_{\overline{u}_r}^{(r^2)}}_{\mathcal S_p(\mathcal F_{r^2}^2,L_{r^2}^2\ominus\mathcal F_{r^2}^2)}
 &\leq C_2\norm{r^{1-d/p}\nabla u_r}_{L^p(\R^d,dv)}.
\end{align*}
Combining these two Hankel estimates with \eqref{eq:factorization}
and the first inequality in \eqref{E7}, we obtain
\begin{equation}\label{E61}
 \begin{aligned}
 \norm{D_{r^2}(u_r,h)}_{\mathcal S_p(\mathcal F_{r^2}^2)}
 &\leq\norm{H_{\overline{u}_r}^{(r^2)}}_{\mathcal S_p(\mathcal F_{r^2}^2,L_{r^2}^2\ominus\mathcal F_{r^2}^2)}
                    \norm{H_h^{(r^2)}}\\
 &\leq Cr\norm{r^{1-d/p}\nabla u_r}_{L^{p}(\R^d,dv)}\norm{\nabla h}_{L^\infty(\R^d)}
 \rightarrow0\quad(r\rightarrow0^{+}).
 \end{aligned}
\end{equation}
Here $C=C_1C_2>0$ depends only on $n,m,M,p$.
The limit in \eqref{E61} follows from \eqref{E205},
since $h$ is fixed and $r\rightarrow0^+$.

Now fix $g\in L^\infty(\R^d)$. By linearity, we may assume
$\norm{g}_{L^\infty(\R^d)}\leq1$.
To prove the same convergence for $g$, we approximate its restriction
to a ball by a smooth function and control the part outside
the ball by Lemma \ref{L4} (b).
Let $C_0$ be the constant in Lemma \ref{L4} (a).
Given $\varepsilon>0$, choose $K,R$ as in
Lemma \ref{L4} (b), and retain its constants $C,c_1$.
Fix $a>0$ and use the
mollifier $\eta_\delta$ from \eqref{E16}, with $\delta>0$, to set
$$
 h=(g\mathbf1_{B(0,R+1)})*\eta_\delta,
 \quad q=g\mathbf1_{B(0,R+1)}-h,
 \quad E=\{x\in\R^d:|q(x)|>a\}.
$$
Since $g\mathbf1_{B(0,R+1)}\in L^\infty(\R^d)\cap L^p(\R^d,dv)$,
Lemma \ref{L6}, \eqref{E53}, and the support formula
\cite[Proposition 4.18, p. 106]{Brezis2011} give
$$
 h\in C_c^\infty(\R^d)\quad\text{and}\quad
 \supp h\subseteq\overline{B(0,R+1+\rho\delta)}.
$$
To control $q$, we use the argument leading to \eqref{E64},
with $v=g\mathbf1_{B(0,R+1)}\in L^p(\R^d,dv)$ and
mollification parameter $\delta$, to obtain
\begin{equation}\label{E142}
 \norm{q}_{L^p(\R^d,dv)}\longrightarrow0
 \quad(\delta\rightarrow0^+).
\end{equation}
The convolution bound in \eqref{E50}, the triangle inequality,
and Chebyshev's inequality give
\begin{equation}\label{E206}
 \norm{h}_{L^\infty(\R^d)}\leq1,\quad
 \norm{q}_{L^\infty(\R^d)}\leq2,\quad
 |E|\leq a^{-p}\norm{q}_{L^p(\R^d,dv)}^p<\infty.
\end{equation}

Since $h+q=g\mathbf1_{B(0,R+1)}$, we have
$g=h+q\mathbf1_{E^c}+q\mathbf1_E+g\mathbf1_{B(0,R+1)^c}$.
By \eqref{E206} and the definition of $E$, we have
$|q\mathbf1_{E^c}|\leq a$ and
$|q\mathbf1_E|\leq2\mathbf1_E$.
We apply \eqref{E34} to $q\mathbf1_{E^c}$ and
\eqref{eq:smooth-smallsets} to $(q/2)\mathbf1_E$ to obtain $ \norm{D_{r^2}(u_r,q\mathbf1_{E^c})}_{\mathcal S_p(\mathcal F_{r^2}^2)}
 \leq C_0a$ and
\begin{align*}
 \norm{D_{r^2}(u_r,q\mathbf1_E)}_{\mathcal S_p(\mathcal F_{r^2}^2)}
 &\leq2C\varepsilon+2CK|E|^{1/p}\leq2C\varepsilon+2CKa^{-1}\norm{q}_{L^p(\R^d,dv)}.
 \end{align*}
The last inequality uses the Chebyshev bound in \eqref{E206}.
Since $|g\mathbf1_{B(0,R+1)^c}|\leq\mathbf1_{B(0,R+1)^c}$,
\eqref{eq:smooth-tails} controls the last term in the
decomposition of $g$. Combining this estimate with the two
bounds above and the Schatten norm triangle inequality gives
\begin{equation}\label{E143}
 \begin{aligned}
 \norm{D_{r^2}(u_r,g)}_{\mathcal S_p(\mathcal F_{r^2}^2)}
 &\leq\norm{D_{r^2}(u_r,h)}_{\mathcal S_p(\mathcal F_{r^2}^2)}
       +3C\varepsilon+C_0a\\
 &\quad+2CKa^{-1}\norm{q}_{L^p(\R^d,dv)}
       +CK|B(0,R)|^{1/p}e^{-c_1/(pr)}.
 \end{aligned}
\end{equation}

For fixed $\varepsilon,K,R,a,\delta$, the function $h$ is
independent of $r$, so \eqref{E61} makes the first term on
the right of \eqref{E143} tend to zero as $r\rightarrow0^+$.
The exponential term also tends to zero, which gives
$$
 \limsup_{r\rightarrow0^{+}}\norm{D_{r^2}(u_r,g)}_{\mathcal S_p(\mathcal F_{r^2}^2)}
 \leq 3C\varepsilon+C_0a+2CKa^{-1}\norm{q}_{L^p(\R^d,dv)}.
$$
With $\varepsilon,K,R,a$ fixed, \eqref{E142} gives
$2CKa^{-1}\norm{q}_{L^p(\R^d,dv)}\rightarrow0$
as $\delta\rightarrow0^+$.
Passing to this limit in the preceding estimate yields
$$
 0\leq\limsup_{r\rightarrow0^{+}}
 \norm{D_{r^2}(u_r,g)}_{\mathcal S_p(\mathcal F_{r^2}^2)}
 \leq3C\varepsilon+C_0a.
$$
This estimate holds for every $a,\varepsilon>0$, with
$C,C_0$ independent of both parameters.
Letting $a\rightarrow0^+$ and then $\varepsilon\rightarrow0^+$
proves the asserted limit.
\end{proof}

\subsection{Proof of the criterion for fixed symbols}
We smooth $f$ at a fixed fraction of the operator scale so that
Lemma \ref{lem:smooth-family-ui} can handle the smooth part.
To estimate the remaining contribution, we first bound positive
Toeplitz operators by local averages and use these bounds to
control semi-commutators.

\begin{lemma}\label{lem:average-toeplitz}
At the fixed scale $t=1$, let $u\geq0$ be bounded.
If $1\leq q<\infty$ and $M_{1,1}(u)\in L^q(\C^n,dv)$, then
\begin{equation}\label{E20}
 \norm{T_u}_{\mathcal S_q(\mathcal F_{1}^2)}
 \leq C_q\norm{M_{1,1}(u)}_{L^{q}(\R^d,dv)}.
\end{equation}
Here $C_q>0$ depends only on $n,m,M,q$ and is independent of $u$.
\end{lemma}
\begin{proof}
To compare $T_u$ with $T_{M_{1,1}(u)}$, let
$h\in\mathcal F_1^2$. We first derive the weighted local
estimate of \cite[Lemma 2.3]{IsralowitzVirtanenWolf2015},
recording its constant for the fixed real-valued weight
$\varphi\in C^2(\C^n)$ satisfying \eqref{eq:weight}. For each fixed $z\in\C^n$, we define
the holomorphic affine function $g_z:\C^n\to\C$ of the
variable $w$ by
$$
 g_z(w)=\varphi(z)+2\sum_{j=1}^n
       \partial_{z_j}\varphi(z)(w_j-z_j).
$$
Writing $z_j=x_{2j-1}+ix_{2j}$ and
$w_j=y_{2j-1}+iy_{2j}$ for $1\leq j\leq n$, we use
\eqref{E13} and the real-valuedness of $\varphi$ to obtain
\begin{align*}
 2\Re\bigl(\partial_{z_j}\varphi(z)(w_j-z_j)\bigr)
 &=\partial_{x_{2j-1}}\varphi(z)(y_{2j-1}-x_{2j-1})+\partial_{x_{2j}}\varphi(z)(y_{2j}-x_{2j}).
\end{align*}
Summing over $j$ and using \eqref{E109}, with $w-z$
expressed in real coordinates, yields
$$
 \Re g_z(w)=\varphi(z)+\nabla\varphi(z)\cdot(w-z).
$$
This is the affine Taylor polynomial of $\varphi$ at $z$.
To estimate its remainder, we set $v=w-z$ in real coordinates
and $F(\theta)=\varphi(z+\theta v)$ for $0\leq\theta\leq1$.
Since $\varphi\in C^2$, the chain rule and \eqref{eq:weight} give
\begin{align*}
 F'(0)&=\nabla\varphi(z)\cdot v\quad\text{and}\quad
 F''(\theta)
 =\bigl(\Hess_{\R}\varphi(z+\theta v)v\bigr)\cdot v
 \leq M|v|^2.
\end{align*}
Integration by parts gives
\begin{align*}
 \varphi(w)-\Re g_z(w)
 &=F(1)-F(0)-F'(0)=\int_0^1(1-\theta)F''(\theta)d\theta\\
 &\leq M|v|^2\int_0^1(1-\theta)d\theta
 =\frac M2|w-z|^2\leq\frac M2
\end{align*}
for $w\in B(z,1)$.

Applying the ball mean inequality to $he^{-g_z}$ and using
the preceding bound on $\varphi-\Re g_z$, we obtain
\begin{align*}
 |h(z)e^{-g_z(z)}|^2
 &=|h(z)|^2e^{-2\varphi(z)}\leq\frac1{|B(0,1)|}
       \int_{B(z,1)}|h(w)|^2e^{-2\Re g_z(w)}dv(w)\\
 &\leq\frac{e^M}{|B(0,1)|}
\int_{B(z,1)}|h(w)|^2e^{-2\varphi(w)}dv(w)\\
&\leq c_0\int_{B(z,1)}|h(w)|^2d\mu_1(w),
\end{align*}
where $c_0=e^M/|B(0,1)|>0$.
Multiplying this local estimate by $u(z)\geq0$ and integrating
with respect to $dv(z)$, we use $Ph=h$ and
$d\mu_1(z)=e^{-2\varphi(z)} dv(z)$ and apply Tonelli's theorem
to obtain
\begin{align*}
 0\leq\langle T_uh,h\rangle_1
 &=\int_{\C^n}u(z)|h(z)|^2d\mu_1(z)\leq c_0\int_{\C^n}u(z)
       \int_{B(z,1)}|h(w)|^2d\mu_1(w) dv(z)\\
 &=c_0\int_{\C^n}|h(w)|^2
       \left(\int_{B(w,1)}u(z) dv(z)\right)d\mu_1(w)\\
 &=c_0|B(0,1)|
       \langle T_{M_{1,1}(u)}h,h\rangle_1.
 \end{align*}
The change of integration order uses
$w\in B(z,1)$ if and only if $z\in B(w,1)$.
The last equality follows from the definition of
$M_{1,1}$ in \eqref{E4}.
As $h$ was arbitrary, this proves the positive operator
comparison
$$
 0\leq T_u\leq c_0|B(0,1)|T_{M_{1,1}(u)}.
$$

Since $u\geq0$, we have $M_{1,1}(u)\geq0$.
This together with the hypothesis
$M_{1,1}(u)\in L^q(\C^n,dv)$ allows us to apply
\eqref{eq:positive} to $M_{1,1}(u)$ with $t=r=1$,
which yields
$T_{M_{1,1}(u)}\in\mathcal S_q(\mathcal F_1^2)$.
Combining the operator comparison above with \eqref{E8},
we obtain
$T_u\in\mathcal S_q(\mathcal F_1^2)$ and
 \begin{align*}
 \norm{T_u}_{\mathcal S_q(\mathcal F_1^2)}
 &\leq c_0|B(0,1)|
       \norm{T_{M_{1,1}(u)}}_{\mathcal S_q(\mathcal F_1^2)}\leq c_0|B(0,1)|b_q
       \norm{M_{1,1}(u)}_{L^q(\R^d,dv)}.
 \end{align*}
The last inequality follows from \eqref{eq:positive}
with $t=r=1$, where $b_q>0$ is the $q$-th root of the
constant in that estimate.
This proves \eqref{E20} with
$C_q=c_0|B(0,1)|b_q>0$, which depends only on $n,m,M,q$. The proof is complete.
\end{proof}

To estimate the semi-commutator of a bounded remainder,
we use the positive Toeplitz operator with its absolute value as symbol.

\begin{lemma}\label{lem:remainder}
Let $p\geq1$ and $f,g\in L^\infty(\C^n)$. If
$T_{|f|}^{(t)}\in\mathcal S_p(\cF)$, then
\begin{equation}\label{E21}
 \norm{D_t(f,g)}_{\mathcal S_p}
 \leq2\norm{g}_{L^\infty(\C^n)}\norm{T_{|f|}^{(t)}}_{\mathcal S_p}.
\end{equation}
\end{lemma}
\begin{proof}
Define $J:\cF\to L_t^2$ by $Jh=|f|^{1/2}h$ for
$h\in\cF$, and let $\theta=f/|f|$ where $f\neq0$,
with value zero elsewhere.
The adjoint formula \eqref{E68},
applied to $|f|^{1/2}$, gives
$$
 J^*J=T_{|f|}^{(t)},\quad
 T_f^{(t)}=J^*M_\theta J,\quad
 T_{fg}^{(t)}=J^*M_{\theta g}J.
$$
By the first identity and \eqref{E9}, we have
$$
 J\in\mathcal S_{2p}(\cF,L_t^2)\quad\text{and}\quad
 \norm{J}_{\mathcal S_{2p}(\cF,L_t^2)}^2
 =\norm{T_{|f|}^{(t)}}_{\mathcal S_p}.
$$
Substituting the two Toeplitz factorizations above into
\eqref{E2}, we apply the two inequalities in \eqref{E7}
and the triangle inequality. The preceding identity for the
Schatten norm of $J$ then gives
\begin{align*}
 \norm{D_t(f,g)}_{\mathcal S_p}
 &\leq\norm{J^*M_\theta J}_{\mathcal S_p}\norm{T_g^{(t)}}
       +\norm{J^*M_{\theta g}J}_{\mathcal S_p}\\
 &\leq\norm{J}_{\mathcal S_{2p}(\cF,L_t^2)}^2
       \bigl(\norm{M_\theta}\norm{T_g^{(t)}}
                    +\norm{M_{\theta g}}\bigr)\\
 &\leq2\norm{T_{|f|}^{(t)}}_{\mathcal S_p}
       \norm{g}_{L^\infty(\C^n)}.
\end{align*}
Here $|\theta|\leq1$ gives $\norm{M_\theta}\leq1$ and
$\norm{M_{\theta g}}\leq\norm{g}_{L^\infty(\C^n)}$.
This proves \eqref{E21}.
\end{proof}

To make this remainder estimate small as the smoothing fraction
tends to zero, we compare local averages at the two scales.

\begin{lemma}\label{lem:average-scale-ui}
Let $u\in L^1_{\rm loc}(\R^d,dv)$ be nonnegative.
The local averages satisfy
$$
 \norm{M_{1,r}(u)}_{L^{p}(\R^d,dv)}\leq2^d\norm{M_{1,s}(u)}_{L^{p}(\R^d,dv)}
$$
for $0<s\leq r$ and $p\geq1$.
\end{lemma}
\begin{proof}
For the averages defined in \eqref{E4}, Tonelli's theorem and
$B(w,s)\subset B(z,r+s)$ for $w\in B(z,r)$ give
$$
 \int_{B(z,r+s)}M_{1,s}(u)(y) dv(y)
 =\int_{\R^d}u(w)\frac{|B(w,s)\cap B(z,r+s)|}{|B(0,s)|} dv(w)
 \geq\int_{B(z,r)}u(w) dv(w).
$$
Dividing by $|B(0,r)|$ gives the pointwise bound
$$
 M_{1,r}(u)(z)
 \leq\bigg(1+\frac sr\bigg)^d
                 M_{1,r+s}(M_{1,s}(u))(z).
$$
If $\norm{M_{1,s}(u)}_{L^p(\R^d,dv)}=\infty$,
there is nothing to prove. Otherwise, \eqref{E56}, applied
with averaging exponent one and radius $r+s$, yields
\begin{align*}
 \norm{M_{1,r}(u)}_{L^p(\R^d,dv)}
 &\leq\bigg(1+\frac sr\bigg)^d
       \norm{M_{1,r+s}(M_{1,s}(u))}_{L^p(\R^d,dv)}\leq2^d\norm{M_{1,s}(u)}_{L^p(\R^d,dv)},
\end{align*}
because $s\leq r$. This proves the assertion.
\end{proof}

\begin{proof}[Proof of Theorem \ref{thm:weakcriterion}]
Proposition \ref{prop:necessary-wcompact} and
Lemma \ref{lem:ui-analytic-real} prove (a)$\Rightarrow$(b) and
(b)$\Leftrightarrow$(c), respectively. It remains to show
(c)$\Rightarrow$(a). Fix $r_0>0$ such that
$\{\omega_s(f):0<s<r_0\}\in\mathcal W$.

Let $\theta=\eta_0$, where $\eta_0$ is fixed before
\eqref{E16}. For each $s>0$, define
$\theta_s(z)=s^{-d}\theta(z/s)=\eta_{s/\rho}(z)$
for $z\in\R^d$, and set $u_s=f*\theta_s$ and $k_s=f-u_s$.
The kernels $\theta_s$ satisfy $\theta_s\geq0$,
$\int_{\R^d}\theta_s dv=1$, and
$\supp\theta_s\subset B(0,s)$ by \eqref{E16} and \eqref{E53}.
Applying Lemma \ref{L6} and \eqref{E50} with $r=s/\rho$,
we obtain
$$
 u_s\in C^\infty(\R^d)\cap L^\infty(\R^d),\quad
 \norm{u_s}_{L^\infty(\R^d)}\leq\norm{f}_{L^\infty(\R^d)},\quad
 \norm{k_s}_{L^\infty(\R^d)}\leq2\norm{f}_{L^\infty(\R^d)}.
$$
To express the gradient in terms of the local oscillation,
we use the calculation in \eqref{E144} with $h\equiv1$
and $r=s/\rho$ in each real coordinate, together with the
chain rule, to obtain
$$
 \int_{B(z,s)}\nabla\theta_s(z-w) dv(w)=0.
$$
For fixed $z$, the average $f_{B(z,s)}$ is constant in $w$.
Using this zero integral in \eqref{E145} with $u=f$ and
$\varepsilon=s/\rho$, we obtain
\begin{align*}
 \nabla u_s(z)
 &=\int_{B(z,s)}(f(w)-f_{B(z,s)})\nabla\theta_s(z-w) dv(w).
\end{align*}
The chain rule gives
$\nabla\theta_s(y)=s^{-d-1}(\nabla\theta)(y/s)$.
Taking absolute values in the gradient identity and using
this scaling, we obtain
\begin{align*}
 |\nabla u_s(z)|
 &\leq s^{-d-1}\norm{\nabla\theta}_{L^\infty(\R^d)}
       \int_{B(z,s)}|f(w)-f_{B(z,s)}| dv(w).
\end{align*}
Using \eqref{E6} and $|B(z,s)|=s^d|B(0,1)|$, we conclude that
\begin{equation}\label{eq:mollified-oscillation-ui}
 |\nabla u_s(z)|\leq C_\theta s^{-1}\operatorname{MO}_{1,s}(f)(z).
\end{equation}
Here $C_\theta>0$ depends only on $d$ and the fixed function $\theta$,
and is independent of $f,s,z$.

To use \eqref{E20} and \eqref{E21}, we first bound the local
averages of $k_s$. For $\gamma\in\R^d$ with $|\gamma|<s$, we have 
$$B(z,s)\cup(B(z,s)-\gamma)\subseteq B(z,2s).$$
With $c=f_{B(z,2s)}$, the triangle inequality, the change of
variables $y=w-\gamma$, and Tonelli's theorem give
\begin{align*}
\int_{B(z,s)}|k_s(w)| dv(w)
 &\leq\int_{B(z,s)}\int_{B(0,s)}\theta_s(\gamma)|f(w)-f(w-\gamma)| dv(\gamma) dv(w)\\
 &\leq\int_{B(z,s)}|f(w)-c| dv(w)
   +\int_{B(0,s)}\theta_s(\gamma)\int_{B(z,s)-\gamma}|f(y)-c| dv(y) dv(\gamma)\\
 &\leq2\int_{B(z,2s)}|f(y)-c| dv(y).
 \end{align*}
Dividing by $|B(0,s)|$ proves
\begin{equation}\label{eq:mollified-remainder-ui}
 M_{1,s}(k_s)(z)\leq2^{d+1}\operatorname{MO}_{1,2s}(f)(z).
\end{equation}
The hypothesis and \eqref{eq:intrinsic-three} give
$b=\sup_{0<\tau<r_0}\norm{\omega_\tau(f)}_{L^{1}(\R^d,dv)}^{1/p}<\infty$.
By \eqref{eq:energy-definitions}, this means that
$$\norm{\operatorname{MO}_{1,\tau}(f)}_{L^p(\R^d,dv)}
\leq b\tau^{d/p}$$ for $0<\tau<r_0$.
Using this bound with $\tau=s$ in
\eqref{eq:mollified-oscillation-ui} and with $\tau=2s$ in
\eqref{eq:mollified-remainder-ui}, we obtain
\begin{equation}\label{E35}
 \norm{M_{1,s}(k_s)}_{L^p(\R^d,dv)}\leq C_{\rm rem}bs^{d/p}
 \quad\text{and}\quad \norm{\nabla u_s}_{L^p(\R^d,dv)}\leq C_\theta bs^{d/p-1}
\end{equation}
for $0<s<r_0/2$.
Here $C_{\rm rem}=2^{d+1+d/p}>0$, and $C_\theta$ is the constant
in \eqref{eq:mollified-oscillation-ui}.

Fix $0<\lambda<1/4$ and, for $0<r<r_0$, set $s=\lambda r$.
Equations \eqref{eq:mollified-oscillation-ui} and
\eqref{eq:energy-definitions} imply
$$
 |r^{1-d/p}\nabla u_{\lambda r}(z)|^p
       \leq C_\theta^p\lambda^{d-p}\omega_{\lambda r}(f)(z).
$$
Since $0<\lambda r<r_0$ and $\lambda$ is fixed, we apply
Lemma \ref{L9} to the preceding bound and the family
$\{\omega_s(f):0<s<r_0\}\in\mathcal W$ to obtain
$$
 \{|r^{1-d/p}\nabla u_{\lambda r}|^p:0<r<r_0\}
       \in\mathcal W.
$$
By Lemma \ref{L6}, applied to $f$ with mollification
parameter $\lambda r/\rho$, the functions $u_{\lambda r}$
are smooth and bounded. Equation \eqref{E35}, with
$s=\lambda r$, puts their gradients in $L^p(\R^d,dv)$.
These facts together with the weak compactness just proved
verify all hypotheses of Lemma \ref{lem:smooth-family-ui},
whose conclusion is
\begin{equation}\label{eq:smooth-component-limit-ui}
 \norm{D_{r^2}(u_{\lambda r},g)}_{\mathcal S_p(\mathcal F_{r^2}^2)}\rightarrow0
 \quad(r\rightarrow0^{+})
\end{equation}
for every fixed bounded $g$.

For the remainder $k_s$, the ratio $s/r=\lambda$ gives the small
factor in a bound uniform in $r$.
Definition \eqref{E4} expresses the change of scale as
$$
 M_{1,1}(|k_s|(r\cdot))(\zeta)
 =M_{1,r}(k_s)(r\zeta),\quad \zeta\in\C^n.
$$
Since $0<s=\lambda r<r$ and $2s<r_0$, Lemma
\ref{lem:average-scale-ui}, applied to $|k_s|$, and
\eqref{E35} show that $M_{1,r}(k_s)\in L^p(\R^d,dv)$.
The estimate \eqref{E20} for the bounded symbol $|k_s|(r\cdot)$,
together with the Toeplitz norm invariance in
Lemma \ref{lem:scaled-kernel} (b), therefore yields
$T_{|k_s|}^{(r^2)}\in\mathcal S_p(\mathcal F_{r^2}^2)$.
Let $c_p>0$ be the constant in \eqref{E20} with $q=p$.
Applying \eqref{E20} and \eqref{E21}, and substituting
$z=r\zeta$, we obtain
\begin{align*}
 \norm{D_{r^2}(k_s,g)}_{\mathcal S_p(\mathcal F_{r^2}^2)}
 &\leq2\norm{g}_{L^\infty(\R^d)}\norm{T_{|k_s|}^{(r^2)}}_{\mathcal S_p(\mathcal F_{r^2}^2)}\leq2c_pr^{-d/p}\norm{M_{1,r}(k_s)}_{L^p(\R^d,dv)}\norm{g}_{L^\infty(\R^d)}\\
 &\leq2^{d+1}c_pr^{-d/p}\norm{M_{1,s}(k_s)}_{L^p(\R^d,dv)}\norm{g}_{L^\infty(\R^d)}\\
 &\leq2^{d+1}c_pC_{\rm rem}b\lambda^{d/p}\norm{g}_{L^\infty(\R^d)}.
 \end{align*}
The third inequality follows from Lemma \ref{lem:average-scale-ui},
and the fourth inequality from the first estimate in \eqref{E35}
with $s=\lambda r$. Set
$C_*=2^{d+1}c_pC_{\rm rem}=2^{2d+2+d/p}c_p>0$.
Since $f=u_{\lambda r}+k_{\lambda r}$, the triangle inequality
and \eqref{eq:smooth-component-limit-ui} give
$$
 \limsup_{r\rightarrow0^{+}}\norm{D_{r^2}(f,g)}_{\mathcal S_p(\mathcal F_{r^2}^2)}
       \leq C_*b\lambda^{d/p}\norm{g}_{L^\infty(\R^d)}.
$$
Note that $C_*$ depends only on $n,m,M,p$ and is
independent of $f,r_0,\lambda,r,g$.
Since $b<\infty$ is independent of $\lambda$, letting
$\lambda\rightarrow0^{+}$ after taking the upper limit in $r$
proves $f\in\mathcal A_p$. The proof is complete.
\end{proof}

The criteria above describe membership in $\mathcal A_p$
without reference to the weight and also give its algebra properties.

\begin{corollary}\label{cor:algebra-weight}
For $0<p<\infty$, the class $\mathcal A_p$ is a unital algebra
closed under complex conjugation and is the same for all
weights satisfying \eqref{eq:weight}.
\end{corollary}
\begin{proof}
For $0<p<d$, Theorem \ref{thm:main} (a) identifies
$\mathcal A_p$ with the almost everywhere constant functions,
so all assertions follow. Suppose $p\geq d$. Theorem \ref{thm:main} (b) for $p=d$
and Theorem \ref{thm:weakcriterion} (c) for $p>d$
characterize membership in $\mathcal A_p$ using only
$f$ and Lebesgue measure. Hence, $\mathcal A_p$ is
independent of the weight.

For $p=d$, the identity
$|\nabla\overline{f}|=|\nabla f|$
preserves the gradient condition in the first criterion.
For $p>d$, the identity
$\operatorname{MO}_{1,r}(\overline{f})
=\operatorname{MO}_{1,r}(f)$
leaves the densities $\omega_r(f)$ in \eqref{eq:energy-definitions}
unchanged, so the second criterion is also preserved.
These identities show that $\overline{f}\in\mathcal A_p$ whenever
$f\in\mathcal A_p$.

To prove the algebra properties, fix $f,h\in\mathcal A_p$
and $g\in L^\infty(\C^n)$.
Since $p\geq d=2n\geq2$, linearity in the first symbol
in \eqref{E2} and the Schatten triangle inequality give
$$
 \norm{D_t(af+bh,g)}_{\mathcal S_p}
 \leq |a|\norm{D_t(f,g)}_{\mathcal S_p}
      +|b|\norm{D_t(h,g)}_{\mathcal S_p}
 \longrightarrow0\quad(t\rightarrow0^{+})
$$
for $a,b\in\C$, where the limit follows from
\eqref{eq:property}. Since $g$ was arbitrary, this proves $af+bh\in\mathcal A_p$.

For the product $fh$, expanding the semi-commutators
by \eqref{E2} gives
$$
 D_t(fh,g)=T_f^{(t)}D_t(h,g)+D_t(f,hg)-D_t(f,h)T_g^{(t)}.
$$
The first inequality in \eqref{E7} and the Toeplitz
operator norm bounds from \eqref{E1} now yield
\begin{align*}
 \norm{D_t(fh,g)}_{\mathcal S_p}
 &\leq\norm{f}_{L^\infty(\C^n)}\norm{D_t(h,g)}_{\mathcal S_p}
       +\norm{D_t(f,hg)}_{\mathcal S_p}\\
 &\quad+\norm{D_t(f,h)}_{\mathcal S_p}\norm{g}_{L^\infty(\C^n)}
 \longrightarrow0\quad(t\rightarrow0^{+}).
\end{align*}
Indeed, the three Schatten norms tend to zero by
\eqref{eq:property}, applied to the symbol pairs
$(h,g)$, $(f,hg)$, and $(f,h)$, respectively.
Here $g,hg,h$ are fixed bounded second symbols.
Since $g$ was arbitrary, \eqref{eq:property} gives
$fh\in\mathcal A_p$.
Finally, \eqref{E1} and \eqref{E2} give
$T_1^{(t)}=I$ and $D_t(1,g)=0$, so
$1\in\mathcal A_p$ is the multiplicative unit.
The proof is complete.
\end{proof}

\section{Uniform convergence and H\"older sufficient conditions}\label{sec:uniform}
Theorem \ref{thm:weakcriterion} characterizes convergence for
each fixed bounded second symbol. We now prove the uniform
criterion in Theorem \ref{thm:uniform}, then apply both
criteria to obtain Theorem \ref{thm:holder}.

\subsection{Uniform convergence}
Lemma \ref{lem:phase} gives the lower estimate in
\eqref{eq:two-sided}. For the upper estimate, the analytic
decomposition at $r=1$ controls the complex derivatives of the
smooth part and the local $L^1$ averages of the remainder by
the same analytic distance. Lemmas \ref{lem:average-toeplitz} and
\ref{lem:remainder} convert the latter control into a Schatten
bound, retaining local exponent one.

\begin{proof}[Proof of Theorem \ref{thm:uniform}]
We first prove the upper estimate for the testing norm
$m_{p,1}(f)$ in \eqref{E5}. Fix $g\in L^\infty(\C^n)$, put
$u=\overline{f}$, and choose the smooth partition in
Lemma \ref{lem:analytic-decomposition} once for each dimension $n$.
At $r=1$, the lemma provides
$u_1\in C^\infty(\C^n)\cap L^\infty(\C^n)$ and
$u_2\in L^\infty(\C^n)$ with $u=u_1+u_2$.
With $c_0>0$ from \eqref{E15}, this decomposition satisfies
\begin{equation}\label{eq:decomposition}
 |\dbar u_1(z)|+M_{1,1}(u_2)(z)
 \leq c_0 G_{1,4}(u)(z).
\end{equation}

Assume first that $G_{1,4}(u)\in L^p(\R^d,dv)$.
Equation \eqref{eq:decomposition} gives
\begin{equation}\label{E155}
 \begin{aligned}
 \max\{\norm{\dbar u_1}_{L^p(\R^d,dv)},\norm{M_{1,1}(u_2)}_{L^p(\R^d,dv)}\}
 &\leq c_0\norm{G_{1,4}(u)}_{L^p(\R^d,dv)}<\infty.
 \end{aligned}
\end{equation}
Since $u_1$ is bounded and smooth and $p\geq2$,
we apply \eqref{E14} with $q=p$ and use \eqref{E155} to obtain
$$
 \norm{\nabla u_1}_{L^p(\R^d,dv)}
 \leq C_{n,p}\norm{\dbar u_1}_{L^p(\R^d,dv)}
 \leq C_{n,p}c_0\norm{G_{1,4}(u)}_{L^p(\R^d,dv)}<\infty.
$$
This verifies the gradient hypothesis of Lemma
\ref{lem:gradient} (a). With $c_1>0$ from
\eqref{eq:hankel-gradient}, that estimate at $t=1$,
the ideal inequality \eqref{E7}, and the factorization
\eqref{eq:factorization} give
$$
 \norm{D_1(\overline{u}_1,g)}_{\mathcal S_p(\mathcal F_{1}^2)}
 \leq\norm{H_{u_1}^{(1)}}_{\mathcal S_p(\mathcal F_1^2,L_1^2\ominus\mathcal F_1^2)}\norm{g}_{L^\infty(\C^n)}
 \leq c_1\norm{\dbar u_1}_{L^{p}(\R^d,dv)}\norm{g}_{L^\infty(\C^n)}.
$$

For the remainder, we apply \eqref{E20} with $q=p$ to
$|u_2|$, using \eqref{E155}, and combine the result with
\eqref{E21} to obtain
\begin{align*}
 \norm{D_1(\overline{u}_2,g)}_{\mathcal S_p(\mathcal F_{1}^2)}
 &\leq2\norm{g}_{L^\infty(\C^n)}\norm{T_{|u_2|}}_{\mathcal S_p(\mathcal F_{1}^2)}\\
 &\leq 2C_p\norm{M_{1,1}(u_2)}_{L^{p}(\R^d,dv)}\norm{g}_{L^\infty(\C^n)}.
 \end{align*}
Set $C=c_0(c_1+2C_p)>0$, where $C_p$ is the constant
in \eqref{E20} with $q=p$.
Since $f=\overline{u}_1+\overline{u}_2$, the triangle inequality
and the preceding estimates give
\begin{align*}
 \norm{D_1(f,g)}_{\mathcal S_p(\mathcal F_1^2)}
 &\leq
 \norm{D_1(\overline{u}_1,g)}_{\mathcal S_p(\mathcal F_1^2)}
 +\norm{D_1(\overline{u}_2,g)}_{\mathcal S_p(\mathcal F_1^2)}\\
 &\leq C\norm{G_{1,4}(\overline{f})}_{L^p(\R^d,dv)}
          \norm{g}_{L^\infty(\C^n)},
 \end{align*}
where the last step uses \eqref{E155}, and
$C$ depends only on $n,m,M,p$, independently of $f$ and $g$.
Taking the supremum over
$\norm{g}_{L^\infty(\C^n)}\leq1$ in \eqref{E5},
we obtain
\begin{equation}\label{E69}
 m_{p,1}(f)
 \leq C\norm{G_{1,4}(\overline{f})}_{L^p(\R^d,dv)}.
\end{equation}
The inequality also holds when its right side is infinite.

Fix $r>0$ and set
$$
 f_r(\zeta)=f(r\zeta),\quad
 g_r(\zeta)=g(r\zeta),\quad \zeta\in\C^n.
$$
The identities
$\norm{g_r}_{L^\infty(\C^n)}=\norm{g}_{L^\infty(\C^n)}$
and $(g_r)(\cdot/r)=g$ show that $g\mapsto g_r$ maps the
$L^\infty(\C^n)$ unit ball bijectively onto itself.
Using \eqref{E55} and the norm invariance in
Lemma \ref{lem:scaled-kernel} (b), we obtain
 \begin{equation}\label{E207}
 \begin{aligned}
 m_{p,r^2}(f)
 &=\sup_{\norm{g}_{L^\infty(\C^n)}\leq1}
   \norm{D_1(f_r,g_r)}_{\mathcal S_p(\mathcal F_1^2)}=m_{p,1}(f_r).
 \end{aligned}
\end{equation}
Equation \eqref{E149}, with $q=1$, $s=4$, and
$h_r=\overline{f_r}$, gives
$$
 G_{1,4}(\overline{f_r})(\zeta)
 =G_{1,4r}(\overline{f})(r\zeta)
$$
for $\zeta\in\C^n$. Applying \eqref{E69} to $f_r$,
using \eqref{E207}, and changing variables by $z=r\zeta$, we obtain
\begin{align}\label{E70}
 m_{p,r^2}(f)
 &\leq C\norm{G_{1,4}(\overline{f_r})}_{L^p(\R^d,dv)}
=Cr^{-d/p}\norm{G_{1,4r}(\overline{f})}_{L^p(\R^d,dv)}.
\end{align}
This proves the upper bound in \eqref{eq:two-sided}
with $R=4$.

For the lower bound, take $\rho>0$ from Lemma \ref{L1}
and choose the lattice spacing $L=3\rho>2\rho$.
Then \eqref{eq:phase} in Lemma \ref{lem:phase} gives
\begin{equation}\label{E71}
 c_2^{-1}r^{-d/p}
 \norm{G_{1,\rho r}(\overline{f})}_{L^p(\R^d,dv)}
 \leq m_{p,r^2}(f),
\end{equation}
where $c_2>0$ is its constant for this choice of $L$.
Then the lower bound holds with $c=c_2^{-1}$.
Since $\rho$ depends only on $n,m,M$, both constants
$c,C>0$ in \eqref{eq:two-sided} depend only on $n,m,M,p$.
This proves \eqref{eq:two-sided}.

To prove \eqref{eq:uniform-characterization}, recall that
$R=4$ in \eqref{eq:two-sided}. Let $s>0$. Substituting $r=s/4$ in the upper
bound \eqref{E70} and $r=s/\rho$
in the lower bound \eqref{E71}, we obtain
\begin{align*}
 s^{-d/p}\norm{G_{1,s}(\overline{f})}_{L^p(\R^d,dv)}
\leq c^{-1}\rho^{-d/p}m_{p,(s/\rho)^2}(f),\quad
 m_{p,s^2/16}(f)
\leq C4^{d/p}s^{-d/p}
       \norm{G_{1,s}(\overline{f})}_{L^p(\R^d,dv)}.
\end{align*}
Both changes of variable $t=(s/\rho)^2$ and $t=s^2/16$
are bijections of $(0,\infty)$ with
$t\rightarrow0^{+}$ if and only if $s\rightarrow0^{+}$.
The two inequalities therefore prove the respective
implications in \eqref{eq:uniform-characterization}. The proof is complete.
\end{proof}

Theorem \ref{thm:uniform} characterizes uniform convergence
through local analytic distance. We compare this distance
with ordinary mean oscillation to obtain a criterion in terms
of $\omega_r(f)$.

\begin{corollary}\label{cor:uniform-real}
For bounded $f$ and $2\leq p<\infty$, the following equivalence holds.
$$
 \lim_{t\rightarrow0^{+}}m_{p,t}(f)=0
 \quad\Longleftrightarrow\quad
 \lim_{r\rightarrow0^{+}}\int_{\R^d}\omega_r(f) dv=0.
$$
\end{corollary}
\begin{proof}
We first compare the $L^p$ norms of
$\operatorname{MO}_{1,r}(f)$ and $G_{1,4r}(\overline{f})$.
For fixed $r>0$, we consider the case
$G_{1,4r}(\overline{f})\in L^p(\R^d,dv)$.
Lemma \ref{lem:analytic-decomposition} provides
$\overline{f}=u_{1,r}+u_{2,r}$ with $u_{1,r}$ smooth and bounded.
By \eqref{E15} and \eqref{E14} with $q=p>1$, we have
$|\dbar u_{1,r}|,|\nabla u_{1,r}|\in L^p(\R^d,dv)$.

The first inequality in \eqref{E51} follows from \eqref{E26}
for the smooth part and \eqref{E6} for the remainder, so it
applies to the present decomposition for every $r>0$.
Multiplying it by $r^{d/p}$ and taking $L^p(\R^d,dv)$ norms,
we use \eqref{E56} with $q=1$, $a=p$, $s=r$,
and $h=|\nabla u_{1,r}|$ to obtain
\begin{equation}\label{E36}
 \begin{aligned}
 \norm{\operatorname{MO}_{1,r}(f)}_{L^{p}(\R^d,dv)}
 &\leq C_Pr\norm{M_{1,r}(|\nabla u_{1,r}|)}_{L^p(\R^d,dv)}
                  +2\norm{M_{1,r}(u_{2,r})}_{L^p(\R^d,dv)}\\
 &\leq C_Pr\norm{\nabla u_{1,r}}_{L^p(\R^d,dv)}
                  +2\norm{M_{1,r}(u_{2,r})}_{L^p(\R^d,dv)}\\
 &\leq C_PC_{n,p}r\norm{\dbar u_{1,r}}_{L^p(\R^d,dv)}
                  +2\norm{M_{1,r}(u_{2,r})}_{L^p(\R^d,dv)}\\
 &\leq C\norm{G_{1,4r}(\overline{f})}_{L^p(\R^d,dv)}.
 \end{aligned}
\end{equation}
Here $C_P$ is the constant in \eqref{E26}.
The third inequality uses \eqref{E14} with $q=p$ to bound
$\norm{\nabla u_{1,r}}_{L^p(\R^d,dv)}$ by
$C_{n,p}\norm{\dbar u_{1,r}}_{L^p(\R^d,dv)}$.
The last inequality uses the decomposition bound
\eqref{eq:scaled-decomposition-ui}, with $C_D>0$ chosen as in
\eqref{E15}, and we take $C=C_D(C_PC_{n,p}+2)>0$.
The constant $C$ depends only on $n,p$ and the fixed
auxiliary choices in Lemma \ref{lem:analytic-decomposition},
and is independent of $f$ and $r$.
The bound \eqref{E36} also holds when its right side is infinite,
so it is valid for every $r>0$.

Integrating \eqref{E152} gives the first inequality in
\begin{equation}\label{E156}
 \int_{\R^d}e_r(f) dv
 \leq\int_{\R^d}\omega_r(f) dv
 \leq C^p4^d\int_{\R^d}e_{4r}(f) dv.
\end{equation}
For the second inequality, we raise \eqref{E36} to the power $p$,
multiply by $r^{-d}$, and use \eqref{eq:energy-definitions}
with $r^{-d}=4^d(4r)^{-d}$. Assume that $m_{p,t}(f)\rightarrow0$ as $t\rightarrow0^+$.
Setting $s=4r$ in \eqref{eq:uniform-characterization} gives
\begin{align*}
 r^{-d/p}\norm{G_{1,4r}(\overline{f})}_{L^p(\R^d,dv)}
 &=4^{d/p}(4r)^{-d/p}
   \norm{G_{1,4r}(\overline{f})}_{L^p(\R^d,dv)}
 \longrightarrow0\quad(r\rightarrow0^+).
\end{align*}
In particular, $G_{1,4r}(\overline{f})$ belongs to $L^p(\R^d,dv)$
for all sufficiently small $r>0$.
Using the second inequality in \eqref{E156} and
\eqref{eq:energy-definitions}, we obtain
$$
 0\leq\int_{\R^d}\omega_r(f) dv
 \leq C^p\left(r^{-d/p}
       \norm{G_{1,4r}(\overline{f})}_{L^p(\R^d,dv)}\right)^p
 \longrightarrow0\quad(r\rightarrow0^+).
$$

Conversely, suppose that
$\int_{\R^d}\omega_r(f) dv\rightarrow0$ as $r\rightarrow0^+$.
The first inequality in \eqref{E156} and
\eqref{eq:energy-definitions} give
$$
 \left(r^{-d/p}
       \norm{G_{1,r}(\overline{f})}_{L^p(\R^d,dv)}\right)^p
 =\int_{\R^d}e_r(f) dv
 \leq\int_{\R^d}\omega_r(f) dv
 \longrightarrow0\quad(r\rightarrow0^+).
$$
Taking $p$-th roots and using \eqref{eq:uniform-characterization},
we obtain $m_{p,t}(f)\rightarrow0$ as $t\rightarrow0^+$.
This proves the equivalence, which completes the proof.
\end{proof}

At $p=d$, a bound for $m_{d,t}(f)$ uniform over
$0<t<t_0$, for some $t_0>0$, already forces
$|\nabla f|\in L^d(\R^d,dv)$.

\begin{corollary}\label{cor:critical-bound}
For bounded $f$ and $d=2n$, the following are equivalent.
\begin{enumerate}[label=\textup{(\alph*)}]
\item $f\in\mathcal A_d$.
\item $|\nabla f|\in L^d(\R^d,dv)$.
\item $\sup_{0<t<t_0}m_{d,t}(f)<\infty$ for some $t_0>0$.
\end{enumerate}
Moreover, if $f$ is nonconstant and satisfies these conditions, then
$$
 0<c\norm{\nabla f}_{L^d(\R^d,dv)}
 \leq\liminf_{t\rightarrow0^{+}}m_{d,t}(f)
 \leq\sup_{t>0}m_{d,t}(f)
 \leq C\norm{\nabla f}_{L^d(\R^d,dv)}.
$$
The constants $c,C>0$ depend only on $n,m,M$ and the fixed
mollifier $\eta$ and lattice spacing $L$ used in
\eqref{eq:derivative-necessary}. They are independent of $f$ and $t_0$.
\end{corollary}
\begin{proof}
The equivalence \textup{(a)$\Leftrightarrow$(b)} follows from
Theorem \ref{thm:main} (b).
Under \textup{(b)}, we apply \eqref{eq:gradient-semicommutator}
at $p=d$ and take the supremum in \eqref{E5} to obtain
$$
 m_{d,t}(f)\leq C\norm{\nabla f}_{L^d(\R^d,dv)}
 \quad(t>0),
$$
where $C>0$ is the constant in that estimate.
This proves \textup{(b)$\Rightarrow$(c)} and the asserted upper bound.

To prove \textup{(c)$\Rightarrow$(b)}, choose $t_0>0$ as in
\textup{(c)} and set
$$
 a=\sup_{0<t<t_0}m_{d,t}(f)<\infty.
$$
Let $C_d>0$ be the constant in \eqref{eq:derivative-necessary}
at $p=d$, and let $C_{n,d}>0$ be the first constant in
\eqref{E14} with $q=d$. By Lemma \ref{L6},
$\overline{f}*\eta_r\in C^\infty(\R^d)\cap L^\infty(\R^d)$.
We may therefore combine \eqref{E14} and
\eqref{eq:derivative-necessary} at $p=d$ to obtain
$$
 \norm{\nabla(\overline{f}*\eta_r)}_{L^d(\R^d,dv)}
 \leq C_{n,d}\norm{\dbar(\overline{f}*\eta_r)}_{L^d(\R^d,dv)}
 \leq C_{n,d}C_d a
$$
for $0<r<\sqrt{t_0}$.
By \eqref{E66},
$\overline{f}*\eta_r\rightarrow\overline{f}$ in
$L^1_{\rm loc}(\R^d,dv)$ as $r\rightarrow0^{+}$.
The uniform $L^d(\R^d,dv)$ gradient bound permits the
weak compactness construction in \eqref{E150}, with
$u_r=\overline{f}*\eta_r$ and $r_0=\sqrt{t_0}$.
The local convergence in \eqref{E66} then allows us to use
the test-function calculation \eqref{E154}, which identifies
the weak limit as $\nabla\overline{f}$ and gives
$|\nabla f|\in L^d(\R^d,dv)$.
Thus, (b) holds, completing the equivalence.

It remains to prove the lower bound under \textup{(b)}.
Set $u_r=\overline{f}*\eta_r$.
Since $|\nabla\overline{f}|=|\nabla f|\in L^d(\R^d,dv)$,
we apply \eqref{E17} and the approximation argument leading
to \eqref{E64}, with $f_R$ replaced by $\overline{f}$ and $\varepsilon$ by $r$, to obtain
\begin{align*}
 \norm{\nabla u_r-\nabla\overline{f}}_{L^d(\R^d,dv)}
 &=\norm{(\nabla\overline{f})*\eta_r-\nabla\overline{f}}_{L^d(\R^d,dv)}\rightarrow0\quad(r\rightarrow0^{+}).
\end{align*}
By \eqref{E57},
$$
 \norm{\dbar u_r-\dbar\overline{f}}_{L^d(\R^d,dv)}
 \leq\frac1{\sqrt2}
      \norm{\nabla u_r-\nabla\overline{f}}_{L^d(\R^d,dv)}
 \longrightarrow0\quad(r\rightarrow0^{+}).
$$
Lemma \ref{L6} gives $u_r\in C^\infty(\R^d)\cap L^\infty(\R^d)$,
so \eqref{E14} applies to $u_r$. Combining this estimate with
\eqref{eq:derivative-necessary} at $p=d$ and passing to the limit, we obtain
\begin{align*}
 \norm{\nabla f}_{L^d(\R^d,dv)}
 &=\lim_{r\rightarrow0^{+}}\norm{\nabla u_r}_{L^d(\R^d,dv)}\leq C_{n,d}\lim_{r\rightarrow0^{+}}
                   \norm{\dbar u_r}_{L^d(\R^d,dv)}\\
 &=C_{n,d}\norm{\dbar\overline{f}}_{L^d(\R^d,dv)}
 \leq C_{n,d}C_d\liminf_{r\rightarrow0^{+}}m_{d,r^2}(f).
 \end{align*}
This proves the lower bound with $c=(C_{n,d}C_d)^{-1}>0$.
If $\norm{\nabla f}_{L^d(\R^d,dv)}=0$, then
$\nabla f=0$ almost everywhere, so \eqref{E3} gives
$\nabla f=0$ in distributions. Since $f$ is bounded,
\eqref{E151} with $v=f$ would imply that $f$ is constant
almost everywhere, contrary to the hypothesis.
This proves that $\norm{\nabla f}_{L^d(\R^d,dv)}>0$,
so the lower bound is strictly positive. The proof is complete.
\end{proof}
\begin{remark}\label{R1}
Under \eqref{eq:weight}, let $0<p\leq d$ and
$f\in L^\infty(\C^n)$.
Theorem \ref{thm:main} (a) for $p<d$ and
Corollary \ref{cor:critical-bound} for $p=d$ show that
$m_{p,t}(f)\to0$ as $t\to0^+$ if and only if
$f$ is constant almost everywhere.
\end{remark}

\subsection{H\"older sufficient conditions}\label{sec:super}
The criteria in Theorems \ref{thm:weakcriterion} and \ref{thm:uniform}
apply directly to compactly supported H\"older symbols.
Domination of the normalized densities
gives convergence for each fixed second symbol, and the
upper estimate in \eqref{E70} gives the uniform rate.

\begin{proof}[Proof of Theorem \ref{thm:holder}]
For each fixed $z\in\C^n$ and $0<s<1$, let
$h_z\in\mathcal O(B(z,s))$ be the constant function
$h_z(w)=\overline{f(z)}$ for $w\in B(z,s)$.
Using $h_z$ in \eqref{E106} and the H\"older bound
\eqref{E44}, we obtain
$$
 G_{1,s}(\overline{f})(z)
 \leq\frac1{|B(z,s)|}\int_{B(z,s)}|f(w)-f(z)| dv(w)
 \leq[f]_\alpha s^\alpha.
$$
To localize this estimate, let
$K=\{z\in\R^d:\operatorname{dist}(z,\supp f)\leq1\}$.
If $z\notin K$, then $$B(z,s)\cap\supp f=\varnothing \quad\text{and}\quad
G_{1,s}(\overline{f})(z)=0.$$ Combining this with the preceding
bound and \eqref{eq:energy-definitions}, we obtain
\begin{equation}\label{E125}
 G_{1,s}(\overline{f})\leq[f]_\alpha s^\alpha\mathbf1_K\quad\text{and}\quad
 0\leq e_s(f)\leq[f]_\alpha^p s^{p\alpha-d}\mathbf1_K
             \leq[f]_\alpha^p\mathbf1_K.
\end{equation}
Here $p\alpha-d\geq0$. Since $\supp f$ is compact,
$K$ is bounded and $|K|<\infty$, so
$$
 \norm{[f]_\alpha^p\mathbf1_K}_{L^1(\R^d,dv)}
 =[f]_\alpha^p|K|<\infty.
$$
The singleton family $\{[f]_\alpha^p\mathbf1_K\}$ satisfies
\eqref{eq:intrinsic-three} and hence belongs to $\mathcal W$.
Using the domination in \eqref{E125}, we apply
Lemma \ref{L9} with $N=1$ and $C=1$ to obtain
$\{e_s(f):0<s<1\}\in\mathcal W$.
Theorem \ref{thm:weakcriterion} now implies $f\in\mathcal A_p$.

Suppose $\alpha>d/p$, and let $C_0>0$ be the constant
in \eqref{E70}. For $0<r<1/4$, that estimate and
\eqref{E125} with $s=4r$ give
\begin{align}\label{E157}
 m_{p,r^2}(f)
 &\leq C_0r^{-d/p}\norm{G_{1,4r}(\overline{f})}_{L^p(\R^d,dv)}\leq C_0 4^\alpha|K|^{1/p}[f]_\alpha r^{\alpha-d/p}.
 \end{align}
For $r\geq1/4$, let $C_1>0$ be the Hankel constant in
\eqref{eq:multiplier}. The boundedness and compact support of $f$ imply that
$\overline{f}\in L^p(\R^d,dv)\cap L^\infty(\C^n)$.
For each $g\in L^\infty(\C^n)$, the factorization
\eqref{eq:factorization} and the ideal inequality \eqref{E7} give
\begin{align*}
 \norm{D_{r^2}(f,g)}_{\mathcal S_p(\mathcal F_{r^2}^2)}
 &\leq
 \norm{H_{\overline{f}}^{(r^2)}}_{\mathcal S_p(\mathcal F_{r^2}^2,L_{r^2}^2\ominus\mathcal F_{r^2}^2)}
 \norm{H_g^{(r^2)}}\leq C_1r^{-d/p}\norm{f}_{L^p(\R^d,dv)}
       \norm{g}_{L^\infty(\C^n)}.
\end{align*}
The second inequality uses \eqref{eq:multiplier} for
$\overline{f}$ and the Hankel norm bound in \eqref{E1}.
Taking the supremum in \eqref{E5} and using
$1\leq4^\alpha r^\alpha$ for $r\geq1/4$, we obtain
\begin{equation}\label{E158}
 m_{p,r^2}(f)
 \leq C_1r^{-d/p}\norm{f}_{L^p(\R^d,dv)}
 \leq C_1 4^\alpha\norm{f}_{L^p(\R^d,dv)}r^{\alpha-d/p}.
\end{equation}
Since both $C_0$ and $C_1$ depend only on $n,m,M,p$ and
$\norm{f}_{L^p(\R^d,dv)}
\leq|K|^{1/p}\norm{f}_{L^\infty(\C^n)}$, we may take
$$
 C_f=4^\alpha\max\{1,C_0|K|^{1/p}[f]_\alpha,
                         C_1\norm{f}_{L^p(\R^d,dv)}\}>0.
$$
Combining \eqref{E157} and \eqref{E158} with $r=\sqrt t$ gives
$m_{p,t}(f)\leq C_ft^{\alpha/2-n/p}$ for every $t>0$,
with the stated dependence of $C_f$.
The exponent is positive, so the bound tends to zero
as $t\rightarrow0^{+}$. This proves uniform convergence and completes the proof.
\end{proof}

\section{Examples and concentration}\label{sec:examples}
For $2\leq p<\infty$, Theorem \ref{thm:weakcriterion}
characterizes convergence as $t\rightarrow0^{+}$ for every
fixed bounded second symbol by relative weak compactness of
the normalized oscillation densities. Corollary \ref{cor:uniform-real}
characterizes uniform convergence over the $L^\infty(\C^n)$
unit ball by their $L^1(\R^d,dv)$ norms tending to zero as
$r\rightarrow0^{+}$.
The examples examine the roles of these conditions.
A H\"older symbol shows that the gradient condition is not
necessary when $p>d$, and a Gaussian calculation distinguishes
the conjugate test from testing every bounded second symbol
when $n<p<d$. For $p>d$, two concentration examples show
that bounded density mass is insufficient, first for a
discontinuous symbol and then for the continuous symbol in
Proposition \ref{prop:continuous-concentration}, for which
convergence to zero holds for the conjugate test.

\subsection{A H\"older symbol without the gradient condition}
For $p>d$, the H\"older condition in Theorem \ref{thm:holder}
need not imply $|\nabla f|\in L^p(\R^d,dv)$.

\begin{example}\label{ex:lacunary}
For every $d<p<\infty$, there is a compactly supported symbol
$f\in\mathcal A_p$ whose distributional derivative
$\partial_{x_1}f$ cannot be represented by a function in
$L^p(\R^d,dv)$.
\end{example}
\begin{proof}
Put $\alpha=d/p\in(0,1)$ and define the $2\pi$-periodic function
\begin{equation}\label{E37}
 u(s)=\sum_{j=1}^\infty2^{-j\alpha}\cos(2^j s),
 \quad s\in\R.
\end{equation}
The series is absolutely and uniformly convergent. We verify its
H\"older regularity directly. For $\tau\in\R$ with $0<|\tau|\leq\tfrac12$, choose
$J\geq1$ so that $2^{-J-1}<|\tau|\leq2^{-J}$. The bounds
\mbox{$|\cos(s+\tau)-\cos s|\leq\min\{|\tau|,2\}$} give
\begin{align*}
 |u(s+\tau)-u(s)|
 &\leq |\tau|\sum_{j=1}^J2^{j(1-\alpha)}
                 +2\sum_{j>J}2^{-j\alpha}\leq C_\alpha\bigl(|\tau|2^{J(1-\alpha)}
                                      +2^{-J\alpha}\bigr)\leq C_\alpha'|\tau|^\alpha.
 \end{align*}
Note that the constants $C_\alpha,C_\alpha'>0$ depend only on $\alpha$.
For $|\tau|>\tfrac12$, boundedness gives
$$
 |u(s+\tau)-u(s)|
 \leq2\norm{u}_{L^\infty(\R)}
 \leq2^{1+\alpha}\norm{u}_{L^\infty(\R)}|\tau|^\alpha.
$$
Increasing $C_\alpha'$ if necessary gives
\begin{equation}\label{E174}
 |u(s+\tau)-u(s)|\leq C_\alpha'|\tau|^\alpha
\end{equation}
for all $s,\tau\in\R$.

Choose $\chi\in C_c^\infty(\R^d)$ such that
\begin{equation}\label{E175}
 0\leq\chi\leq1,\quad
 \chi=1\text{ on }[-3\pi,3\pi]^d,\quad
 \supp\chi\subset(-4\pi,4\pi)^d.
\end{equation}
This choice gives $\norm{\chi}_{L^\infty(\R^d)}=1$ and
$\norm{\nabla\chi}_{L^\infty(\R^d)}<\infty$.
Set $f(x)=\chi(x)u(x_1)$ for
$x=(x_1,\ldots,x_d)\in\R^d$.
For $x,y\in\R^d$, the fundamental theorem of calculus gives
\begin{align*}
 |\chi(x)-\chi(y)|
 &\leq |x-y|\int_0^1|\nabla\chi(y+t(x-y))|dt\leq\norm{\nabla\chi}_{L^\infty(\R^d)}|x-y|.
\end{align*}
The bound $0\leq\chi\leq1$ also gives
$|\chi(x)-\chi(y)|\leq1$.
With
$
 C_\chi=\max\{\norm{\nabla\chi}_{L^\infty(\R^d)},1\}>0
$,
these bounds imply
$|\chi(x)-\chi(y)|\leq C_\chi\min\{|x-y|,1\}$.
Using this estimate, \eqref{E174}, and
$\min\{|x-y|,1\}\leq|x-y|^\alpha$, we obtain
\begin{align*}
 |f(x)-f(y)|
 &\leq\norm{\chi}_{L^\infty(\R^d)}|u(x_1)-u(y_1)|
       +\norm{u}_{L^\infty(\R)}|\chi(x)-\chi(y)|\\
 &\leq\bigl(\norm{\chi}_{L^\infty(\R^d)}C_\alpha'
       +C_\chi\norm{u}_{L^\infty(\R)}\bigr)|x-y|^\alpha.
 \end{align*}
Then $f\in C_c^{0,\alpha}(\R^d)$, and Theorem
\ref{thm:holder} gives $f\in\mathcal A_p$.

Recall from \eqref{E160} that an $L^p$ representative of
the distributional derivative $\partial_{x_1}f$ is the weak
derivative of $f$ in $x_1$. Suppose, to obtain a contradiction,
that $f$ has a weak partial derivative with respect to $x_1$
and that $\partial_{x_1}f\in L^p(\R^d,dv)$.
Let $ds$ and $dy$ denote Lebesgue measure on $\R$ and
$\R^{d-1}$, respectively. To recover a weak derivative of $u$
on $(-2\pi,2\pi)$, take $\theta\in C_c^\infty((-\pi,\pi)^{d-1})$
with $\int_{\R^{d-1}}\theta(y)dy=1$.
By Fubini's theorem, we define $g$ almost everywhere on
$(-2\pi,2\pi)$ by
\begin{equation}\label{E176}
 g(s)=\int_{\R^{d-1}}\partial_{x_1}f(s,y)\theta(y)dy,
 \quad |s|<2\pi.
\end{equation}
With $p'=p/(p-1)$, H\"older's inequality in $y$ gives
\begin{equation}\label{E161}
 \norm{g}_{L^p((-2\pi,2\pi),ds)}
 \leq\norm{\theta}_{L^{p'}(\R^{d-1},dy)}
     \norm{\partial_{x_1}f}_{L^p((-2\pi,2\pi)\times\supp\theta,dv)}
 <\infty.
\end{equation}

To verify that $g$ is the weak derivative of $u$ on
$(-2\pi,2\pi)$, fix $\psi\in C_c^\infty(-2\pi,2\pi)$ and use
its zero extension to $\R$ defined by \eqref{E93}.
The test function
$(s,y)\mapsto\psi(s)\theta(y)$ belongs to $C_c^\infty(\R^d)$
and satisfies
$\partial_{x_1}(\psi(s)\theta(y))=\psi'(s)\theta(y)$.
Since $f$ and its weak derivative $\partial_{x_1}f$ are
locally integrable, Fubini's theorem applies to the compactly
supported integrands below. Using \eqref{E176},
we apply \eqref{E3} with this test function to obtain
\begin{align*}
 \int_{-2\pi}^{2\pi}g(s)\psi(s)ds
 &=\int_{-2\pi}^{2\pi}\int_{\R^{d-1}}
       \partial_{x_1}f(s,y)\psi(s)\theta(y)dyds\\
 &=-\int_{-2\pi}^{2\pi}\int_{\R^{d-1}}
       f(s,y)\psi'(s)\theta(y)dyds\\
 &=-\int_{-2\pi}^{2\pi}u(s)\psi'(s)ds,
\end{align*}
where we use $f(s,y)=u(s)$ on the support of
$\psi'(s)\theta(y)$, since \eqref{E175} gives $\chi=1$ there, and
$\int_{\R^{d-1}}\theta(y)dy=1$ in the final step.
Since $\psi$ is arbitrary, \eqref{E3} identifies $g$ as the
weak derivative of $u$ on $(-2\pi,2\pi)$.

The boundedness of $u$ and \eqref{E161} give
$u,g\in L^p((-2\pi,2\pi),ds)$.
Since $u$ is continuous,
\cite[Theorem 8.2, p. 204]{Brezis2011} gives
$$
 u(s)-u(0)=\int_0^s g(\xi)d\xi,\quad |s|<2\pi.
$$
To deduce absolute continuity on $[-\pi,\pi]$ from this
representation, we use \eqref{E161} and H\"older's inequality
to obtain
\begin{equation}\label{E167}
 g\in L^1((-\pi,\pi),ds).
\end{equation}
The fundamental theorem for Lebesgue integrals
\cite[Theorem 3.35, p. 106]{Folland1999} now shows that $u$
is absolutely continuous on $[-\pi,\pi]$, with classical
derivative $u'(s)=g(s)$ for almost every $s\in(-\pi,\pi)$.

For $h\in L^1((-\pi,\pi),ds)$, define
\begin{equation}\label{E177}
 \widehat h(k)=\frac1{2\pi}\int_{-\pi}^{\pi}h(s)e^{-iks}ds,
 \quad k\in\Z.
\end{equation}
For $I=(-\pi,\pi)$, we write
$W^{1,p}(I)=\{v\in L^p(I,ds):v'\in L^p(I,ds)\}$,
where $v'$ denotes the weak derivative in the sense of
\eqref{E3}, with integrals over $I$ and test functions in
$C_c^\infty(I)$.
The boundedness of $u$, \eqref{E161}, and
$e^{-iks}\in C^\infty(\R)$ give
$u,\Re(e^{-iks}),\Im(e^{-iks})\in W^{1,p}(I)$.
We apply the integration by parts formula
\cite[Corollary 8.10, p. 215]{Brezis2011} on $\overline I$
to $u$ and each of the real and imaginary parts of $e^{-iks}$.
Combining the two identities by linearity gives
\begin{equation}\label{E162}
 \begin{aligned}
 \widehat{u'}(k)
 &=\frac1{2\pi}\bigl[u(s)e^{-iks}\bigr]_{-\pi}^{\pi}
   +\frac{ik}{2\pi}\int_{-\pi}^{\pi}u(s)e^{-iks}ds=ik\widehat u(k),\quad k\in\Z.
 \end{aligned}
\end{equation}
The boundary term vanishes in the last equality because
$u(-\pi)=u(\pi)$ and $e^{-ik\pi}=e^{ik\pi}$ for $k\in\Z$.
The uniform convergence in \eqref{E37} permits termwise
integration. For $j\geq1$, expanding each cosine gives
\begin{align*}
 \widehat u(2^j)
 &=\sum_{\ell=1}^{\infty}\frac{2^{-\ell\alpha}}{4\pi}
   \int_{-\pi}^{\pi}\bigl(e^{i(2^\ell-2^j)s}
                         +e^{-i(2^\ell+2^j)s}\bigr)ds.
\end{align*}
Using the orthogonality of integer frequencies on $[-\pi,\pi]$
and \eqref{E162} with $k=2^j$, we obtain
$$
 \widehat u(2^j)=2^{-j\alpha-1}\quad\text{and}\quad
 |\widehat{u'}(2^j)|=2^{j(1-\alpha)-1}.
$$
Since $u'=g$ almost everywhere on $(-\pi,\pi)$,
\eqref{E167} gives $u'\in L^1((-\pi,\pi),ds)$.
By \eqref{E177},
$$
 2^{j(1-\alpha)-1}
 =|\widehat{u'}(2^j)|
 \leq\frac1{2\pi}\norm{u'}_{L^1((-\pi,\pi),ds)}<\infty
 \quad(j\geq1),
$$
which is impossible as $j\rightarrow\infty$ because $\alpha<1$.
This rules out an $L^p(\R^d,dv)$ representative of
$\partial_{x_1}f$ and completes the proof.
\end{proof}

\subsection{A Gaussian test in every Schatten class}\label{sec:gaussian}
The following calculation compares the test $g=\overline{f}$ with
the condition for every bounded second symbol. In this subsection, take
$$
 \varphi(z)=\tfrac12|z|^2,\quad f(z)=e^{-|z|^2}.
$$
The calculation uses nonnegative scalar series and is valid
for every $p>0$, including $0<p<1$.

Write $\N_0=\{0,1,2,\ldots\}$. For a multiindex
$\beta=(\beta_1,\ldots,\beta_n)\in\N_0^n$, set
$$
 |\beta|=\sum_{j=1}^n\beta_j,\quad
 \beta!=\prod_{j=1}^n\beta_j!,\quad
 z^\beta=\prod_{j=1}^nz_j^{\beta_j}.
$$
The monomial basis in \cite[Section 2.1, (2)]{KellerLuef2021},
with complex dimension $n$ and $\hbar=t/2$, is orthonormal for
the measure $e^{-|z|^2/t}dv$.
Multiplying its elements by $t^{n/2}$ accounts for the factor
$t^{-n}$ in \eqref{E82} and gives an orthonormal basis
$(e_{\beta,t})_{\beta\in\N_0^n}$ of $\mathcal F_t^2$, with
monomial norms and basis elements given by
$$
 \norm{z^\beta}_{L_t^2}^2=\pi^n\beta!t^{|\beta|}\quad\text{and}\quad
 e_{\beta,t}(z)=\frac{z^\beta}
 {\pi^{n/2}\sqrt{\beta!}t^{|\beta|/2}}.
$$
For $c>0$, the radial symbol $e^{-c|z|^2}$ gives a diagonal
Toeplitz operator in this basis. Indeed, angular integration
makes the off-diagonal matrix entries zero, and the diagonal entry
is
\begin{align*}
 \langle T_{e^{-c|z|^2}}^{(t)}e_{\beta,t},e_{\beta,t}\rangle_t
 &=\frac{t^{-n}}{\pi^n\beta!t^{|\beta|}}
   \int_{\C^n}|z^\beta|^2e^{-(c+1/t)|z|^2} dv(z)\\
 &=\frac{t^{-n}}{\pi^n\beta!t^{|\beta|}}
   \pi^n\beta!(c+1/t)^{-(|\beta|+n)}
 =(1+ct)^{-(|\beta|+n)}.
 \end{align*}
The second equality uses the monomial norm formula with $t$
replaced by $(c+1/t)^{-1}$.
For $k=|\beta|+n$, \eqref{E2} and the diagonal formula
with $c=1,2$ give
\begin{align*}
 D_t(f,f)e_{\beta,t}
 &=\bigl((T_f^{(t)})^2-T_{f^2}^{(t)}\bigr)e_{\beta,t}=\bigl((1+t)^{-2k}-(1+2t)^{-k}\bigr)e_{\beta,t}.
\end{align*}
Since $(1+t)^2>1+2t$ for $t>0$, we have $D_t(f,f)\leq0$.
Its eigenvalue at $e_{\beta,t}$ is $-\lambda_k(t)$, whose
absolute value is
\begin{equation}\label{E27}
 \lambda_k(t)=(1+2t)^{-k}-(1+t)^{-2k}>0,
 \quad k=n,n+1,\ldots.
\end{equation}
The number of multiindices with $|\beta|=k-n$ is
$\binom{k-1}{n-1}$. For fixed $t>0$, the entries tend to zero
as $|\beta|\rightarrow\infty$, so $D_t(f,f)$ is a compact diagonal
operator. Its singular values are the numbers $\lambda_k(t)$,
with these multiplicities, and their order does not affect
the nonnegative sum. The formula for diagonal operators
\cite[Section 1.6, Exercise 6]{Zhu1990} gives
\begin{equation}\label{eq:gaussian-series}
 \norm{D_t(f,f)}_{\mathcal S_p}^p
 =\sum_{k=n}^\infty\binom{k-1}{n-1}\lambda_k^p(t).
\end{equation}
For fixed $t>0$, the bounds
$\binom{k-1}{n-1}\leq k^{n-1}/(n-1)!$ and
$0\leq\lambda_k(t)\leq(1+2t)^{-k}$ give
$$
 \sum_{k=n}^{\infty}\binom{k-1}{n-1}\lambda_k^p(t)
 \leq\frac1{(n-1)!}
       \sum_{k=n}^{\infty}k^{n-1}(1+2t)^{-kp}<\infty,
 \quad p>0.
$$
Together with \eqref{eq:gaussian-series}, this bound shows that
$D_t(f,f)$ belongs to $\mathcal S_p(\cF)$ for every $p>0$
before any limit is taken.

The series in \eqref{eq:gaussian-series} gives the leading term
of the Schatten norm as $t$ tends to zero.
We use the Gamma function
$\Gamma(s)=\int_0^\infty u^{s-1}e^{-u}du$ for $s>0$,
with $\Gamma(n)=(n-1)!$ for positive integers $n$
\cite[(5.2.1) and (5.4.1)]{OlverEtAl2010}.

\begin{proposition}\label{prop:gaussian}
For every $p>0$, the Gaussian symbol $f(z)=e^{-|z|^2}$ satisfies
\begin{equation}\label{eq:gaussian-asymptotic}
 \norm{D_t(f,f)}_{\mathcal S_p}^p
 =\frac{\Gamma(n+p)}{\Gamma(n)(2p)^{n+p}}
       t^{p-n}\bigl(1+o(1)\bigr)
 \qquad(t\rightarrow0^{+}).
\end{equation}
In particular, $\norm{D_t(f,f)}_{\mathcal S_p}$ tends to zero
if and only if $p>n$.
\end{proposition}
\begin{proof}
Throughout this proof, $n,p$ are fixed and limits without an
indicated variable are taken as $t\rightarrow0^{+}$.
We normalize \eqref{eq:gaussian-series} to obtain a Riemann sum
with spacing $t$. For this, set
$$
 a_t=\log(1+2t)\quad\text{and}\quad
 c_t=2\log(1+t)-\log(1+2t).
$$
Equation \eqref{E27} gives
\begin{equation}\label{E178}
 \lambda_k(t)=e^{-ka_t}\bigl(1-e^{-kc_t}\bigr).
\end{equation}
For grid points $x=kt$ in a fixed interval
$[a,R]\subset(0,\infty)$, we have
$k\geq a/t\rightarrow\infty$ as $t\rightarrow0^+$.
Since $c_t=\log(1+t^2/(1+2t))>0$ and
$\log(1+s)/s\rightarrow1$ as $s\rightarrow0$, we have
$a_t/t\rightarrow2$ and $c_t/t^2\rightarrow1$.
Substituting $x=kt$ in \eqref{E178} and using these limits, we obtain
\begin{equation}\label{E96}
 \frac{\lambda_k(t)}{t}
 =e^{-x a_t/t}
   x\frac{c_t}{t^2}
   \frac{1-e^{-x c_t/t}}{x c_t/t}
 \rightarrow xe^{-2x}\quad(t\rightarrow0^+)
\end{equation}
uniformly in the grid points $x=kt$. The quotient in the last
factor is assigned its continuous value one at zero.

To control the normalized series outside compact intervals,
choose constants $c,C>0$, independent of $t$ and $k$,
such that $a_t\geq ct$ and $0<c_t\leq Ct^2$ whenever $t>0$
is sufficiently small.
By \eqref{E178} and the scalar inequality
$1-e^{-y}\leq y$ for $y\geq0$,
\begin{equation}\label{eq:gaussian-domination}
 0\leq\frac{\lambda_k(t)}{t}
 \leq C(kt)e^{-c kt}.
\end{equation}
For the multiplicities in \eqref{eq:gaussian-series}, interpreting
the empty product as one when $n=1$, we have
\begin{equation}\label{E38}
 t^{n-1}\binom{k-1}{n-1}
 =\frac{\prod_{\ell=1}^{n-1}(kt-\ell t)}{(n-1)!}
 \leq\frac{(kt)^{n-1}}{\Gamma(n)}.
\end{equation}
On compact intervals of $x=kt$, this expression converges
uniformly to $x^{n-1}/\Gamma(n)$. Multiplying
\eqref{eq:gaussian-series} by $t^{n-p}$ gives
\begin{equation}\label{E39}
 t^{n-p}\norm{D_t(f,f)}_{\mathcal S_p}^p
 =t\sum_{k=n}^\infty
  \left[t^{n-1}\binom{k-1}{n-1}\right]
  \left[\frac{\lambda_k(t)}t\right]^p.
\end{equation}
For $k\geq n$ and $t>0$, put
\begin{equation}\label{E179}
 b_{k,t}=t^{n-1}\binom{k-1}{n-1}
             \left(\frac{\lambda_k(t)}t\right)^p,
 \quad h(x)=\frac{x^{n+p-1}e^{-2px}}{\Gamma(n)},\quad x>0.
\end{equation}
For fixed $0<a<R<\infty$, the uniform convergence in
\eqref{E96} and that of the multiplicity factor in \eqref{E38},
together with the uniform continuity of $y\mapsto y^p$ on bounded
nonnegative intervals, also for $0<p<1$, give
$$
 \sup_{\substack{k\geq n\\a\leq kt\leq R}}
       |b_{k,t}-h(kt)|\rightarrow0.
$$
Since $t\times\#\{k\geq n:a\leq kt\leq R\}\leq R-a+2t$, the
corresponding part of \eqref{E39} satisfies
\begin{equation}\label{E120}
\begin{aligned}
 \bigg|t\sum_{\substack{k\geq n\\a\leq kt\leq R}}b_{k,t}
       -\int_a^R h(x)dx\bigg|
 &\leq(R-a+2t)\sup_{\substack{k\geq n\\a\leq kt\leq R}}
                      |b_{k,t}-h(kt)|\\
 &\quad+\bigg|t\sum_{\substack{k\geq n\\a\leq kt\leq R}}h(kt)
              -\int_a^R h(x)dx\bigg|
 \rightarrow0.
 \end{aligned}
\end{equation}
Since $h$ is continuous on $[a,R]$, the last term tends to
zero by Riemann sum convergence as $t\rightarrow0^+$.

To pass from compact intervals to $(0,\infty)$, put
$\gamma=n+p-1>0$. By \eqref{eq:gaussian-domination}
and \eqref{E38},
\begin{equation}\label{E163}
 0\leq b_{k,t}\leq C_{n,p}(kt)^\gamma e^{-cpkt},
 \quad C_{n,p}=C^p/\Gamma(n)>0,
\end{equation}
where $C$ is the constant in \eqref{eq:gaussian-domination},
so $C_{n,p}$ is independent of $t,k,a,R$.
Since $\xi\mapsto\xi^\gamma$ is increasing on $(0,\infty)$,
each term with $kt<a$ satisfies
$t(kt)^\gamma\leq\int_{kt}^{(k+1)t}\xi^\gamma d\xi$. The intervals
are disjoint, and hence
$$
 t\sum_{n\leq k<a/t}(kt)^\gamma
 \leq\int_0^{a+t}\xi^\gamma d\xi
 =\frac{(a+t)^{\gamma+1}}{\gamma+1}.
$$
Combining this integral bound with \eqref{E163}, we obtain
\begin{equation}\label{E121}
 \limsup_{t\rightarrow0^{+}}
 t\sum_{\substack{k\geq n\\kt<a}}b_{k,t}
 \leq\frac{C_{n,p}a^{n+p}}{n+p}
 \longrightarrow0\quad(a\rightarrow0^{+}).
\end{equation}
Choose $R>1+\gamma/(cp)$, so $\xi\mapsto\xi^\gamma e^{-cp\xi}$ is decreasing
on $[R-1,\infty)$. If $0<t<1$ and $kt>R$, monotonicity gives
$$
 t(kt)^\gamma e^{-cpkt}
 \leq\int_{(k-1)t}^{kt}\xi^\gamma e^{-cp\xi}d\xi.
$$
Summing over these disjoint intervals in $[R-1,\infty)$
and using \eqref{E163}, we obtain
\begin{equation}\label{E122}
 t\sum_{\substack{k\geq n\\kt>R}}b_{k,t}
 \leq C_{n,p}\int_{R-1}^\infty \xi^\gamma e^{-cp\xi}d\xi
 \longrightarrow0\quad(R\rightarrow\infty),
\end{equation}
uniformly for all sufficiently small $0<t<1$.

Splitting the sum in \eqref{E39} into $kt<a$, $a\leq kt\leq R$,
and $kt>R$, we combine \eqref{E120}, \eqref{E121}, and \eqref{E122} to obtain
\begin{align*}
 \limsup_{t\rightarrow0^{+}}
 \left|t^{n-p}\norm{D_t(f,f)}_{\mathcal S_p}^p
       -\int_0^\infty h(\xi)d\xi\right|
 &\leq \frac{C_{n,p}a^{n+p}}{n+p}
       +C_{n,p}\int_{R-1}^\infty \xi^\gamma e^{-cp\xi}d\xi\\
 &\quad+\int_0^a h(\xi)d\xi+\int_R^\infty h(\xi)d\xi.
\end{align*}
Since the function $h$ in \eqref{E179} belongs to
$L^1((0,\infty),d\xi)$ by the Gamma integral,
the right side tends to zero as $a\rightarrow0^{+}$
and $R\rightarrow\infty$. We then have
\begin{align*}
 \lim_{t\rightarrow0^{+}}t^{n-p}
                   \norm{D_t(f,f)}_{\mathcal S_p}^p
 &=\frac1{\Gamma(n)}\int_0^\infty
                          \xi^{n+p-1}e^{-2p\xi}d\xi=\frac{\Gamma(n+p)}{\Gamma(n)(2p)^{n+p}}.
 \end{align*} Since the constant is strictly positive, taking
$p$-th roots shows that the norm diverges for $p<n$, has a
positive finite limit for $p=n$, and tends to zero for $p>n$. The proof is complete.
\end{proof}

When $n<p<2n$, Proposition \ref{prop:gaussian} gives
$\norm{D_t(f,f)}_{\mathcal S_p}\rightarrow0$, although
Theorem \ref{thm:main} (a) shows that $f\notin\mathcal A_p$.
By \eqref{eq:property}, the limit fails for some fixed bounded
second symbol.
The identity $|\nabla f(x)|=2|x|e^{-|x|^2}$ and polar coordinates
give $|\nabla f|\in L^q(\R^d,dv)$ for every $1\leq q<\infty$.
For $p=2n$, Theorem \ref{thm:main} (b) therefore gives
$f\in\mathcal A_p$. For $p>2n$, \eqref{eq:super-gradient} gives
convergence uniformly over
$g\in L^\infty(\C^n)$ with $\norm{g}_{L^\infty(\C^n)}\leq1$. At $p=2$, the constant in
\eqref{eq:gaussian-asymptotic} reduces to
$$
 \frac{\Gamma(n+2)}{\Gamma(n)4^{n+2}}
 =\frac{n(n+1)}{16\cdot4^n},
$$
which is the coefficient of $t^{2-n}$ in the asymptotic formula for
$\norm{D_t(f,f)}_{\mathcal S_2}^2$ as $t\rightarrow0^{+}$.

\subsection{Bounded mass and a discontinuous symbol}
We next examine whether bounded density mass suffices for
membership in $\mathcal A_p$. Lemma \ref{lem:baire} and the
lower bound in \eqref{eq:two-sided} show that uniform boundedness of
$r^{-d/p}\norm{G_{1,r}(\overline{f})}_{L^p(\R^d,dv)}$
for all sufficiently small $r>0$ is necessary.
At $p=d$, the corresponding upper bound and Corollary
\ref{cor:critical-bound} give sufficiency, but the following
discontinuous example shows that this fails for $p>d$.

\begin{example}\label{ex:concentration}
Suppose $p>d$. There are a real-valued bounded compactly supported
function $f$ and a number $t_0>0$ such that
$$
 \sup_{0<t<t_0}m_{p,t}(f)<\infty
 \quad\text{but}\quad
 \liminf_{t\rightarrow0^{+}}\norm{D_t(f,f)}_{\mathcal S_p}>0.
$$
This example works for every weight satisfying \eqref{eq:weight}.
\end{example}
\begin{proof}
In real coordinates, we define $u(0)=0$ and 
$u(x)=x_1/|x|$ for $x\neq0$.
Choose a real $\chi_0\in C_c^\infty(\R^d)$ equal to one near
zero, and set $f=\chi_0u$. The positive constants $C,C'$ in this
proof may depend on $n,m,M,p$ and the fixed cutoff $\chi_0$,
but are independent of $r$, and $C$ may increase between estimates.
The function $f$ is bounded and
compactly supported. We estimate its local analytic distance
to bound $m_{p,r^2}(f)$ by \eqref{eq:two-sided}.
On $\R^d\setminus\{0\}$, both $u$ and $f$ are smooth,
and the classical derivatives of $u$ satisfy
\begin{equation}\label{E180}
 \partial_{x_j}u(x)=\frac{\delta_{1j}}{|x|}
                     -\frac{x_1x_j}{|x|^3},
 \quad |\nabla u(x)|\leq\frac{C_d}{|x|},\quad x\neq0
\end{equation}
for $1\leq j\leq d$.
Here $C_d>0$ depends only on $d$.
Choose a fixed $R_0>3$ large enough that
$\supp f\subset B(0,R_0-1)$. For $x\neq0$, the product rule gives
$$
 \nabla f(x)=u(x)\nabla\chi_0(x)+\chi_0(x)\nabla u(x).
$$
Using $|u|\leq1$ and the bound for $|\nabla u|$ in
\eqref{E180}, we obtain
\begin{equation}\label{E181}
\begin{aligned}
 |\nabla f(x)|
 &\leq\norm{\nabla\chi_0}_{L^\infty(\R^d)}
       +\frac{C_d\norm{\chi_0}_{L^\infty(\R^d)}}{|x|}\leq\frac{C}{|x|},\quad 0<|x|<R_0+1,
\end{aligned}
\end{equation}
where we may take
$$
 C=(R_0+1)\norm{\nabla\chi_0}_{L^\infty(\R^d)}
   +C_d\norm{\chi_0}_{L^\infty(\R^d)}.
$$
Since $R_0$ and $\chi_0$ are fixed, this choice is independent of $r$.

Fix $0<r<1$. If $2r<|x|<R_0$ and $y\in B(x,r)$, then
$$
 |x+s(y-x)|\geq|x|-|y-x|>|x|-r>|x|/2,\qquad 0\leq s\leq1.
$$
The segment lies in $B(0,R_0+1)\setminus\{0\}$, so the gradient
estimate in \eqref{E181} and the fundamental theorem of calculus give
$$
 |f(y)-f(x)|
 \leq |y-x|\int_0^1|\nabla f(x+s(y-x))|ds
 \leq\frac{Cr}{|x|}.
$$
If $|x|\geq R_0$, then $B(x,r)\cap\supp f=\varnothing$
because $r<1$ and $\supp f\subset B(0,R_0-1)$.
Using the constant holomorphic function $w\mapsto f(x)$
on $B(x,r)$ when $2r<|x|<R_0$, and zero otherwise, we
obtain from \eqref{E106} that
\begin{equation}\label{E168}
 0\leq G_{1,r}(f)(x)\leq
 \begin{cases}
  \norm{f}_{L^\infty(\R^d)},&|x|\leq2r,\\
  Cr/|x|,&2r<|x|<R_0,\\
  0,&|x|\geq R_0,
 \end{cases}
 \quad 0<r<1.
\end{equation}
Integrating the $p$-th power of \eqref{E168} and using polar
coordinates, we obtain
\begin{equation}\label{E164}
\begin{aligned}
 \norm{G_{1,r}(f)}_{L^{p}(\R^d,dv)}^p
 &\leq\norm{f}_{L^\infty(\R^d)}^p|B(0,2r)|
       +C^p r^p\int_{\{2r<|x|<R_0\}}|x|^{-p}dv(x)\\
 &\leq C r^d+C r^p\int_{2r}^{R_0}s^{d-p-1}ds\\
 &=C r^d+\frac{C r^p}{p-d}
                         \bigl((2r)^{d-p}-R_0^{d-p}\bigr)\leq C' r^d.
 \end{aligned}
\end{equation}
The last inequality uses $p>d$ and
$r^p(2r)^{d-p}=2^{d-p}r^d$.
Let $R$ be the fixed radius in \eqref{eq:two-sided}.
For $0<r<1/R$, we combine the upper bound in \eqref{eq:two-sided}
with \eqref{E164} at radius $Rr$, using $\overline{f}=f$, to obtain
$$
 m_{p,r^2}(f)\leq c_pr^{-d/p}\norm{G_{1,Rr}(f)}_{L^{p}(\R^d,dv)}
 \leq c_p(C')^{1/p}R^{d/p},\quad 0<r<1/R.
$$
Here $c_p>0$ is the upper bound constant in \eqref{eq:two-sided},
and $C'>0$ is the constant in \eqref{E164}.
This proves the required uniform bound with $t_0=R^{-2}>0$.

To prove failure of the limit, put $f_r(\zeta)=f(r\zeta)$.
Homogeneity gives
$f_r(\zeta)=\chi_0(r\zeta)u(\zeta)\rightarrow u(\zeta)$
for every fixed $\zeta\neq0$ as $r\rightarrow0^{+}$.
Let $k_0=k_{1,0}$ be the normalized reproducing kernel at zero in
$\mathcal F_1^2$, as defined before \eqref{E10}.
Since $|u|\leq1$ and
$\sup_{r>0}\norm{f_r}_{L^\infty(\R^d)}
\leq\norm{\chi_0}_{L^\infty(\R^d)}$, we have
$$
 |(f_r-u)k_0|^2
 \leq\bigl(\norm{\chi_0}_{L^\infty(\R^d)}+1\bigr)^2|k_0|^2.
$$
The right side is integrable because
$\int_{\C^n}|k_0|^2d\mu_1=1$.
Moreover, the definition of the Hankel operators and
$\norm{I-P}\leq1$ give
$$
 \norm{H_{f_r}^{(1)}k_0-H_u^{(1)}k_0}_{L_1^2}
 =\norm{(I-P)((f_r-u)k_0)}_{L_1^2}
 \leq\norm{(f_r-u)k_0}_{L_1^2}.
$$
Since $f_r\to u$ almost everywhere, dominated convergence
and the preceding estimate yield
\begin{equation}\label{E95}
 f_r k_0\rightarrow uk_0\quad\text{in }L_1^2
 \quad\text{and}\quad
 H_{f_r}^{(1)}k_0\rightarrow H_u^{(1)}k_0
                                  \quad\text{in }L_1^2
\end{equation}
as $r\rightarrow0^+$.
To obtain a positive lower limit from this convergence, we
show that $H_u^{(1)}k_0\neq0$. The lower bound \eqref{E105}, with $r=1$ and $a=0$, shows that $k_0$ has no
zeros on $B(0,\rho)$. If $H_u^{(1)}k_0=0$, then
$$
 h=\frac{P(uk_0)}{k_0}\in\mathcal O(B(0,\rho))
 \quad\text{and}\quad h=u\quad\text{almost everywhere on }B(0,\rho).
$$
Since $u$ is real, continuity gives $\operatorname{Im}h=0$
on the ball. The Cauchy--Riemann equations imply $h\equiv c$
for some $c\in\R$, so $u=c$ almost everywhere on $B(0,\rho)$.
By continuity of $u(x)=x_1/|x|$ away from the origin,
this equality holds throughout $B(0,\rho)\setminus\{0\}$.
With $e_1=(1,0,\ldots,0)$, we obtain the contradiction
$$
 1=u\bigl((\rho/2)e_1\bigr)
  =c=u\bigl(-(\rho/2)e_1\bigr)=-1.
$$
So $H_u^{(1)}k_0\neq0$.

By Lemma \ref{lem:scaled-kernel} (b) and
\eqref{eq:factorization}, we obtain
\begin{align*}
 \norm{D_{r^2}(f,f)}_{\mathcal S_p(\mathcal F_{r^2}^2)}
 &=\norm{D_1(f_r,f_r)}_{\mathcal S_p(\mathcal F_{1}^2)}\geq\bigl|\langle D_1(f_r,f_r)k_0,k_0\rangle_1\bigr|\\
 &=\norm{H_{f_r}^{(1)}k_0}_{L_1^2}^2
 \rightarrow\norm{H_u^{(1)}k_0}_{L_1^2}^2>0
 \quad(r\rightarrow0^{+}).
 \end{align*}
The inequality follows from \eqref{E128} and
$\norm{k_0}_{L_1^2}=1$, while the limit follows from
\eqref{E95}.
This proves the asserted positive lower limit. The proof is complete.
\end{proof}

\subsection{A continuous concentrated symbol}\label{sec:continuous-concentration}
In Example \ref{ex:concentration}, boundedness of $m_{p,t}(f)$
does not ensure convergence even for the conjugate test.
We now prove Proposition \ref{prop:continuous-concentration},
where the symbol is continuous and convergence holds for the
conjugate test, but fails for another fixed bounded second symbol.

Fix $p>d=2n$.
Choose a real $\chi_0\in C_c^\infty(B(0,1))$ with
$0\leq\chi_0\leq1$ and $\chi_0=1$ on $B(0,1/2)$.
Define its zero extension, as in \eqref{E93}, by
$$
 \chi(x)=
 \begin{cases}
 \chi_0(x),&x\in B(0,1),\\
 0,&x\in\R^d\setminus B(0,1).
 \end{cases}
$$
Since $\supp\chi_0\Subset B(0,1)$, the zero-extension property
stated after \eqref{E93} gives $\chi\in C_c^\infty(\R^d)$.
Let $e_1=(1,0,\ldots,0)\in\R^d$ be the first real
coordinate vector. For each integer $k\geq1$, set
\begin{equation}\label{E28}
 \begin{gathered}
 c_k=2^{-k}e_1,\quad R_k=2^{-k-4},\quad
 r_k=2^{-k^2-10k-20},\quad
 a_k=(r_k/R_k)^{d/p}.
 \end{gathered}
\end{equation}
The relation $(R_k/r_k)^d a_k^p=1$ will compensate for the small
amplitudes when we sum the diagonal lower bounds over the lattice.
For $\ell>k$, the inequality
$|c_k-c_\ell|=2^{-k}-2^{-\ell}\geq2^{-k-1}>R_k+R_\ell$
shows that the balls $B(c_k,R_k)$ are disjoint.
The scale choices also give
\begin{equation}\label{E169}
 \frac{r_k}{R_k}=2^{-k^2-9k-16}\rightarrow0,\quad
 \frac{r_{k+1}}{r_k}=2^{-2k-11}\leq\frac14
 \quad(k\geq1).
\end{equation}

Writing $(c_k)_1=2^{-k}$ for the first real coordinate
of $c_k$, we define $g_k,f_k,f,g$ on $\R^d$ by
\begin{equation}\label{E40}
 \begin{gathered}
 g_k(x)=\chi((x-c_k)/R_k)\cos((x_1-(c_k)_1)/r_k),\\
 f_k(x)=a_kg_k(x),\quad
 f(x)=\sum_{k=1}^\infty f_k(x),\quad
 g(x)=\sum_{k=1}^\infty g_k(x).
 \end{gathered}
\end{equation}
The disjoint supports and $a_k\rightarrow0$ give
\begin{equation}\label{E165}
 \norm{g}_{L^\infty(\R^d)}\leq1\quad\text{and}\quad
 \sup_{x\in\R^d}\left|f(x)-\sum_{k=1}^Nf_k(x)\right|
 \leq\sup_{k>N}a_k\longrightarrow0\quad(N\rightarrow\infty).
\end{equation}
Then $f$ is the uniform limit of smooth functions supported
in $\overline{B(0,1)}$, so $f\in C_c(\R^d)$.

The following uniform bound for every tail
allows us to pass from the finite sums to the conjugate test for $f$.

\begin{lemma}\label{L7}
For the functions in \eqref{E40} and every weight satisfying
\eqref{eq:weight}, there is a constant $C_1>0$, depending only
on $n,m,M,p,\chi$, such that
\begin{equation}\label{eq:packet-tail-bound}
 \sup_{r>0}m_{p,r^2}\bigg(\sum_{k>N}f_k\bigg)\leq C_1
 \quad\text{for every integer }N\geq0.
\end{equation}
For the conjugate test, we have
\begin{equation}\label{E97}
 \norm{D_{r^2}(f,f)}_{\mathcal S_p}\rightarrow0
 \quad(r\rightarrow0^{+}),
\end{equation}
with the same $f$ for every such weight.
\end{lemma}
\begin{proof}
We first bound the testing norm of each summand $f_k$.
We choose
$$C_\chi=\max\{\norm{\chi}_{L^p(\R^d,dv)},
\norm{\nabla\chi}_{L^p(\R^d,dv)}\}>0,$$
which depends only on $d,p,\chi$.
The change of variables $x=c_k+R_ky$ and \eqref{E28} give
\begin{equation}\label{E170}
 \norm{f_k}_{L^p(\R^d,dv)}
 \leq C_\chi a_kR_k^{d/p}=C_\chi r_k^{d/p}.
\end{equation}
Taking $p$th roots in \eqref{eq:positive} with $q=p$ and
$u=|f_k|$ gives the Schatten bound for $T_{|f_k|}^{(r^2)}$
needed in \eqref{E21}. Taking the supremum in \eqref{E5}
and using \eqref{E170}, we obtain
\begin{align}\label{E172}
 m_{p,r^2}(f_k)\leq2\norm{T_{|f_k|}^{(r^2)}}_{\mathcal S_p}\leq c_{p,0}r^{-d/p}\norm{f_k}_{L^p(\R^d,dv)}
 \leq c_{p,0}C_\chi(r_k/r)^{d/p}
\end{align}
for every $r>0$, where $c_{p,0}>0$ is twice the $p$th root
of the constant in \eqref{eq:positive} with $q=p$.

For the gradient bound, put
$\vartheta_k(x)=(x_1-(c_k)_1)/r_k$. The product and chain
rules applied to \eqref{E40} give
$$
 \nabla f_k(x)
 =\frac{a_k}{R_k}(\nabla\chi)((x-c_k)/R_k)\cos\vartheta_k(x)
 -\frac{a_k}{r_k}\chi((x-c_k)/R_k)\sin\vartheta_k(x)e_1.
$$
Using this formula, the same change of variables, and the
triangle inequality, we obtain
\begin{equation}\label{E171}
 \begin{aligned}
 \norm{\nabla f_k}_{L^p(\R^d,dv)}
 &\leq a_kR_k^{d/p}\bigl(R_k^{-1}\norm{\nabla\chi}_{L^p(\R^d,dv)}
                       +r_k^{-1}\norm{\chi}_{L^p(\R^d,dv)}\bigr)\\
 &\leq C_\chi a_kR_k^{d/p}(R_k^{-1}+r_k^{-1})\leq2C_\chi r_k^{d/p-1}.
 \end{aligned}
\end{equation}
Here $a_kR_k^{d/p}=r_k^{d/p}$ follows from \eqref{E28},
and $r_k<R_k$ follows from \eqref{E169}.
Since $f_k$ is smooth and bounded, \eqref{E171} allows us
to use \eqref{eq:gradient-semicommutator}. Taking the supremum
in \eqref{E5}, we obtain
\begin{align}\label{E173}
 m_{p,r^2}(f_k)
 &\leq c_{p,1}r^{1-d/p}\norm{\nabla f_k}_{L^p(\R^d,dv)}\leq2c_{p,1}C_\chi(r/r_k)^{1-d/p}
\end{align}
for every $r>0$, where $c_{p,1}>0$ is the constant in
\eqref{eq:gradient-semicommutator}.
Both $c_{p,0}$ and $c_{p,1}$ depend only on $n,m,M,p$.
Combining \eqref{E172} and \eqref{E173}, with
$C=C_\chi\max\{c_{p,0},2c_{p,1}\}$, we obtain
\begin{equation}\label{eq:packet-bound}
 m_{p,r^2}(f_k)
 \leq C\min\{(r_k/r)^{d/p},(r/r_k)^{1-d/p}\}.
\end{equation}
Note that this constant is independent of $r,k$ and the bounded
second symbol used in $m_{p,r^2}$.

To sum \eqref{eq:packet-bound} uniformly in $r$, put
$\alpha=d/p\in(0,1)$. Fix $r$ and let $j$ be the last index
with $r_j\geq r$, with $j=0$ if $r>r_1$.
The second bound in \eqref{E169} gives
$$
 \begin{cases}
 4^{j-k}r\leq4^{j-k}r_j\leq r_k,&k\leq j,\\[2pt]
 r_k\leq4^{-(k-j-1)}r_{j+1}<4^{-(k-j-1)}r,&k>j.
 \end{cases}
$$
Using these bounds for $(r/r_k)^{1-\alpha}$ when $k\leq j$
and for $(r_k/r)^\alpha$ when $k>j$, we sum the minima in
\eqref{eq:packet-bound} to obtain
\begin{equation}\label{E166}
\begin{aligned}
 \sum_{k=1}^\infty
       \min\{(r_k/r)^\alpha,(r/r_k)^{1-\alpha}\}
 &\leq\sum_{\nu=0}^\infty4^{-\nu(1-\alpha)}
                  +\sum_{\nu=0}^\infty4^{-\nu\alpha}=\frac1{1-4^{-(1-\alpha)}}+\frac1{1-4^{-\alpha}}.
 \end{aligned}
\end{equation}
The range $k\leq j$ is empty when $j=0$. This bound is
independent of $r$ and also applies to every tail of the series.
For each $r>0$ and $h\in L^\infty(\R^d)$,
\eqref{E5}, \eqref{eq:packet-bound}, and \eqref{E166} give
$$
 \sum_{k=1}^\infty\norm{D_{r^2}(f_k,h)}_{\mathcal S_p}
 \leq\norm{h}_{L^\infty(\R^d)}
       \sum_{k=1}^\infty m_{p,r^2}(f_k)<\infty.
$$
Since $p>d\geq2$, completeness
\cite[Section 1.4, p. 26]{Zhu1990} gives an absolutely convergent
series in $\mathcal S_p(\mathcal F_{r^2}^2)$.
To identify its sum as $D_{r^2}(f,h)$, we use \eqref{E1},
\eqref{E2}, and \eqref{E165} to obtain
\begin{align*}
 \bigg\|\sum_{k=1}^N D_{r^2}(f_k,h)-D_{r^2}(f,h)\bigg\|
 &\leq2\norm{h}_{L^\infty(\R^d)}
       \bigg\|\sum_{k=1}^N f_k-f\bigg\|_{L^\infty(\R^d)}\rightarrow0\quad(N\rightarrow\infty).
\end{align*}
By \eqref{E128}, the series also converges in operator norm,
so the preceding limit identifies its sum as $D_{r^2}(f,h)$.
For each integer $N\geq0$, the triangle inequality and
\eqref{E5} give
\begin{align*}
\bigg\|D_{r^2}\bigg(\sum_{k>N}f_k,h\bigg)\bigg\|_{\mathcal S_p}
 &\leq\sum_{k>N}\norm{D_{r^2}(f_k,h)}_{\mathcal S_p}\leq\norm{h}_{L^\infty(\R^d)}
                \sum_{k>N}m_{p,r^2}(f_k)
 \leq C_1\norm{h}_{L^\infty(\R^d)}.
\end{align*}
The last inequality uses \eqref{eq:packet-bound} and
\eqref{E166}, where
$C_1=C\bigl((1-4^{-(1-\alpha)})^{-1}+(1-4^{-\alpha})^{-1}\bigr)>0$
depends only on $n,m,M,p,\chi$ and is independent of
$N,r$. Taking the supremum over
$\norm{h}_{L^\infty(\R^d)}\leq1$ in \eqref{E5} proves
\eqref{eq:packet-tail-bound}.

For the conjugate test, write $f=h_N+f_{>N}$, where
$h_N=\sum_{k\leq N}f_k$ and $f_{>N}=f-h_N$.
Since every $f_k$ is real, so is $f_{>N}$.
The tail bound \eqref{eq:packet-tail-bound}, the conjugate
factorization \eqref{eq:factorization}, and \eqref{E9} yield
\begin{equation}\label{E182}
 \norm{H_{f_{>N}}^{(r^2)}}_{\mathcal S_{2p}(\mathcal F_{r^2}^2,L_{r^2}^2\ominus\mathcal F_{r^2}^2)}^2
 =\norm{D_{r^2}(f_{>N},f_{>N})}_{\mathcal S_p(\mathcal F_{r^2}^2)}
 \leq C_1\norm{f_{>N}}_{L^\infty(\R^d)}.
\end{equation}
As $r\rightarrow0^{+}$ with $N$ fixed, we apply
\eqref{eq:hankel-gradient} and \eqref{E57} with exponent $2p$
to the smooth compactly supported function $h_N$ and obtain
\begin{equation}\label{E183}
 \norm{H_{h_N}^{(r^2)}}_{\mathcal S_{2p}(\mathcal F_{r^2}^2,L_{r^2}^2\ominus\mathcal F_{r^2}^2)}
 \leq C_{H,p}r^{1-d/(2p)}\norm{\nabla h_N}_{L^{2p}(\R^d,dv)}
 \rightarrow0\quad(r\rightarrow0^{+}).
\end{equation}
Here $C_{H,p}>0$ depends only on $n,m,M,p$.
Combining \eqref{E182} and \eqref{E183} with the triangle
inequality and \eqref{E165}, we obtain
\begin{align*}
 \limsup_{r\rightarrow0^{+}}
 \norm{H_f^{(r^2)}}_{\mathcal S_{2p}(\mathcal F_{r^2}^2,L_{r^2}^2\ominus\mathcal F_{r^2}^2)}
 &\leq C_1^{1/2}\norm{f_{>N}}_{L^\infty(\R^d)}^{1/2}\leq C_1^{1/2}\bigl(\sup_{k>N}a_k\bigr)^{1/2} \rightarrow0\quad(N\rightarrow\infty).
 \end{align*}
Since $f$ is real, \eqref{eq:conjugate-identity} gives \eqref{E97}. The proof is complete.
\end{proof}

To show that the fixed function $g$ in \eqref{E40} prevents
convergence, we estimate the diagonal of $D_{r_k^2}(f,g)$
on $B(c_k,R_k/4)$.

\begin{lemma}\label{L8}
Let $f,g$ be the functions in \eqref{E40}. For every weight
satisfying \eqref{eq:weight}, there is a constant $c_0>0$,
depending only on $n,m,M$, such that
\begin{equation}\label{E98}
 \bigl|\langle D_{r_k^2}(f,g)k_{r_k^2,z},k_{r_k^2,z}
                            \rangle_{r_k^2}\bigr|
 \geq(c_0/2)a_k
\end{equation}
for $z\in B(c_k,R_k/4)$ and all sufficiently large $k$.
\end{lemma}
\begin{proof}
Recall the kernel radius $\rho>0$ from Lemma \ref{L1}.
Set $t=r_k^2$ and let $k_{t,z}$ be the normalized reproducing
kernel defined before \eqref{E10}. To separate the contribution
of the $k$-th support, put $u_k=f-a_kg$. The disjoint supports
in \eqref{E40} and $0<a_\ell\leq1$ from \eqref{E28} give
\begin{equation}\label{E188}
 u_k=\sum_{\ell\ne k}(a_\ell-a_k)g_\ell,\quad
 u_k|_{B(c_k,R_k)}=0,\quad
 \norm{u_k}_{L^\infty(\R^d)}
 \leq\sup_{\ell\ne k}|a_\ell-a_k|\leq2.
\end{equation}
Since $f,g,u_k$ are real, \eqref{eq:factorization} yields
\begin{equation}\label{E193}
 \langle D_t(f,g)k_{t,z},k_{t,z}\rangle_t
 =-a_k\norm{H_g^{(t)}k_{t,z}}_{L_t^2}^2
  -\langle H_g^{(t)}k_{t,z},H_{u_k}^{(t)}k_{t,z}\rangle_t.
\end{equation}
By \eqref{E1}, the contractivity of $I-P_t$, and \eqref{E165},
\begin{equation}\label{E194}
 \norm{H_g^{(t)}k_{t,z}}_{L_t^2}
 \leq\norm{gk_{t,z}}_{L_t^2}
 \leq\norm{g}_{L^\infty(\R^d)}\norm{k_{t,z}}_{L_t^2}\leq1.
\end{equation}
In view of \eqref{E193} and \eqref{E194}, it remains to bound
$\norm{H_g^{(t)}k_{t,z}}_{L_t^2}^2$ away from zero and to show
that $\norm{H_{u_k}^{(t)}k_{t,z}}_{L_t^2}=o(a_k)$ uniformly for
$z\in B(c_k,R_k/4)$ as $k\rightarrow\infty$.

For the lower bound, \eqref{E169} gives $\rho r_k<R_k/4$
for all sufficiently large $k$, so
$$
 B(z,\rho r_k)\subset B(c_k,R_k/2)
 \quad(z\in B(c_k,R_k/4)).
$$
On this ball, the cutoff $\chi((w-c_k)/R_k)$ equals one, and all
the other terms in $g$ vanish by the disjoint supports in \eqref{E40}.
With $y=(y_1,\ldots,y_d)\in\R^d$ and
$\vartheta=(\operatorname{Re}z_1-(c_k)_1)/r_k$, we have
\begin{equation}\label{E184}
 g(z+r_ky)=\cos(\vartheta+y_1),\quad |y|<\rho.
\end{equation}
To measure the distance of this cosine from holomorphic
functions, define
\begin{equation}\label{E185}
 \delta(\vartheta)=
 \inf_{h\in\mathcal O(B(0,\rho))\cap L^2(B(0,\rho),dv)}
     \norm{\cos(\vartheta+y_1)-h(y)}_{L^2(B(0,\rho),dv)}.
\end{equation}
The choice $h=0$ gives $\delta(\vartheta)\leq|B(0,\rho)|^{1/2}$.
The triangle inequality, followed by taking the infimum in
\eqref{E185} and interchanging $\vartheta,\vartheta'$, gives
\begin{align*}
 |\delta(\vartheta)-\delta(\vartheta')|
 &\leq\norm{\cos(\vartheta+y_1)-\cos(\vartheta'+y_1)}_{L^2(B(0,\rho),dv)}\leq |B(0,\rho)|^{1/2}|\vartheta-\vartheta'|.
\end{align*}
This estimate proves continuity of $\delta$, while its
$2\pi$-periodicity follows from \eqref{E185}.

To prove positivity, use \cite[Lemma 3.3]{HuVirtanen2023}
with $q=2$ for the locally square-integrable function
$y\mapsto\cos(\vartheta+y_1)$. It provides a holomorphic
minimizer $h_\vartheta$ on $B(0,\rho)$, and the triangle inequality gives
\begin{align*}
 \norm{h_\vartheta}_{L^2(B(0,\rho),dv)}
 &\leq\norm{\cos(\vartheta+y_1)}_{L^2(B(0,\rho),dv)}
       +\norm{\cos(\vartheta+y_1)-h_\vartheta(y)}_{L^2(B(0,\rho),dv)}\\
 &\leq2|B(0,\rho)|^{1/2}<\infty.
\end{align*}
Here the last inequality uses the competitor $h=0$ in the
minimization. Multiplication by the normalization factor
$|B(0,\rho)|^{1/2}$ does not change the minimizer, so
$$
 \delta(\vartheta)
 =\norm{\cos(\vartheta+y_1)-h_\vartheta(y)}_{L^2(B(0,\rho),dv)}.
$$
If $\delta(\vartheta)=0$, continuity makes the almost everywhere
identity $h_\vartheta(y)=\cos(\vartheta+y_1)$ valid throughout
$B(0,\rho)$, contradicting \eqref{E13} because
$$
 \tfrac12(\partial_{y_1}+i\partial_{y_2})\cos(\vartheta+y_1)
 =-\tfrac12\sin(\vartheta+y_1)\not\equiv0
 \quad\text{on }B(0,\rho).
$$
Continuity, periodicity, and positivity of $\delta$ now imply
\begin{equation}\label{E186}
 \delta_0=\min_{[0,2\pi]}\delta>0,
\end{equation}
where $\delta_0$ depends only on $n$ and the fixed radius $\rho$.

To transfer \eqref{E186} to the Hankel operator, fix
$z\in B(c_k,R_k/4)$ and set
\begin{equation}\label{E187}
 h_z(w)=\frac{P_t(gk_{t,z})(w)}{k_{t,z}(w)},
 \quad w\in B(z,\rho r_k).
\end{equation}
The numerator belongs to $\mathcal F_t^2$, and \eqref{E105}
makes the denominator nonzero on this ball, so $h_z$ is
holomorphic there. Writing $c>0$ for the lower bound constant
in \eqref{E105}, we use \eqref{E82} to obtain
\begin{align*}
 \int_{B(z,\rho r_k)}|h_z(w)|^2 dv(w)
 &\leq c^{-1}\int_{B(z,\rho r_k)}
       |P_t(gk_{t,z})(w)|^2e^{-2\varphi(w/r_k)} dv(w)\\
 &\leq c^{-1}r_k^d\norm{P_t(gk_{t,z})}_{L_t^2}^2<\infty.
\end{align*}
The last norm is finite because $gk_{t,z}\in L_t^2$ and
$P_t$ is an orthogonal projection. Rescaling shows that
$y\mapsto h_z(z+r_ky)$ is admissible in \eqref{E185}.

By \eqref{E1} and \eqref{E187},
$$
 (H_g^{(t)}k_{t,z})(w)=(g(w)-h_z(w))k_{t,z}(w),
 \quad w\in B(z,\rho r_k).
$$
Restricting the squared norm to this ball and using
\eqref{E82} and \eqref{E105}, we obtain
\begin{align*}
 \norm{H_g^{(t)}k_{t,z}}_{L_t^2}^2
 &\geq cr_k^{-d}\int_{B(z,\rho r_k)}|g(w)-h_z(w)|^2 dv(w)\\
 &=c\int_{B(0,\rho)}|\cos(\vartheta+y_1)-h_z(z+r_ky)|^2 dv(y)
 \geq c\delta_0^2.
 \end{align*}
The equality follows from $w=z+r_ky$ and \eqref{E184},
and the last inequality follows from \eqref{E185} and \eqref{E186}.
With $c_0=c\delta_0^2>0$ depending only on $n,m,M$,
this estimate gives
\begin{equation}\label{E41}
 \norm{H_g^{(r_k^2)}k_{r_k^2,z}}_{L_{r_k^2}^2}^2\geq c_0
 \qquad (z\in B(c_k,R_k/4))
\end{equation}
for all sufficiently large $k$.

For the remainder in \eqref{E193}, the vanishing property
in \eqref{E188} gives $\supp u_k\subseteq B(c_k,R_k)^c$.
The triangle inequality gives
\begin{equation}\label{E189}
 |w-z|\geq|w-c_k|-|z-c_k|\geq3R_k/4
\end{equation}
for $z\in B(c_k,R_k/4)$ and $w\in\supp u_k$.
For $t=r^2$, the reproducing identity \eqref{E10}, the
definition of $k_{t,z}$, and the first two estimates in
\eqref{eq:gram} give
\begin{align*}
 |k_{t,z}(w)|^2r^{-d}e^{-2\varphi(w/r)}
 &=|\langle k_{t,z},k_{t,w}\rangle_t|^2
                         K_t(w,w)r^{-d}e^{-2\varphi(w/r)}\\
 &\leq C_3r^{-d}e^{-C_5|w-z|/r},
 \end{align*}
where $C_3,C_5>0$ depend only on $n,m,M$.
This bounds the density of $|k_{t,z}|^2d\mu_t$ by \eqref{E82}.

Using this density bound with \eqref{E188} and \eqref{E189}, followed by
\eqref{E60} with exponential coefficient $C_5/2$,
$\delta=3R_k/4$, and $r=r_k$, we obtain
\begin{align*}
 \norm{u_kk_{r_k^2,z}}_{L_{r_k^2}^2}^2
 &\leq4C_3r_k^{-d}
   \int_{\{w\in\C^n:|w-z|\geq3R_k/4\}}
                 e^{-C_5|w-z|/r_k} dv(w)\\
 &\leq4C_3C_6e^{-3C_5R_k/(8r_k)},
 \end{align*}
where $C_6=\int_{\R^d}e^{-(C_5/2)|y|} dv(y)<\infty$.
Taking square roots and using the contractivity of $I-P_t$
in \eqref{E1} gives
\begin{equation}\label{E42}
 \norm{H_{u_k}^{(r_k^2)}k_{r_k^2,z}}_{L_{r_k^2}^2}
 \leq\norm{u_kk_{r_k^2,z}}_{L_{r_k^2}^2}
 \leq C_2e^{-C_4R_k/r_k}
 \quad(z\in B(c_k,R_k/4)),
\end{equation}
with $C_2=2(C_3C_6)^{1/2}$ and $C_4=3C_5/16$, depending
only on $n,m,M$.

To compare this remainder with the main term, \eqref{E28}
and \eqref{E169} give
$$
 a_k^{-1}e^{-C_4R_k/r_k}
 =(R_k/r_k)^{d/p}e^{-C_4R_k/r_k}
 \longrightarrow0\quad(k\rightarrow\infty).
$$
The preceding limit gives $C_2e^{-C_4R_k/r_k}\leq(c_0/2)a_k$
for all sufficiently large $k$. Applying the reverse triangle inequality and
Cauchy--Schwarz to \eqref{E193}, and then using \eqref{E194},
\eqref{E41}, and \eqref{E42}, we obtain
\begin{align*}
 \bigl|\langle D_{r_k^2}(f,g)k_{r_k^2,z},k_{r_k^2,z}
                           \rangle_{r_k^2}\bigr|
 &\geq a_k\norm{H_g^{(r_k^2)}k_{r_k^2,z}}_{L_{r_k^2}^2}^2
       -\norm{H_{u_k}^{(r_k^2)}k_{r_k^2,z}}_{L_{r_k^2}^2}
          \norm{H_g^{(r_k^2)}k_{r_k^2,z}}_{L_{r_k^2}^2}\\
 &\geq c_0a_k-C_2e^{-C_4R_k/r_k}\geq(c_0/2)a_k
 \end{align*}
for all such $k$ and $z\in B(c_k,R_k/4)$. This proves \eqref{E98}. The proof is complete.
\end{proof}

We combine Lemmas \ref{L7} and \ref{L8} with the lattice
estimate \eqref{eq:lattice-diagonal} to prove
Proposition \ref{prop:continuous-concentration}.

\begin{proof}[Proof of Proposition \ref{prop:continuous-concentration}]
The construction in \eqref{E40} and \eqref{E165} gives
real-valued $f\in C_c(\R^d)$ and $g\in L^\infty(\R^d)$
with $\norm{g}_{L^\infty(\R^d)}\leq1$.
The case $N=0$ of \eqref{eq:packet-tail-bound} gives
$\sup_{r>0}m_{p,r^2}(f)\leq C_1$, and \eqref{E97} gives
$\norm{D_{r^2}(f,f)}_{\mathcal S_p}\to0$ as $r\to0^+$.
These prove the first two assertions after setting $t=r^2$.
By \eqref{E5}, the same bound also gives
\begin{equation}\label{E195}
 \norm{D_{r^2}(f,g)}_{\mathcal S_p(\mathcal F_{r^2}^2)}
 \leq m_{p,r^2}(f)\leq C_1\quad(r>0).
\end{equation}

To obtain a positive lower bound independent of $k$, we sum
\eqref{E98} over lattice centers in $B(c_k,R_k/4)$.
Fix $L>2\rho$ as in Lemma \ref{lem:lattice}, and define
$$
 z_j=Lr_kj,\quad Q_j=z_j+[-Lr_k/2,Lr_k/2)^d,\quad j\in\Z^d.
$$
We select the cubes indexed by
$$
 J_k=\{j\in\Z^d:Q_j\cap B(c_k,R_k/8)\neq\varnothing\}.
$$
For $j\in J_k$, choose $y\in Q_j\cap B(c_k,R_k/8)$.
The definition of $Q_j$ gives
$$
 |z_j-c_k|\leq|z_j-y|+|y-c_k|
 \leq R_k/8+\tfrac{\sqrt d}{2}Lr_k.
$$
Since the lattice spacing $Lr_k$ is positive for fixed $k$,
only finitely many $z_j$ satisfy this bound, so $J_k$ is finite.
By \eqref{E169}, $r_k/R_k\to0$, and the same estimate gives
\begin{equation}\label{E196}
 z_j\in B(c_k,R_k/4)\quad(j\in J_k)
\end{equation}
for all sufficiently large $k$.

To count these centers, use the partition
$\R^d=\bigsqcup_{j\in\Z^d}Q_j$, where $\bigsqcup$ denotes
a disjoint union. Each $x\in\R^d$ belongs to exactly one $Q_j$.
If $x\in B(c_k,R_k/8)$, this cube meets the ball, so
$j\in J_k$ by definition. Hence
$$
 B(c_k,R_k/8)\subseteq\bigcup_{j\in J_k}Q_j.
$$
Recall that $\#J_k$ denotes the number of elements of $J_k$,
whereas $|E|$ denotes the Lebesgue measure of a measurable set
$E\subseteq\R^d$.
Since $|Q_j|=(Lr_k)^d$, monotonicity and additivity of
Lebesgue measure give
$$
 |B(c_k,R_k/8)|
 \leq\bigg|\bigcup_{j\in J_k}Q_j\bigg|
 =\sum_{j\in J_k}|Q_j|=\#J_k(Lr_k)^d.
$$
Dividing by $(Lr_k)^d$ and using
$|B(c_k,R_k/8)|=|B(0,1)|(R_k/8)^d$, we obtain
\begin{equation}\label{E190}
 \#J_k\geq b_1(R_k/r_k)^d\quad\text{and}\quad
 b_1=\frac{|B(0,1)|}{(8L)^d}>0.
\end{equation}

Let $c_0$ be the constant in Lemma \ref{L8}, let $C_p>0$
be the constant in \eqref{eq:lattice-diagonal}, and put
$b_2=b_1(c_0/2)^p/C_p>0$.
In view of \eqref{E195}, we may use
\eqref{eq:lattice-diagonal} with $r=r_k$ and lattice translation
$b=0$. Since $J_k\subseteq\Z^d$ and the summands are nonnegative,
we obtain
\begin{align*}
 \sum_{j\in J_k}
 \bigl|\langle D_{r_k^2}(f,g)k_{r_k^2,z_j},k_{r_k^2,z_j}
                    \rangle_{r_k^2}\bigr|^p
 &\leq\sum_{j\in\Z^d}
 \bigl|\langle D_{r_k^2}(f,g)k_{r_k^2,z_j},k_{r_k^2,z_j}
                    \rangle_{r_k^2}\bigr|^p\\
 &\leq C_p\norm{D_{r_k^2}(f,g)}_{\mathcal S_p(\mathcal F_{r_k^2}^2)}^p.
\end{align*}
For all sufficiently large $k$, we divide by $C_p$ and use
\eqref{E98} at the centers in \eqref{E196} to obtain
\begin{align*}
 \norm{D_{r_k^2}(f,g)}_{\mathcal S_p(\mathcal F_{r_k^2}^2)}^p
 &\geq C_p^{-1}\sum_{j\in J_k}
 \bigl|\langle D_{r_k^2}(f,g)k_{r_k^2,z_j},k_{r_k^2,z_j}
                    \rangle_{r_k^2}\bigr|^p
\geq C_p^{-1}\#J_k(c_0a_k/2)^p\\&\geq b_2(R_k/r_k)^d a_k^p=b_2>0.
\end{align*}
The third inequality uses \eqref{E190}, and the equality
uses $(R_k/r_k)^d a_k^p=1$ from \eqref{E28}.
The constant $b_2$ depends only on $n,m,M,p,L$ and is
independent of $k$. Since $r_k^2\to0$ by \eqref{E28}, we obtain
$$
 \limsup_{t\rightarrow0^+}
 \norm{D_t(f,g)}_{\mathcal S_p(\mathcal F_t^2)}
 \geq\liminf_{k\rightarrow\infty}
 \norm{D_{r_k^2}(f,g)}_{\mathcal S_p(\mathcal F_{r_k^2}^2)}
 \geq b_2^{1/p}>0.
$$
This proves the last assertion with the fixed test $g$.
The construction in \eqref{E28} and \eqref{E40} is independent
of the weight, so the same pair works for every weight satisfying
\eqref{eq:weight}. The proof is complete.
\end{proof}

\begin{remark}\label{R2}
The function $f$ in \eqref{E40} satisfies both conditions in
\eqref{eq:super-necessary}, whereas
$\{\omega_r(f):0<r<r_0\}\notin\mathcal W$ for every $r_0>0$.
Thus, these conditions are not sufficient, even for continuous
compactly supported symbols.

To exhibit the concentration of $\omega_r(f)$, we first
estimate the mean oscillation of the local cosine.
In the following integral, $y$ and $\xi$ are vectors in $\R^d$,
and $y_1,\xi_1$ denote their first real coordinates. For $\vartheta\in\R$, define
\begin{equation}\label{E191}
 q(\vartheta)=\frac1{|B(0,1)|}\int_{B(0,1)}
   \left|\cos(\vartheta+y_1)
       -\frac1{|B(0,1)|}\int_{B(0,1)}
                    \cos(\vartheta+\xi_1) dv(\xi)\right| dv(y).
\end{equation}
To bound $q$ from below uniformly in $\vartheta$, we first
apply the triangle inequality to its definition and obtain
\begin{align*}
 |q(\vartheta)-q(\vartheta')|
 &\leq\frac{2}{|B(0,1)|}\int_{B(0,1)}
       |\cos(\vartheta+y_1)-\cos(\vartheta'+y_1)| dv(y)\leq2|\vartheta-\vartheta'|.
\end{align*}
Then $q$ is continuous and $2\pi$-periodic.
If $q(\vartheta)=0$, the continuous nonnegative integrand
in its definition vanishes, giving
$$
 \cos(\vartheta+y_1)
 =\frac1{|B(0,1)|}\int_{B(0,1)}\cos(\vartheta+\xi_1) dv(\xi),
 \quad y\in B(0,1).
$$
Differentiation in $y_1$ would give
$\sin(\vartheta+y_1)=0$ for every $|y_1|<1$, a contradiction.
Hence $b_3=\min_{[0,2\pi]}q>0$, with $b_3$ depending only on $d$.

For sufficiently large $k$, \eqref{E169} gives $r_k<R_k/4$,
so $B(z,r_k)\subset B(c_k,R_k/2)$ whenever $z\in B(c_k,R_k/4)$.
The disjoint supports in \eqref{E40} give $f=f_k$ on this
ball, where the cutoff is one. With
$\vartheta=(\operatorname{Re}z_1-(c_k)_1)/r_k$, we have
$$
 f(z+r_ky)=a_k\cos(\vartheta+y_1),\quad y\in B(0,1).
$$
The substitution $w=z+r_ky$ and
$|B(z,r_k)|=r_k^d|B(0,1)|$ give the ball average
\begin{align*}
 f_{B(z,r_k)}
 &=\frac{r_k^d}{|B(z,r_k)|}
       \int_{B(0,1)}f(z+r_ky) dv(y)=\frac{a_k}{|B(0,1)|}
       \int_{B(0,1)}\cos(\vartheta+y_1) dv(y).
\end{align*}
Substituting this average into \eqref{E6} and using the same
change of variables, we obtain by \eqref{E191}
$$
 \operatorname{MO}_{1,r_k}(f)(z)=a_kq(\vartheta)\geq a_kb_3
 \quad(z\in B(c_k,R_k/4)).
$$

By \eqref{eq:energy-definitions}, this lower bound yields
\begin{equation}\label{E192}
 \begin{aligned}
 \int_{B(c_k,R_k/4)}\omega_{r_k}(f) dv
 &\geq r_k^{-d}(a_kb_3)^p|B(c_k,R_k/4)|=b_4r_k^{-d}R_k^da_k^p=b_4>0.
 \end{aligned}
\end{equation}
Here $b_4=4^{-d}|B(0,1)|b_3^p>0$ depends only on $d,p$ and is
independent of $k$. The equalities use
$|B(c_k,R_k/4)|=4^{-d}|B(0,1)|R_k^d$ and
$r_k^{-d}R_k^da_k^p=1$ from \eqref{E28}, which also gives
$R_k\to0$. The same volume formula therefore yields
$$
 |B(c_k,R_k/4)|\rightarrow0\quad(k\rightarrow\infty).
$$

For any $r_0,\varepsilon>0$, choose $k$ so large
that $r_k<r_0$ and $|B(c_k,R_k/4)|\leq\varepsilon$.
Equation \eqref{E192} gives
$$
 \sup_{0<r<r_0}\sup_{|E|\leq\varepsilon}
       \int_E\omega_r(f) dv
 \geq\int_{B(c_k,R_k/4)}\omega_{r_k}(f) dv
 \geq b_4,
$$
where $E$ ranges over measurable subsets of $\R^d$.
Since this holds for every $\varepsilon>0$, the uniform absolute
continuity condition in \eqref{eq:intrinsic-three} fails for
$\{\omega_r(f):0<r<r_0\}$ with every $r_0>0$.
These density families are therefore not uniformly integrable and do not belong to $\mathcal W$.

To verify the two conditions in \eqref{eq:super-necessary},
let $\rho>0$ be the kernel radius from Lemma \ref{L1} and
let $C_1$ be the constant in Lemma \ref{L7}.
Let $C_7,C_8>0$ denote the constants in \eqref{E36} and
the lower bound of \eqref{eq:two-sided}, respectively.
Since $f$ is real, the lower bound in \eqref{eq:two-sided}
with scale $4r/\rho$ and \eqref{eq:packet-tail-bound} with $N=0$
show that $G_{1,4r}(f)\in L^p(\R^d,dv)$ for every $r>0$.
Taking $p$-th powers in \eqref{E36} and using the lower bound
in \eqref{eq:two-sided} and \eqref{eq:packet-tail-bound}, we obtain
\begin{align*}
 \int_{\R^d}\omega_r(f) dv
 &\leq C_7^p r^{-d}\norm{G_{1,4r}(f)}_{L^p(\R^d,dv)}^p\leq C_7^p C_8^{-p}(4/\rho)^d m_{p,(4r/\rho)^2}^p(f)\leq C_7^p C_8^{-p}(4/\rho)^d C_1^p
\end{align*}
for $r>0$.
To obtain the first condition in \eqref{eq:super-necessary},
we use \eqref{eq:energy-definitions}, the pointwise bound
\eqref{E152}, and the preceding mass estimate to write
\begin{align*}
 r^{-d/p}\norm{G_{1,r}(f)}_{L^p(\R^d,dv)}
 =\left(\int_{\R^d}e_r(f) dv\right)^{1/p}
 \leq\left(\int_{\R^d}\omega_r(f) dv\right)^{1/p}
 \leq C_7C_8^{-1}(4/\rho)^{d/p}C_1
\end{align*}
for $r>0.$

For the second condition, the definition in
Proposition \ref{prop:conjugate} and $f=\overline{f}$ give
$\delta_{p,t}(f)=\norm{D_t(f,f)}_{\mathcal S_p}$.
Using this identity in \eqref{E31} at scale $r/\rho$ and the
conjugate test limit \eqref{E97}, we obtain
$$
 r^{-d/(2p)}\norm{G_{2,r}(f)}_{L^{2p}(\R^d,dv)}
 \leq C_{p,0}\rho^{-d/(2p)}
       \norm{D_{(r/\rho)^2}(f,f)}_{\mathcal S_p}^{1/2}
 \longrightarrow0\quad(r\rightarrow0^{+}).
$$
Thus, both conditions in \eqref{eq:super-necessary} hold,
whereas the weak compactness required by
Theorem \ref{thm:weakcriterion} fails.
\end{remark}

\section{Consequences for Toeplitz quantization}\label{sec:quantization}
The classification of $\mathcal A_p$ and its algebra properties
in Corollary \ref{cor:algebra-weight} give the finite product
and commutator limits in Corollary \ref{P1}.
We prove these limits and express them in an algebraic quotient
of operator families in Proposition \ref{P2}.
The quantitative estimate in Theorem \ref{thm:holder} also gives
the singular value bounds and finite rank approximations in
Proposition \ref{P3}.
Throughout this section, the weight satisfies \eqref{eq:weight},
and all symbols are independent of $t$. For families indexed by $t$,
limits without an indicated variable are taken as $t\rightarrow0^{+}$.

\subsection{Finite products and commutators}
Recall from \eqref{eq:property} that $f\in\mathcal A_p$ means
$\norm{D_t(f,g)}_{\mathcal S_p}\to0$ for every fixed bounded $g$.
We use closure under multiplication to iterate this limit for
finite products, and closure under conjugation to control the
reverse order of two Toeplitz factors.

\begin{proof}[Proof of Corollary \ref{P1}]
If $p<d$, Theorem \ref{thm:main} (a) gives
$\mathcal A_p=\C$. By \eqref{E1}, $T_c^{(t)}=cI$ for constant
$c$, so both differences vanish identically in this case.
We may therefore assume $p\geq d\geq2$ and use the
triangle inequality and the ideal inequality \eqref{E7}
in $\mathcal S_p$.

We first prove \textup{(a)}. For $1\leq j\leq k$, set
$$
 h_j=\prod_{i=1}^j f_i\quad\text{and}\quad
 A_j(t)=T_{f_1}^{(t)}\cdots T_{f_j}^{(t)}-T_{h_j}^{(t)}.
$$
Corollary \ref{cor:algebra-weight} gives
$h_j\in\mathcal A_p$ for $j<k$. Inserting
$T_{h_{j-1}}^{(t)}T_{f_j}^{(t)}$ into the difference
defining $A_j(t)$ and using \eqref{E2}, we obtain
\begin{equation}\label{E208}
 \begin{aligned}
 A_j(t)
 &=\bigl(T_{f_1}^{(t)}\cdots T_{f_{j-1}}^{(t)}
                -T_{h_{j-1}}^{(t)}\bigr)T_{f_j}^{(t)}
       +D_t(h_{j-1},f_j)\\
 &=A_{j-1}(t)T_{f_j}^{(t)}+D_t(h_{j-1},f_j),\quad 2\leq j\leq k.
 \end{aligned}
\end{equation}
Since the $f_j$ are fixed and bounded, \eqref{eq:property} gives
$$
 \max_{2\leq j\leq k}
 \norm{D_t(h_{j-1},f_j)}_{\mathcal S_p}
 \longrightarrow0\quad(t\rightarrow0^{+}).
$$

Choose $t_0>0$ such that
$D_t(h_{j-1},f_j)\in\mathcal S_p$ for $2\leq j\leq k$
and $0<t<t_0$. Starting with $A_1(t)=0$, \eqref{E208}
and \eqref{E7} give $A_j(t)\in\mathcal S_p$ for
$1\leq j\leq k$ on this interval.
Taking Schatten norms in \eqref{E208}, we use the triangle
inequality, \eqref{E7}, and the bound
$\norm{T_{f_j}^{(t)}}\leq\norm{f_j}_{L^\infty(\C^n)}$
from \eqref{E1} to obtain
$$
 \norm{A_j(t)}_{\mathcal S_p}
 \leq\norm{A_{j-1}(t)}_{\mathcal S_p}\norm{f_j}_{L^\infty(\C^n)}
       +\norm{D_t(h_{j-1},f_j)}_{\mathcal S_p},\qquad 2\leq j\leq k.
$$
Iterating this inequality from $A_1(t)=0$ gives
$$
 \norm{A_k(t)}_{\mathcal S_p}
 \leq\sum_{j=2}^k
       \norm{D_t(h_{j-1},f_j)}_{\mathcal S_p}
       \prod_{i=j+1}^k\norm{f_i}_{L^\infty(\C^n)}
 \longrightarrow0,
$$
where an empty product equals one. The sum has finitely many
terms, and its other factors are independent of $t$.
This proves \textup{(a)}.

For \textup{(b)}, the orthogonality of $P_t$ in \eqref{E1}
implies
$$
 \langle T_f^{(t)}u,v\rangle_t
 =\int_{\C^n}fu\overline{v}d\mu_t
 =\langle u,T_{\overline{f}}^{(t)}v\rangle_t
$$
for $u,v\in\cF$.
This calculation applies to every bounded symbol $f$ and proves
\begin{equation}\label{E103}
 (T_f^{(t)})^*=T_{\overline{f}}^{(t)}\quad\text{on }\cF.
\end{equation}
Using \eqref{E2} and \eqref{E103}, we obtain
\begin{align}\label{E197}
 [T_f^{(t)},T_g^{(t)}]
 &=D_t(f,g)-D_t(g,f)\quad\text{and}\quad
 D_t(g,f)^*
 =T_{\overline{f}}^{(t)}T_{\overline{g}}^{(t)}
       -T_{\overline{fg}}^{(t)}
 =D_t(\overline{f},\overline{g}).
\end{align}

Since $f,\overline{f}\in\mathcal A_p$ by
Corollary \ref{cor:algebra-weight}, \eqref{eq:property}
applied to the fixed tests $g,\overline{g}$ puts
$D_t(f,g),D_t(\overline{f},\overline{g})$ in $\mathcal S_p$
for all sufficiently small $t$, with both norms tending to zero.
In the canonical decomposition
\cite[Section 1.3, p. 16--17]{Zhu1990}, interchanging the
two orthonormal systems gives the adjoint. For a compact
operator $S$ on a complex Hilbert space $X$, this gives
\begin{equation}\label{E153}
 s_j(S^*)=s_j(S)\quad\text{and}\quad
 \norm{S^*}_{\mathcal S_p(X)}=\norm{S}_{\mathcal S_p(X)}
\end{equation}
for $j\geq1$ and $0<p<\infty$.
For $S=D_t(\overline{f},\overline{g})\in\mathcal S_p$,
the adjoint identity in \eqref{E197} gives $S^*=D_t(g,f)$.
Equation \eqref{E153} now yields
$D_t(g,f)\in\mathcal S_p$ and
$$
 \norm{D_t(g,f)}_{\mathcal S_p}
 =\norm{D_t(\overline{f},\overline{g})}_{\mathcal S_p}.
$$
The commutator identity in \eqref{E197}, this norm equality,
and the triangle inequality yield
$$
 \norm{[T_f^{(t)},T_g^{(t)}]}_{\mathcal S_p}
 \leq\norm{D_t(f,g)}_{\mathcal S_p}
     +\norm{D_t(\overline{f},\overline{g})}_{\mathcal S_p}
 \longrightarrow0.
$$
This proves \textup{(b)}. The proof is complete.
\end{proof}

\subsection{An algebra of operator families}
The preceding limits can be expressed by identifying two
operator families when their difference tends to zero in
$\mathcal S_p$. To specify the algebra in which this
identification takes place, let $\mathcal B(X)$ denote the
bounded operators on a Hilbert space $X$, and define
$$
 \mathfrak B=
 \left\{A=(A_t)_{0<t\leq1}:
 A_t\in\mathcal B(\cF),\quad
 \sup_{0<t\leq1}\norm{A_t}<\infty\right\}.
$$
Addition, scalar multiplication, multiplication, and adjoints
are defined at each parameter. In particular, the product and adjoint
families satisfy $(AB)_t=A_tB_t$ and $(A^*)_t=A_t^*$, so each composition
acts on a single space $\cF$. The identity family is
$\mathbf I=(I_{\cF})_{0<t\leq1}$.
For $0<p<\infty$, set
\begin{equation}\label{E99}
 \mathfrak J_p=
 \left\{A\in\mathfrak B:
 \lim_{t\rightarrow0^{+}}\norm{A_t}_{\mathcal S_p(\cF)}=0\right\}.
\end{equation}

For $A\in\mathfrak J_p$, the definition of $\mathfrak B$ and
\eqref{E99}, with the extended quantity from
Section \ref{sec:schatten}, give
$$
 \norm{A_t}\leq C_A\quad(0<t\leq1),\qquad
 A_t\in\mathcal S_p(\cF)\quad(0<t<t_0)
$$
for some $C_A>0$ and $0<t_0\leq1$.
The convergence in \eqref{E99} concerns only $t\to0^+$ and
imposes no Schatten condition for $t_0\leq t\leq1$.

\begin{proposition}\label{P2}
For every $0<p<\infty$, $\mathfrak J_p$ is a two-sided
ideal of $\mathfrak B$ closed under adjoints. With
$[A]=A+\mathfrak J_p$, the map
\begin{equation}\label{E100}
 Q_p:\mathcal A_p\longrightarrow\mathfrak B/\mathfrak J_p,
 \quad Q_p(f)=[(T_f^{(t)})_{0<t\leq1}]
\end{equation}
is a unital $*$-algebra homomorphism. The quotient is
understood algebraically.
\end{proposition}
\begin{proof}
We first prove that $\mathfrak J_p$ is closed under addition for every $p>0$.
For compact operators $S,T$ on a Hilbert space $X$, the
approximation number formula \eqref{E102} gives
\begin{equation}\label{E209}
 s_{j+k-1}(S+T)\leq s_j(S)+s_k(T),\quad j,k\geq1.
\end{equation}
Indeed, if $F,G$ have ranks less than $j,k$, respectively,
then $\rank(F+G)<j+k-1$ and
$$
 s_{j+k-1}(S+T)
 \leq\norm{S+T-F-G}\leq\norm{S-F}+\norm{T-G}.
$$
Taking the infimum over $F,G$ proves \eqref{E209}.
Taking $k=j$ in \eqref{E209} bounds $s_{2j-1}(S+T)$ by
$s_j(S)+s_j(T)$. Monotonicity of singular values and
$(a+b)^p\leq\max\{1,2^{p-1}\}(a^p+b^p)$ for $a,b\geq0$
then give
\begin{equation}\label{E210}
 \begin{aligned}
 \norm{S+T}_{\mathcal S_p}^p
 &=\sum_{j=1}^\infty
       \bigl(s_{2j-1}^p(S+T)+s_{2j}^p(S+T)\bigr)\leq2\sum_{j=1}^\infty s_{2j-1}^p(S+T)\\
 &\leq2\sum_{j=1}^\infty(s_j(S)+s_j(T))^p
\leq2\max\{1,2^{p-1}\}
       \bigl(\norm{S}_{\mathcal S_p}^p+\norm{T}_{\mathcal S_p}^p\bigr).
 \end{aligned}
\end{equation}
For $A,G\in\mathfrak J_p$, the family $A+G$ belongs to
$\mathfrak B$ because
$$
 \sup_{0<t\leq1}\norm{A_t+G_t}
 \leq\sup_{0<t\leq1}\norm{A_t}
       +\sup_{0<t\leq1}\norm{G_t}<\infty.
$$
For all sufficiently small $t$, both $A_t$ and $G_t$ belong to
$\mathcal S_p(\cF)$, so \eqref{E210} on $X=\cF$, with $S=A_t$
and $T=G_t$, yields
$$
 \norm{A_t+G_t}_{\mathcal S_p}^p
 \leq2\max\{1,2^{p-1}\}
       \bigl(\norm{A_t}_{\mathcal S_p}^p
              +\norm{G_t}_{\mathcal S_p}^p\bigr)
 \longrightarrow0.
$$
By \eqref{E99}, $A+G\in\mathfrak J_p$. This together with
homogeneity proves linearity for every $p>0$.

To prove stability under left and right multiplication,
take $S\in\mathcal S_p(X)$, $B\in\mathcal B(X)$, and a finite rank operator $F$.
We have $\rank(BF),\rank(FB)\leq\rank F$ and
$$
 \norm{BS-BF}\leq\norm{B}\norm{S-F},\quad
 \norm{SB-FB}\leq\norm{B}\norm{S-F}.
$$
Taking the infimum over $\rank F<j$ in \eqref{E102}, we obtain
$$
 \max\{s_j(BS),s_j(SB)\}\leq\norm{B}s_j(S),\quad j\geq1.
$$
Raising these bounds to the power $p>0$ and summing over $j$ gives
$$
 \max\bigg\{\sum_{j=1}^\infty s_j^p(BS),
             \sum_{j=1}^\infty s_j^p(SB)\bigg\}
 \leq\norm{B}^p\sum_{j=1}^\infty s_j^p(S).
$$
Taking the $p$th root, we obtain
\begin{equation}\label{E211}
 \max\{\norm{BS}_{\mathcal S_p},\norm{SB}_{\mathcal S_p}\}
 \leq\norm{B}\norm{S}_{\mathcal S_p}.
\end{equation}
For $A\in\mathfrak J_p$ and $B\in\mathfrak B$, we apply
\eqref{E211} on $X=\cF$ to $S=A_t$ and the bounded operator
$B_t$ for all sufficiently small $t$. Using the uniform operator
norm bound for the family $B$, we obtain
$$
 \max\{\norm{B_tA_t}_{\mathcal S_p},
        \norm{A_tB_t}_{\mathcal S_p}\}
 \leq\left(\sup_{0<s\leq1}\norm{B_s}\right)
       \norm{A_t}_{\mathcal S_p}\longrightarrow0.
$$
Since $A,B\in\mathfrak B$, we also have
$$
 \sup_{0<t\leq1}\max\{\norm{B_tA_t},\norm{A_tB_t}\}
 \leq\left(\sup_{0<t\leq1}\norm{B_t}\right)
       \left(\sup_{0<t\leq1}\norm{A_t}\right)<\infty,
$$
so \eqref{E99} gives $BA,AB\in\mathfrak J_p$.

For all sufficiently small $t$, $A_t$ is compact, so
\eqref{E153} applies to $S=A_t$. Since
$\norm{A_t^*}=\norm{A_t}$ for every bounded operator $A_t$,
we obtain
$$
 \norm{A_t^*}_{\mathcal S_p}=\norm{A_t}_{\mathcal S_p}
 \longrightarrow0\quad\text{and}\quad
 \sup_{0<t\leq1}\norm{A_t^*}
 =\sup_{0<t\leq1}\norm{A_t}<\infty.
$$
By \eqref{E99}, $A^*\in\mathfrak J_p$. To see that multiplication and adjoints do not depend on
the representatives, take $A,A',B,B'\in\mathfrak B$ with
$A-A',B-B'\in\mathfrak J_p$. The properties just proved give
$$
 AB-A'B'=(A-A')B+A'(B-B')\in\mathfrak J_p\quad\text{and}\quad
 A^*-(A')^*=(A-A')^*\in\mathfrak J_p.
$$
These identities show that both operations are well defined on the
algebraic quotient.

It remains to verify the assertions about $Q_p$.
For $f\in\mathcal A_p\subseteq L^\infty(\C^n)$,
\eqref{E1} gives
$$
 \sup_{0<t\leq1}\norm{T_f^{(t)}}
 \leq\norm{f}_{L^\infty(\C^n)}<\infty,
$$
so \eqref{E100} is well defined.
Linearity follows from \eqref{E1}, and
$T_1^{(t)}=I_{\cF}$ gives $Q_p(1)=[\mathbf I]$.
For $f,h\in\mathcal A_p$, Corollary \ref{cor:algebra-weight}
ensures that $fh,\overline{f}\in\mathcal A_p$.
Equations \eqref{E1}, \eqref{E2}, and \eqref{eq:property} give
$$
 \sup_{0<t\leq1}\norm{D_t(f,h)}
 \leq2\norm{f}_{L^\infty}\norm{h}_{L^\infty}<\infty\quad\text{and}\quad
 \norm{D_t(f,h)}_{\mathcal S_p}\longrightarrow0
 \quad(t\rightarrow0^+).
$$
By \eqref{E99}, these bounds imply
$$
 (T_f^{(t)}T_h^{(t)}-T_{fh}^{(t)})_{0<t\leq1}
 =(D_t(f,h))_{0<t\leq1}\in\mathfrak J_p.
$$
Taking its class in the quotient and
using the Toeplitz adjoint identity \eqref{E103}, we conclude that
$$
 Q_p(f)Q_p(h)=Q_p(fh)\quad\text{and}\quad
 Q_p(\overline{f})=Q_p(f)^*.
$$
This proves the stated homomorphism properties. The proof is complete.
\end{proof}

The product relation for $Q_p$ concerns each fixed pair
$f,h\in\mathcal A_p$. At $p=d$, Corollary \ref{cor:critical-bound}
shows that every nonconstant $f\in\mathcal A_d$ satisfies
$\liminf_{t\to0^{+}}m_{d,t}(f)>0$, although
$\norm{D_t(f,g)}_{\mathcal S_d}\to0$ as $t\to0^{+}$ for every
fixed bounded $g$.

\subsection{Singular values and finite rank approximation}
Theorem \ref{thm:holder} supplies a rate uniform over
$\norm{g}_{L^\infty(\C^n)}\leq1$ when $\alpha>d/p$.
We now convert that Schatten bound into estimates for the singular
values and the operator norm of the remainder after a finite rank
approximation.

\begin{proposition}\label{P3}
Let $p>d$, $d/p<\alpha\leq1$, and
$f\in C_c^{0,\alpha}(\R^d)$. Put
$\beta=\alpha/2-n/p>0$.
There is $C_f>0$ such that
\begin{equation}\label{E101}
 s_N(D_t(f,g))
 \leq C_f\norm{g}_{L^\infty(\C^n)}N^{-1/p}t^\beta
\end{equation}
for every $t>0$, $g\in L^\infty(\C^n)$, and integer $N\geq1$.
For these $t,g,N$, there is an operator
$A_{t,N}:\cF\to\cF$ with $\rank A_{t,N}\leq N$ such that
$$
 \norm{D_t(f,g)-A_{t,N}}
 \leq C_f\norm{g}_{L^\infty(\C^n)}(N+1)^{-1/p}t^\beta.
$$
The constant has the dependence stated in
Theorem \ref{thm:holder} and is independent of $t,g,N$.
\end{proposition}
\begin{proof}
For every $g\in L^\infty(\C^n)$ and $t>0$,
Theorem \ref{thm:holder} and linearity in $g$ give
\begin{equation}\label{E212}
 \norm{D_t(f,g)}_{\mathcal S_p(\cF)}
 \leq C_f\norm{g}_{L^\infty(\C^n)}t^\beta<\infty.
\end{equation}
In particular, $D_t(f,g)$ is compact.
The definition in Section \ref{sec:schatten} and monotonicity
of its singular values give
\begin{equation}\label{E104}
 N s_N^p(D_t(f,g))
 \leq\sum_{j=1}^N s_j^p(D_t(f,g))
 \leq\norm{D_t(f,g)}_{\mathcal S_p(\cF)}^p.
\end{equation}
Taking $p$th roots in \eqref{E104} and using \eqref{E212},
we obtain
$$
 s_N(D_t(f,g))
 \leq N^{-1/p}\norm{D_t(f,g)}_{\mathcal S_p(\cF)}
 \leq C_f\norm{g}_{L^\infty(\C^n)}N^{-1/p}t^\beta.
$$
This proves \eqref{E101}.

Fix $t,g$, and write $D=D_t(f,g)$.
Its canonical decomposition
\cite[Section 1.3, p. 16--17]{Zhu1990}, applied on $\cF$,
provides orthonormal systems $(u_j),(v_j)$ such that
$$
 Dh=\sum_j s_j(D)\langle h,u_j\rangle_t v_j,
 \qquad h\in\cF.
$$
The series converges in $L_t^2$ for every $h\in\cF$ and
runs over the positive singular values, with a finite sum
when $D$ has finite rank. Define the truncation by
$$
 A_{t,N}h=\sum_{j=1}^N s_j(D)\langle h,u_j\rangle_t v_j,
 \qquad h\in\cF,
$$
omitting terms beyond the rank when necessary.
Its range lies in the span of at most $N$ vectors $v_j$,
so $\rank A_{t,N}\leq N$. Taking limits in the Pythagorean
identity for the finite partial sums and applying Bessel's
inequality to the orthonormal system $(u_j)$, we obtain
$$
 \norm{(D-A_{t,N})h}_{L_t^2}^2
 =\sum_{j>N}s_j^2(D)|\langle h,u_j\rangle_t|^2
 \leq s_{N+1}^2(D)\norm{h}_{L_t^2}^2.
$$
If $s_{N+1}(D)>0$, equality is attained at $h=u_{N+1}$.
If $s_{N+1}(D)=0$, the remainder vanishes.
Thus, $\norm{D-A_{t,N}}=s_{N+1}(D)$, and
\eqref{E101} with index $N+1$ proves
the asserted finite rank bound.
\end{proof}

If $p>d$ and $f\in C_c^{0,d/p}(\R^d)$, Theorem
\ref{thm:holder} gives $f\in\mathcal A_p$. Then, for every fixed bounded $g$, the operator $D_t(f,g)$ belongs
to $\mathcal S_p$ for all sufficiently small $t$, with its Schatten
norm tending to zero by \eqref{eq:property}.
Estimate \eqref{E104} requires only this Schatten membership,
so it also applies here. For each integer $N\geq1$, taking
$p$th roots in \eqref{E104} gives
$$
 s_N(D_t(f,g))
 \leq N^{-1/p}\norm{D_t(f,g)}_{\mathcal S_p}
 \longrightarrow0\quad(t\rightarrow0^{+}).
$$
This conclusion uses convergence for the fixed symbol $g$, whereas
\eqref{E101} gives a positive power of $t$
uniformly over $\norm{g}_{L^\infty}\leq1$ under the strict
inequality $\alpha>d/p$.

\section{Applications to other quantization models}\label{sec:other-models}
We first transfer the Fock classification to Gaussian anti-Wick
operators by an explicitly normalized Bargmann transform.
For compact K\"ahler manifolds, a Bergman kernel estimate and
$L^2$ approximation yield a H\"older sufficient condition, which
we compare with an exact calculation on $\mathbb{CP}^1$.

\subsection{Gaussian anti-Wick quantization}\label{sec:antiwick}
In this subsection, take $\varphi(z)=|z|^2/2$ in \eqref{E82}.
For $\hbar>0$, the normalized Gaussian window in
\cite[(12)]{KellerLuef2021}, with center $(0,0)$ and
$\varepsilon=\hbar$, is
$$
 \phi_\hbar(x)=(\pi\hbar)^{-n/4}e^{-|x|^2/(2\hbar)},
 \quad x\in\R^n.
$$
We define the short-time Fourier transform with this window,
as in \cite[Section 2.1, (6)]{KellerLuef2021}, by
$$
 (\mathcal V_\hbar u)(q,\xi)
 =(2\pi\hbar)^{-n/2}
   \int_{\R^n}u(x)\phi_\hbar(x-q)e^{-ix\cdot\xi/\hbar}dx
$$
for $u\in L^2(\R^n,dx)$ and $(q,\xi)\in\R^{2n}.$
Plancherel's identity in the frequency variable, with
$\xi=\hbar\eta$, followed by Tonelli's theorem and
$\norm{\phi_\hbar}_{L^2(\R^n,dx)}=1$, gives
\begin{align*}
 \norm{\mathcal V_\hbar u}_{L^2(\R^{2n},dq d\xi)}^2
 &=\int_{\R^n}\int_{\R^n}
   |u(x)|^2|\phi_\hbar(x-q)|^2dx dq\\
 &=\norm{u}_{L^2(\R^n,dx)}^2
   \norm{\phi_\hbar}_{L^2(\R^n,dx)}^2
 =\norm{u}_{L^2(\R^n,dx)}^2.
\end{align*}
This proves that
$\mathcal V_\hbar:L^2(\R^n,dx)\to L^2(\R^{2n},dq d\xi)$
is an isometry.
For $a\in L^\infty(\R^{2n})$, we follow the notation of
\cite[Section 2.2, (15)]{KellerLuef2021} and display the
dependence on $\hbar$ by writing
\begin{equation}\label{E226}
 \operatorname{op}_{\mathrm{aw},\hbar}(a)
 =\mathcal V_\hbar^*M_a\mathcal V_\hbar
 \quad\text{on }L^2(\R^n,dx),
\end{equation}
with the multiplier notation of Remark \ref{R3}.
The isometry gives
$\norm{\operatorname{op}_{\mathrm{aw},\hbar}(a)}\leq\norm{a}_{L^\infty(\R^{2n})}$.

To compare this operator with \eqref{E1}, let $\mathcal B_\hbar$
be the Bargmann transform
$$
 (\mathcal B_\hbar u)(z)
 =(2\pi\hbar)^{-n/2}(\pi\hbar)^{-n/4}
 \int_{\R^n}u(x)
 e^{(x\cdot z-z\cdot z/4-|x|^2/2)/\hbar}dx,
 \quad z\in\C^n,
$$
where $z\cdot w=\sum_{j=1}^n z_jw_j$ denotes the complex bilinear
pairing on $\C^n$.
By \cite[Section 2.1]{KellerLuef2021}, this transform is unitary
onto the holomorphic subspace of
$L^2(\C^n,e^{-|z|^2/(2\hbar)}dv)$.
The factor $t^{-n}$ in \eqref{E82} is accounted for by setting
\begin{equation}\label{E227}
 \mathcal U_\hbar=(2\hbar)^{n/2}\mathcal B_\hbar:
 L^2(\R^n,dx)\longrightarrow\mathcal F_{2\hbar}^2.
\end{equation}
Indeed, $\norm{\mathcal U_\hbar u}_{L_{2\hbar}^2}
=\norm{\mathcal B_\hbar u}_{L^2(\C^n,e^{-|z|^2/(2\hbar)}dv)}
=\norm{u}_{L^2(\R^n,dx)}$, so $\mathcal U_\hbar$ is unitary. For $f\in L^\infty(\C^n)$, write
$\widetilde f(q,\xi)=f(q-i\xi)$ for $(q,\xi)\in\R^{2n}$.
Then $\widetilde{fg}=\widetilde f\widetilde g$ for
$f,g\in L^\infty(\C^n)$.
 
\begin{corollary}\label{cor:antiwick}
Let $0<p<\infty$ and $f\in L^\infty(\C^n)$. Then
$$
 \lim_{\hbar\to0^+}
 \norm{\operatorname{op}_{\mathrm{aw},\hbar}(\widetilde f)\operatorname{op}_{\mathrm{aw},\hbar}(a)
       -\operatorname{op}_{\mathrm{aw},\hbar}(\widetilde f a)}_{\mathcal S_p(L^2(\R^n,dx))}=0
$$
for every fixed $a\in L^\infty(\R^{2n})$ if and only if
$f\in\mathcal A_p$.
More precisely, for bounded $f,g$ and every $\hbar>0$,
\begin{equation}\label{E214}
 \begin{aligned}
 &\norm{\operatorname{op}_{\mathrm{aw},\hbar}(\widetilde f)\operatorname{op}_{\mathrm{aw},\hbar}(\widetilde g)
       -\operatorname{op}_{\mathrm{aw},\hbar}(\widetilde{fg})}_{\mathcal S_p(L^2(\R^n,dx))}=\norm{D_{2\hbar}(f,g)}_{\mathcal S_p(\mathcal F_{2\hbar}^2)}.
 \end{aligned}
\end{equation}
\end{corollary}

\begin{proof}
We replace the frequency variable $p$ used in
\cite[(7)]{KellerLuef2021} by $-\xi$ and obtain
$$
 (\mathcal V_\hbar u)(q,\xi)
 =e^{-iq\cdot\xi/(2\hbar)}
  e^{-(|q|^2+|\xi|^2)/(4\hbar)}
  (\mathcal B_\hbar u)(q-i\xi).
$$
For $u,v\in L^2(\R^n,dx)$ and $z=q-i\xi\in\C^n$,
cancellation of the phase factors gives
\begin{equation}\label{E228}
 (\mathcal V_\hbar u)(q,\xi)
 \overline{(\mathcal V_\hbar v)(q,\xi)}
 =(\mathcal B_\hbar u)(z)\overline{(\mathcal B_\hbar v)(z)}
   e^{-|z|^2/(2\hbar)}.
\end{equation}
By \eqref{E226} and the adjoint identity, we have
\begin{align*}
 \langle\operatorname{op}_{\mathrm{aw},\hbar}(\widetilde f)u,v\rangle_{L^2(\R^n,dx)}
 &=\langle M_{\widetilde f}\mathcal V_\hbar u,
                \mathcal V_\hbar v\rangle_{L^2(\R^{2n},dq d\xi)}\\
 &=\int_{\R^{2n}}f(q-i\xi)(\mathcal V_\hbar u)(q,\xi)
       \overline{(\mathcal V_\hbar v)(q,\xi)}dq d\xi\\
 &=\int_{\C^n}f(z)(\mathcal B_\hbar u)(z)
       \overline{(\mathcal B_\hbar v)(z)}e^{-|z|^2/(2\hbar)}dv(z)\\
 &=\langle T_f^{(2\hbar)}\mathcal U_\hbar u,
                \mathcal U_\hbar v\rangle_{2\hbar}.
\end{align*}
The third equality uses \eqref{E228} and $dv(z)=dq d\xi$,
since the change of variables $z=q-i\xi$ has absolute Jacobian one.
The last equality uses $P_{2\hbar}\mathcal U_\hbar v=\mathcal U_\hbar v$
and $(2\hbar)^{n/2}(2\hbar)^{n/2}(2\hbar)^{-n}=1$, with the
factors coming from \eqref{E227} and \eqref{E82}.
It follows that
\begin{equation}\label{E229}
 \mathcal U_\hbar\operatorname{op}_{\mathrm{aw},\hbar}(\widetilde f)\mathcal U_\hbar^{-1}
 =T_f^{(2\hbar)}.
\end{equation}

Set
$$
 A_\hbar
 =\operatorname{op}_{\mathrm{aw},\hbar}(\widetilde f)
  \operatorname{op}_{\mathrm{aw},\hbar}(\widetilde g)
  -\operatorname{op}_{\mathrm{aw},\hbar}(\widetilde{fg}).
$$
Equation \eqref{E229}, applied to $f,g,fg$, gives
$$
 \mathcal U_\hbar A_\hbar\mathcal U_\hbar^{-1}
 =T_f^{(2\hbar)}T_g^{(2\hbar)}-T_{fg}^{(2\hbar)}
 =D_{2\hbar}(f,g).
$$
Conjugation by $\mathcal U_\hbar$ is a rank-preserving bijection
between finite-rank operators on $L^2(\R^n,dx)$ and
$\mathcal F_{2\hbar}^2$.
By \eqref{E102} and unitarity, we obtain
\begin{align*}
 s_j(D_{2\hbar}(f,g))
 &=\inf_{\rank R<j}
   \norm{D_{2\hbar}(f,g)-\mathcal U_\hbar R\mathcal U_\hbar^{-1}}\\
 &=\inf_{\rank R<j}
   \norm{\mathcal U_\hbar(A_\hbar-R)\mathcal U_\hbar^{-1}}\\
 &=\inf_{\rank R<j}\norm{A_\hbar-R}
 =s_j(A_\hbar),\quad j\geq1,
\end{align*}
where $R$ ranges over bounded operators on $L^2(\R^n,dx)$.
The definition of the Schatten quantity now gives
$$
 \norm{A_\hbar}_{\mathcal S_p(L^2(\R^n,dx))}^p
 =\sum_{j=1}^\infty s_j(A_\hbar)^p
 =\sum_{j=1}^\infty s_j(D_{2\hbar}(f,g))^p
 =\norm{D_{2\hbar}(f,g)}_{\mathcal S_p(\mathcal F_{2\hbar}^2)}^p,
$$
which proves \eqref{E214}, also when the series diverges.

To transfer the condition on every fixed second symbol, take
$a\in L^\infty(\R^{2n})$ and define
$g(z)=a(\Re z,-\Im z)$ for $z\in\C^n$. Then
$$
 \widetilde g(q,\xi)=g(q-i\xi)=a(q,\xi),
$$
so $a=\widetilde g$. The change of variables $z=q-i\xi\in\C^n$
preserves Lebesgue measure, with $dv(z)=dq d\xi$, and gives
$$
 \norm{g}_{L^\infty(\C^n)}
 =\norm{a}_{L^\infty(\R^{2n})}<\infty.
$$
This correspondence is bijective and independent of $\hbar$.
Since $\widetilde f a=\widetilde{fg}$, \eqref{E214} with
$t=2\hbar$ shows that the limit for every fixed $a$ is equivalent to
$$
 \lim_{t\to0^+}
 \norm{D_t(f,g)}_{\mathcal S_p(\mathcal F_t^2)}=0
 \quad\text{for every fixed }g\in L^\infty(\C^n),
$$
which is precisely $f\in\mathcal A_p$ by \eqref{eq:property}. This completes the proof.
\end{proof}

Consequently, Theorems \ref{thm:main} and \ref{thm:weakcriterion}
classify the first symbols in Corollary \ref{cor:antiwick}.
Taking suprema over $\norm{g}_{L^\infty(\C^n)}\leq1$ in
\eqref{E214} also transfers the uniform criterion
of Theorem \ref{thm:uniform} to anti-Wick operators for $2\leq p<\infty$.

\subsection{Compact K\"ahler manifolds}\label{sec:compact-model}
We prove Proposition \ref{P4} after recalling the geometric
setting and the associated operators.

\subsubsection{Geometric setting}
Let $(X,\omega)$ be a compact connected K\"ahler manifold of complex
dimension $n\geq1$, with complex structure $J$, as in
\cite[Section 2.1]{Schlichenmaier2010}.
Its Riemannian metric and volume are
$$
 g_X(v,w)=\omega(v,Jw),\quad dv_X=\frac{\omega^n}{n!},
$$
as in \cite[p. 2]{MaMarinescu2012}.
Let $d_X$ be the induced distance, write $B_X(x,r)$ for its open
metric balls, and set $\operatorname{vol}_X(E)=\int_E dv_X$
for measurable $E\subseteq X$.
The H\"older seminorm and space used in the introduction are given by
\begin{equation}\label{E230}
 [f]_\alpha=\sup_{x\neq y}
       \frac{|f(x)-f(y)|}{d_X^\alpha(x,y)},\quad
 C^{0,\alpha}(X)=\{f\in C(X):[f]_\alpha<\infty\}
\end{equation}
for $0<\alpha\leq1$.

Write $\operatorname{Tan}_xX$ for the real tangent space of $X$
at $x\in X$.
Let $L\to X$ be a holomorphic line bundle with a smooth Hermitian
metric $h^L$ and Chern connection $\nabla^L$
\cite[Section 2.3]{Schlichenmaier2010}.
We assume that its curvature $R^L=(\nabla^L)^2$ is strictly positive
in the sense of \cite[Section 2.8, p. 6]{Schlichenmaier2010}.
Writing $(R^L)_x$ for the curvature form at $x$, we express
this condition as
$$
 i(R^L)_x(v,Jv)>0,\quad x\in X,\quad 0\neq v\in \operatorname{Tan}_xX.
$$

\subsubsection{Sections and the Bergman projection}
For each integer $k\geq1$, give $L^k=L^{\otimes k}$ the induced
metric $h^{L^k}$.
In a local holomorphic frame $e$ of $L$ on an open set $U\subseteq X$,
a smooth section $u\in C^\infty(X,L^k)$ has the representation
$$
 u|_U=u_e e^{\otimes k},\quad
 |u|_{h^{L^k}}^2=|u_e|^2|e|_{h^L}^{2k},
 \quad u_e\in C^\infty(U),
$$
as in \cite[Sections 2.9 and 3.1]{Schlichenmaier2010}.
The section is holomorphic precisely when all these local coefficients
are holomorphic. Following \cite[(0.2), p. 2]{MaMarinescu2012}, let
$\mathcal L_k=L^2(X,L^k;dv_X)$ be the completion of the smooth
sections in the inner product
$$
 \langle u,v\rangle_{\mathcal L_k}
 =\int_X h^{L^k}(u(x),v(x))dv_X(x),
 \quad u,v\in C^\infty(X,L^k),
$$
with the fiber metric and inner product linear in the first variable.
Since $X$ is compact, \cite[Section 2.2, p. 16]{BaerBallmann2012},
applied to the smooth Hermitian bundle $L^k$ and volume $dv_X$,
identifies $\mathcal L_k$ isometrically with the measurable sections
$u$ satisfying $$\int_X|u|_{h^{L^k}}^2dv_X<\infty,$$ modulo equality
almost everywhere.
The holomorphic sections form the finite-dimensional subspace
$\mathcal H_k=H^0(X,L^k)\subset\mathcal L_k$
\cite[Section 3.1]{Schlichenmaier2010}.
Let $P_k:\mathcal L_k\to\mathcal L_k$ be its orthogonal projection,
called the Bergman projection.
Compactness and curvature positivity will provide the uniform
geometric bounds and projection kernel estimate used in the proof.

\subsubsection{Toeplitz operators and Schatten estimates}
For $a\in L^\infty(X)$ and $u\in\mathcal L_k$, multiplication
is given locally by $(M_au)|_U=(au_e)e^{\otimes k}$ almost everywhere
on $U$, where $u_e$ is measurable.
Integrating $|au|_{h^{L^k}}^2\leq
\norm{a}_{L^\infty(X)}^2|u|_{h^{L^k}}^2$ gives
$$\norm{M_a}_{\mathcal L_k\to\mathcal L_k}\leq\norm{a}_{L^\infty(X)}.$$
This bound extends the Toeplitz construction in
\cite[(0.4), p. 2]{MaMarinescu2012} to bounded measurable symbols.
Restricting to $\mathcal H_k$ and defining the associated Hankel
operator as in \eqref{E1}, we set
$$
 T_{a,k}=P_kM_a:\mathcal H_k\to\mathcal H_k,
 \quad
 H_{a,k}=(I-P_k)M_a:\mathcal H_k\to\mathcal L_k.
$$
Recall that $D_k(f,g)=T_{f,k}T_{g,k}-T_{fg,k}$.
We use the unnormalized Schatten norms from Section \ref{sec:schatten}.

\begin{proof}[Proof of Proposition \ref{P4}]
With $L_x^k$ denoting the fiber of $L^k$ at $x\in X$, write
$P_k(x,y)\in\operatorname{Hom}(L_y^k,L_x^k)$ for the
kernel of $P_k$ with respect to $dv_X(y)$, as in
\cite[(0.3), p. 2]{MaMarinescu2012}, where
$\operatorname{Hom}(L_y^k,L_x^k)$ denotes the space of complex
linear maps from $L_y^k$ to $L_x^k$.
We use the induced fiber norm for $|P_k(x,y)|$.
Compactness makes $(X,g_X)$ complete, with positive injectivity
radius, and gives uniform bounds for the curvature tensors of
$(X,g_X)$ and $(L,h^L)$ and their covariant derivatives of every order.
Since $J$ is parallel for the K\"ahler metric, these verify the
bounded-geometry hypotheses in
\cite[Theorem 1, p. 1328]{MaMarinescu2015}, as also noted in the proof
of \cite[Corollary 3, p. 1340]{MaMarinescu2015}.
On the compact unit tangent bundle, strict positivity gives
$$
 2\varepsilon
 =\min_{\substack{x\in X,\ v\in \operatorname{Tan}_xX\\ |v|_{g_X}=1}}
       i(R^L)_x(v,Jv)>0,
 \quad i(R^L)_x(v,Jv)\geq2\varepsilon|v|_{g_X}^2.
$$
In particular, $i(R^L)_x(v,Jv)>\varepsilon|v|_{g_X}^2$ for $v\neq0$,
which is the curvature hypothesis in
\cite[(1.2), p. 1328]{MaMarinescu2015}.
Taking derivative order zero and the trivial auxiliary line bundle
in \cite[Theorem 5, p. 1338, and (1.3), p. 1328]{MaMarinescu2015},
with tensor power $k$, gives constants $C_0,c_0>0$ and $k_0\geq1$,
depending only on $(X,\omega,J)$ and $(L,h^L)$, such that
\begin{equation}\label{E216}
 |P_k(x,y)|\leq C_0 k^n e^{-c_0\sqrt{k}d_X(x,y)},
 \quad x,y\in X,\quad k\geq k_0.
\end{equation}
Recall that compactness of $X$ gives
$N_k=\dim\mathcal H_k<\infty$ for each fixed $k\geq1$
\cite[Section 3.1, p. 8]{Schlichenmaier2010}.
For an orthonormal basis $(s_{k,j})_{j=1}^{N_k}$ of
$\mathcal H_k$, set
$B_k(x)=\sum_{j=1}^{N_k}|s_{k,j}(x)|^2$.
For $x\in X$, let $I_{L_x^k}$ denote the identity on the fiber $L_x^k$.
The diagonal kernel identity
\cite[(4.19), p. 1344]{MaMarinescu2015}, written in this orthonormal
basis, gives
$$
 \begin{aligned}
 P_k(x,x)v
 &=\sum_{j=1}^{N_k}h^{L^k}(v,s_{k,j}(x))s_{k,j}(x)\\
 &=\sum_{j=1}^{N_k}|s_{k,j}(x)|^2v=B_k(x)v,
 \quad v\in L_x^k,
 \end{aligned}
$$
where the second equality uses $\dim_{\C}L_x^k=1$.
Hence $P_k(x,x)=B_k(x)I_{L_x^k}$ and $|P_k(x,x)|=B_k(x)$.
Taking $y=x$ in \eqref{E216} gives the first bound below.
Orthonormality gives
$$\int_X B_kdv_X=\sum_{j=1}^{N_k}\norm{s_{k,j}}_{\mathcal L_k}^2=N_k,$$
so integration of that bound yields
\begin{equation}\label{E217}
 B_k(x)\leq C_0k^n,\quad
 N_k=\int_X B_k(x)dv_X(x)\leq C_0\operatorname{vol}_X(X)k^n.
\end{equation}

We first estimate $H_{f,k}$ in operator norm.
Fix $r_*>0$, depending only on $g_X$, smaller than the injectivity radius.
In normal coordinates, write $dv_X=\kappa_x(Z)dZ$ as in
\cite[(2.4), p. 11]{MaMarinescu2012}, where
$Z\in\operatorname{Tan}_xX$, $dZ$ is the Euclidean volume induced by
$g_X$, and $\kappa_x$ is the smooth positive volume density.
This density is jointly smooth in $(x,Z)$ on a neighborhood of
the closed disk bundle $|Z|_{g_X}\leq r_*$ over $X$.
Since $X$ is compact, this disk bundle is compact, so we may take
$$
 0<C_{\rm vol}=\max_{\substack{x\in X\\ |Z|_{g_X}\leq r_*}}
                 \kappa_x(Z)<\infty.
$$
This constant depends only on $(X,g_X)$ and the chosen $r_*$.
With $C_{\rm ball}=C_{\rm vol}|B_{\R^{2n}}(0,1)|$, the change of variables gives
$$
 \begin{aligned}
 \operatorname{vol}_X(B_X(x,r))
 &=\int_{\{|Z|_{g_X}<r\}}\kappa_x(Z)dZ\\
 &\leq C_{\rm vol}|B_{\R^{2n}}(0,1)|r^{2n}
 =C_{\rm ball}r^{2n},\quad 0<r\leq r_*.
 \end{aligned}
$$
For $r>r_*$, use $B_X(x,r)\subseteq X$ and $(r/r_*)^{2n}>1$.
With $C_X=\max\{C_{\rm ball},\operatorname{vol}_X(X)r_*^{-2n}\}$,
depending only on $(X,g_X)$ and independent of $x$ and $r$,
the two ranges give
\begin{equation}\label{E231}
 \operatorname{vol}_X(B_X(x,r))
 \leq
 \begin{cases}
 C_{\rm ball}r^{2n},&0<r\leq r_*,\\
 \operatorname{vol}_X(X)r_*^{-2n}r^{2n},&r>r_*,
 \end{cases}
 \leq C_Xr^{2n}.
\end{equation}
For the kernel integral, set
$$
 C_\alpha=C_X\sum_{j=0}^\infty(j+1)^{2n+\alpha}e^{-c_0j}<\infty.
$$
The constant $C_\alpha$ depends only on
$(X,\omega,J)$, $(L,h^L)$, and $\alpha$.
Splitting $X$ into the sets
$j/\sqrt{k}\leq d_X(x,y)<(j+1)/\sqrt{k}$, $j\geq0$, and using
\eqref{E231}, we obtain
\begin{equation}\label{E232}
 \begin{aligned}
 &\qquad k^n\int_X d_X^\alpha(x,y)
           e^{-c_0\sqrt{k}d_X(x,y)}dv_X(y)\\
 &\quad\leq k^n\sum_{j=0}^\infty
       \left(\frac{j+1}{\sqrt{k}}\right)^\alpha e^{-c_0j}
       \operatorname{vol}_X\!\left(B_X\!\left(x,\frac{j+1}{\sqrt{k}}\right)\right)\\
 &\quad\leq C_X k^{-\alpha/2}
       \sum_{j=0}^\infty (j+1)^{2n+\alpha}e^{-c_0j}
       \leq C_\alpha k^{-\alpha/2}.
 \end{aligned}
\end{equation}
The bound in \eqref{E232} is uniform in $x$ and $k$.
Set $C_{\alpha,1}=2C_0C_\alpha$, with the same dependencies as
$C_\alpha$ and independent of $k,f,g$.
By \eqref{E230}, $|f(x)-f(y)|\leq[f]_\alpha d_X^\alpha(x,y)$.
Combining this inequality with \eqref{E216} and \eqref{E232},
also with $x,y$ interchanged, gives
\begin{equation}\label{E233}
 \begin{aligned}
 &\sup_x\int_X |f(x)-f(y)||P_k(x,y)|dv_X(y)\\
 &\qquad+\sup_y\int_X |f(x)-f(y)||P_k(x,y)|dv_X(x)
 \leq C_{\alpha,1}[f]_\alpha k^{-\alpha/2}.
 \end{aligned}
\end{equation}

The commutator $M_fP_k-P_kM_f$ has kernel
$(f(x)-f(y))P_k(x,y)$. For $u\in\mathcal L_k$, its fiber norm satisfies
$$
 |((M_fP_k-P_kM_f)u)(x)|
 \leq\int_X |f(x)-f(y)||P_k(x,y)||u(y)|dv_X(y).
$$
For $[f]_\alpha>0$, \eqref{E233} verifies both integral
conditions of Schur's test \cite[Corollary 3.7, p. 53]{Zhu1990}
on $L^2(X,dv_X)$, with auxiliary function one and both constants
equal to $C_{\alpha,1}[f]_\alpha k^{-\alpha/2}$.
As in the passage from \eqref{E133} to \eqref{E134},
the preceding pointwise bound gives
$$
 \norm{M_fP_k-P_kM_f}_{\mathcal L_k\to\mathcal L_k}
 \leq C_{\alpha,1}[f]_\alpha k^{-\alpha/2}.
$$
If $[f]_\alpha=0$, the kernel vanishes and the same estimate holds.
Since $(M_fP_k-P_kM_f)|_{\mathcal H_k}=H_{f,k}$, we obtain
\begin{equation}\label{E218}
 \norm{H_{f,k}}_{\mathcal H_k\to\mathcal L_k}
 \leq C_{\alpha,1}[f]_\alpha k^{-\alpha/2}.
\end{equation}
Equation \eqref{E218} also holds for $\overline f$ because
$[\overline f]_\alpha=[f]_\alpha$.
From $P_k^2=P_k=P_k^*$, we have
\begin{equation}\label{E219}
 D_k(f,g)=-P_kM_f(I-P_k)M_g|_{\mathcal H_k}
         =-H_{\overline f,k}^*H_{g,k}.
\end{equation}
Since $\rank D_k(f,g)\leq\dim\mathcal H_k=N_k$, \eqref{E102}
gives $s_j(D_k(f,g))\leq\norm{D_k(f,g)}$ for $j\geq1$ and
$s_j(D_k(f,g))=0$ for $j>N_k$, by taking the approximating
operators to be $0$ and $D_k(f,g)$, respectively.
The Schatten norm definition in Section \ref{sec:schatten}
and the factorization \eqref{E219} yield
$$
 \begin{aligned}
 \norm{D_k(f,g)}_{\mathcal S_p(\mathcal H_k)}^p
 &=\sum_{j=1}^{N_k}s_j(D_k(f,g))^p\leq N_k\norm{D_k(f,g)}^p\leq N_k\norm{H_{\overline f,k}}^p\norm{H_{g,k}}^p,
 \end{aligned}
$$
where the last inequality uses
$\norm{H_{\overline f,k}^*}=\norm{H_{\overline f,k}}$.
Taking $p$-th roots and substituting \eqref{E217}, \eqref{E218},
and $\norm{H_{g,k}}\leq\norm{g}_{L^\infty(X)}$, we obtain
$$
 \norm{D_k(f,g)}_{\mathcal S_p(\mathcal H_k)}
 \leq N_k^{1/p}\norm{H_{\overline f,k}}\norm{H_{g,k}}
 \leq C[f]_\alpha\norm{g}_{L^\infty(X)}k^{n/p-\alpha/2}.
$$
Here $C=C_{\alpha,1}(C_0\operatorname{vol}_X(X))^{1/p}$ depends
only on $(X,\omega,J)$, $(L,h^L)$, $\alpha$, and $p$.
This proves \eqref{E215}.

To include the critical exponent, fix $g\in L^\infty(X)$.
For any $h\in C^\infty(X)$, apply the Hilbert--Schmidt basis
formula \cite[Theorem 1.22, p. 18]{Zhu1990} to the extension of
$H_{g-h,k}$ by zero on $\mathcal H_k^\perp\subset\mathcal L_k$,
which preserves its nonzero singular values as in \eqref{E138}.
The contractivity of $I-P_k$, the definition of $B_k$, and
\eqref{E217} then give
\begin{equation}\label{E234}
\begin{aligned}
 \norm{H_{g-h,k}}_{\mathcal S_2(\mathcal H_k,\mathcal L_k)}^2
 &=\sum_{j=1}^{N_k}\norm{H_{g-h,k}s_{k,j}}_{\mathcal L_k}^2\\
 &\leq\sum_{j=1}^{N_k}\norm{(g-h)s_{k,j}}_{\mathcal L_k}^2
   =\int_X|g-h|^2B_kdv_X\\
 &\leq C_0k^n\norm{g-h}_{L^2(X,dv_X)}^2.
\end{aligned}
\end{equation}
Write $(dh)_x:\operatorname{Tan}_xX\to\C$ for the real differential
of $h$ at $x$. Since $h$ is smooth and $X$ is compact,
integration along piecewise smooth curves and the definition of $d_X$ give
$$
 [h]_1\leq\max_{x\in X}
   \sup_{\substack{v\in\operatorname{Tan}_xX\\|v|_{g_X}=1}}
       |(dh)_x(v)|<\infty.
$$
Using $\rank H_{h,k}\leq N_k$, the dimension bound in \eqref{E217},
and \eqref{E218} with $\alpha=1$, we obtain
\begin{equation}\label{E235}
\begin{aligned}
 \norm{H_{h,k}}_{\mathcal S_2(\mathcal H_k,\mathcal L_k)}^2
 &\leq N_k\norm{H_{h,k}}^2\leq C_0\operatorname{vol}_X(X)k^n
       \bigl(C_{1,1}[h]_1k^{-1/2}\bigr)^2=C_hk^{n-1},
\end{aligned}
\end{equation}
where $C_h=C_0\operatorname{vol}_X(X)C_{1,1}^2[h]_1^2$
depends only on $(X,\omega,J)$, $(L,h^L)$, and $h$, and is
independent of $k$ and $g$.
The triangle inequality, \eqref{E234}, and \eqref{E235} give
\begin{align*}
 k^{-n/2}\norm{H_{g,k}}_{\mathcal S_2(\mathcal H_k,\mathcal L_k)}
 &\leq k^{-n/2}\norm{H_{g-h,k}}_{\mathcal S_2(\mathcal H_k,\mathcal L_k)}
       +k^{-n/2}\norm{H_{h,k}}_{\mathcal S_2(\mathcal H_k,\mathcal L_k)}\\
 &\leq C_0^{1/2}\norm{g-h}_{L^2(X,dv_X)}
       +C_h^{1/2}k^{-1/2}.
\end{align*}
Keeping $h$ fixed and letting $k\to\infty$, we obtain
\begin{equation}\label{E225}
 \limsup_{k\to\infty}k^{-n/2}
 \norm{H_{g,k}}_{\mathcal S_2(\mathcal H_k,\mathcal L_k)}
 \leq C_0^{1/2}\norm{g-h}_{L^2(X,dv_X)}.
\end{equation}

On the closed manifold $X$ with its complex orientation,
\cite[Corollary 5.3, p. 249]{Clarke2013}, applied with $g_X$
and exponent two to real and imaginary parts, gives density
of $C^\infty(X)$ in $L^2(X,dv_X)$.
Since $g\in L^\infty(X)\subset L^2(X,dv_X)$, choose a sequence
$(h_\nu)_{\nu\geq1}\subset C^\infty(X)$ with
$$
 \norm{g-h_\nu}_{L^2(X,dv_X)}\longrightarrow0
 \quad(\nu\to\infty).
$$
For this fixed $g$, substituting $h=h_\nu$ in \eqref{E225}
and then letting $\nu\to\infty$ gives
\begin{equation}\label{E220}
 \norm{H_{g,k}}_{\mathcal S_2(\mathcal H_k,\mathcal L_k)}^2
 =o(k^n).
\end{equation}
For $p>2$, we apply \eqref{E124} to the finite rank operator
$H_{g,k}$. Using $\norm{H_{g,k}}\leq\norm{g}_{L^\infty(X)}$
and \eqref{E220}, we obtain
$$
 \norm{H_{g,k}}_{\mathcal S_p(\mathcal H_k,\mathcal L_k)}^p
 \leq\norm{g}_{L^\infty(X)}^{p-2}
       \norm{H_{g,k}}_{\mathcal S_2(\mathcal H_k,\mathcal L_k)}^2
 =o(k^n).
$$
Together with \eqref{E220} for $p=2$, this yields
$\norm{H_{g,k}}_{\mathcal S_p(\mathcal H_k,\mathcal L_k)}
=o(k^{n/p})$ for each fixed $2\leq p<\infty$.
The ideal inequality \eqref{E7}, applied to \eqref{E219}, together
with \eqref{E218} and this estimate for $H_{g,k}$, gives
\begin{align*}
 \norm{D_k(f,g)}_{\mathcal S_p(\mathcal H_k)}
 &\leq\norm{H_{\overline f,k}^*}
       \norm{H_{g,k}}_{\mathcal S_p(\mathcal H_k,\mathcal L_k)}\\
 &\leq C_{\alpha,1}[f]_\alpha k^{-\alpha/2}
       \norm{H_{g,k}}_{\mathcal S_p(\mathcal H_k,\mathcal L_k)}
 =o(k^{n/p-\alpha/2})\quad(k\to\infty).
\end{align*}
Since $p\geq2n/\alpha\geq2$, the last expression tends to zero. The proof is complete.
\end{proof}

\subsection{A trace norm obstruction on the projective line}

The endpoint conclusion of Proposition \ref{P4} concerns each fixed
bounded second symbol. The following example on $\mathbb{CP}^1$
gives this convergence in $\mathcal S_2$ but not in $\mathcal S_1$,
even though the first symbol is smooth.

Let $X=\mathbb{CP}^1=\C\cup\{\infty\}$ with homogeneous
coordinates $[z_0:z_1]$, where nonzero pairs $(z_0,z_1)\in\C^2$
are identified up to multiplication by a nonzero complex scalar.
Write $\mathcal O(-1)$ for the tautological line bundle of
\cite[Section 2.4]{Schlichenmaier2010}, whose fiber at
$[z_0:z_1]$ is the complex line
$\mathcal O(-1)_{[z_0:z_1]}=\C(z_0,z_1)\subset\C^2$,
and set $L=\mathcal O(1)=\mathcal O(-1)^*$, with $*$ denoting
the fiberwise complex linear dual.
On $z_0\ne0$, the affine coordinate $z=z_1/z_0$ identifies
points with $[1:z]$, and the remaining point $[0:1]$ represents
$\infty$ \cite[Section 2.5]{Schlichenmaier2010}.
Take the frame $e(z)$ dual to $(1,z)$ on this chart and the
metric induced by the Euclidean metric on $\C^2$.
The dual metric gives
$|e(z)|^2=|(1,z)|^{-2}=(1+|z|^2)^{-1}$.
We normalize the Fubini--Study volume in
\cite[(2.6)]{Schlichenmaier2010} by
$$
 d\nu(z)=\frac{\omega_{\mathrm{FS}}}{2\pi}
        =\frac{dA(z)}{\pi(1+|z|^2)^2},
$$
where $dA$ is planar Lebesgue measure.
This rescaling multiplies every inner product by the same
positive constant, so the orthogonal projection and the Toeplitz
operators are unchanged.

With $\mathcal O(k)=L^{\otimes k}$, the correspondence in
\cite[Section 5.2]{Schlichenmaier2010} identifies
$\mathcal H_k=H^0(X,\mathcal O(k))$ with the homogeneous
polynomials of degree $k$ on $\C^2$.
Restricting these polynomials to $(1,z)$ gives the polynomials
$u(z)$ of degree at most $k$, with local section
$u e^{\otimes k}$.
For two such polynomials $u,v$, we have
$h^{L^k}(ue^{\otimes k},ve^{\otimes k})
=u\overline v(1+|z|^2)^{-k}$.
Substituting this identity and $d\nu$ into the inner product of
Section \ref{sec:compact-model} gives
\begin{equation}\label{E221}
 \langle u,v\rangle_{\mathcal H_k}
 =\frac1\pi\int_\C
       u(z)\overline{v(z)}(1+|z|^2)^{-k-2}dA(z).
\end{equation}

\begin{example}\label{Eg1}
Define $x:\mathbb{CP}^1\to[0,1]$ by
$x(z)=|z|^2/(1+|z|^2)$ on $\C$ and $x(\infty)=1$.
Set $f=x$, $g=\mathbf1_{\{x>1/2\}}$, and
$D_k=D_k(f,g)=T_{f,k}T_{g,k}-T_{fg,k}$.
Then $f\in C^\infty(\mathbb{CP}^1)$ and
$g\in L^\infty(\mathbb{CP}^1,d\nu)$, with both symbols
independent of $k$. As $k\to\infty$, the corresponding
semi-commutator satisfies
\begin{equation}\label{E222}
 \begin{gathered}
 \norm{D_k}_{\mathcal S_1(\mathcal H_k)}
 =\frac{k+1}{4(k+2)}\longrightarrow\frac14,\\[6pt]
 \begin{aligned}
 \norm{D_k}_{\mathcal S_2(\mathcal H_k)}^2
 &=\frac1{16}\left(\frac{k+1}{k+2}\right)^2
       4^{-k}\binom{2k}{k}\sim\frac1{16\sqrt{\pi k}}.
 \end{aligned}
 \end{gathered}
\end{equation}
\end{example}

\begin{proof}
The expression for $f$ is smooth on $\C$.
In the coordinate $\zeta=1/z$ near infinity, we have
$$
 f(1/\zeta)=\frac{|\zeta|^{-2}}{1+|\zeta|^{-2}}
          =\frac1{1+|\zeta|^2},\quad \zeta\ne0.
$$
This expression extends smoothly to $\zeta=0$ with value
$1=f(\infty)$, so $f\in C^\infty(\mathbb{CP}^1)$.
Moreover, $g$ is the indicator of
$\{z\in\C:|z|>1\}\cup\{\infty\}$, so it is bounded and
discontinuous along $|z|=1$.
To compute the two norms, we diagonalize the radial Toeplitz
operators in a monomial basis.
In polar coordinates $z=re^{i\theta}$, the monomials
$1,z,\ldots,z^k$ are orthogonal for the inner product in \eqref{E221}.
For their norms, the substitution $\xi=r^2/(1+r^2)$ satisfies
$d\xi=2r(1+r^2)^{-2}dr$ and maps $0<r<\infty$ onto $0<\xi<1$.
The beta integral \cite[(5.12.1)]{OlverEtAl2010} then gives
$$
 \begin{aligned}
 \norm{z^j}_{\mathcal H_k}^2
 &=2\int_0^\infty r^{2j+1}(1+r^2)^{-k-2}dr\\
 &=\int_0^1 \xi^j(1-\xi)^{k-j}d\xi
 =\frac{j!(k-j)!}{(k+1)!},\quad 0\leq j\leq k.
 \end{aligned}
$$
Then $e_{j,k}(z)=\sqrt{(k+1)\binom{k}{j}}z^j$, $0\leq j\leq k$,
is an orthonormal basis of $\mathcal H_k$.
For a bounded measurable function $a:[0,1]\to\C$, the
orthogonality of $P_k$ and the radiality of $a\circ x$ give
$\langle T_{a\circ x,k}e_{j,k},e_{\ell,k}\rangle_{\mathcal H_k}=0$
for $j\ne\ell$ by \eqref{E221}.
Since $(e_{j,k})_{j=0}^k$ is an orthonormal basis, we have
\begin{equation}\label{E223}
 T_{a\circ x,k}e_{j,k}=t_{j,k}(a)e_{j,k},\quad 0\leq j\leq k.
\end{equation}
By \eqref{E221} and the normalization of $e_{j,k}$,
the same substitution $\xi=r^2/(1+r^2)$ gives
\begin{equation}\label{E238}
 \begin{aligned}
 t_{j,k}(a)
 &=\langle T_{a\circ x,k}e_{j,k},e_{j,k}\rangle_{\mathcal H_k}\\
 &=2(k+1)\binom{k}{j}\int_0^\infty
      a\!\left(\frac{r^2}{1+r^2}\right)
      r^{2j+1}(1+r^2)^{-k-2}dr\\
 &=(k+1)\binom{k}{j}
      \int_0^1a(\xi)\xi^j(1-\xi)^{k-j}d\xi.
 \end{aligned}
\end{equation}
In particular, \eqref{E238} and the beta integral with
$a(\xi)=\xi$ give
$T_{f,k}e_{j,k}=(j+1)(k+2)^{-1}e_{j,k}$.

By \eqref{E223}, write $D_ke_{j,k}=d_{j,k}e_{j,k}$.
Using \eqref{E238} and the identity
$$
 \frac{d}{d\xi}\bigl(\xi^{j+1}(1-\xi)^{k-j+1}\bigr)
 =\bigl(j+1-(k+2)\xi\bigr)\xi^j(1-\xi)^{k-j},
$$
we obtain
\begin{equation}\label{E224}
 \begin{aligned}
 d_{j,k}
 &=\frac{j+1}{k+2}\langle T_{g,k}e_{j,k},e_{j,k}\rangle_{\mathcal H_k}
   -\langle T_{fg,k}e_{j,k},e_{j,k}\rangle_{\mathcal H_k}\\
 &=(k+1)\binom{k}{j}\int_{1/2}^1
       \left(\frac{j+1}{k+2}-\xi\right)\xi^j(1-\xi)^{k-j}d\xi\\
 &=\frac{k+1}{k+2}\binom{k}{j}
       \bigl[\xi^{j+1}(1-\xi)^{k-j+1}\bigr]_{1/2}^1\\
 &=-\frac{k+1}{k+2}2^{-k-2}\binom{k}{j}.
 \end{aligned}
\end{equation}
Since $D_k$ is diagonal in the orthonormal basis
$(e_{j,k})_{j=0}^k$, formula \eqref{E224} shows that $D_k$
is negative definite. Then its singular values are the numbers
$|d_{j,k}|=-d_{j,k}$, counted with multiplicity.
The definition of the Schatten quantity now gives
\begin{equation}\label{E237}
 \norm{D_k}_{\mathcal S_p(\mathcal H_k)}^p
 =\left(\frac{k+1}{4(k+2)}\right)^p
       \sum_{j=0}^k\left(2^{-k}\binom{k}{j}\right)^p,
 \quad 0<p<\infty.
\end{equation}

The binomial identities
$$
 \sum_{j=0}^k\binom{k}{j}=2^k,
 \quad
 \sum_{j=0}^k\binom{k}{j}^2=\binom{2k}{k}
$$
give the two exact formulas in \eqref{E222} by taking
$p=1,2$ in \eqref{E237}. The second identity follows by
comparing the coefficients of $t^k$ in
$(1+t)^k(1+t)^k=(1+t)^{2k}$.
Stirling's formula \cite[(5.11.3)]{OlverEtAl2010} yields
$$
 4^{-k}\binom{2k}{k}
 =4^{-k}\frac{(2k)!}{(k!)^2}
 \sim\frac1{\sqrt{\pi k}}
 \quad(k\rightarrow\infty),
$$
which completes \eqref{E222}.
\end{proof}

The ratio
$\binom{k}{j+1}/\binom{k}{j}=(k-j)/(j+1)$ for $0\leq j<k$
shows that $\binom{k}{j}$ is maximal at $j=\lfloor k/2\rfloor$,
where $\lfloor k/2\rfloor$ denotes the greatest integer not exceeding $k/2$.
For this pair, \eqref{E224} and Stirling's formula therefore give
$$
 \norm{D_k}
 =\max_{0\leq j\leq k}|d_{j,k}|
 =\frac{k+1}{4(k+2)}2^{-k}\binom{k}{\lfloor k/2\rfloor}
 =O(k^{-1/2})\quad(k\to\infty),
$$
where the implied constant is absolute.
The operator norm bound and \eqref{E222} imply
$$
 \norm{D_k}_{\mathcal S_p(\mathcal H_k)}^p
 =\sum_{j=0}^k|d_{j,k}|^p
 \leq\norm{D_k}^{p-1}\norm{D_k}_{\mathcal S_1(\mathcal H_k)}
 \longrightarrow0\quad(k\rightarrow\infty)
$$
for every fixed $p>1$ in this example.

\end{document}